\documentclass[numbers=enddot,12pt,final,onecolumn,notitlepage]{scrartcl}%
\usepackage[headsepline,footsepline,manualmark]{scrlayer-scrpage}
\usepackage{amssymb}
\usepackage{amsmath}
\usepackage{amsthm}
\usepackage{framed}
\usepackage{comment}
\usepackage{color}
\usepackage[breaklinks=True]{hyperref}
\usepackage[sc]{mathpazo}
\usepackage[T1]{fontenc}
\usepackage{needspace}
\usepackage{tabls}
\usepackage{wasysym}
\usepackage{ytableau}
\usepackage{caption}
\usepackage[labelfont=bf]{subcaption}
\providecommand{\U}[1]{\protect\rule{.1in}{.1in}}
\theoremstyle{definition}
\newtheorem{theo}{Theorem}[section]
\newenvironment{theorem}[1][]
{\begin{theo}[#1]\begin{leftbar}}
{\end{leftbar}\end{theo}}
\newtheorem{lem}[theo]{Lemma}
\newenvironment{lemma}[1][]
{\begin{lem}[#1]\begin{leftbar}}
{\end{leftbar}\end{lem}}
\newtheorem{prop}[theo]{Proposition}
\newenvironment{proposition}[1][]
{\begin{prop}[#1]\begin{leftbar}}
{\end{leftbar}\end{prop}}
\newtheorem{defi}[theo]{Definition}
\newenvironment{definition}[1][]
{\begin{defi}[#1]\begin{leftbar}}
{\end{leftbar}\end{defi}}
\newtheorem{remk}[theo]{Remark}
\newenvironment{remark}[1][]
{\begin{remk}[#1]\begin{leftbar}}
{\end{leftbar}\end{remk}}
\newtheorem{coro}[theo]{Corollary}
\newenvironment{corollary}[1][]
{\begin{coro}[#1]\begin{leftbar}}
{\end{leftbar}\end{coro}}
\newtheorem{conv}[theo]{Convention}
\newenvironment{convention}[1][]
{\begin{conv}[#1]\begin{leftbar}}
{\end{leftbar}\end{conv}}
\newtheorem{warn}[theo]{Warning}

\newtheorem{conj}[theo]{Conjecture}

\newtheorem{exam}[theo]{Example}
\newenvironment{example}[1][]
{\begin{exam}[#1]\begin{leftbar}}
{\end{leftbar}\end{exam}}
\newenvironment{statement}{\begin{quote}}{\end{quote}}

\let\sumnonlimits\sum
\let\prodnonlimits\prod
\renewcommand{\sum}{\sumnonlimits\limits}
\renewcommand{\prod}{\prodnonlimits\limits}
\excludecomment{verlong}
\includecomment{vershort}
\excludecomment{noncompile}
\newcommand{\NN}{\mathbb{N}}
\newcommand{\ZZ}{\mathbb{Z}}
\newcommand{\id}{\operatorname{id}}
\newcommand{\exc}{\operatorname{exc}}

\newcommand{\SYT}{\operatorname{SYT}}
\newcommand{\SSYT}{\operatorname{SSYT}}
\newcommand{\FSSYT}{\operatorname{FSSYT}}
\newcommand{\calE}{\mathcal{E}}
\newcommand{\calF}{\mathcal{F}}
\newcommand{\bb}{\mathbf{b}}
\newcommand{\zz}{\mathbf{z}}
\newcommand{\DD}{\mathbf{D}}
\newcommand{\set}[1]{\left\{ #1 \right\}}
\newcommand{\abs}[1]{\left| #1 \right|}
\newcommand{\tup}[1]{\left( #1 \right)}
\newcommand{\ive}[1]{\left[ #1 \right]}
\newcommand{\email}[1]{\href{mailto:#1}{\texttt{#1}}}
\definecolor{dbluecolor}{rgb}{0.01,0.02,0.7}
\definecolor{dgreencolor}{rgb}{0.2,0.4,0.0}
\definecolor{darkred}{rgb}{0.7,0,0}
\ihead{The Pak--Postnikov and Naruse skew hook length formulas: a new proof}
\ohead{page \thepage}
\begin{document}

\title{The Pak--Postnikov and Naruse skew hook length formulas: a new proof}
\author{Darij Grinberg\footnote{Department of Mathematics, Drexel University, Philadelphia, U.S.A. (\email{darijgrinberg@gmail.com})}, Nazar Korniichuk\footnote{Kyiv Natural-Scientific Lyceum No. 145, Kyiv, Ukraine (\email{n.korniychuk.a@gmail.com})},\\Kostiantyn Molokanov\footnote{Kyiv Natural-Scientific Lyceum No. 145, Kyiv, Ukraine (\email{kostyamolokanov@gmail.com})}, and Severyn Khomych\footnote{Brucknergymnasium Wels, Austria  (\email{severyn.khomych@gmail.com})}}
\subtitle{[Yulia's Dream research paper, 2023]}
\date{version 1.0, 27 October 2023}
\maketitle
\begin{abstract}
\textbf{Abstract.}
The classical hook length formula of enumerative combinatorics expresses the number of standard Young tableaux of a given partition shape as a single fraction. In recent years, two generalizations of this formula have emerged: one by Pak and Postnikov, replacing the number by a (rational) generating function, and one by Naruse, which generalizes the setting from a partition to a skew partition. Both generalizations appear to lie significantly deeper, with no simple proofs known. We combine them into a generating-function identity for skew partitions, and prove it in a fairly elementary way using recursion, determinants and simple combinatorics.
\end{abstract}

\tableofcontents

\section{Introduction}

The hook length formula is one of the most prominent results in
algebraic combinatorics.
Since its discovery by Frame, Robinson and Thrall in 1953, it has
seen several proofs and various applications.
It expresses the number of standard tableaux of shape $\lambda$,
where $\lambda$ is a given integer partition, as the fraction
\begin{align}
\dfrac{n!}{\prod_{c\in Y\left(  \lambda\right)  }h_{\lambda}\left(  c\right)  },
\label{eq.intro.HLF.RHS}
\end{align}
where $n$ is the size (i.e., number of boxes) of $\lambda$,
where $Y\tup{\lambda}$ is the Young diagram of $\lambda$ (as a
set of boxes), and where $h_\lambda\tup{c}$ denotes the
so-called hook length of a given box $c$ in $\lambda$.
Viewed in isolation, this is a surprising enough result, yet it has
also found applications in probability theory \cite{Romik15},
representation theory \cite[\S 9.5.2]{Procesi}, \cite{JamKer09},
and enumerative algebraic geometry \cite[\S 4.6]{Gillespie}.

A natural generalization of a partition $\lambda$ is a skew partition $\lambda/\mu$.
For a long time, no analogue of the hook length formula was known for skew partitions, until Hiroshi Naruse announced one in 2014.
He expressed the number of standard tableaux of shape $\lambda/\mu$ not as a simple fraction (such an expression is easily seen to be impossible), but as a sum
\begin{align}
\sum_{E\in\mathcal{E}\left(  \lambda/\mu\right)  }\dfrac{n!}{\prod_{c\in Y\left(
\lambda\right)  \setminus E} h_{\lambda}\left(  c\right)  }
\label{eq.intro.naruse.RHS}
\end{align}
over all \emph{excitations} $E$ of $Y\tup{\mu}$ that fit inside $Y\tup{\lambda}$ (see Theorem~\ref{thm.naruse} for a precise statement).
The latter excitations (also known as \emph{excited diagrams}) are a (finite) set of diagrams obtained from $Y\tup{\mu}$ through a simple combinatorial transition rule.
When $\mu = \varnothing$, there is only one such excitation, and the sum \eqref{eq.intro.naruse.RHS} boils down to the single fraction \eqref{eq.intro.HLF.RHS}. Thus, Naruse's formula generalizes the classical hook length formula.

Naruse's formula has garnered a reputation of being hard to prove. Naruse never published his proof, which is believed to have used Schubert calculus. Several proofs have since been published by Morales, Pak and Panova in \cite{MPP1} and \cite{MPP2} and by Konvalinka in \cite{Konva1} and \cite{Konva2}, all relying on intricate combinatorics or fairly deep symmetric function theory. Compared with the many elementary proofs of the original hook length formula, these paint a rather desolate picture.

Another generalization of the hook length formula has surfaced in a rather different place. In his 2001 work \cite{Pak01}, Igor Pak has found a new approach to the original hook length formula using discrete geometry -- specifically, using volume-preserving maps on certain polytopes associated to any poset. Alexander Postnikov observed that, as a side effect of his approach, the formula could be refined to an equality between two rational functions (one being defined as a sum over all standard tableaux, and the other being a multivariate generalization of \eqref{eq.intro.HLF.RHS}). This generalization, too, has withstood many attempts at a simple proof; the only proofs known so far are Pak's original geometric proof and Sam Hopkins's proof \cite{Hopkins} using P-partitions.\footnote{The two proofs are more akin than they might seem: The P-partitions can be viewed as the integer points in the polytopes considered by Pak, and Hopkins's toggle-based map $\mathcal{RSK}$ is essentially Pak's geometric bijection $\xi_\lambda$.}

The present work aims to ameliorate the situation by proving both generalizations (the Naruse one and the Pak--Postnikov one) in an elementary way that requires no combinatorial virtuosity from the reader. In fact, we combine them into a single result, which generalizes both: a generating-function hook length formula for skew diagrams. This result has been discovered by the first author and proved by Konvalinka, who outlined his proof in \cite[Section 5]{Konva2}, but nothing more has been published on it so far.

The ingredients of our proof are a bijection between excitations and flagged semistandard tableaux (essentially the same bijection that appears in all existing proofs of the formula); a refined version of the Jacobi--Trudi formula (already known to Gessel and Viennot); a few elementary determinantal identities; and a handful of recursions and elementary properties of partitions. To keep this paper self-contained, we prove all of them, trusting the expert reader to skip anything that is known or clear. We hope that the level of detail makes this paper accessible to the undergraduate or the non-combinatorialist.

This paper is formatted as a sequence of intermediate results (lemmas, propositions and corollaries), each of which is not too hard to prove using what has been shown before.
Some hints to these proofs can be found in \ref{sec.hints}, while detailed proofs are given in Section \ref{sec.proofs}.
We believe the latter are most useful as a backup, as the reader will attain a better understanding of the material by constructing these proofs on his own.
(Thus, our intermediate results serve a similar role as the ``Examples'' in Macdonald's text \cite{Macdon95}.)

\subsection*{Related questions}

Naruse's formula is not the only formula for the number of standard tableaux of a given shape $\lambda / \mu$.
A different formula was found by Okounkov and Olshanski in 1996 \cite[Theorem 8.1]{OkoOls98}, and remarkably can also be reformulated in a similar vein as Naruse's formula,
using \emph{reverse excitations} instead of excitations.
We refer to \cite{MorZhu20} for the details of this formula and various ways to state it.
We are not aware of any Pak--Postnikov-type generalizations of this formula so far, but it appears natural to try extending it in this direction.

Yet another Naruse--like formula, true however only for a narrower class of skew partitions (``slim diagrams''), can be found in \cite[\S 9.1]{Pak21}.

\subsection*{Acknowledgments}

We thank the Yulia's Dream program and, in particular, its organizers Pavel Etingof, Slava Gerovich, Vasily Dolgushev and Dmytro Matvieievskyi.
The first author further appreciates conversations with Alexander Postnikov (who made him aware of the Pak--Postnikov formula), Matja\v{z} Konvalinka (who first confirmed his conjecture that has become Theorem~\ref{thm.main}) and Igor Pak (for several relevant pointers to the literature).

\section{\label{sec.nots}Notations and terminology}

The hook length formula concerns \emph{standard tableaux}. What
follows is a brief introduction to the subject, sufficient for
the purposes of the present work.
More comprehensive treatments of the theory of standard tableaux
can be found in \cite{Fulton97}, \cite[Chapter 7]{Sagan19} and \cite[Chapter 7]{EC2}.

\subsection{Partitions and skew partitions}

We let $\NN := \set{0,1,2,\ldots}$. The size of a set $S$ is denoted by $\abs{S}$.

A \emph{partition} means an infinite sequence $\lambda=\left(  \lambda
_{1},\lambda_{2},\lambda_{3},\ldots\right)  $ of nonnegative integers such
that $\lambda_{1}\geq\lambda_{2}\geq\lambda_{3}\geq\cdots$ and $\lambda_{i}=0$
for all sufficiently large $i$. For instance, $\left(  5,2,2,1,0,0,0,\ldots
\right)  $ is a partition. We usually omit the zeroes when we write down a
partition; e.g., the partition we just mentioned can be rewritten as $\left(
5,2,2,1\right)  $. In particular, the sequence $\left(  0,0,0,\ldots\right)  $
is a partition, denoted by $\varnothing$.

We write $\lambda_{i}$ for the $i$-th entry of any partition $\lambda$. Thus,
$\lambda=\left(  \lambda_{1},\lambda_{2},\lambda_{3},\ldots\right)  $ for any
partition $\lambda$.

Two partitions $\lambda$ and $\mu$ are said to satisfy $\lambda\supseteq\mu$
if $\lambda_{i}\geq\mu_{i}$ for all $i$. We also write $\mu\subseteq\lambda$
for $\lambda\supseteq\mu$. A pair $\left(  \lambda,\mu\right)  $ of two
partitions satisfying $\lambda\supseteq\mu$ is called a \emph{skew partition},
and is denoted by $\lambda/\mu$.

\begin{convention}
    For Sections \ref{sec.nots} and \ref{sec.thms}, we fix two partitions $\lambda$ and $\mu$. (We shall sometimes -- but not always -- require that $\lambda\supseteq\mu$.)
\end{convention}

\subsection{Diagrams and Young diagrams}

A \emph{box} (or \emph{cell}) will mean a pair $\tup{i, j} \in \ZZ^2$ of two integers.
A \emph{diagram} will mean a finite set of boxes, i.e., a finite subset of $\ZZ^2$.

\emph{Young diagrams} are a particularly crucial type of diagrams, defined as follows:
If $\lambda \supseteq \mu$, then the \emph{(skew) Young diagram }$Y\left(  \lambda/\mu\right)  $
means the set of all pairs $\left(  i,j\right)  $ of positive integers satisfying
$\lambda_{i}\geq j>\mu_{i}$. For instance,
\[
Y\tup{\tup{4,2,1,1} / \tup{2,1,1}}
= \set{\tup{1,3},\ \tup{1,4},\ \tup{2,2},\ \tup{4,1}} .
\]
The Young diagram $Y\left(  \lambda/\mu\right)  $ is clearly a diagram.
But there are also diagrams that are not Young diagrams (for instance, $\left\{  \left(  1,1\right)  ,\ \left(
2,3\right)  \right\}  $).

If $\lambda \supseteq \mu$, then the elements $\tup{i,j}$ of the skew Young diagram
$Y\left( \lambda/\mu\right)  $ will be called the \emph{boxes} (or \emph{cells}) of
$Y\left(  \lambda/\mu\right)  $.

Boxes and diagrams will be visually represented in a specific way.
Namely, we visualize each box $\tup{i, j}$ as a square box of
sidelength $1$, placed in the Cartesian plane in such a way that its center has
Cartesian coordinates $\tup{i, j}$.
However, we let the x-axis go north-to-south and the y-axis go west-to-east%
\footnote{We shall use the words ``north'', ``east'', ``south'' and ``west'' as synonyms for ``up'', ``right'', ``down'' and ``left''.}
(so that the eastern neighbor of the box $\tup{i,j}$ is $\tup{i,j+1}$, whereas its southern neighbor is $\tup{i+1,j}$).
Thus, for instance, the skew Young
diagram $Y\left(  \left(  5,4,3,3,1\right)  /\left(  2,1,1\right)  \right)  $
is the following conglomeration of boxes:
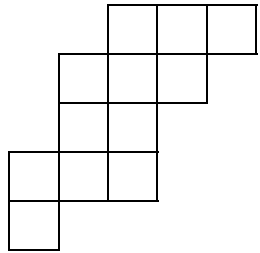
\begin{equation*}
\ydiagram{2+3,1+3,1+2,0+3,0+1}
\end{equation*}
(where, e.g., the three boxes in the topmost row correspond to the pairs
$\left(  1,3\right)  $, $\left(  1,4\right)  $ and $\left(  1,5\right)  $ from
left to right).
This way of representing boxes and diagrams is known as the \emph{English notation} or the \emph{matrix notation} (as it imitates the way that the entries of a matrix are commonly indexed).
It makes each diagram look like a table (although usually not of rectangular shape),
and thus allows us to speak of ``rows'' and ``columns'' of a diagram (e.g., the set of all boxes $\tup{i, j}$ with a given $i$ is called the \emph{$i$-th row}),
and also to fill numbers into the boxes (which we shall do later).

Any partition $\lambda$ satisfies $\lambda \supseteq \varnothing$, and thus the skew partition $\lambda / \varnothing$ is well-defined.
We will write $Y\left(\lambda\right)$ for the skew Young diagram $Y\left(\lambda / \varnothing\right)$.
We observe that it consists of all pairs $\tup{i, j}$ of positive integers satisfying $\lambda_i \geq j$.
Young diagrams of the form $Y\tup{\lambda}$ are called \emph{straight Young diagrams}.
In a straight Young diagram, the rows are ``left-aligned'' (i.e., each row has its westernmost box in the $1$-st column).

We note that any skew partition $\lambda / \mu$ satisfies $Y\left(\lambda / \mu\right) = Y\left(\lambda\right) \setminus Y\left(\mu\right)$. Furthermore, two partitions $\lambda$ and $\mu$ satisfy $\mu \subseteq \lambda$ if and only if $Y\tup{\mu} \subseteq Y\tup{\lambda}$.

\subsection{Hooks and their lengths}

If $c = \tup{i, j}$ is a box of a Young diagram $Y\left(  \lambda\right)  $, then its \emph{hook} $H_{\lambda}\left(  c\right)  $ is
defined to be the set of all boxes of $Y\left(  \lambda\right)  $ that lie
due east or due south of $c$, including $c$ itself.
Formally speaking, $H_{\lambda} \tup{c}$ is defined by
\[
H_{\lambda} \tup{c} :=
\set{ \tup{i, k} \in Y\tup{\lambda} \ \mid\ k \geq j }
\cup
\set{ \tup{k, j} \in Y\tup{\lambda} \ \mid\ k \geq i }.
\]


For example, if $\lambda = \tup{5,4,3,3,1}$, then the hook $H_{\lambda}\left(3,2\right)$ of the
box $\left(  3,2\right)  $ consists of three boxes, which are marked with
asterisks in the picture below:
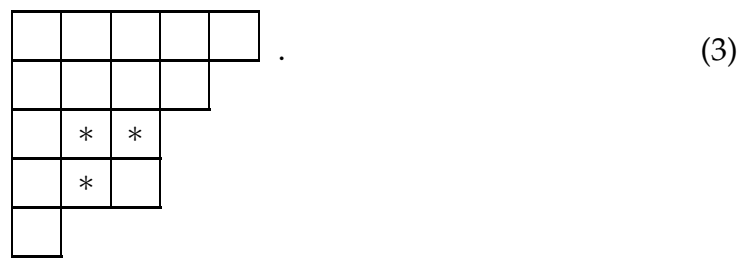
\begin{equation}
\begin{ytableau} \phantom{x} & & & & \\ \phantom{x}  & & & \\ \phantom{x}  &\ast&\ast\\ &\ast&\\ \phantom{x} \end{ytableau}\ \ .
\label{eq.nots.hook.1}
\end{equation}
Formally speaking, this hook is the set
$\set{ \tup{3,2}, \tup{3,3}, \tup{4,2} }$.

The \emph{hook length} $h_{\lambda}\left(  c\right)  $ of a box $c\in
Y\left(  \lambda\right)  $ is defined to be $\abs{ H_{\lambda}\tup{c} }$, that is, the number of all boxes in the hook
of $c$. Thus, $h_{\lambda}\left(  3,2\right)  =3$ in our above example
(\ref{eq.nots.hook.1}). Likewise, in the same example, $h_{\lambda}\left(  2,2\right)  =5$ and $h_{\lambda}\left(  1,3\right)  =6$.

\subsection{Standard tableaux}

If $\lambda \supseteq \mu$, then a \emph{standard tableau} of shape $\lambda/\mu$ is defined to be a way to put
a positive integer into each box of $Y\left(  \lambda/\mu\right)  $ such that

\begin{itemize}
\item the integers are $1,2,\ldots,n$ (where $n$ is the number of boxes of
$Y\left(  \lambda/\mu\right)  $), and each appears exactly once;

\item the integers increase left-to-right in each row;

\item the integers increase top-to-bottom in each column.
\end{itemize}

Formally speaking, this means that a standard tableau of shape $\lambda/\mu$
is a bijection $T : Y\left(  \lambda/\mu\right)  \rightarrow\left\{
1,2,\ldots,n\right\}  $ (where $n = \abs{Y\tup{\lambda/\mu}}$) such that
\[
T\left(  i,j\right)  <T\left(  i,j+1\right)  \ \ \ \ \ \ \ \ \ \ \text{for all
}\left(  i,j\right)  \in Y\left(  \lambda/\mu\right)  \text{ satisfying
}\left(  i,j+1\right)  \in Y\left(  \lambda/\mu\right)
\]
and
\[
T\left(  i,j\right)  <T\left(  i+1,j\right)  \ \ \ \ \ \ \ \ \ \ \text{for all
}\left(  i,j\right)  \in Y\left(  \lambda/\mu\right)  \text{ satisfying
}\left(  i+1,j\right)  \in Y\left(  \lambda/\mu\right) .
\]
We imagine each value $T\left(  i,j\right)  $ of this bijection $T$ to be
written into the box $\left(  i,j\right)  $, so that $T$ is visualized as
a filling of the diagram $Y\left(  \lambda/\mu\right)  $ with the numbers
$1,2,\ldots,n$.

For example, here is a standard tableau of shape $\left(  5,4,3,3,1\right)  /\left(  2,1,1\right)$:
\begin{equation}
\begin{ytableau} \none&\none&1&3&10\\ \none &2&4&11\\ \none&5&6\\ 7&8&12\\ 9 \end{ytableau}\ \ .
\label{eq.nots.skew-tab.3}%
\end{equation}

We also agree that if $\lambda$ and $\mu$ do \textbf{not} satisfy $\lambda \supseteq \mu$, then there are no standard tableaux of shape $\lambda / \mu$.

We let $\operatorname*{SYT}\left(  \lambda/\mu\right)  $ denote the set of all
standard tableaux of shape $\lambda/\mu$. Clearly, this set is finite (and empty unless $\lambda \supseteq \mu$). The
famous \emph{hook length formula} of Frame, Robinson and Thrall gives a beautifully simple expression for its size when $\mu=\varnothing$:

\begin{theorem}[hook length formula]
\label{thm.hlf-classical}
Let $\lambda$ be a partition, and let $\operatorname*{SYT}\left(  \lambda\right)  $ denote $\operatorname*{SYT}\left(  \lambda/\varnothing\right)  $. Then,
\[
\abs{\operatorname*{SYT}\left(  \lambda\right) }
=
\dfrac{n!}{\prod_{c\in Y\left(  \lambda\right)  }h_{\lambda}\left(  c\right)  },
\]
where $n$ is the number of boxes in $Y\left(  \lambda\right)  $.
\end{theorem}

Much about this theorem is surprising, starting with the fact that the right
hand side is not obviously an integer. Various proofs are known; the easiest
are probably the ones in \cite{Glass-Ng} (see \cite{Zhang10} for a simplification) and in \cite[\S 7.3]{Sagan19}.
A zoology of different proofs of Theorem~\ref{thm.hlf-classical} can be found in Pak's recent work \cite[\S 11.2]{Pak22}.

In the following, we agree to abbreviate $\SYT\tup{\lambda / \varnothing}$ as $\SYT\tup{\lambda}$.
Likewise, a ``standard tableau of shape $\lambda$'' will mean a standard tableau of shape $\lambda / \varnothing$.

\subsection{The Pak--Postnikov refined hook length formula}

A recurring motif in algebraic combinatorics is to replace counting by
summing. In particular, one might hope to replace the $\left\vert
\operatorname*{SYT}\left(  \lambda\right)  \right\vert $ on the left hand side
of Theorem \ref{thm.hlf-classical} by a sum of some functions, one for each
standard tableau in $\operatorname*{SYT}\left(  \lambda\right)  $, such that
the sum of these functions simplifies to a sufficiently neat expression, and
such that substituting $1$ for all variables in the resulting equality would
recover the hook length formula. Such a generalization has been found by Igor
Pak and Alexander Postnikov. To state it, we need one more piece of notation:

If $T$ is a standard tableau (of any shape), and if $k$ is a positive integer,
then $c_{T}\left(  k\right)  $ shall denote the difference $j-i$, where
$\left(  i,j\right)  $ is the box of $T$ that contains the entry $k$ (assuming
that such a box exists\footnote{Such box is certainly unique, since $T$ is a standard tableau.}). For example, if $T$ is the tableau given in
(\ref{eq.nots.skew-tab.3}), then $c_{T}\left(  1\right)  =3-1=2$ and
$c_{T}\left(  2\right)  =2-2=0$ and $c_{T}\left(  3\right)  =4-1=3$ and so on.
Now we can state the formula:

\begin{theorem}[Pak--Postnikov hook length formula]
\label{thm.pak-1}
Let $\lambda$ be a
partition such that $Y\left(  \lambda\right)  $ has $n$ boxes. Let
$\ldots,z_{-2},z_{-1},z_{0},z_{1},z_{2},\ldots$ be an infinite family of
commuting indeterminates.

For any standard tableau $T$ of shape $\lambda/\varnothing$, we define the fraction
\[
\zz_{T}:=\dfrac{1}{\prod_{k=1}^{n}\left(  z_{c_{T}\left(  k\right)  }%
+z_{c_{T}\left(  k+1\right)  }+\cdots+z_{c_{T}\left(  n\right)  }\right)  }
\]
(a rational function in our indeterminates).

Furthermore, for any box $c=\left(  i,j\right)  $ in $Y\left(
\lambda\right)  $, we define the \emph{algebraic
hook length} $h_{\lambda}\left(  c;z\right)  $ by%
\[
h_{\lambda}\left(  c;z\right)  :=\sum_{\left(  i,j\right)
\in H_{\lambda}\left(  c\right)  }z_{j-i}.
\]

Then,
\begin{equation}
\sum_{T\in\operatorname*{SYT}\left(  \lambda\right)  }\zz_{T}=\prod_{c\in
Y\left(  \lambda\right)  }\dfrac{1}{h_{\lambda}\left(  c;z\right)  }.
\label{eq.thm.pak-1.claim}%
\end{equation}

\end{theorem}

For example, if $\lambda=\left(  2,2\right)  $, then (\ref{eq.thm.pak-1.claim}%
) says that%
\begin{align*}
&  \dfrac{1}{ \left(  z_{0}+z_{1}+z_{-1}+z_{0}\right)  \left(  z_{1}+z_{-1}%
+z_{0}\right) \left(  z_{-1}+z_{0}\right) z_{0} }\\
&  \ \ \ \ \ \ \ \ \ \ +\dfrac{1}{ \left(  z_{0}+z_{-1}+z_{1}+z_{0}\right) \left(
z_{-1}+z_{1}+z_{0}\right) \left(  z_{1}+z_{0}\right) z_{0}  }\\
&  =\dfrac{1}{\left(  z_{0}+z_{1}+z_{-1}\right)  \left(  z_{-1}+z_{0}\right)
\left(  z_{1}+z_{0}\right)  z_{0}},
\end{align*}
since there are only two standard tableaux of shape $\lambda$ (namely,
$\ytableaushort{12,34}$ and $\ytableaushort{13,24}$), and since the algebraic
hook lengths for $\lambda=\left(  2,2\right)  $ are
\begin{align*}
h_{\lambda}\left(  \left(  1,1\right)  ;z\right)   &  =z_{1-1}+z_{2-1}%
+z_{1-2}=z_{0}+z_{1}+z_{-1},\\
h_{\lambda}\left(  \left(  2,1\right)  ;z\right)   &  =z_{1-2}+z_{2-2}%
=z_{-1}+z_{0},\\
h_{\lambda}\left(  \left(  1,2\right)  ;z\right)   &  =z_{2-1}+z_{2-2}%
=z_{1}+z_{0},\\
h_{\lambda}\left(  \left(  2,2\right)  ;z\right)   &  =z_{2-2}=z_{0}.
\end{align*}

When we set all the variables $z_{i}$ to $1$, each fraction $\zz_{T}$ in Theorem~\ref{thm.pak-1} simplifies
to $\dfrac{1}{n!}$, so that the left hand side of (\ref{eq.thm.pak-1.claim})
becomes $\sum_{T\in\operatorname*{SYT}\left(  \lambda\right)  }\dfrac{1}%
{n!}=\dfrac{\left\vert \operatorname*{SYT}\left(  \lambda\right)  \right\vert
}{n!}$, whereas the right hand side becomes $\prod_{c\in Y\left(
\lambda\right)  }\dfrac{1}{h_{\lambda}\left(  c\right)  }$ (since each
$h_{\lambda}\left(  c;z\right)  $ becomes $\sum_{\left(  i,j\right)  \in
H_{\lambda}\left(  c\right)  }1=\left\vert H_{\lambda}\left(  c\right)
\right\vert = h_{\lambda}\tup{c}$), and thus the equality (\ref{eq.thm.pak-1.claim}) becomes $\dfrac{\left\vert \operatorname*{SYT}\left(  \lambda\right)  \right\vert
}{n!} = \prod_{c\in Y\left(
\lambda\right)  }\dfrac{1}{h_{\lambda}\left(  c\right)  } = \dfrac{1}{\prod_{c \in Y\tup{\lambda}} h_{\lambda}\left(  c\right)  }$. But this is precisely the claim of Theorem
\ref{thm.hlf-classical}. Hence,
we recover the original hook length formula (Theorem
\ref{thm.hlf-classical}) from Theorem~\ref{thm.pak-1}.

While the original hook length formula has seen many proofs, including some
pretty short and elementary ones, its Pak--Postnikov generalization (Theorem
\ref{thm.pak-1}) has so far resisted all simple approaches. The only proof
that, to our knowledge, appears explicitly in the present literature is
Hopkins's algebraic proof \cite[Theorem 18]{Hopkins} using toggles and
P-partitions. The original proof by Pak and Postnikov, using volumes of
polytopes, has never been published, but all the tools for it have been built
in \cite{Pak01}. Finally, Konvalinka's recent paper \cite[Section 5]{Konva2}
outlines an elaborate combinatorial proof of a more general fact (Theorem \ref{thm.main} below) using a
complex RSK-like bijection. None of these proofs is easy or particularly
elementary. One of the purposes of the present paper is to give a new proof
that relies on no advanced tools or vertiginous bijective combinatorics.

\subsection{Excited moves and excitations}

Another recent development in hook length formulas is the discovery of a
\emph{skew hook length formula} by Naruse, later proved by Morales, Pak and
Panova in two ways (\cite[Theorem 1.2]{MPP1}, \cite[Theorem 1.2]{MPP2}). It
gives an expression for $\left\vert \operatorname*{SYT}\left(  \lambda
/\mu\right)  \right\vert $ rather than merely for $\left\vert
\operatorname*{SYT}\left(  \lambda/\varnothing\right)  \right\vert $. The
expression is more intricate, relying on a notion of \emph{excitations} (also known as \emph{excited diagrams}).
Let us start by defining this notion.

For any box $c = \tup{i,j} \in \ZZ^2$, we define
\begin{itemize}
    \item its \emph{southern neighbor} $c_{\downarrow} := \tup{i+1,j}$;
    \item its \emph{eastern neighbor} $c_{\to} := \tup{i,j+1}$;
    \item its \emph{southeastern neighbor} $c_{\searrow} := \tup{i+1,j+1}$.
\end{itemize}
These neighbors are arranged as follows:
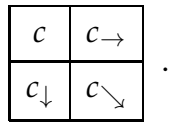
\[%
\begin{tabular}
[c]{|c|c|}\hline
$c$ & $c_{\to}$\\\hline
$c_{\downarrow}$ & $c_{\searrow}$\\\hline
\end{tabular}
\ \ .
\]

If $D$ is a diagram that contains some box $c$ but
contains none of its three neighbors $c_{\downarrow}, c_{\to}, c_{\searrow}$,
then we can replace the box $c$ by its
southeastern neighbor $c_{\searrow}$ in $D$. The resulting diagram
$\left(  D\setminus\left\{  c\right\}  \right)  \cup\left\{ 
c_{\searrow} \right\}  $ is denoted by $\operatorname*{exc}\nolimits_{c}D$,
and we say that $\operatorname*{exc}\nolimits_{c}D$ is obtained from $D$ by an
\emph{excited move}. Thus,%
\[
\operatorname*{exc}\nolimits_{c}D=\left(  D\setminus\left\{  c\right\}
\right)  \cup\left\{  c_{\searrow}\right\}  .
\]

\begin{example}
Let $D$ be the diagram $\left\{  \left(  1,1\right)  ,\ \left(  1,3\right)
,\ \left(  2,1\right)  \right\}  $. Then, we can obtain the diagram
\begin{align*}
\operatorname*{exc}\nolimits_{\left(  1,3\right)  }D  &  =\left(
D\setminus\left\{  \left(  1,3\right)  \right\}  \right)  \cup\left\{  \left(
2,4\right)  \right\} \\
&  =\left\{  \left(  1,1\right)  ,\ \left(  2,4\right)  ,\ \left(  2,1\right)
\right\}
\end{align*}
by an excited move from $D$. However, the diagram $\left(  D\setminus\left\{
\left(  1,1\right)  \right\}  \right)  \cup\left\{  \left(  2,2\right)
\right\}  $ cannot be obtained by an excited move from $D$, since the southern
neighbor $\left(  2,1\right)  $ of $\left(  1,1\right)  $ is in $D$ (so that
$\operatorname*{exc}\nolimits_{\left(  1,1\right)  }D$ is not defined).

Another way to apply an excited move to $D$ results in%
\begin{align*}
\operatorname*{exc}\nolimits_{\left(  2,1\right)  }D  &  =\left(
D\setminus\left\{  \left(  2,1\right)  \right\}  \right)  \cup\left\{  \left(
3,2\right)  \right\} \\
&  =\left\{  \left(  1,1\right)  ,\ \left(  1,3\right)  ,\ \left(  3,2\right)
\right\}  .
\end{align*}

\end{example}

We can visualize an excited move as a diagonal (bishop) move\footnote{``Bishop move'' refers to the conventions of chess, but it should be noticed that bishops in chess can move in all four diagonal directions, whereas excited moves only go in one direction.} by one unit to
the east and one unit to the south simultaneously; it is only allowed if the
two ``intermediate''\ boxes (i.e., the
southern and eastern neighbors of the original box) as well as the target
box are not in the diagram.

If a diagram $E$ is obtained from a diagram $D$ by a sequence of excited
moves, then we say that $E$ is an \emph{excitation} of $D$. (This sequence can
be empty, so that $D$ itself is an excitation of $D$.)

\begin{example}
Let $D$ be the diagram $\left\{  \left(  1,1\right)  ,\ \left(  1,3\right)
,\ \left(  2,1\right)  \right\}  $, and let $E$ be the diagram $\left\{
\left(  1,3\right)  ,\ \left(  2,2\right)  ,\ \left(  4,3\right)  \right\}  $.
Then, $E$ is an excitation of $D$. In fact, it is easy to check that
$
E=\operatorname*{exc}\nolimits_{\left(  1,1\right)  }\left(
\operatorname*{exc}\nolimits_{\left(  3,2\right)  }\left(  \operatorname*{exc}%
\nolimits_{\left(  2,1\right)  }D\right)  \right)
$.

\end{example}

Now, recall that $\lambda$ and $\mu$ are two partitions.

\begin{definition}
\label{def.Elamu}
We define $\mathcal{E}
\left(  \lambda/\mu\right)  $ to be the set of all excitations $E$ of
$Y\left(  \mu\right)  $ that satisfy $E\subseteq Y\left(  \lambda\right)  $.
\end{definition}

\begin{example}
\label{exa.Elamu.1}
Let $\lambda=\left(  4,4,3\right)  $ and $\mu=\left(  3,1\right)  $. Then, the
skew diagram $Y\left(  \lambda/\mu\right)  $ looks as follows:
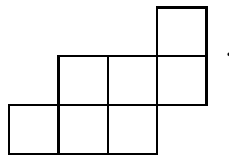
\[
\ydiagram{3+1,1+3,0+3}\ \ .
\]
The set $\mathcal{E}\left(  \lambda/\mu\right)  $ consists of
the following diagrams:
\begin{align*}
Y\left(  \mu\right)   &  =\left\{  \left(  1,1\right)  ,\ \left(  1,2\right)
,\ \left(  1,3\right)  ,\ \left(  2,1\right)  \right\}  ,\\
\operatorname*{exc}\nolimits_{\left(  2,1\right)  }\left(  Y\left(
\mu\right)  \right)   &  =\left\{  \left(  1,1\right)  ,\ \left(  1,2\right)
,\ \left(  1,3\right)  ,\ \left(  3,2\right)  \right\}  ,\\
\operatorname*{exc}\nolimits_{\left(  1,3\right)  }\left(  Y\left(
\mu\right)  \right)   &  =\left\{  \left(  1,1\right)  ,\ \left(  1,2\right)
,\ \left(  2,4\right)  ,\ \left(  2,1\right)  \right\}  ,\\
\operatorname*{exc}\nolimits_{\left(  1,3\right)  }\left(  \operatorname*{exc}%
\nolimits_{\left(  2,1\right)  }\left(  Y\left(  \mu\right)  \right)  \right)
&  =\left\{  \left(  1,1\right)  ,\ \left(  1,2\right)  ,\ \left(  2,4\right)
,\ \left(  3,2\right)  \right\}  ,\\
\operatorname*{exc}\nolimits_{\left(  1,2\right)  }\left(  \operatorname*{exc}%
\nolimits_{\left(  1,3\right)  }\left(  Y\left(  \mu\right)  \right)  \right)
&  =\left\{  \left(  1,1\right)  ,\ \left(  2,3\right)  ,\ \left(  2,4\right)
,\ \left(  2,1\right)  \right\}  ,\\
\operatorname*{exc}\nolimits_{\left(  1,2\right)  }\left(  \operatorname*{exc}%
\nolimits_{\left(  1,3\right)  }\left(  \operatorname*{exc}\nolimits_{\left(
2,1\right)  }\left(  Y\left(  \mu\right)  \right)  \right)  \right)   &
=\left\{  \left(  1,1\right)  ,\ \left(  2,3\right)  ,\ \left(  2,4\right)
,\ \left(  3,2\right)  \right\}  ,\\
\operatorname*{exc}\nolimits_{\left(  1,1\right)  }\left(  \exc_{\left(  1,2\right)  }\left(  \operatorname*{exc}\nolimits_{\left(
1,3\right)  }\left(  \operatorname*{exc}\nolimits_{\left(  2,1\right)
}\left(  Y\left(  \mu\right)  \right)  \right)  \right)  \right)   &
=\left\{  \left(  2,2\right)  ,\ \left(  2,3\right)  ,\ \left(  2,4\right)
,\ \left(  3,2\right)  \right\}  .
\end{align*}
(The order in which the excited moves are made can sometimes be altered, but this does not affect the resulting diagrams.) Here are these diagrams, drawn
inside the shape $Y\left(  \lambda\right)  $ (we mark the boxes in each
diagram with asterisks):
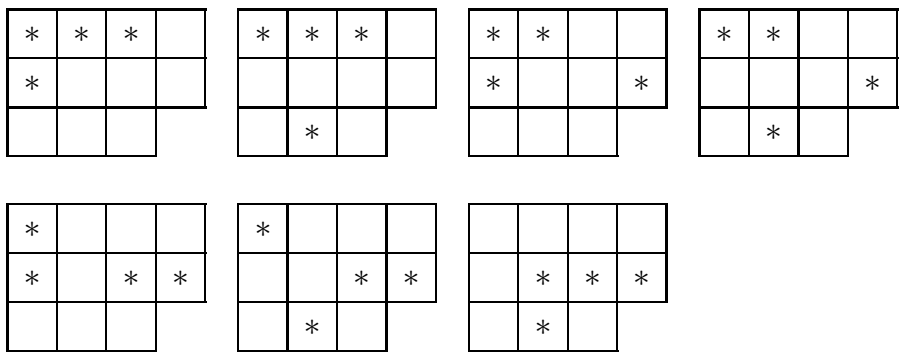
\[%
\begin{tabular}
[c]{cccc}%
$\begin{ytableau} \ast&\ast&\ast& \\ \ast&&& \\ && \end{ytableau}$ &
$\begin{ytableau} \ast&\ast&\ast& \\ &&& \\ &\ast& \end{ytableau}$ &
$\begin{ytableau} \ast&\ast&& \\ \ast&&&\ast \\ && \end{ytableau}$ &
$\begin{ytableau} \ast&\ast&& \\ &&&\ast \\ &\ast& \end{ytableau}$\\
&  &  & \\
$\begin{ytableau} \ast&&& \\ \ast&&\ast&\ast \\ && \end{ytableau}$ &
$\begin{ytableau} \ast&&& \\ &&\ast&\ast \\ &\ast& \end{ytableau}$ &
$\begin{ytableau} \phantom{a}&&& \\ &\ast&\ast&\ast \\ &\ast& \end{ytableau}$
&
\end{tabular}
\
\]
It is easy to see that no other excitations are possible (without outgrowing
the diagram $Y\left(  \lambda\right)  $).
\end{example}

We will soon (in Section \ref{sec.flag}) learn of an equivalent model for the
excitations in $\mathcal{E}\left(  \lambda/\mu\right)  $ which is easier to
handle.

Note that, despite the notation, the set $\mathcal{E}\left(  \lambda
/\mu\right)  $ is not directly related to the skew diagram $Y\left(
\lambda/\mu\right)  $; it could just as well be reasonably called
$\mathcal{E}\left(  \mu,\lambda\right)  $. Nevertheless, we call it
$\mathcal{E}\left(  \lambda/\mu\right)  $ in order to follow the existing literature.

At this point, let us state three easy observations, which will be useful
later. Proofs for these observations (just as for all others in this paper)
are given in Section \ref{sec.proofs} below.

\begin{lemma}
\label{lem.exc.straight}
Let $\lambda$ be any partition. Then, $\mathcal{E}%
\left(  \lambda/\varnothing\right)  =\left\{  \varnothing\right\}  $.
\end{lemma}

\begin{lemma}
\label{lem.exc.empty}
Let $\lambda$ and $\mu$ be two partitions that don't satisfy $\lambda \supseteq \mu$. Then, $\mathcal{E}%
\left(  \lambda/\mu\right)  =\varnothing$.
\end{lemma}

\begin{lemma}
\label{lem.exc.equal}
Let $\lambda$ be any partition. Then, $\mathcal{E}%
\left(  \lambda/\lambda\right)  =\left\{  Y\tup{\lambda}\right\}  $.
\end{lemma}

\subsection{The Naruse skew hook length formula}

In a conference talk in 2014, Hiroshi Naruse stated the following
generalization of the hook length formula to skew diagrams (see \cite{Naruse},
\cite[Theorem 1.2]{MPP1}, \cite[Theorem 1.2]{MPP2}, \cite[(2)]{Konva1}):

\begin{theorem}[Naruse skew hook length formula]
\label{thm.naruse}
For any skew partition $\lambda/\mu$, we have
\[
\left\vert \operatorname*{SYT}\left(  \lambda/\mu\right)  \right\vert
=n!\sum_{E\in\mathcal{E}\left(  \lambda/\mu\right)  }\ \ \prod_{c\in Y\left(
\lambda\right)  \setminus E}\dfrac{1}{h_{\lambda}\left(  c\right)  },
\]
where $n$ is the number of boxes in $Y\left(  \lambda/\mu\right)  $.
\end{theorem}

For $\mu=\varnothing$, this theorem simplifies to the classical hook length
formula (Theorem \ref{thm.hlf-classical}), since the set $\mathcal{E}\left(
\lambda/\varnothing\right)  $ consists of the single diagram $\varnothing
=\left\{  {}\right\}  $. Once again, the generality comes at a cost: All known
proofs of Theorem \ref{thm.naruse} are complex, using either deep
combinatorics or (for Naruse's original proof) algebraic geometry. We hope
that the proof we give in the present paper will be somewhat more accessible.

Note that the expression \eqref{eq.intro.naruse.RHS} is just a rewritten form of the right hand side of Theorem~\ref{thm.naruse}.

We remark that Matjaz Konvalinka has generalized Theorem \ref{thm.naruse}
further to shifted skew partitions of types B and D in \cite{Konva2}. We shall
not follow this direction of generalization here; the applicability of our
method to these settings is an interesting question that deserves further work.

\section{\label{sec.thms}The main theorem}

The main protagonist of this paper is a common generalization of Theorem
\ref{thm.naruse} and Theorem \ref{thm.pak-1}:

\begin{theorem}
\label{thm.main}
Let $\lambda/\mu$ be a skew partition.
Let $\ldots,z_{-2},z_{-1},z_{0},z_{1},z_{2},\ldots$
be an infinite family of commuting indeterminates.

For any standard tableau $T$ of shape $\lambda/\mu$, we define the fraction%
\[
\zz_{T}:=\dfrac{1}{\prod_{k=1}^{n}\left(  z_{c_{T}\left(  k\right)  }%
+z_{c_{T}\left(  k+1\right)  }+\cdots+z_{c_{T}\left(  n\right)  }\right)  }%
\]
(a rational function in our indeterminates),
where $n = \abs{Y\tup{\lambda / \mu}}$ is the number of boxes of $Y\tup{\lambda / \mu}$.

Furthermore, for any box $c=\left(  i,j\right)  $ in $Y\left(  \lambda\right)
$, we define the \emph{algebraic hook length} $h_{\lambda}\left(  c;z\right)  $ by
\begin{equation}
h_{\lambda}\left(  c;z\right)  :=
\sum_{\left(  i,j\right)  \in H_{\lambda} \left(  c\right)  }z_{j-i}.
\label{eq.def.hlcz}
\end{equation}
(Note that this does not depend on $\mu$.)

Then,
\begin{equation}
\sum_{T\in\operatorname*{SYT}\left(  \lambda/\mu\right)  } \zz_{T}
= \sum_{E\in\mathcal{E}\left(  \lambda/\mu\right)  }
\ \ \prod_{c\in Y\left(  \lambda\right)  \setminus E}
\dfrac{1}{h_{\lambda}\left(  c;z\right)  }.
\label{eq.thm.main.claim}
\end{equation}

\end{theorem}

\begin{example}
Let $\lambda=\left(  3,2\right)  $ and $\mu=\left(  1\right)  $. Thus,
$Y\left(  \lambda/\mu\right)  $ looks as follows:%
\begin{equation}
\ydiagram{1+2,0+2}\ \ .
\end{equation}

Then, there are five standard tableaux of shape $\lambda/\mu$, namely%
\begin{align*}
A &
=\begin{ytableau} \none & 1 & 2 \\ 3 & 4 \end{ytableau}\ \ ,\ \ \ \ \ \ \ \ \ \ B=\begin{ytableau} \none & 1 & 3 \\ 2 & 4 \end{ytableau}\ \ ,\ \ \ \ \ \ \ \ \ \ C=\begin{ytableau} \none & 2 & 3 \\ 1 & 4 \end{ytableau}\ \ ,\\
& \\
D &
=\begin{ytableau} \none & 1 & 4 \\ 2 & 3 \end{ytableau}\ \ ,\ \ \ \ \ \ \ \ \ \ E=\begin{ytableau} \none & 2 & 4 \\ 1 & 3 \end{ytableau}\ \ .
\end{align*}
Thus, $\operatorname*{SYT}\left(  \lambda/\mu\right)  =\left\{
A,B,C,D,E\right\}  $. Furthermore,%
\begin{align*}
\zz_{A} &  =\dfrac{1}{\left(  z_{1}+z_{2}+z_{-1}+z_{0}\right)  \left(
z_{2}+z_{-1}+z_{0}\right)  \left(  z_{-1}+z_{0}\right)  z_{0}},\\
\zz_{B} &  =\dfrac{1}{\left(  z_{1}+z_{-1}+z_{2}+z_{0}\right)  \left(
z_{-1}+z_{2}+z_{0}\right)  \left(  z_{2}+z_{0}\right)  z_{0}},\\
\zz_{C} &  =\dfrac{1}{\left(  z_{-1}+z_{1}+z_{2}+z_{0}\right)  \left(
z_{1}+z_{2}+z_{0}\right)  \left(  z_{2}+z_{0}\right)  z_{0}},\\
\zz_{D} &  =\dfrac{1}{\left(  z_{1}+z_{-1}+z_{0}+z_{2}\right)  \left(
z_{-1}+z_{0}+z_{2}\right)  \left(  z_{0}+z_{2}\right)  z_{2}},\\
\zz_{E} &  =\dfrac{1}{\left(  z_{-1}+z_{1}+z_{0}+z_{2}\right)  \left(
z_{1}+z_{0}+z_{2}\right)  \left(  z_{0}+z_{2}\right)  z_{2}}.
\end{align*}
Meanwhile,%
\[
\mathcal{E}\left(  \lambda/\mu\right)  =\left\{  \left\{  \left(  1,1\right)
\right\}  ,\ \left\{  \left(  2,2\right)  \right\}  \right\}  ,
\]
since the only excitations of $Y\left(  \mu\right)  =\left\{  \left(
1,1\right)  \right\}  $ that are subsets of $Y\left(  \lambda\right)  $ are
$Y\left(  \mu\right)  =\left\{  \left(  1,1\right)  \right\}  $ itself and
$\operatorname*{exc}\nolimits_{\left(  1,1\right)  }\left(  Y\left(
\mu\right)  \right)  =\left\{  \left(  2,2\right)  \right\}  $. Furthermore,
the algebraic hook lengths of the boxes
of $\lambda=\left(  3,2\right)  $ are%
\begin{align*}
h_{\lambda}\left(  \left(  1,1\right)  ;z\right)   &  =z_{0}+z_{1}%
+z_{2}+z_{-1},\\
h_{\lambda}\left(  \left(  2,1\right)  ;z\right)   &  =z_{-1}+z_{0},\\
h_{\lambda}\left(  \left(  1,2\right)  ;z\right)   &  =z_{1}+z_{2}+z_{0},\\
h_{\lambda}\left(  \left(  2,2\right)  ;z\right)   &  =z_{0},\\
h_{\lambda}\left(  \left(  1,3\right)  ;z\right)   &  =z_{2}.
\end{align*}
Now, in view of $\operatorname*{SYT}\left(  \lambda/\mu\right)  =\left\{
A,B,C,D,E\right\}  $ and $\mathcal{E}\left(  \lambda/\mu\right)  =\left\{
\left\{  \left(  1,1\right)  \right\}  ,\ \left\{  \left(  2,2\right)
\right\}  \right\}  $, the claim of Theorem \ref{thm.main} says that%
\begin{align*}
&  \zz_{A}+\zz_{B}+\zz_{C}+\zz_{D}+\zz_{E}\\
&  =\prod_{c\in Y\left(  \lambda\right)  \setminus\left\{  \left(  1,1\right)
\right\}  }\dfrac{1}{h_{\lambda}\left(  c;z\right)  }+\prod_{c\in Y\left(
\lambda\right)  \setminus\left\{  \left(  2,2\right)  \right\}  }\dfrac
{1}{h_{\lambda}\left(  c;z\right)  }\\
&  =\dfrac{1}{h_{\lambda}\left(  \left(  2,1\right)  ;z\right)  }\cdot
\dfrac{1}{h_{\lambda}\left(  \left(  1,2\right)  ;z\right)  }\cdot\dfrac
{1}{h_{\lambda}\left(  \left(  2,2\right)  ;z\right)  }\cdot\dfrac
{1}{h_{\lambda}\left(  \left(  1,3\right)  ;z\right)  }\\
&  \ \ \ \ \ \ \ \ \ \ +\dfrac{1}{h_{\lambda}\left(  \left(  1,1\right)
;z\right)  }\cdot\dfrac{1}{h_{\lambda}\left(  \left(  2,1\right)  ;z\right)
}\cdot\dfrac{1}{h_{\lambda}\left(  \left(  1,2\right)  ;z\right)  }\cdot
\dfrac{1}{h_{\lambda}\left(  \left(  1,3\right)  ;z\right)  }\\
&  =\dfrac{1}{\left(  z_{-1}+z_{0}\right)  \left(  z_{1}+z_{2}+z_{0}\right)
z_{0}z_{2}}\\
&  \ \ \ \ \ \ \ \ \ \ +\dfrac{1}{\left(  z_{0}+z_{1}+z_{2}+z_{-1}\right)
\left(  z_{-1}+z_{0}\right)  \left(  z_{1}+z_{2}+z_{0}\right)  z_{2}}.
\end{align*}
This is indeed easily checked using our above expressions for
$\zz_A, \zz_B, \zz_C, \zz_D, \zz_E$.
\end{example}

We note how Theorem \ref{thm.main} generalizes the previous hook length formulas:

\begin{itemize}
\item If we set $\mu=\varnothing$ in Theorem \ref{thm.main}, then the sum on
the right hand side of (\ref{eq.thm.main.claim}) simplifies to $\prod_{c\in
Y\left(  \lambda\right)  }\dfrac{1}{h_{\lambda}\left(  c;z\right)  }$ (by
Lemma \ref{lem.exc.straight}), whereas the set $\operatorname*{SYT}\left(
\lambda/\mu\right)  $ on the left hand side becomes $\operatorname*{SYT}%
\left(  \lambda\right)  $. Thus, the equality (\ref{eq.thm.main.claim}) turns
into (\ref{eq.thm.pak-1.claim}), and we recover Theorem \ref{thm.pak-1}.

\item Let $n = \abs{Y\tup{\lambda/\mu}}$.
If we set $z_{i}=1$ for all $i\in\mathbb{Z}$ in Theorem \ref{thm.main},
then each $\zz_{T}$ becomes $\dfrac{1}{n!}$ (since the denominator $\prod
_{k=1}^{n}\left(  z_{c_{T}\left(  k\right)  }+z_{c_{T}\left(  k+1\right)
}+\cdots+z_{c_{T}\left(  n\right)  }\right)  $ becomes $\prod_{k=1}^{n}\left(
n-k+1\right)  =n\cdot\left(  n-1\right)  \cdot\left(  n-2\right)  \cdot
\cdots\cdot1=n!$). Hence, the sum on the left hand side of
(\ref{eq.thm.main.claim}) becomes $\sum_{T\in\operatorname*{SYT}\left(
\lambda/\mu\right)  }\dfrac{1}{n!}=\left\vert \operatorname*{SYT}\left(
\lambda/\mu\right)  \right\vert \cdot\dfrac{1}{n!}$, whereas each
algebraic hook length $h_{\lambda}\left(
c;z\right)  $ on the right hand side becomes $h_{\lambda}\left(  c\right)  $.
Thus, the equality (\ref{eq.thm.main.claim}) becomes%
\[
\left\vert \operatorname*{SYT}\left(  \lambda/\mu\right)  \right\vert
\cdot\dfrac{1}{n!}=\sum_{E\in\mathcal{E}\left(  \lambda/\mu\right)  }%
\ \ \prod_{c\in Y\left(  \lambda\right)  \setminus E}\dfrac{1}{h_{\lambda
}\left(  c\right)  },
\]
and we recover Theorem \ref{thm.naruse} (after multiplying both sides by $n!$).

\item Likewise, if we set both $\mu=\varnothing$ and $z_{i}=1$ for all
$i\in\mathbb{Z}$ in Theorem \ref{thm.main}, then Theorem
\ref{thm.hlf-classical} results.
\end{itemize}

Thus, a new proof of Theorem \ref{thm.main} will yield new proofs of all three
previously known hook length formulas.

Theorem \ref{thm.main} was originally conjectured by one of the authors in
2018, and was proved by Matjaz Konvalinka; his proof (in a type-B variant) has
been sketched in \cite[Section 5]{Konva2}. Our below proof of Theorem
\ref{thm.main} will have some commonalities with his; in particular, we shall
derive it from the same recursion as he did, although the recursion will be
proved quite differently.

\section{\label{sec.recz}A recursion for the $\zz_T$}

\begin{definition}
Let $\mu$ and $\nu$ be two partitions.
Then, we write $\mu \lessdot \nu$ if
$\mu \subseteq \nu$ and $\abs{Y\tup{\nu / \mu}} = 1$.
Equivalently, $\mu \lessdot \nu$ holds if and only if $\nu$ can be obtained
from $\mu$ by increasing one entry by $1$.
\end{definition}

For example, $\tup{4,2,2,1} \lessdot \tup{4,3,2,1}$ and
$\tup{5,2} \lessdot \tup{5,2,1}$.
(In the latter example, the ``invisible'' third entry is being increased from $0$ to $1$.)
But neither $\tup{2,1} \lessdot \tup{3,2}$ nor
$\tup{2,1} \lessdot \tup{4}$ holds.
\medskip

For the rest of this section, we shall use the following notations:

\begin{convention}
\label{conv.z-and-zT}
Fix an infinite family of commuting indeterminates
$\ldots,z_{-2},z_{-1},z_{0},z_{1},z_{2},\ldots$. We shall be working in the
field of rational functions in these indeterminates (with rational coefficients).

We shall furthermore use the fractions $\mathbf{z}_T$ defined in
Theorem~\ref{thm.main}.
\end{convention}

\begin{lemma}
\label{recus.z}
Let $\lambda/\mu$ be a skew partition with $\lambda \neq \mu$ (so that $Y\tup{\lambda/\mu}$ has at least one box).
Let $T \in \SYT\tup{\lambda/\mu}$.

If we remove the box with the entry $1$ from $T$, and subtract $1$ from all
remaining entries, then we obtain a new standard tableau $T'$, which has shape
$\lambda / \nu$ for some partition $\nu$ satisfying $\mu \lessdot \nu \subseteq \lambda$.
It satisfies
\begin{align}
\zz_{T} = \frac{1}{\sum\limits_{\tup{i,j} \in Y\left( \lambda/\mu\right)} z_{j-i}}\cdot \zz_{T'}.
\label{eq.recus_z.eq}
\end{align}
\end{lemma}

\begin{example}
Let $\lambda = \tup{3,3,2}$ and $\mu = \tup{2,1}$.
Let $T$ be the standard tableau
\[
\ytableausetup{centertableaux}
\begin{ytableau}
\none & \none & 2 \\
\none & 1 & 3 \\
4 & 5 \\
\end{ytableau}
\in \SYT\tup{\lambda / \mu} .
\]
Then, $T'$ is
\[
\ytableausetup{centertableaux}
\begin{ytableau}
\none & \none & 1 \\
\none & \none & 2 \\
3 & 4 \\
\end{ytableau}
\in \SYT\tup{\lambda / \nu}
\]
for $\nu = \tup{2,2}$.

We have
\[
\zz_{T'} = \frac{1}{\left(z_2 + z_1 + z_{-2} + z_{-1}\right) \left(z_1 + z_{-2} + z_{-1}\right) \left(z_{-2} + z_{-1}\right) z_{-1}}
\]
and
\[
\zz_{T} = \frac{1}{\left(z_0 + z_2 + z_1 + z_{-2} + z_{-1}\right) \left(z_2 + z_1 + z_{-2} + z_{-1}\right) \left(z_1 + z_{-2} + z_{-1}\right) \left(z_{-2} + z_{-1}\right) z_{-1}}.
\]
Thus,
\[
\zz_{T} = \frac{1}{z_0 + z_2 + z_1 + z_{-2} + z_{-1}}\cdot \zz_{T'},
\]
as predicted by Lemma \ref{recus.z}
(since 
$\sum\limits_{\tup{i,j} \in Y\left( \lambda/\mu\right)} z_{j-i} = z_0 + z_2 + z_1 + z_{-2} + z_{-1}$).
\end{example}

Now, it is easy to see that the left hand side of \eqref{eq.thm.main.claim} satisfies a recursion relation:

\begin{lemma}
\label{recursion.main}
Let $\lambda/\mu$ be a skew partition.

\begin{enumerate}
\item[\textbf{(a)}] If $\lambda=\mu$, then
\[
\sum_{T\in\SYT \left(  \lambda/\mu\right)  } \zz_T =1.
\]

\item[\textbf{(b)}] If $\lambda\neq\mu$, then
\[
\sum_{T\in\SYT \left(  \lambda/\mu\right)  } \zz_T
= \frac{1}{\sum\limits_{\tup{i,j}\in Y\left(
\lambda/\mu\right)  }z_{j-i}}
\cdot
\sum\limits_{\mu\lessdot\nu\subseteq\lambda
}
\ \ \sum_{T\in\SYT \left(  \lambda/\nu\right)  } \zz_T.
\]
Here, the second sum on the right hand side is a sum over all partitions $\nu$
satisfying $\mu\lessdot\nu\subseteq\lambda$.
\end{enumerate}
\end{lemma}

Note that the skew partitions $\lambda/\nu$ on the right hand side of Lemma
\ref{recursion.main} \textbf{(b)} have one fewer box than the skew partition
$\lambda/\mu$ (in the sense that $\left\vert Y\left(  \lambda/\nu\right)
\right\vert =\left\vert Y\left(  \lambda/\mu\right)  \right\vert -1$). Thus,
Lemma \ref{recursion.main} yields a recursive algorithm for computing the left hand side $\sum_{T\in\SYT \left(  \lambda/\mu\right)  } \zz_T$ of \eqref{eq.thm.main.claim} for all skew partitions $\lambda/\mu$. We
shall now show that the right hand side  $\sum\limits_{E\in\mathcal{E}\left(
\lambda/\mu\right)  }\ \ \prod\limits_{c\in Y\left(  \lambda\right)  \setminus
E}\dfrac{1}{h_{\lambda}\left(  c;z\right)  }$ of (\ref{eq.thm.main.claim})
satisfies an analogous recursion (Lemma \ref{rec.2}). The proof of this recursion is much
subtler, and the preparations for it will occupy most of this paper. Once it
is proved, (\ref{eq.thm.main.claim}) will easily follow.

\section{\label{sec.konva}Conjugates, Delta-sets and the Konvalinka recursion}

We shall next explore some fundamental features of partitions, which will become basic ingredients in our proof.

\subsection{The conjugate partition}

We recall one of the most fundamental concepts in the theory of partitions:

\begin{definition}
\label{def.lambdat}
    Let $\lambda$ be a partition. Then, $\lambda^t$ is the partition whose Young diagram is the reflection of the diagram of $\lambda$ across the main diagonal -- i.e., it is given by
    \begin{equation}
    Y\tup{\lambda^t} = \set{\tup{j, i} \mid \tup{i, j} \in Y\tup{\lambda} }.
    \label{eq.def.lambdat.1}
    \end{equation}
    Explicitly, it is given by
    \begin{align}
    \lambda^t_k = \abs{\set{i \geq 1 \mid \lambda_i \geq k }}
    \qquad \text{for all } k \geq 1.
    \label{eq.def.lambdat.2}
    \end{align}
    Equivalently,
    \begin{align}
    \lambda^t_k = \tup{\text{number of boxes in the $k$-th column of $Y\tup{\lambda}$}}
    \label{eq.def.lambdat.5}
    \end{align}
    for all $k \geq 1$.
    Equivalently,
    \begin{align}
    \lambda^t_k = \max\set{i \geq 1 \mid \lambda_i \geq k }
    \qquad \text{for all } k \geq 1,
    \label{eq.def.lambdat.3}
    \end{align}
    with the understanding that the maximum of an empty set is $0$ here.

    The partition $\lambda^t$ is called the \emph{conjugate} (or \emph{transpose}) of $\lambda$.
\end{definition}

\begin{example}
    Let $\lambda$ be the partition $\tup{5, 2, 2, 1}$. Then, $\lambda^t = \tup{4, 3, 1, 1, 1}$. The Young diagrams $Y\tup{\lambda}$ and $Y\tup{\lambda^t}$ are \[
    \ydiagram{5,2,2,1} \qquad \text{ and } \qquad
    \ydiagram{4,3,1,1,1} .
    \]
\end{example}

The following lemma follows easily from this definition:

\begin{lemma}
    \label{lem.conj.uniprop}
    Let $\lambda$ be a partition. Let $i$ and $j$ be two positive integers. Then, we have the logical equivalence
    \[
    \tup{\lambda^t_i \geq j}
    \ \Longleftrightarrow\ %
    \tup{\lambda_j \geq i}.
    \]
\end{lemma}

Note that the conjugate partition $\lambda^t$ is denoted $\lambda'$ in \cite{Macdon95} and denoted $\widetilde{\lambda}$ in \cite{Fulton97}.

\subsection{Delta-sets}

The following definition is less standard but no less important to us:

\begin{definition}
\label{def.Delta-lambda}
Let $\lambda$ be any partition. Then, we define a set $\Delta\tup{\lambda}$ by
\[
\Delta\tup{\lambda} := \set{\lambda_i - i \mid i \geq 1} .
\]

\end{definition}

\begin{example}
    Let $\lambda$ be the partition $\tup{5, 2, 2, 1}$. Then,
    \[
    \Delta\tup{\lambda} = \set{4, 0, -1, -3, -5, -6, -7, \ldots}
    \]
    (where the ``$\ldots$'' are just the negative integers from $-8$ on downwards).
\end{example}

\subsection{The $\mathbf{s}_\lambda\ive{\nu}$ polynomial}

\begin{definition}
\label{def.snu}
Let $\lambda$ be a partition.

Let $x_{1},x_{2},x_3,\ldots$ and $y_{1},y_{2},y_3,\ldots$ be two infinite families of commuting indeterminates.

For any partition $\nu$, we set
\[
\mathbf{s}_\lambda\left[  \nu\right]  :=\sum_{D\in\mathcal{E}\left(  \lambda
/\nu\right)  }\ \ \prod_{\left(  i,j\right)  \in D}\left(  x_{i}+y_{j}\right)
.
\]
Also, if $\nu$ is not a partition, then we set
\[
\mathbf{s}_\lambda\left[  \nu\right]  :=0.
\]
\end{definition}

\begin{example}
Let $\lambda=\left(  3,3,1\right)  $ and $\mu=\left(  1\right)  $. Then,
\[
\mathcal{E}\left(  \lambda/\mu\right)  =\left\{  \left\{  \left(  1,1\right)
,\ \left(  1,2\right)  \right\}  ,\ \left\{  \left(  1,1\right)  ,\ \left(
2,3\right)  \right\}  ,\ \left\{  \left(  2,2\right)  ,\ \left(  2,3\right)
\right\}  \right\}  .
\]
Drawn using asterisks as in Example \ref{exa.Elamu.1}, these three excitations
look as follows:%
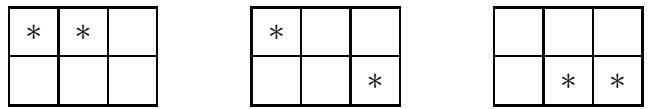
\[
\begin{ytableau}
\ast&\ast& \\ && \\
\end{ytableau}\ \ \ \ \ \ \ \ \ \ \begin{ytableau}
\ast&& \\ &&\ast \\
\end{ytableau}\ \ \ \ \ \ \ \ \ \ \begin{ytableau}
\vphantom{x}&& \\ &\ast&\ast \\
\end{ytableau}\ \ .
\]
Definition \ref{def.snu} yields
\begin{align*}
&\mathbf{s}_{\lambda}\left[  \mu\right]    \\
& =\sum_{D\in\mathcal{E}\left(
\lambda/\mu\right)  }\ \ \prod_{\left(  i,j\right)  \in D}\left(  x_{i}%
+y_{j}\right)  \\
& =\prod_{\left(  i,j\right)  \in\left\{  \left(  1,1\right)  ,\ \left(
1,2\right)  \right\}  }\left(  x_{i}+y_{j}\right)  +\prod_{\left(  i,j\right)
\in\left\{  \left(  1,1\right)  ,\ \left(  2,3\right)  \right\}  }\left(
x_{i}+y_{j}\right)  +\prod_{\left(  i,j\right)  \in\left\{  \left(
2,2\right)  ,\ \left(  2,3\right)  \right\}  }\left(  x_{i}+y_{j}\right)  \\
& =\left(  x_{1}+y_{1}\right)  \left(  x_{1}+y_{2}\right)  +\left(
x_{1}+y_{1}\right)  \left(  x_{2}+y_{3}\right)  +\left(  x_{2}+y_{2}\right)
\left(  x_{2}+y_{3}\right)  \\
& =x_{1}^{2}+x_{2}^{2}+x_{1}x_{2}+x_{1}y_{1}+x_{1}y_{2}+x_{2}y_{1}+x_{1}%
y_{3}+x_{2}y_{2}+x_{2}y_{3} \\
& \qquad \qquad +y_{1}y_{2}+y_{1}y_{3}+y_{2}y_{3}.
\end{align*}
\end{example}

In the lingo of symmetric functions, the polynomials $\mathbf{s}_\lambda\ive{\nu}$ can be called \emph{flagged factorial Schur polynomials}.
They turn into the usual flagged Schur polynomials if we set all the $y_i$ to $0$, and into the factorial Schur polynomials if we let $\lambda_i \to +\infty$.
(This follows from Corollary~\ref{corollaryFlaggedExc} below.)

\subsection{The Konvalinka recursion}

Now, we can state the \emph{Konvalinka recursion} (\cite[Theorem 5]{Konva1}), which was used by Konvalinka in his proof of Naruse's formula:

\begin{theorem}[Konvalinka recursion]
\label{mainKonvalinka}
Let $\lambda/\mu$ be any skew partition, and let $x_{1},x_{2},x_3,\ldots$ and $y_{1},y_{2},y_3,\ldots$ be two infinite families of commuting indeterminates.

Set
\[
\ell_i := \lambda_i - i \qquad \text{ and } \qquad \ell^t_i := \lambda^t_i - i \qquad  \text{ for all } i \geq 1.
\]
Then,
\begin{align*}
\left(\sum_{\substack{k \geq 1;\\ \ell_k \notin \Delta\tup{\mu}}} x_k + \sum_{\substack{k \geq 1;\\ \ell^t_k \notin \Delta\tup{\mu^t}}} y_k\right) \mathbf{s}_\lambda\ive{\mu}
= \sum_{\mu \lessdot \nu \subseteq \lambda} \mathbf{s}_\lambda\ive{\nu} .
\end{align*}
Here, the sum on the right hand side ranges over all partitions $\nu$ that satisfy $\mu \lessdot \nu \subseteq \lambda$.
\end{theorem}

Our proof of Theorem~\ref{thm.main} will not make direct use of the Konvalinka recursion as we just stated it, but instead use a simpler (if less aesthetically pleasant) variant (Lemma \ref{lem.konvalinka-bi}, which we will prove below).
We will then (in Section \ref{sec.konva-forreal}) derive the Konvalinka recursion from this variant (with some extra work).

\section{\label{sec.flag}Flagged semistandard tableaux}

We next define the notion of \emph{flagged semistandard tableaux}, which will (in a specific case) be a more manageable model for excitations.

\subsection{Semistandard tableaux}

We begin with the concept of semistandard tableaux (a relative of that of standard tableaux):

\begin{definition}
\label{SSYT_def}
Let $\mu$ be a partition.
A \emph{semistandard tableau} of shape $\mu$ means a way to put
a positive integer into each box of $Y\left( \mu\right)  $ (that is, formally speaking, a map $T : Y\tup{\mu} \to \set{1,2,3,\ldots}$) such that

\begin{itemize}
\item the integers weakly increase left-to-right in each row (i.e., we have $T\tup{i,j} \leq T\tup{i,j+1}$ whenever $\tup{i,j}$ and $\tup{i,j+1}$ belong to $Y\tup{\mu}$);

\item the integers strictly increase top-to-bottom in each column (i.e., we have $T\tup{i,j} < T\tup{i+1,j}$ whenever $\tup{i,j}$ and $\tup{i+1,j}$ belong to $Y\tup{\mu}$).

\end{itemize}

(Just as before, the value $T\tup{i,j}$ is regarded as the entry of $T$ in the box $\tup{i,j}$.)

We let $\SSYT\tup{\mu}$ denote the set of all semistandard tableaux of shape $\mu$.

\end{definition}

\begin{example}
    If $\mu=\tup{4,3,3}$, then $T\in \SSYT\tup{\mu}$ can be:
    \[
        \ytableaushort{1123,347,888}
    \]
    but cannot be
    \[
        \ytableaushort{1123,147,888}
    \]
    because its property $T\tup{1,1} = T\tup{2,1}$ would violate the second condition in Definition~\ref{SSYT_def}.
\end{example}

The following lemma exposes two basic properties of semistandard tableaux, which will be used later on and also make for good warmup exercises:

\begin{lemma}
\label{lem.ssyt.geq}
Let $\mu$ be a partition.
Let $T \in \SSYT\tup{\mu}$ be a semistandard tableau. Then:

\begin{enumerate}
    \item[\textbf{(a)}]
    We have $T\tup{i,j} \geq i$ for each $\tup{i,j} \in Y\tup{\mu}$.

    \item[\textbf{(b)}]
    Let $\tup{i,j}$ and $\tup{u,v}$ be two boxes in $Y\tup{\mu}$ such that $i \leq u$ and $j \leq v$.
    Then,
    \[
    T\tup{u,v} - u \geq T\tup{i,j} - i .
    \]
\end{enumerate}
\end{lemma}

\subsection{Flagged semistandard tableaux in general}

Flagged semistandard tableaux are semistandard tableaux in which, for each $i$, the entries in the $i$-th row are bounded from above by a given integer $b_i$:

\begin{definition}

\begin{enumerate}
\item[\textbf{(a)}]
A \emph{flagging} means a sequence $\tup{b_1, b_2, b_3, \ldots}$, where each $b_i$ is a positive integer.

\item[\textbf{(b)}]
A flagging $\tup{b_1, b_2, b_3, \ldots}$ is said to be \emph{weakly increasing} if $b_1 \leq b_2 \leq b_3 \leq \cdots$.

\item[\textbf{(c)}]
Let $\bb = \tup{b_1, b_2, b_3, \ldots}$ be a flagging, and let $\mu$ be a partition. A semistandard tableau $T$ of shape $\mu$ is said to be \emph{$\bb$-flagged} if and only if it satisfies
\[
T\tup{i, j} \leq b_i \qquad \text{for all } \tup{i, j} \in Y\tup{\mu}
\]
(that is, all entries in row $i$ are $\leq b_i$).

We let $\FSSYT\tup{\mu, \bb}$ be the set of all $\bb$-flagged semistandard tableaux of shape $\mu$.

\end{enumerate}

\end{definition}

\begin{example}
    Let $\mu=\tup{3,2,1}$ and $\bb=\tup{2,3,3,3,3,\ldots}$. Then, $\FSSYT\tup{\mu, \bb}$ consists of the five semistandard tableaux
    \[
    \ytableaushort{111,22,3} \qquad
    \ytableaushort{111,23,3} \qquad
    \ytableaushort{112,22,3} \qquad
    \ytableaushort{112,23,3} \qquad
    \ytableaushort{122,23,3}
    \]
\end{example}

\begin{definition}
\label{diag_defin}
In the following, for any $k \in \ZZ$, we define the \emph{$k$-th diagonal} to be the set of all boxes $\tup{i, j} \in \ZZ^2$ such that $j-i = k$.
For instance, the $1$-st diagonal consists of the boxes $\ldots,\ \tup{-2, -1},\ \tup{-1, 0},\ \tup{0, 1},\ \tup{1, 2},\ \ldots$.
\end{definition}

We note that if $\lambda$ is a partition and $i \geq 1$, then the $i$-th row of $Y\tup{\lambda}$
has boxes in the $\tup{1-i}$-th, $\tup{2-i}$-th, ..., $\tup{\lambda_i-i}$-th diagonals.

\subsection{The flagging induced by $\lambda / \mu$}

\begin{convention}
    \textbf{For the rest of this section}, we fix two partitions $\lambda$ and $\mu$.
    (We do not require that $\lambda \supseteq \mu$.)
\end{convention}

\begin{definition}
\label{def.flagging-of-lm}
Let $\lambda$ and $\mu$ be two partitions (not necessarily satisfying $\mu \subseteq \lambda$).

For each $i \geq 1$, we set
\begin{align*}
b_i :=
\max\set{k \geq 0 \mid \lambda_k - k \geq \mu_i - i },
\end{align*}
where we understand $\lambda_0$ to be $+\infty$ (so that the set on the right hand side always includes $0$).
The maximum here is well-defined because of Lemma~\ref{lem.flagging-of-lm.wd} below.

We define the flagging $\bb$ to be $\tup{b_1, b_2, b_3, \ldots}$,
and we call it the \emph{flagging induced by $\lambda / \mu$}.


Note that if $\mu \subseteq \lambda$ and $\mu_i > 0$, then $b_i$ is $i$ plus the maximum number of diagonal moves that the box $\tup{i, \mu_i}$ could make along its diagonal without leaving the Young diagram $Y\tup{\lambda}$.
(A \emph{diagonal move} is a move that takes a box to its southeastern neighbor on the same diagonal, without regard for any other boxes in the diagram.)
This description applies to the general case as well, if we extend $Y\tup{\lambda}$ by all boxes $\tup{i,j}$ with $i<0$ or $j<0$, and allow ``negative'' diagonal moves (which are just reverse diagonal moves).

We define $\mathcal{F}\tup{\lambda/\mu}$ to be the set $\FSSYT\tup{\mu, \bb}$, with $\bb$ defined as above.
\end{definition}

\begin{example}
    Let $\mu=\tup{3,2,1,1}$ and $\lambda=\tup{7,6,6,5,5,3,1}$. Then,
    the flagging induced by $\lambda / \mu$ is
    $\bb = \tup{3,5,5,6,6,7,8,9,\ldots}$. (Note that $b_i = i$ for all sufficiently large $i$, by Lemma~\ref{lem.flagging-of-lm.0}.) Visually, $b_2 = 5$ can be seen by applying diagonal moves to the box $g = \tup{2, \mu_2}$, and $b_3 = 5$ can be seen by applying diagonal moves to the box $h = \tup{3, \mu_3}$ in the following picture:
    \[
        \ytableausetup{centertableaux}
        \ytableaushort
        {\none , \none g, h\none g,\none h\none g,\none \none h\none g}
        * {7,6,6,5,5,3,1}
        * [*(green)]{3,2,1,1}
    \]

\end{example}

The flagging induced by $\lambda / \mu$ has a few basic properties:

\begin{lemma}
\label{lem.flagging-of-lm.wd}
The maximum $\max\set{k \geq 0 \mid \lambda_k - k \geq \mu_i - i }$ in Definition~\ref{def.flagging-of-lm} is well-defined.
\end{lemma}

\begin{lemma}
\label{lem.flagging-of-lm.uniprop}
Let $\bb=\left(  b_{1},b_{2},b_{3},\ldots\right)  $ be the flagging induced by $\lambda/\mu$.
Let $i$ and $j$ be two positive integers. Then, we have the logical
equivalence
\[
\left(  j\leq b_{i}\right)
\ \Longleftrightarrow\ \left(  \lambda_{j} 
-j\geq\mu_{i}-i\right)  .
\]

\end{lemma}

\begin{lemma}
\label{lem.flagging-of-lm.inc}
The flagging $\bb$ induced by $\lambda/\mu$ is weakly increasing.
\end{lemma}

\begin{lemma}
\label{lem.flagging-of-lm.0}
Let $\bb=\left(  b_{1},b_{2},b_{3},\ldots\right)  $ be the flagging induced by $\lambda/\mu$.

Let $i \geq 1$ satisfy $\mu_i = 0$ and $\lambda_{i+1} = 0$.
Then, $b_i = i$.
\end{lemma}

\subsection{Flagged semistandard tableaux vs. excitations}

Now we get to the core of this section.
There is a connection between flagged semistandard tableaux and excitations.
This connection was first noticed by Kreiman \cite[\S 6]{Kreiman}, who stated its main properties but left them unproved\footnote{Kreiman refers to the excited moves as ``ladder moves''. We suspect that his $\SSYT_{\mu,\lambda}$ is our $\calF\tup{\lambda/\mu}$.}.
Essentially the same properties (in a slightly modified form\footnote{The flagging $\mathbf{f}^{\tup{\lambda/\mu}} = \tup{\mathsf{f}_1, \mathsf{f}_2, \ldots, \mathsf{f}_{\ell\tup{\mu}}}$ in \cite{MPP1} is not the same as our flagging $\bb$, but it is not hard to see that the two flaggings determine the same flagged semistandard tableaux.}) later appeared in \cite[Proposition 3.6]{MPP1} with a sketched proof and in \cite[Section 4]{Konva1} with an assurance of the proof being ``easy''.
We shall give complete proofs of these properties, although they are indeed highly intuitive and easy to the combinatorially trained reader (even as their simplicity gets lost in writing).

\begin{definition}
    \label{def.DD}
    Let $T \in \SSYT\tup{\mu}$ be a semistandard tableau.
    Then:
    
    \begin{enumerate}
    \item[\textbf{(a)}]
    If $c = \tup{i,j}$ is any box in $Y\tup{\mu}$, then we define a new box
    \[
    c_{+T} := \tup{T\tup{i,j}, \ T\tup{i,j}+j-i} \in \ZZ^2 .
    \]
    (Recall that $T\tup{i,j}$ denotes the entry of $T$ in the box $\tup{i,j}$.)

    The box $c_{+T}$ can be equivalently characterized as the unique box in the $T\tup{c}$-th row that lies on the same diagonal as $c$.

    \item[\textbf{(b)}]
    The diagram $\DD\tup{T}$ is defined by
    \begin{align}
        \DD\tup{T} := \set{ c_{+T} \ \mid \ c \in Y\tup{\mu} } .
        \label{eq.def.DD.DDT=}
    \end{align}
    \end{enumerate}
\end{definition}

\begin{example}
    Let $\mu = \tup{4, 2, 1}$, and let $T \in \SSYT\tup{\mu}$ be the following semistandard tableau:
    \[
    T = \ytableaushort{1123,23,4}.
    \]
    Then, the corresponding boxes $c_{+T}$ are
    \begin{align*}
        &
        \tup{1,1}_{+T} = \tup{1,1}, \quad
        \tup{1,2}_{+T} = \tup{1,2}, \quad
        \tup{1,3}_{+T} = \tup{2,4}, \quad
        \tup{1,4}_{+T} = \tup{3,6}, \\
        &
        \tup{2,1}_{+T} = \tup{2,1}, \quad
        \tup{2,2}_{+T} = \tup{3,3}, \quad
        \tup{3,1}_{+T} = \tup{4,2}.
    \end{align*}
    Thus, the diagram $\DD\tup{T}$ is the following collection of boxes (the northwesternmost of which is $\tup{1,1}$):
    \[
    \ytableaushort{{\phantom{\ast}}{\phantom{\ast}},{\phantom{\ast}}\none\none{\phantom{\ast}},\none\none{\phantom{\ast}}\none\none{\phantom{\ast}},\none{\phantom{\ast}}}
    \]
\end{example}

\begin{example}
    Let $\mu=\tup{3,2,1}$ and $\lambda=\tup{4,4,3}$.
    Then, the flagging $\bb$ induced by $\lambda/\mu$ is $\tup{2,3,3,4,5,6,\ldots}$.
    Here are all flagged semistandard tableaux $T \in \calF\tup{\lambda / \mu}$ and the corresponding diagrams $\DD\tup{T}$:
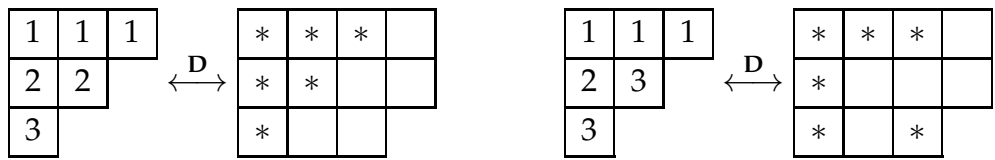
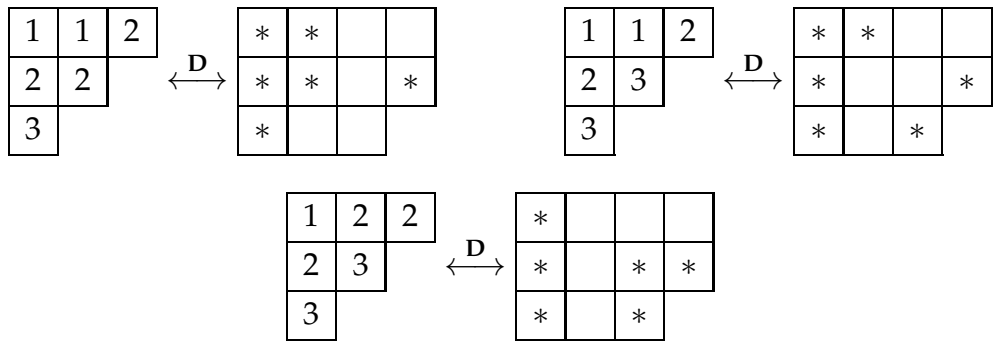
    \[
        \ytableaushort{111,22,3}
        \overset{\mathbf{D}}{\longleftrightarrow}
        \ytableaushort{\ast\ast\ast,\ast\ast,\ast} *{4,4,3}
        \qquad\qquad
        \ytableaushort{111,23,3}
        \overset{\mathbf{D}}{\longleftrightarrow}        \ytableaushort{\ast\ast\ast,\ast,\ast\none\ast} *{4,4,3}
    \]

    \[   
        \ytableaushort{112,22,3}
        \overset{\mathbf{D}}{\longleftrightarrow}
        \ytableaushort{\ast\ast,\ast\ast\none\ast,\ast} *{4,4,3}
        \qquad\qquad
        \ytableaushort{112,23,3}
        \overset{\mathbf{D}}{\longleftrightarrow}        \ytableaushort{\ast\ast,\ast\none\none\ast,\ast\none\ast} *{4,4,3}
    \] 
    
    \[
        \ytableaushort{122,23,3}
        \overset{\mathbf{D}}{\longleftrightarrow}
        \ytableaushort{\ast,\ast\none\ast\ast,\ast\none\ast} *{4,4,3}
    \]
  
\end{example}

We are now ready to present the relevant properties of excitations.
However, for the ease of their proofs, we first state two technical lemmas about semistandard tableaux.
The first lemma characterizes the cases in which two boxes of $\DD\tup{T}$ are adjacent:

\begin{lemma}
\label{lem.ssyt.neighbors}
Let $\mu$ be a partition. Let $T\in
\operatorname*{SSYT}\left(  \mu\right)  $ be a semistandard tableau. Let $c$
and $d$ be two boxes in $Y\left(  \mu\right)  $.

\begin{enumerate}
\item[\textbf{(a)}] If $d_{+T}=c_{+T}$, then $d=c$.

\item[\textbf{(b)}] If $d_{+T}=\left(  c_{+T}\right)  _{\rightarrow}$, then
$d=c_{\rightarrow}$ and $T\left(  d\right)  =T\left(  c\right)  $.

\item[\textbf{(c)}] If $d_{+T}=\left(  c_{+T}\right)  _{\downarrow}$, then
$d=c_{\downarrow}$ and $T\left(  d\right)  =T\left(  c\right)  +1$.

\item[\textbf{(d)}] If $d_{+T}=\left(  c_{+T}\right)  _{\searrow}$, then
$d=c_{\searrow}$ and $T\left(  d\right)  =T\left(  c\right)  +1$.
\end{enumerate}
\end{lemma}

The next technical lemma (which is nearly trivial) says that if we change a single entry in a semistandard tableau $S$ (while preserving its semistandardness), then the diagram $\DD\tup{S}$ changes just by a single box:

\begin{lemma}
\label{lem.DS-vs-DT}
Let $\left(  i,j\right)  \in Y\left(  \mu\right)  $ be a
box. Let $T\in\operatorname*{SSYT}\left(  \mu\right)  $ and $S\in
\operatorname*{SSYT}\left(  \mu\right)  $ be two semistandard tableaux. Assume
that%
\begin{equation}
T\left(  c\right)  =S\left(  c\right)  \ \ \ \ \ \ \ \ \ \ \text{for all }c\in
Y\left(  \mu\right)  \text{ distinct from }\left(  i,j\right)
.\label{eq.lem.DS-vs-DT.ass}%
\end{equation}
Then, the diagram $\mathbf{D}\left(  T\right)  $ can be obtained from
$\mathbf{D}\left(  S\right)  $ by replacing the box $\left(  i,j\right)
_{+S}$ by the box $\left(  i,j\right)  _{+T}$.
\end{lemma}

We can now connect excitations of $Y\tup{\mu}$ with semistandard tableaux:

\begin{lemma}
\label{transition_wd}
    Let $T \in \SSYT\tup{\mu}$.
    Then:
    \begin{enumerate}
        
        \item[\textbf{(a)}] The diagram $\DD\tup{T}$ is an excitation of the diagram $Y\tup{\mu}$.
        
        \item[\textbf{(b)}] We have
        \[
        \prod_{\tup{i,j} \in \DD\tup{T}} \tup{x_{i}+y_{j}}
        = \prod_{\tup{i,j} \in Y\tup{\mu}} \tup{x_{T\tup{i,j}}+y_{T\tup{i,j}+j-i}} .
        \]
        
        \item[\textbf{(c)}] We have $\DD\tup{T} \in \calE\tup{\lambda / \mu}$ if and only if $T \in \calF\tup{\lambda/\mu}$.
    
    \end{enumerate}
\end{lemma}

\begin{lemma}
\label{transition_formula}
    The map
    \begin{align*}
        \SSYT\tup{\mu} &\to \set{\text{all excitations of $Y\tup{\mu}$}}, \\
        T &\mapsto \DD\tup{T}
    \end{align*}
    is well-defined and is a bijection.
\end{lemma}

\begin{lemma}
\label{transition_formula_flagged}
    The map
    \begin{align*}
        \calF(\lambda/\mu) &\to \calE(\lambda/\mu), \\
        T &\mapsto \DD\tup{T}
    \end{align*}
    is well-defined and is a bijection.
\end{lemma}

As a consequence of the above, we can rewrite the polynomial $\mathbf{s}_\lambda\ive{\mu}$ from Definition~\ref{def.snu} in terms of flagged semistandard tableaux:

\begin{corollary}
\label{corollaryFlaggedExc}
We have
\[
    \mathbf{s}_\lambda\ive{\mu} = \sum_{T \in \mathcal{F}(\lambda/\mu)}\ \ \prod_{\tup{i,j} \in Y\tup{\mu}} \tup{x_{T\tup{i,j}}+y_{T\tup{i,j}+j-i}} .
\]
\end{corollary}

\section{\label{sec.habc}The $h\left(  a,b,c\right)  $ polynomials}

Now we introduce three further pieces of notation.

\begin{definition}
    The \emph{Iverson bracket} is the function that assigns to each statement its truth value (i.e., the number $1$ if the statement is true and $0$ otherwise). More formally: If $X$ is a statement, then
    \[
    \left[ X \right]=        
    \begin{cases}
        1, & \text{if $X$ is true;}
        \\
        0, & \text{if $X$ is false.}
    \end{cases} 
    \]
\end{definition}

For example, $[2 = 3] = 0$ and $[5 \ne 1] = 1$.

\begin{definition}
    If $N$ is any integer, then $\ive{N}$ will denote the set $\set{1, 2, \ldots, N}$. This set is empty when $N \leq 0$.
\end{definition}

\begin{definition}
\label{defh}
\ \ \ \ %

\begin{enumerate}

\item[\textbf{(a)}]
Let $R$ be the polynomial ring over $\ZZ$ in countably many commuting indeterminates
\begin{align*}
    x_1, x_2, x_3, \ldots, \\
    y_1, y_2, y_3, \ldots .
\end{align*}

\item[\textbf{(b)}]
We furthermore set $x_i = 0$ and $y_i = 0$ for any integer $i \leq 0$. Thus, $x_i$ and $y_i$ are defined for any integer $i$.

\item[\textbf{(c)}]
For all $a,c \in \ZZ$ and $b \in \NN$, we set 
\begin{align*}
h\tup{a,b,c}
:= 
\sum_{\substack{\tup{i_1,i_2,\ldots,i_a} \in \ive{b}^a; \\  i_1 \leq i_2 \leq \cdots \leq i_a}}
\ \ \prod_{j=1}^a \tup{x_{i_j} + y_{i_j+\tup{j-1}+c}}.
\end{align*}
This is a polynomial in $R$, and will be called an \emph{$h$-polynomial}.

\end{enumerate}
\end{definition}

It is clear that $h\tup{0,b,c} = 1$ for any $b$ and $c$ (since the set $[b]^0$ has only one element). We furthermore understand $h\tup{a,b,c}$ to be $0$ when $a < 0$. Thus, we have
\begin{equation}
h(a,b,c) = \ive{a=0} \qquad \text{whenever }a\leq0 .
\label{eq.h.triv.aleq0}
\end{equation}
But we also have
\[
h(a,b,c) = 0 \qquad \text{whenever }a>0 \text{ and } b = 0
\]
(since $\ive{b}^a = \ive{0}^a = \varnothing^a = \varnothing$ in this case).
Combining these two equalities, we obtain
\begin{equation}
h(a,b,c) = \ive{a=0} \qquad \text{whenever }b=0 .
\label{eq.h.triv.b=0}
\end{equation}

Another simple example of $h$-polynomials are the $h\tup{1,b,c}$:

\begin{lemma}
    \label{lem.h1bc}
    Let $c \in \ZZ$ and $b \in \NN$. Then,
    \[
    h\tup{1,b,c} = \sum_{i=1}^b x_i + \sum_{j=c+1}^{c+b} y_j.
    \]
\end{lemma}

Let us now derive three recursive formulas for $h$-polynomials.

\begin{lemma}
\label{1hprop}
For all integers $a$ and $c$ and all positive integers $b$, we have
\begin{align}
h\tup{a,b,c} = \tup{x_b+y_{a+b+c-1}} \cdot h\tup{a-1,b,c} + h\tup{a,b-1,c}.
\end{align}
\end{lemma}

\begin{lemma}
\label{3hprop}
For all integers $a$ and $c$ and all nonnegative integers $b$, we have
\begin{align}
h\tup{a, b, c} - h\tup{a, b, c-1}
= \tup{y_{a+b+c-1} - y_c}\cdot h\tup{a-1, b, c}.
\label{eq.3hprop.eq}
\end{align}
\end{lemma}

In what follows, we will use two corollaries that follow easily from the above lemmas:

\begin{corollary}
\label{2hprop}
For all integers $a$ and $c$ and all positive integers $b$, we have
\begin{align}
h\tup{a,b-1,c} = h\tup{a,b,c-1} - \tup{x_b+y_c} \cdot h\tup{a-1,b,c} .
\end{align}
\end{corollary}

\begin{corollary}
\label{3hprop2}
For all integers $a$ and $c$ and all nonnegative integers $b$, we have
\begin{align*}
h\tup{a+1, b, c+1}
= h\tup{a+1, b, c}
+ \tup{y_{a+b+c+1} - y_{c+1}} \cdot h\tup{a, b, c+1}.
\end{align*}
\end{corollary}

\section{\label{sec.jt}The flagged Jacobi--Trudi identity}

\subsection{\label{subsec.jt.fjt}A flagged Jacobi--Trudi identity for $\FSSYT(\mu, \bb)$}

\begin{definition}
    Let $n \in \NN$. Then, the notation $\tup{a_{i,j}}_{i, j \in [n]}$ shall mean the $n\times n$-matrix whose $\tup{i,j}$-th entry is $a_{i,j}$ for all $i, j \in [n]$.
\end{definition}

We now state a crucial formula that helps us rewrite certain sums over flagged semistandard tableaux as determinants:

\begin{proposition}
    \label{prop.flagJT.f}
    Let $\mu = \tup{\mu_1, \mu_2, \ldots, \mu_n}$ be a partition.
    Let $\bb = \tup{b_1, b_2, b_3, \ldots}$ be a weakly increasing flagging.
    Then,
    \begin{align*}
    \sum_{T \in \FSSYT(\mu, \bb)}\ \ \prod_{\tup{i,j}\in Y\tup{\mu}} \tup{x_{T\tup{i,j}} + y_{T\tup{i,j}+j-i}} =
    \det\tup{h\tup{\mu_i - i + j, \ \ b_i, \ \ 1-j}}_{i, j \in \ive{n}}
    \end{align*}
    (where $x_1, x_2, x_3, \ldots$ and $y_1, y_2, y_3, \ldots$ are indeterminates).
\end{proposition}

This proposition is a variant of the flagged Jacobi--Trudi identity \cite[Theorem 1.3]{Wachs85}, and in fact is a particular case of Gessel's and Viennot's generalization of the latter (\cite[Theorem 3]{GesVie89}).
(See Subsection \ref{subsec.gv-apx.1} for how Proposition \ref{prop.flagJT.f} can be derived
from \cite[Theorem 3]{GesVie89}.)
It can also be derived from \cite[second bullet point after Theorem 4.2]{ChLiLo02}
(see Subsection \ref{subsec.gv-apx.2} for a hint).
However, the sake of completeness, we shall give a standalone proof
(which essentially comes from \cite[proof of Theorem 11]{GesVie89}).
Our proof is a close relative of the well-known proof of the Jacobi--Trudi identities using the Lindstr\"om--Gessel--Viennot (LGV) lemma (see, e.g., \cite[First proof of Theorem 7.16.1]{EC2}), and indeed a reader familiar with the latter lemma will find it easy to adapt the latter proof to Proposition~\ref{prop.flagJT.f}.
In order to keep the present work self-contained, we are abstaining from the use of the LGV lemma in favor of a direct combinatorial argument using a sign-reversing involution.

\subsection{\label{subsec.jt.proof}Proof of Proposition \ref{prop.flagJT.f}}

We will actually prove a slight generalization of Proposition
\ref{prop.flagJT.f}:

\begin{theorem}
\label{thm.flagJT.gen}
Let $R$ be a commutative ring. Let $u_{i,j}$ be an
element of $R$ for each pair $\left(  i,j\right)  \in\mathbb{Z}\times
\mathbb{Z}$. For each $b\in\mathbb{N}$ and $q,d\in\mathbb{Z}$, we define an
element $h_{b;\ q}\left[  d\right]  \in R$ by%
\[
h_{b;\ q}\left[  d\right]  :=\sum_{\substack{\left(  i_{1},i_{2},\ldots
,i_{q}\right)  \in\left[  b\right]  ^{q};\\i_{1}\leq i_{2}\leq\cdots\leq
i_{q}}}\ \ \prod_{j=1}^{q}u_{i_{j},\ j-d}.
\]
This sum is understood to be $0$ if $q<0$, and to be $1$ if $q=0$.

Let $\mu=\left(  \mu_{1},\mu_{2},\ldots,\mu_{n}\right)  $ be a partition. Let
$\mathbf{b}=\left(  b_{1},b_{2},b_{3},\ldots\right)  $ be a weakly increasing
flagging. Then,
\begin{equation}
\sum_{T\in\operatorname{FSSYT}\left(  \mu,\mathbf{b}\right)  }\ \ 
\prod_{\tup{i,j}\in Y\tup{\mu}} u_{T\left(  i,j\right)  ,\ j-i}
=\det\left(  h_{b_{i};\ \mu_{i}-i+j}\left[  j\right]  \right)  _{i,j\in\left[  n\right] }.
\nonumber
\end{equation}

\end{theorem}

Our proof of this theorem will follow \cite[proof of Theorem 11]{GesVie89}. We
begin with some notations:%
\footnote{We note that Theorem~\ref{thm.flagJT.gen} can be generalized even further (see Theorem~\ref{thm.flagJT.gen-skew} below), but we shall not need this generality.}

\begin{convention}
For the rest of Subsection \ref{subsec.jt.proof}, we fix a number
$n\in\mathbb{N}$, a partition $\mu=\left(  \mu_{1},\mu_{2},\ldots,\mu
_{n}\right)  $ and a weakly increasing flagging $\mathbf{b}=\left(
b_{1},b_{2},b_{3},\ldots\right)  $. We also fix a commutative ring $R$ and an
element $u_{i,j}$ of $R$ for each pair $\left(  i,j\right)  \in\mathbb{Z}%
\times\mathbb{Z}$.

We let $S_{n}$ denote the $n$-th symmetric group, i.e., the group of
permutations of $\left\{  1,2,\ldots,n\right\}  $. For any permutation
$\sigma\in S_{n}$, we let $\left(  -1\right)  ^{\sigma}$ denote the sign of
$\sigma$. Thus, the Leibniz formula for a determinant says that%
\begin{equation}
\det\left(  a_{i,j}\right)  _{i,j\in\left[  n\right]  }=\sum_{\sigma\in S_{n}%
}\left(  -1\right)  ^{\sigma}\prod_{i=1}^{n}a_{i,\sigma\left(  i\right)  }
\label{eq.det.leibniz}%
\end{equation}
for any $n\times n$-matrix $\left(  a_{i,j}\right)  _{i,j\in\left[  n\right]
}\in R^{n\times n}$.
\end{convention}

\begin{definition}
\label{def.jt.sig-array}
Let $\sigma\in S_{n}$ be any permutation. Then:

\begin{enumerate}
\item[\textbf{(a)}] We say that $\sigma$ is \emph{legitimate} if each
$i\in\left[  n\right]  $ satisfies $\mu_{\sigma\left(  i\right)  }%
-\sigma\left(  i\right)  +i\geq0$.

\item[\textbf{(b)}] If $\sigma$ is legitimate, then we define $P\left(
\sigma\right)  $ to be the set of all pairs $\left(  i,j\right)  $ of positive
integers satisfying $i\in\left[  n\right]  $ and $j\leq\mu_{\sigma\left(
i\right)  }-\sigma\left(  i\right)  +i$. This set $P\left(  \sigma\right)  $
is a diagram. (It has $\mu_{\sigma\left(  i\right)  }-\sigma\left(  i\right)
+i$ boxes in the $i$-th row for each $i\in\left[  n\right]  $; all rows are left-aligned.)

(If $\sigma$ is not legitimate, then we don't define the diagram $P\left(
\sigma\right)  $, since it would have \textquotedblleft negative-length
rows\textquotedblright.)

\item[\textbf{(c)}] If $\sigma$ is legitimate, then a $\sigma$\emph{-array}
will mean a filling $T$ of the diagram $P\left(  \sigma\right)  $ with positive
integers (i.e., a map $T:P\left(  \sigma\right)  \rightarrow\left\{
1,2,3,\ldots\right\}  $) that weakly increase left-to-right along each row (i.e., that
satisfy $T\left(  i,j\right)  \leq T\left(  i,j+1\right)  $ whenever
$\left(  i,j\right)  $ and $\left(  i,j+1\right)  $ are two elements of
$P\left(  \sigma\right)  $). Note that we do not require the entries of $T$ to
strictly increase down the columns.

If $\sigma$ is not legitimate, then we agree that there are no $\sigma$-arrays.

\item[\textbf{(d)}] If $T$ is a $\sigma$-array, then the \emph{weight} of $T$
is defined to be the product $\prod_{(i,j)\in P(\sigma)}u_{T\left(
i,j\right)  ,\ j-i}$. We denote it by $w\left(  T\right)  $.

\item[\textbf{(e)}] We say that a $\sigma$-array $T$ is $\mathbf{b}%
$\emph{-flagged} if it has the property that for each $i\in\left[  n\right]
$, every entry of $T$ in the $i$-th row is $\leq b_{\sigma\left(  i\right)  }$
(that is, if we have $T\left(  i,j\right)  \leq b_{\sigma\left(  i\right)  }$
for every $\left(  i,j\right)  \in P\left(  \sigma\right)  $).
\end{enumerate}
\end{definition}

\newpage

\begin{example}
Assume that $n=3$ and $\mu=\left(  4,2,1\right)  $. We write each permutation
$\sigma\in S_{3}$ as the triple $\left[  \sigma\left(  1\right)
,\ \sigma\left(  2\right)  ,\ \sigma\left(  3\right)  \right]  $ (delimited by
square brackets instead of parentheses in order to avoid confusion with a
partition). The symmetric group $S_{3}$ is generated by the two permutations
$s_{1}:=\left[  2,1,3\right]  $ and $s_{2}:=\left[  1,3,2\right]  $.

The permutation $s_{1}s_{2}s_{1}=\left[  3,2,1\right]  \in S_{3}$ is not
legitimate, because $i=1$ does not satisfy $\mu_{\sigma\left(  i\right)
}-\sigma\left(  i\right)  +i\geq0$ for $\sigma=s_{1}s_{2}s_{1}$ (indeed, for
this $i$ and this $\sigma$, we have $\mu_{\sigma\left(  i\right)  }%
-\sigma\left(  i\right)  +i=\mu_{3}-3+1=1-3+1=-1<0$). For a similar reason,
the permutation $s_{2}s_{1}=\left[  3,1,2\right]  $ is not legitimate either.
All four remaining permutations in $S_{3}$ are legitimate. Here are these four
permutations $\sigma$ along with the corresponding diagrams $P\left(
\sigma\right)  $:%
\[%
\begin{tabular}
[c]{|c|c|}\hline
$\sigma$ & $P\left(  \sigma\right)  $\\\hline\hline
$\operatorname*{id}=\left[  1,2,3\right]  $ &
\multicolumn{1}{|l|}{$\ydiagram{4,2,1}$}\\\hline
$s_{1}=\left[  2,1,3\right]  $ & \multicolumn{1}{|l|}{$\ydiagram{1,5,1}$%
}\\\hline
$s_{2}=\left[  1,3,2\right]  $ & \multicolumn{1}{|l|}{$\ydiagram{4,0,3}$%
}\\\hline
$s_{1}s_{2}=\left[  2,3,1\right]  $ & \multicolumn{1}{|l|}{$\ydiagram{1,0,6}$%
}\\\hline
\end{tabular}
\ \ \ \ \ .
\]

Thus, for instance, an $s_{1}$-array is a filling of $P\left(  s_{1}\right)  $
whose entries weakly increase along each row, i.e., a filling $T$ of the form%
\[
\ytableaushort{a,bcdef,g}
\]
with $b\leq c\leq d\leq e\leq f$. (No conditions on the columns are made!) The
weight of this $s_{1}$-array $T$ is%
\[
w\left(  T\right)  =u_{a,0}u_{b,-1}u_{c,0}u_{d,1}u_{e,2}u_{f,3}u_{g,-2}.
\]
This $s_{1}$-array $T$ is $\mathbf{b}$-flagged if and only if $a\leq b_{2}$
and $f\leq b_{1}$ and $g\leq b_{3}$. (Of course, in order to check that every
entry in the $i$-th row is $\leq b_{\sigma\left(  i\right)  }$, we only need
to check this for the last entry.)
\end{example}

We can use $\mathbf{b}$-flagged $\sigma$-arrays to interpret the right hand
side of Theorem \ref{thm.flagJT.gen}:

\begin{lemma}
\label{lem.flagJT.det=sum}\ \ 

\begin{enumerate}
\item[\textbf{(a)}] For each $\sigma\in S_{n}$, we have%
\begin{equation}
\prod_{i=1}^{n}h_{b_{\sigma\left(  i\right)  };\ \mu_{\sigma\left(  i\right)
}-\sigma\left(  i\right)  +i}\left[  i\right]  =\sum_{\substack{T\text{ is a
}\mathbf{b}\text{-flagged}\\\sigma\text{-array}}}w\left(  T\right)  .
\label{eq.lem.flagJT.det=sum.a}%
\end{equation}

\item[\textbf{(b)}] We have%
\[
\det\left(  h_{b_{i};\ \mu_{i}-i+j}\left[  j\right]  \right)  _{i,j\in\left[
n\right]  }=\sum_{\sigma\in S_{n}}\left(  -1\right)  ^{\sigma}\sum
_{\substack{T\text{ is a }\mathbf{b}\text{-flagged}\\\sigma\text{-array}%
}}w\left(  T\right)  .
\]

\end{enumerate}
\end{lemma}

We bring the claim of Lemma \ref{lem.flagJT.det=sum} \textbf{(b)} into a more
convenient form using the following definition:

\begin{definition}
A \emph{twisted array} will mean a pair $\left(  \sigma,T\right)  $, where
$\sigma\in S_{n}$ and where $T$ is a $\sigma$-array. (Of course, $\sigma$ is
necessarily legitimate if $\left(  \sigma,T\right)  $ is a twisted array,
since otherwise there are no $\sigma$-arrays.)

We say that a twisted array $\left(  \sigma,T\right)  $ is $\mathbf{b}%
$\emph{-flagged} if the $\sigma$-array $T$ is $\mathbf{b}$-flagged.
\end{definition}

Thus, Lemma \ref{lem.flagJT.det=sum} \textbf{(b)} becomes the following:

\begin{lemma}
\label{lem.flagJT.det=sum-twisted}
We have
\[
\det\left(  h_{b_{i};\ \mu_{i}-i+j}\left[  j\right]  \right)  _{i,j\in\left[
n\right]  }=\sum_{\substack{\left(  \sigma,T\right)  \text{ is a }%
\mathbf{b}\text{-flagged}\\\text{twisted array}}}\left(  -1\right)  ^{\sigma
}w\left(  T\right)  .
\]

\end{lemma}

It remains to connect the left hand side of Theorem \ref{thm.flagJT.gen} with
twisted arrays. In a sense, the connection is obvious: A $\mathbf{b}$-flagged
semistandard tableau $T\in\operatorname*{FSSYT}\left(  \mu,\mathbf{b}\right)
$ is a specific kind of $\operatorname*{id}$-array, so that $\left(
\operatorname*{id},T\right)  $ is a twisted array. It just remains to somehow
get rid of all the other twisted arrays $\left(  \sigma,T\right)  $ (i.e.,
those for which $\sigma\neq\operatorname*{id}$, but also those for which
$\sigma=\operatorname*{id}$ but $T$ is not semistandard). This is what we
shall do next. Indeed, we will pair up these unwanted twisted arrays with each
other in such a way that their contributions to the sum%
\[
\sum_{\substack{\left(  \sigma,T\right)  \text{ is a }\mathbf{b}%
\text{-flagged}\\\text{twisted array}}}\left(  -1\right)  ^{\sigma}w\left(
T\right)
\]
cancel out in each pair (i.e., each unwanted twisted array is cancelled out by
its partner in our pairing). This will reduce this sum to only the wanted
part
\[
\sum_{T\in\operatorname*{FSSYT}\left(  \mu,\mathbf{b}\right)  }w\left(
T\right)  ,
\]
which is easily seen to be the left hand side of Theorem \ref{thm.flagJT.gen}.
In combination with Lemma \ref{lem.flagJT.det=sum-twisted}, this will
prove Theorem \ref{thm.flagJT.gen}.

In order to construct our pairing, we introduce the concept of \emph{failure}
of a twisted array. Such a failure will exist exactly when the twisted array
is unwanted.

\begin{definition}
\label{def.flagJT.fail}
Let $\left(  \sigma,T\right)  $ be a twisted array
(i.e., let $\sigma\in S_{n}$, and let $T$ be a $\sigma$-array).

\begin{enumerate}
\item[\textbf{(a)}] A box $\left(  i,j\right)  \in P\left(  \sigma\right)  $
is said to be an \emph{outer failure} of $\left(  \sigma,T\right)  $ if $i>1$
and $\left(  i-1,j\right)  \notin P\left(  \sigma\right)  $. (In other words,
a box $c$ of $P\left(  \sigma\right)  $ is an outer failure of $\left(
\sigma,T\right)  $ if it does not lie in the first row, but its northern
neighbor fails to belong to $P\left(  \sigma\right)  $.)

(This notion does not depend on $T$.)

\item[\textbf{(b)}] A box $\left(  i,j\right)  \in P\left(  \sigma\right)  $
is said to be an \emph{inner failure} of $\left(  \sigma,T\right)  $ if
$\left(  i-1,j\right)  \in P\left(  \sigma\right)  $ and $T\left(
i-1,j\right)  \geq T\left(  i,j\right)  $. (In other words, a box $c$ of
$P\left(  \sigma\right)  $ is an inner failure of $\left(  \sigma,T\right)  $
if its northern neighbor is another box $d\in P\left(  \sigma\right)  $ but
satisfies $T\left(  d\right)  \geq T\left(  c\right)  $.)

\item[\textbf{(c)}] A box $\left(  i,j\right)  \in P\left(  \sigma\right)  $
is said to be a \emph{failure} of $\left(  \sigma,T\right)  $ if it is an
outer failure or an inner failure of $\left(  \sigma,T\right)  $.

\item[\textbf{(d)}] A \emph{leftmost failure} of $\left(  \sigma,T\right)  $
means a failure $\left(  i,j\right)  $ of $\left(  \sigma,T\right)  $ for
which $j$ is minimum (i.e., which lies as far west as a failure of $\left(
\sigma,T\right)  $ can lie).

\item[\textbf{(e)}] A \emph{bottommost leftmost failure} of $\left(
\sigma,T\right)  $ means a leftmost failure $\tup{i,j}$ of $\left(  \sigma,T\right)  $ for
which $i$ is maximum (i.e., which lies as far south as a leftmost failure of
$\left(  \sigma,T\right)  $ can lie). Note that this requirement uniquely
determines the failure (since all leftmost failures of $\left(  \sigma
,T\right)  $ lie in the same column).

\item[\textbf{(f)}] We say that the twisted array $\left(  \sigma,T\right)  $
is \emph{failing} if it has a failure. Otherwise, we say that it is
\emph{unfailing}.
\end{enumerate}
\end{definition}

\begin{example}
\label{exa.flagJT.fail.1}
Let $n=5$ and $\mu=\left(  4,4,3,3,3\right)  $. Let
$\sigma\in S_{5}$ be the permutation that swaps $3$ with $5$ while keeping the
remaining elements unchanged. Let $T$ be the following $\sigma$-array:%
\[
\ytableaushort{1235,2244,5,689,79999}\ \ .
\]
Then, the twisted array $\left(  \sigma,T\right)  $ is failing. It has both
inner and outer failures. Its outer failures are $\left(  4,2\right)  $,
$\left(  4,3\right)  $, $\left(  5,4\right)  $ and $\left(  5,5\right)  $. Its
inner failures are $\left(  2,2\right)  $ (since $T\left(  2,1\right)  \geq
T\left(  2,2\right)  $) as well as $\left(  2,4\right)  $ and $\left(
5,3\right)  $. Its leftmost failures are therefore $\left(  2,2\right)  $ and
$\left(  4,2\right)  $. Hence, its bottommost leftmost failure is $\left(
4,2\right)  $. (Note that the failure $\left(  5,3\right)  $ lies further
south, but does not count as leftmost since it also lies further east. Thus,
the bottommost leftmost failure is not the leftmost bottommost failure!)
\end{example}

\begin{remark}
\label{rmk.flagJT.fail.over1}
If $\left(  i,j\right)  $ is any failure of a
twisted array $\left(  \sigma,T\right)  $, then $i>1$. (This is clear for
outer failures, and follows for inner failures from $\left(  i-1,j\right)  \in
P\left(  \sigma\right)  $.)
\end{remark}

\begin{lemma}
\label{lem.flagJT.fail}
Let $\left(  \sigma,T\right)  $ be a twisted array. Then:

\begin{enumerate}
\item[\textbf{(a)}] If $\sigma\neq\operatorname*{id}$, then $\left(
\sigma,T\right)  $ has an outer failure.

\item[\textbf{(b)}] If $\sigma=\operatorname*{id}$ and $T\notin%
\operatorname*{SSYT}\left(  \mu\right)  $, then $\left(  \sigma,T\right)  $
has an inner failure.

\item[\textbf{(c)}] If $\left(  \sigma,T\right)  $ is unfailing, then
$\sigma=\operatorname*{id}$ and $T\in\operatorname*{SSYT}\left(  \mu\right)  $.

\end{enumerate}
\end{lemma}

As a consequence of Lemma~\ref{lem.flagJT.fail}, we easily obtain:

\begin{lemma}
\label{lem.flagJT.failsum}
We have%
\[
\sum_{\substack{\left(  \sigma,T\right)  \text{ is an}\\\text{unfailing
}\mathbf{b}\text{-flagged}\\\text{twisted array}}}\left(  -1\right)  ^{\sigma
}w\left(  T\right)  =\sum_{T\in\operatorname{FSSYT}\left(  \mu,\mathbf{b}%
\right)  }\ \ \prod_{(i,j)\in Y(\mu)}u_{T\left(  i,j\right)  ,\ j-i}.
\]
\end{lemma}

Now, we promised to pair up the unwanted (i.e., failing) twisted arrays
$\left(  \sigma,T\right)  $. This is achieved by the following transformation:

\begin{definition}
\label{def.flagJT.flip}
Let $\left(  \sigma,T\right)  $ be a failing twisted
array. Let $c=\left(  i,j\right)  \in P\left(  \sigma\right)  $ be the
bottommost leftmost failure of $\left(  \sigma,T\right)  $. To \emph{flip}
$\left(  \sigma,T\right)  $ means to perform the following transformations on
$\sigma$ and on $T$:

\begin{enumerate}
\item We exchange the values of $\sigma$ on $i-1$ and on $i$. (In other words,
we replace $\sigma$ by $\sigma\circ s_{i-1}$, where $s_{i-1}\in S_{n}$ is the
transposition that swaps $i-1$ with $i$.) The resulting permutation will be
called $\sigma^{\prime}$. (Note that $s_{i-1}$ is well-defined, since Remark
\ref{rmk.flagJT.fail.over1} yields $i>1$.)

\item We define the \emph{top floor} of the failure $c$ to be the part of the
$\left(  i-1\right)  $-st row of $T$ that consists of the entries
\[
T\left(  i-1,j\right)  ,\ \ T\left(  i-1,j+1\right)  ,\ \ T\left(
i-1,j+2\right)  ,\ \ \ldots
\]
(that is, the entries from $T\left(  i-1,j\right)  $ on eastwards).

We define the \emph{bottom floor} of the failure $c$ to be the part of the
$i$-th row of $T$ that consists of the entries%
\[
T\left(  i,j+1\right)  ,\ \ T\left(  i,j+2\right)  ,\ \ T\left(  i,j+3\right)
,\ \ \ldots
\]
(that is, the entries from $T\left(  i,j+1\right)  $ on eastwards).

Now, we swap the top floor of the failure $c$ with the bottom floor (i.e., we
move the entries at the positions $\left(  i-1,j\right)  ,\ \ \left(
i-1,j+1\right)  ,\ \ \left(  i-1,j+2\right)  ,\ \ \ldots$ to the positions
$\left(  i,j+1\right)  ,\ \ \left(  i,j+2\right)  ,\ \ \left(  i,j+3\right)
,\ \ \ldots$ (respectively), and vice versa). In other words, we replace the
filling $T$ by the filling $T^{\prime}$ of $P\left(  \sigma^{\prime}\right)  $
which is given by%
\begin{align*}
T^{\prime} \tup{p, q} &=
\begin{cases}
T \tup{i, q+1}, & \text{if } p = i-1 \text{ and } q \geq j; \\
T \tup{i-1, q-1}, & \text{if } p = i \text{ and } q > j; \\
T \tup{p, q}  , & \text{otherwise}
\end{cases} \\
&
\qquad \qquad \qquad \qquad \qquad \qquad
\text{for all } \tup{p, q} \in P\tup{\sigma^{\prime}} .
\end{align*}
The resulting filling $T^{\prime}$ can be shown to be a
$\sigma^{\prime}$-array (see Lemma \ref{lem.flagJT.flip1} \textbf{(a)} for a proof).
\end{enumerate}

\noindent
The resulting pair $\left(  \sigma^{\prime},T^{\prime}\right)  $ will be
denoted by $\operatorname*{flip}\left(  \sigma,T\right)  $, and is said to
arise from $\left(  \sigma,T\right)  $ by \emph{flipping}.
\end{definition}

\begin{figure}
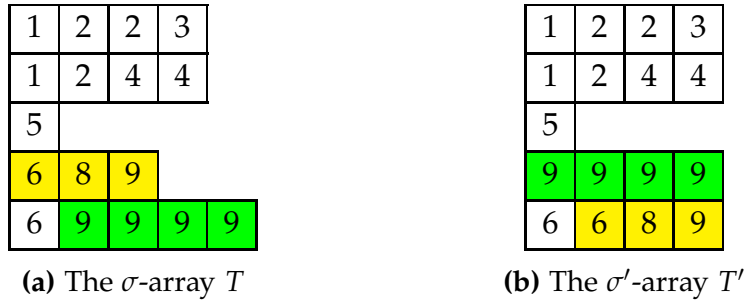

     \centering
     \begin{subfigure}[b]{0.3\textwidth}
         \centering
         \ytableaushort{1223,1244,5,{*(yellow)6}{*(yellow)8}{*(yellow)9},6{*(green)9}{*(green)9}{*(green)9}{*(green)9}}
         \caption[\textbf{(a)}]{The $\sigma$-array $T$}
         \label{fig.flip.1a}
     \end{subfigure}
     \qquad \qquad
     \begin{subfigure}[b]{0.3\textwidth}
         \centering
         \ytableaushort{1223,1244,5,{*(green)9}{*(green)9}{*(green)9}{*(green)9},6{*(yellow)6}{*(yellow)8}{*(yellow)9}}
         \caption{The $\sigma^{\prime}$-array $T^{\prime}$}
         \label{fig.flip.1b}
     \end{subfigure}
     \caption{Arrays for Example \ref{exa.flip.three-flips} \textbf{(a)}. The boxes constituting the top floor of $c$ in $T$ are marked in yellow, whereas the boxes constituting the bottom floor are marked in green. In $T^{\prime}$, the yellow entries and the green entries trade places.}
     \label{fig.flip.1}
\end{figure}

\begin{figure}
     \centering
     \begin{subfigure}[b]{0.3\textwidth}
         \centering
		 \ytableaushort{1223,2244,5,68{*(green)9},79999}
         \caption[\textbf{(a)}]{The $\sigma$-array $T$}
         \label{fig.flip.2a}
     \end{subfigure}
     \qquad \qquad
     \begin{subfigure}[b]{0.3\textwidth}
         \centering
         \ytableaushort{1223,2244,5{*(green)9},68,79999}
         \caption{The $\sigma^{\prime}$-array $T^{\prime}$}
         \label{fig.flip.2b}
     \end{subfigure}
     \caption{Arrays for Example \ref{exa.flip.three-flips} \textbf{(b)}. The colors are as in Figure \ref{fig.flip.1}.}
     \label{fig.flip.2}
\end{figure}

\begin{figure}
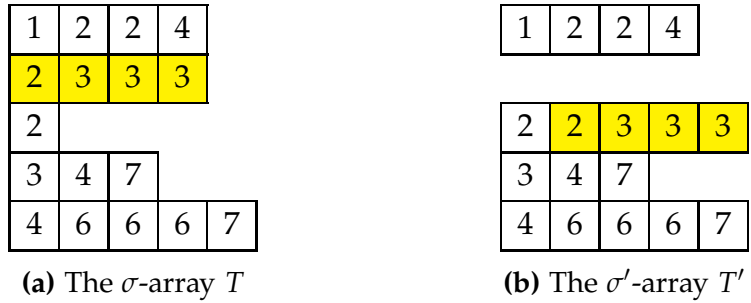

     \centering
     \begin{subfigure}[b]{0.3\textwidth}
         \centering
		 \ytableaushort{1224,{*(yellow)2}{*(yellow)3}{*(yellow)3}{*(yellow)3},2,347,46667}
         \caption[\textbf{(a)}]{The $\sigma$-array $T$}
         \label{fig.flip.3a}
     \end{subfigure}
     \qquad \qquad
     \begin{subfigure}[b]{0.3\textwidth}
         \centering
         \ytableaushort{1224,\none,2{*(yellow)2}{*(yellow)3}{*(yellow)3}{*(yellow)3},347,46667}
         \caption{The $\sigma^{\prime}$-array $T^{\prime}$}
         \label{fig.flip.3b}
     \end{subfigure}
     \caption{Arrays for Example \ref{exa.flip.three-flips} \textbf{(c)}. The colors are as in Figure \ref{fig.flip.1}.}
     \label{fig.flip.3}
\end{figure}

\begin{example}
\label{exa.flip.three-flips}
Let $n$, $\mu$ and $\sigma$ be as in Example \ref{exa.flagJT.fail.1}. We shall
flip three different failing arrays.

\begin{enumerate}

\item[\textbf{(a)}] Let $T$ be the $\sigma$-array
shown in Figure~\ref{fig.flip.1}~\textbf{(a)}.
Its bottommost leftmost failure is $\left(  5,1\right)  $. The top floor of
this failure consists of the entries $6,8,9$ in the $4$-th row, whereas the
bottom floor consists of the entries $9,9,9,9$ in the $5$-th row. Thus,
flipping $\left(  \sigma,T\right)  $ results in the pair $\left(
\sigma^{\prime},T^{\prime}\right)  $, where $\sigma^{\prime}$ is the
permutation $\sigma\circ s_{4}$ (which sends $1,2,3,4,5$ to $1,2,5,3,4$) and
where $T^{\prime}$ is the $\sigma^{\prime}$-array
shown in Figure~\ref{fig.flip.1}~\textbf{(b)}
(obtained from $T$ by swapping the top floor with the bottom floor). Note that
$\left(  \sigma^{\prime},T^{\prime}\right)  $ is still a failing twisted
array, and $\left(  5,1\right)  $ is still its bottommost leftmost failure.

\item[\textbf{(b)}] Let $T$ instead be the $\sigma$-array
shown in Figure~\ref{fig.flip.2}~\textbf{(a)}.
Its bottommost leftmost failure is $\left(  4,2\right)  $. The top floor of
this failure is empty (i.e., consists of no entries at all), whereas the
bottom floor consists of the single entry $9$ in the $4$-th row. Thus,
flipping $\left(  \sigma,T\right)  $ results in the pair $\left(
\sigma^{\prime},T^{\prime}\right)  $, where $\sigma^{\prime}$ is the
permutation $\sigma\circ s_{3}$ (which sends $1,2,3,4,5$ to $1,2,4,5,3$) and
where $T^{\prime}$ is the $\sigma^{\prime}$-array
shown in Figure~\ref{fig.flip.2}~\textbf{(b)}
(obtained from $T$ by swapping the top floor with the bottom floor). Note that
$\left(  \sigma^{\prime},T^{\prime}\right)  $ is still a failing twisted
array, and $\left(  4,2\right)  $ is still its bottommost leftmost failure
(but is now an inner failure).

\item[\textbf{(c)}] Let $T$ instead be the $\sigma$-array
shown in Figure~\ref{fig.flip.3}~\textbf{(a)}.
Its bottommost leftmost failure is $\left(  3,1\right)  $. The top floor of
this failure consists of the entire $2$-nd row, whereas the bottom floor is
empty. Thus, flipping $\left(  \sigma,T\right)  $ results in the pair $\left(
\sigma^{\prime},T^{\prime}\right)  $, where $\sigma^{\prime}$ is the
permutation $\sigma\circ s_{2}$ (which sends $1,2,3,4,5$ to $1,5,2,4,3$) and
where $T^{\prime}$ is the $\sigma^{\prime}$-array
shown in Figure~\ref{fig.flip.3}~\textbf{(b)}
(obtained from $T$ by swapping the top floor with the bottom floor). Note that
$\left(  \sigma^{\prime},T^{\prime}\right)  $ is still a failing twisted
array, and $\left(  3,1\right)  $ is still its bottommost leftmost failure
(but is now an outer failure).

\end{enumerate}
\end{example}

Whatever the name might suggest, flipping a failing twisted array
does not remove the failure; quite to the contrary, the failure is preserved:

\begin{lemma}
\label{lem.flagJT.flip1}
Let $\left(  \sigma,T\right)  $ be a failing twisted array.
Let $\left(  \sigma^{\prime},T^{\prime}\right)  $ be the pair
$\operatorname*{flip}\left(  \sigma,T\right)  $. Then:

\begin{enumerate}
\item[\textbf{(a)}] This pair $\left(  \sigma^{\prime},T^{\prime}\right)  $ is
again a failing twisted array.

\item[\textbf{(b)}] If $c$ is the
bottommost leftmost failure of $\left(  \sigma,T\right)  $, then
the same box $c$ is again the bottommost leftmost failure of
$\left(  \sigma^{\prime},T^{\prime}\right)  $.

\item[\textbf{(c)}] We have $\operatorname*{flip}\left(  \sigma^{\prime
},T^{\prime}\right)  =\left(  \sigma,T\right)  $. In other words, if we flip
$\left(  \sigma,T\right)  $, and then flip the result, then we return back to
$\left(  \sigma,T\right)  $.

\item[\textbf{(d)}] We have $\left(  -1\right)  ^{\sigma^{\prime}}=-\left(
-1\right)  ^{\sigma}$.

\item[\textbf{(e)}] We have $w\left(  T^{\prime}\right)  =w\left(  T\right)  $.

\item[\textbf{(f)}] If $\left(  \sigma,T\right)  $ is $\mathbf{b}$-flagged,
then so is $\left(  \sigma^{\prime},T^{\prime}\right)  $.
\end{enumerate}
\end{lemma}

Lemma \ref{lem.flagJT.flip1} allows us to pair up the failing twisted arrays
with each other, leading to the following:

\begin{lemma}
\label{lem.flagJT.cancel}
We have
\[
\sum_{\substack{\left(  \sigma,T\right)  \text{ is a }\mathbf{b}%
\text{-flagged}\\\text{twisted array}}}\left(  -1\right)  ^{\sigma}w\left(
T\right)  =\sum_{\substack{\left(  \sigma,T\right)  \text{ is an}%
\\\text{unfailing }\mathbf{b}\text{-flagged}\\\text{twisted array}}}\left(
-1\right)  ^{\sigma}w\left(  T\right)  .
\]

\end{lemma}

Theorem \ref{thm.flagJT.gen} now follows by combining the results of Lemma
\ref{lem.flagJT.det=sum-twisted}, Lemma \ref{lem.flagJT.cancel} and Lemma
\ref{lem.flagJT.failsum}.

In turn, deriving Proposition \ref{prop.flagJT.f} from Theorem \ref{thm.flagJT.gen} is
now an exercise in substitution.

\subsection{A formula for $\mathbf{s}_{\lambda}\ive{\nu}$}

Having proved Proposition \ref{prop.flagJT.f}, we can apply it to the
polynomials $\mathbf{s}_{\lambda}\left[  \nu\right]  $ introduced in
Definition~\ref{def.snu}:

\begin{corollary}
\label{cor.slam.det}
Let $n\in\mathbb{N}$. Let $\mu=\left(  \mu_{1},\mu
_{2},\ldots,\mu_{n}\right)  $ be any partition with at most $n$ nonzero
entries. Let $\lambda$ be a further partition. Let $\mathbf{b} = \left(  b_1, b_2, b_3,
\ldots\right)  $ be the flagging induced by $\lambda/\mu$. Then,
\begin{equation}
\mathbf{s}_{\lambda}\left[  \mu\right]  =\det\left(  h(\mu_i-i+j,\ \ b_i,\ \ 1-j)\right)  _{i,j\in\left[  n\right]  }.
\label{eq.snu=.det}
\end{equation}

\end{corollary}

\section{\label{sec.det}Determinantal identities}

In this section, we shall prove some general properties of determinants,
which we will later use in reducing the determinants that arise from Corollary~\ref{cor.slam.det}.
The proofs will require nothing but the Laplace expansion of a determinant.

\begin{convention}
\label{conv.sec-det}
In all proofs in this section, we shall use the following notation:

Let $n\in\mathbb{N}$. For any $k,\ell\in\left[  n\right]  $, and for any
$n\times n$-matrix $A$, we let $A_{\sim k,\sim\ell}$ denote the matrix
obtained from $A$ by removing the $k$-th row and the $\ell$-th column. We also
let $A_{k,\ell}$ denote the $\left(  k,\ell\right)  $-th entry of the matrix
$A$.
\end{convention}

For example,
if $A = \begin{pmatrix} a & b & c \\ d & e & f \\ g & h & i \end{pmatrix}$, then $A_{\sim 2,\sim 3} = \begin{pmatrix} a & b \\ g & h \end{pmatrix}$ and $A_{2, 3} = f$.

Using these notations, we can state Laplace expansion as follows:

\begin{itemize}
\item For any $n\times n$-matrix $A$ and any $k\in\left[  n\right]  $, we can
compute $\det A$ using Laplace expansion along the $k$-th row:%
\begin{equation}
\det A=\sum_{\ell=1}^{n}\left(  -1\right)  ^{k+\ell}A_{k,\ell}\det\left(
A_{\sim k,\sim\ell}\right)  .
\label{pf.lem.det-u.lap}
\end{equation}

\item For any $n\times n$-matrix $A$ and any $\ell\in\left[  n\right]  $, we
can compute $\det A$ using Laplace expansion along the $\ell$-th column:%
\begin{equation}
\det A=\sum_{k=1}^{n}\left(  -1\right)  ^{k+\ell}A_{k,\ell}\det\left(  A_{\sim
k,\sim\ell}\right)  .
\label{pf.lem.det-u.lapc}
\end{equation}

\end{itemize}

We now move on to less well-trodden ground.
We begin with the following general lemma about determinants (the $r=1$ case of \cite[\S 319]{MuiMet60}):

\begin{lemma}
\label{determinant.sum}

Let $P$ and $Q$ be two $n\times n$-matrices over
some commutative ring.

For each $k\in\left[  n\right]  $, we let 
$P\underset{\operatorname*{row}}{\overset{k}{\leftarrow}}Q$
denote the $n\times n$-matrix
that is obtained from $P$ by replacing the $k$-th row by the $k$-th row of $Q$.
(That is, the $\tup{i,j}$-th entry of this matrix is $\begin{cases} P_{i,j}, & \text{ if } i \neq k; \\ Q_{i,j}, & \text{ if } i = k \end{cases}$ for every $i, j \in \ive{n}$.)

For each $k\in\left[  n\right]  $, we let
$P\underset{\operatorname*{col}}{\overset{k}{\leftarrow}}Q$
denote the matrix
that is obtained from $P$ by replacing the $k$-th column by the $k$-th column of $Q$.
(That is, the $\tup{i,j}$-th entry of this matrix is $\begin{cases} P_{i,j}, & \text{ if } j \neq k; \\ Q_{i,j}, & \text{ if } j = k \end{cases}$ for every $i, j \in \ive{n}$.)

Then,
\[
\sum_{k=1}^{n}\det\left(  P\underset{\operatorname*{row}%
}{\overset{k}{\leftarrow}}Q\right)  =\sum_{k=1}^{n}\det\left(
P\underset{\operatorname*{col}}{\overset{k}{\leftarrow}}Q\right)  .
\]

\end{lemma}

\begin{example}
For $n=3$, this is saying that
\begin{align*}
&  \det
\begin{pmatrix}
Q_{1,1} & Q_{1,2} & Q_{1,3}\\
P_{2,1} & P_{2,2} & P_{2,3}\\
P_{3,1} & P_{3,2} & P_{3,3}
\end{pmatrix}
+\det
\begin{pmatrix}
P_{1,1} & P_{1,2} & P_{1,3}\\
Q_{2,1} & Q_{2,2} & Q_{2,3}\\
P_{3,1} & P_{3,2} & P_{3,3}
\end{pmatrix}
+\det
\begin{pmatrix}
P_{1,1} & P_{1,2} & P_{1,3}\\
P_{2,1} & P_{2,2} & P_{2,3}\\
Q_{3,1} & Q_{3,2} & Q_{3,3}
\end{pmatrix}
\\
&  =\det
\begin{pmatrix}
Q_{1,1} & P_{1,2} & P_{1,3}\\
Q_{2,1} & P_{2,2} & P_{2,3}\\
Q_{3,1} & P_{3,2} & P_{3,3}
\end{pmatrix}
+\det
\begin{pmatrix}
P_{1,1} & Q_{1,2} & P_{1,3}\\
P_{2,1} & Q_{2,2} & P_{2,3}\\
P_{3,1} & Q_{3,2} & P_{3,3}
\end{pmatrix}
+\det
\begin{pmatrix}
P_{1,1} & P_{1,2} & Q_{1,3}\\
P_{2,1} & P_{2,2} & Q_{2,3}\\
P_{3,1} & P_{3,2} & Q_{3,3}
\end{pmatrix}
.
\end{align*}

\end{example}

The interested reader can find the general case of \cite[\S 319]{MuiMet60} in Section~\ref{sec.odds} below (as Lemma~\ref{determinant.sumr}); but Lemma~\ref{determinant.sum} will suffice for us.
\medskip

We recall the Iverson bracket notation.
The following lemma is an easy consequence of Lemma \ref{determinant.sum}.

\begin{lemma}
\label{lem.det-u-cor}
Let $n$ be a positive integer. Let $R$ be a commutative ring.

Let $u_{i,j}$ be an element of $R$ for each $i\in\left[  n\right]  $ and each
$j\in\left[  n+1\right]  $. Then,%
\[
\sum_{k=1}^{n}\det\left(  u_{i,j+\left[  k=i\right]  }\right)  _{i,j\in\left[
n\right]  }=\det\left(  u_{i,j+\left[  n=j\right]  }\right)  _{i,j\in\left[
n\right]  }.
\]

\end{lemma}

\begin{example}
For $n=3$, the claim of Lemma \ref{lem.det-u-cor} is%
\begin{align*}
& \det\left(
\begin{array}
[c]{ccc}%
u_{1,2} & u_{1,3} & u_{1,4}\\
u_{2,1} & u_{2,2} & u_{2,3}\\
u_{3,1} & u_{3,2} & u_{3,3}%
\end{array}
\right)  +\det\left(
\begin{array}
[c]{ccc}%
u_{1,1} & u_{1,2} & u_{1,3}\\
u_{2,2} & u_{2,3} & u_{2,4}\\
u_{3,1} & u_{3,2} & u_{3,3}%
\end{array}
\right)  +\det\left(
\begin{array}
[c]{ccc}%
u_{1,1} & u_{1,2} & u_{1,3}\\
u_{2,1} & u_{2,2} & u_{2,3}\\
u_{3,2} & u_{3,3} & u_{3,4}%
\end{array}
\right)  \\
& =\det\left(
\begin{array}
[c]{ccc}%
u_{1,1} & u_{1,2} & u_{1,4}\\
u_{2,1} & u_{2,2} & u_{2,4}\\
u_{3,1} & u_{3,2} & u_{3,4}%
\end{array}
\right)  .
\end{align*}

\end{example}

\begin{lemma}
\label{lem.det-uc}
Let $n$ be a positive integer. Let $R$ be a commutative ring.

Let $u_{i,j}$ be an element of $R$ for each $i\in\left[  n\right]  $ and each
$j\in\left[  n+1\right]  $. Let $p_{1},p_{2},\ldots,p_{n}$ be $n$ further
elements of $R$.

Then,%
\begin{align*}
& \sum_{k=1}^{n}\det\left(  u_{i,j+\left[  k=i\right]  }-p_{i}u_{i,j}\left[
k=i\right]  \right)  _{i,j\in\left[  n\right]  }\\
& =\det\left(  u_{i,j+\left[  n=j\right]  }\right)  _{i,j\in\left[  n\right]
}-\left(  \sum_{k=1}^{n}p_{k}\right)  \det\left(  u_{i,j}\right)
_{i,j\in\left[  n\right]  }.
\end{align*}

\end{lemma}

\begin{example}
For $n=3$, the claim of Lemma \ref{lem.det-uc} is%
\begin{align*}
& \det\left(
\begin{array}
[c]{ccc}%
u_{1,2}-p_{1}u_{1,1} & u_{1,3}-p_{1}u_{1,2} & u_{1,4}-p_{1}u_{1,3}\\
u_{2,1} & u_{2,2} & u_{2,3}\\
u_{3,1} & u_{3,2} & u_{3,3}%
\end{array}
\right)  \\
& \ \ \ \ \ \ \ \ \ \ +\det\left(
\begin{array}
[c]{ccc}%
u_{1,1} & u_{1,2} & u_{1,3}\\
u_{2,2}-p_{2}u_{2,1} & u_{2,3}-p_{2}u_{2,2} & u_{2,4}-p_{2}u_{2,3}\\
u_{3,1} & u_{3,2} & u_{3,3}%
\end{array}
\right)  \\
& \ \ \ \ \ \ \ \ \ \ +\det\left(
\begin{array}
[c]{ccc}%
u_{1,1} & u_{1,2} & u_{1,3}\\
u_{2,1} & u_{2,2} & u_{2,3}\\
u_{3,2}-p_{3}u_{3,1} & u_{3,3}-p_{3}u_{3,2} & u_{3,4}-p_{3}u_{3,3}%
\end{array}
\right)  \\
& =\det\left(
\begin{array}
[c]{ccc}%
u_{1,1} & u_{1,2} & u_{1,4}\\
u_{2,1} & u_{2,2} & u_{2,4}\\
u_{3,1} & u_{3,2} & u_{3,4}%
\end{array}
\right)  -\left(  p_{1}+p_{2}+p_{3}\right)  \det\left(
\begin{array}
[c]{ccc}%
u_{1,1} & u_{1,2} & u_{1,3}\\
u_{2,1} & u_{2,2} & u_{2,3}\\
u_{3,1} & u_{3,2} & u_{3,3}%
\end{array}
\right)  .
\end{align*}

\end{example}

As a curiosity, we observe that Lemma~\ref{lem.det-uc} would also hold if the $p_i$ on the left hand side were replaced by $p_j$; see Lemma~\ref{lem.det-u} below.
\medskip

Finally, we state a particularly simple and well-known property of determinants:

\begin{lemma}
\label{lem.det.lastrow=01}Let $n$ be a positive integer. Let $R$ be a
commutative ring.

Let $a_{i,j}$ be an element of $R$ for each $i,j\in\left[  n\right]  $. Assume
that
\begin{equation}
a_{n,\ell}=0\qquad\text{for each }\ell\in\left[  n-1\right]
.\label{eq.lem.det.lastrow=01.ass}%
\end{equation}
Then,%
\[
\det\left(  a_{i,j}\right)  _{i,j\in\left[  n\right]  }=a_{n,n}\cdot
\det\left(  a_{i,j}\right)  _{i,j\in\left[  n-1\right]  }.
\]

\end{lemma}

\section{\label{sec.further}Combinatorial lemmas}

We shall next prove some combinatorial lemmas.

\subsection{The numbers $\ell_i$, $m_i$, $\ell^t_i$, $m^t_i$}

\begin{convention}
    \label{conv.further.convs}
    From now on \textbf{for the rest of this paper},
    we fix two partitions $\lambda$ and $\mu$.
    (We do not require $\lambda \supseteq \mu$.)

    We furthermore set
    \begin{align*}
    & \ell_i := \lambda_i - i,
    \qquad \qquad \qquad
    m_i := \mu_i - i, \\
    & \ell^t_i := \lambda^t_i - i
    \qquad \text{and} \qquad
    m^t_i := \mu^t_i - i
    \qquad \text{for all } i \geq 1.
    \end{align*}
\end{convention}

Thus, Definition~\ref{def.Delta-lambda} yields
\begin{align*}
    \Delta\tup{\lambda} = \set{\ell_1, \ell_2, \ell_3, \ldots},
    &\qquad \qquad
    \Delta\tup{\mu} = \set{m_1, m_2, m_3, \ldots}, \\
    \Delta\tup{\lambda^t} = \set{\ell^t_1, \ell^t_2, \ell^t_3, \ldots},
    &\qquad \qquad
    \Delta\tup{\mu^t} = \set{m^t_1, m^t_2, m^t_3, \ldots}.
\end{align*}

\begin{lemma}
    \label{lem.l-decrease}
    We have
    \begin{align*}
        & \ell_1 > \ell_2 > \ell_3 > \cdots \qquad \text{ and} \\
        & m_1 > m_2 > m_3 > \cdots \qquad \text{ and} \\
        & \ell^t_1 > \ell^t_2 > \ell^t_3 > \cdots \qquad \text{ and} \\
        & m^t_1 > m^t_2 > m^t_3 > \cdots .
    \end{align*}
\end{lemma}

\subsection{$\operatorname*{ER}\left(  \mu\right)  $ and $\mu^{+k}$}

Next, we shall get a better understanding of the partitions $\nu$ that satisfy
$\mu\lessdot\nu$.

\begin{definition}
\ \ 

\begin{enumerate}
\item[\textbf{(a)}] If $k$ is a positive integer, then $\mu^{+k}$ shall denote
the sequence%
\[
\left(  \mu_{1},\ \mu_{2},\ \ldots,\ \mu_{k-1},\ \mu_{k}+1,\ \mu_{k+1}%
,\ \mu_{k+2},\ \ldots\right)  ,
\]
which is obtained from $\mu$ by incrementing the $k$-th entry by $1$. Note
that this sequence $\mu^{+k}$ is a partition if and only if $k=1$ or $\mu
_{k}\neq\mu_{k-1}$.

\item[\textbf{(b)}] We let $\operatorname*{ER}\left(  \mu\right)  $ be the set
of all positive integers $k$ that satisfy $k=1$ or $\mu_{k}\neq\mu_{k-1}$.
\end{enumerate}
\end{definition}

Thus, $\operatorname*{ER}\left(  \mu\right)  $ is the set of all positive
integers $k$ for which $\mu^{+k}$ is a partition. (The notation
$\operatorname*{ER}$ is short for \textquotedblleft extensible
rows\textquotedblright, since $\operatorname*{ER}\left(  \mu\right)  $
consists exactly of those $k\geq1$ such that the $k$-th row of $Y\left(
\mu\right)  $ can be extended by a new box at its end and the result will
still be the diagram of a partition.)

\begin{example}
If $\mu=\left(  5,2,2,1\right)  $, then $\operatorname*{ER}\left(  \mu\right)
=\left\{  1,2,4,5\right\}  $ and
\begin{align*}
\mu^{+1}  & =\left(  6,2,2,1\right)  ,\ \ \ \ \ \ \ \ \ \ \mu^{+2}=\left(
5,3,2,1\right)  ,\\
\mu^{+4}  & =\left(  5,2,2,2\right)  ,\ \ \ \ \ \ \ \ \ \ \mu^{+5}=\left(
5,2,2,1,1\right)  .
\end{align*}
(Do not forget that $\mu_{5}=0$ can be incremented by $1$, so that
$5\in\operatorname*{ER}\left(  \mu\right)  $.)
\end{example}

The following two lemmas are very easy:

\begin{lemma}
\label{lem.nu=mu+k}
The partitions $\nu$ that satisfy $\mu\lessdot\nu$ are
precisely the partitions $\mu^{+k}$ for the elements $k\in\operatorname*{ER}%
\left(  \mu\right)  $.
\end{lemma}

\begin{lemma}
\label{lem.mu+ki}
Let $k\geq1$ and $i\geq1$. Then, the $i$-th entry of the
sequence $\mu^{+k}$ is $\tup{\mu^{+k}}_i = \mu_{i}+\left[  k=i\right]  $.
\end{lemma}

\subsection{Some easy properties of $\Delta\left(\lambda\right)$ and $\Delta\left(\mu\right)$}

We will now state some further easy lemmas, which will help us simplify some sums later on.

\begin{lemma}
\label{lem.Deltas.if-k-then-i}
Let $n\in\mathbb{N}$ be such that
$\mu=\left(  \mu_{1},\mu_{2},\ldots,\mu_{n}\right)  $
(that is, $\mu_{n+1}=0$).
Let $j \in \ive{n}$, and let $k$ be a positive integer that satisfies $\ell_j = m_k$. 
Then, $k \in \left[  n\right]  $.
\end{lemma}

%
%

\begin{lemma}
\label{lem.Deltas.sum1i}
Let $n\in\mathbb{N}$ be such that
$\mu=\left(  \mu_{1},\mu_{2},\ldots,\mu_{n}\right)  $
(that is, $\mu_{n+1}=0$).

Let $j\in\left[  n\right]  $ be such that $\ell_{j}\in\Delta\left(
\mu\right)  $. Then,
\[
\sum_{\substack{i\in\left[  n\right]  ;\\\ell_{j}=m_{i}}}1=1.
\]

\end{lemma}

\begin{lemma}
\label{lem.Deltas.if-not-then-k}
Let $n\in\mathbb{N}$ be such that $\lambda=\left(  \lambda_{1},\lambda
_{2},\ldots,\lambda_{n}\right)  $ (that is, $\lambda_{n+1}=0$) and
$\mu=\left(  \mu_{1},\mu_{2},\ldots,\mu_{n}\right)  $ (that is, $\mu_{n+1}%
=0$).
Let $k$ be a positive integer that satisfies $\ell_{k}\notin \Delta\left(  \mu\right)  $.
Then, $k\in\left[  n\right]  $.
\end{lemma}

\begin{lemma}
\label{lem.Deltas.f=g}
Let $n\in\mathbb{N}$ be such that $\lambda=\left(  \lambda_{1},\lambda
_{2},\ldots,\lambda_{n}\right)  $ (that is, $\lambda_{n+1}=0$) and
$\mu=\left(  \mu_{1},\mu_{2},\ldots,\mu_{n}\right)  $ (that is, $\mu_{n+1}=0$).

Let $R$ be an additive group. Let $f,g:\left[  n\right]  \rightarrow R$ be two
maps such that if $i,j\in\left[  n\right]  $ satisfy $m_{i}=\ell_{j}$, then
\begin{equation}
f\left(  i\right)  =g\left(  j\right)  .
\label{eq.lem.Deltas.f=g.as}
\end{equation}
Then,%
\[
\sum_{\substack{i\in\left[  n\right]  ;\\m_{i}\in\Delta\left(  \lambda\right)
}}f\left(  i\right)  =\sum_{\substack{j\in\left[  n\right]  ;\\\ell_{j}%
\in\Delta\left(  \mu\right)  }}g\left(  j\right)  .
\]

\end{lemma}

\subsection{Some properties of the flagging induced by $\lambda / \mu$}

The next few lemmas involve the flagging induced by $\lambda / \mu$. (As we recall, it was defined in Definition~\ref{def.flagging-of-lm} to be the flagging $\bb = \tup{b_1, b_2, b_3, \ldots}$, where $b_i$ is the largest $k \geq 0$ satisfying $\lambda_k - k \geq \mu_i - i$. Here, $\lambda_0$ is understood to be $+\infty$.)

\begin{lemma}
\label{lem.bj=i}
Let $\mathbf{b}=\left(  b_{1},b_{2},b_{3},\ldots\right)  $ be
the flagging induced by $\lambda/\mu$. Let $i$ and $j$ be positive integers
such that $m_i = \ell_j$. Then, $b_i = j$.
\end{lemma}

\begin{lemma}
\label{lem.Deltas.x}
Let $\mathbf{b}=\left(  b_{1},b_{2},b_{3},\ldots\right)  $ be the flagging
induced by $\lambda/\mu$.

Let $n\in\mathbb{N}$ be such that $\lambda=\left(  \lambda_{1},\lambda
_{2},\ldots,\lambda_{n}\right)  $ (that is, $\lambda_{n+1}=0$) and
$\mu=\left(  \mu_{1},\mu_{2},\ldots,\mu_{n}\right)  $ (that is, $\mu_{n+1}%
=0$). Then,%
\[
\sum_{i=1}^{n}x_{i}-\sum_{\substack{i\in\left[  n\right]  ;\\m_{i}\in
\Delta\left(  \lambda\right)  }}x_{b_{i}}=\sum_{\substack{k\in\ive{n}; \\\ell
_{k}\notin\Delta\left(  \mu\right)  }}x_{k}.
\]

\end{lemma}

\begin{lemma}
\label{lem.Deltas.mi+1+bi}
Let $\mathbf{b}=\left(  b_{1},b_{2},b_{3},\ldots\right)  $ be the flagging
induced by $\lambda/\mu$.
Let $i$ be a positive integer that satisfies $m_{i}%
\notin\Delta\left(  \lambda\right)  $. Then,
\[
m_{i}+1+b_{i}\geq1 \qquad \text{ and } \qquad
\ell_{m_{i}+1+b_{i}}^{t}=-1-m_{i} .
\]
\end{lemma}

The next lemma specifically concerns the case when the partition
$\mu$ has two equal entries (i.e., when $\mu_{j-1} = \mu_j$).

\begin{lemma}
\label{lem.ER-bb}
Let $\mathbf{b}=\left(  b_{1},b_{2},b_{3},\ldots\right)  $ be
the flagging induced by $\lambda/\mu$. Let $j>1$ be an integer such that
$\mu_{j-1}=\mu_{j}$. Then:

\begin{enumerate}
\item[\textbf{(a)}] If $m_j\notin\Delta\left(  \lambda\right)  $, then
$b_{j}=b_{j-1}$.

\item[\textbf{(b)}] If $m_j\in\Delta\left(  \lambda\right)  $, then
$b_{j}=b_{j-1}+1$.
\end{enumerate}
\end{lemma}

\subsection{The flagging induced by $\lambda/\mu^{+k}$}

How does the flagging induced by $\lambda/\mu$ change when we replace $\mu$ by
a partition of the form $\mu^{+k}$ (that is, when we add $1$ to the $k$-th
entry of $\mu$)? The following lemma gives a simple answer:

\begin{lemma}
\label{lem.b-vs-bk}
Let $k\in\operatorname*{ER}\left(  \mu\right)  $. Let
$\mathbf{b}=\left(  b_{1},b_{2},b_{3},\ldots\right)  $ be the flagging induced
by $\lambda/\mu$. Let $\mathbf{b}^{\ast}=\left(  b_{1}^{\ast},b_{2}^{\ast
},b_{3}^{\ast},\ldots\right)  $ be the flagging induced by $\lambda/\mu^{+k}$. Then:

\begin{enumerate}
\item[\textbf{(a)}] If $m_k\notin\Delta\left(  \lambda\right)  $, then
$b_{i}^{\ast}=b_{i}$ for each $i\geq1$.

\item[\textbf{(b)}] If $m_k\in\Delta\left(  \lambda\right)  $, then
$b_{i}^{\ast}=b_{i}-\left[  k=i\right]  $ for each $i\geq1$.

\item[\textbf{(c)}] If $i$ is a positive integer such that $i \neq k$, then
$b_{i}^{\ast}=b_{i}$.
\end{enumerate}
\end{lemma}

\section{\label{sec.kpf}A variant of the Konvalinka recursion}

We shall now approach a variant of the Konvalinka recursion, which follows fairly easily from the results of the previous sections.

\begin{convention}
In addition to Convention~\ref{conv.further.convs}, we agree on the following:

\begin{enumerate}
\item[\textbf{(a)}]
We let $\mathbf{b} = \left(  b_{1},b_{2},b_{3},\ldots\right)  $ be the flagging
induced by $\lambda/\mu$.

\item[\textbf{(b)}]
We fix a positive integer $n$ that is so large that $\lambda_{n}=0$ and $\mu_{n}=0$.
Thus, $\lambda=\left(  \lambda_{1},\lambda_{2},\ldots,\lambda_{n}\right)  $
and $\mu=\left(  \mu_{1},\mu_{2},\ldots,\mu_{n}\right)  =\left(  \mu_{1}%
,\mu_{2},\ldots,\mu_{n-1}\right)  $.

\end{enumerate}
\end{convention}

From Lemma~\ref{lem.flagging-of-lm.0}, we obtain $b_n = n$.

We observe that the outer sum on the right hand side of Theorem
\ref{mainKonvalinka} can be rewritten by dropping the $\nu\subseteq\lambda$
condition under the summation sign:

\begin{lemma}
\label{lem.nu-sub-lambda-irrelevant}
We have
\[
\sum_{\mu\lessdot\nu\subseteq\lambda}\mathbf{s}_{\lambda}\left[  \nu\right]
=\sum_{\mu\lessdot\nu}\mathbf{s}_{\lambda}\left[  \nu\right]
\]
(where $\nu$ is the summation index in both outer sums, and is supposed to be
a partition).
\end{lemma}

Using Lemma \ref{lem.nu=mu+k} and Lemma \ref{lem.nu-sub-lambda-irrelevant}, we
can easily see the following:

\begin{lemma}
\label{lem.konval-RHS-as-sum}
We have
\[
\sum_{\mu\lessdot\nu\subseteq\lambda}\mathbf{s}_{\lambda}\left[  \nu\right]
=\sum_{k=1}^{n}\mathbf{s}_{\lambda}\left[  \mu^{+k}\right]  .
\]

\end{lemma}

We continue with some lemmas that help us understand the right hand side.

\begin{lemma}
\label{lem.mu+k.sdet}
Let $k\in\left[  n\right]  $. Then:

\begin{enumerate}
\item[\textbf{(a)}] If $m_{k}\notin\Delta\left(  \lambda\right)  $, then%
\[
\mathbf{s}_{\lambda}\left[  \mu^{+k}\right]  =\det\left(  h(\mu_{i}%
-i+j+\left[  k=i\right]  ,\ \ b_{i},\ \ 1-j)\right)  _{i,j\in\left[  n\right]
}.
\]

\item[\textbf{(b)}] If $m_{k}\in\Delta\left(  \lambda\right)  $, then%
\[
\mathbf{s}_{\lambda}\left[  \mu^{+k}\right]  =\det\left(  h(\mu_{i}%
-i+j+\left[  k=i\right]  ,\ \ b_{i}-\left[  k=i\right]  ,\ \ 1-j)\right)
_{i,j\in\left[  n\right]  }.
\]

\end{enumerate}
\end{lemma}

\begin{convention}
\label{conv.u-and-p}
Set
\[
u_{i,j}:=h\left(  \mu_{i}-i+j,\ \ b_{i},\ \ 1-j\right)
\qquad \qquad
\text{for all $i\in\left[  n\right]  $ and $j\in\left[  n+1\right]  $.}
\]

Furthermore, set
\[
p_{i}:=%
\begin{cases}
x_{b_{i}}, & \text{if }m_{i}\in\Delta\left(  \lambda\right)  ;\\
-y_{m_{i}+1+b_{i}}, & \text{if }m_{i}\notin\Delta\left(  \lambda\right)
\end{cases}
\qquad \qquad
\text{for any $i\in\left[  n\right]  $.}
\]
\end{convention}

\begin{lemma}
\label{lem.k-th-det-0}
Define $u_{i,j}$ as in Convention~\ref{conv.u-and-p}.
Then,%
\begin{align}
\mathbf{s}_{\lambda}\left[  \mu\right]
&  =\det\left(  u_{i,j}\right)
_{i,j\in\left[  n\right]  }
\label{eq.lem.k-th-det-0.0}\\
&  =\det\left(  u_{i,j}\right)  _{i,j\in\left[  n-1\right]  }.
\label{eq.lem.k-th-det-0.-1}
\end{align}

\end{lemma}

\begin{lemma}
\label{lem.k-th-det-1}
Define $u_{i,j}$ and $p_{i}$ as in
Convention~\ref{conv.u-and-p}. Let $k\in\left[  n\right]  $. Then,%
\[
\mathbf{s}_{\lambda}\left[  \mu^{+k}\right]  =\det\left(  u_{i,j+\left[
k=i\right]  }-p_{i}u_{i,j}\left[  k=i\right]  \right)  _{i,j\in\left[
n\right]  }.
\]

\end{lemma}

We can now reap some rewards:

\begin{lemma}
\label{lem.k-th-det-2}
Define $u_{i,j}$ and $p_{i}$ as in Convention~\ref{conv.u-and-p}. Then,
\[
\sum_{k=1}^{n}\mathbf{s}_{\lambda}\left[  \mu^{+k}\right]  =\det\left(
u_{i,j+\left[  n=j\right]  }\right)  _{i,j\in\left[  n\right]  }-\left(
\sum_{k=1}^{n}p_{k}\right)  \mathbf{s}_{\lambda}\left[  \mu\right]  .
\]

\end{lemma}

Now, let us simplify the $\det$ and $\sum$ terms on the right hand side of Lemma \ref{lem.k-th-det-2}.

\begin{lemma}
\label{lem.k-th-det-3}
Define $u_{i,j}$ as in Convention~\ref{conv.u-and-p}.
Then,
\[
\det\left(  u_{i,j+\left[  n=j\right]  }\right)  _{i,j\in\left[  n\right]
}=\left(  \sum_{i=1}^{n}x_{i}\right)  \cdot\mathbf{s}_{\lambda}\left[
\mu\right]  .
\]

\end{lemma}

\begin{lemma}
\label{lem.k-th-det-4}
Define $p_{i}$ as in Convention~\ref{conv.u-and-p}.
Then,
\[
\sum_{i=1}^{n}x_{i}-\sum_{k=1}^{n}p_{k}
= \sum_{\substack{k \in \ive{n};\\ \ell_k \notin \Delta\tup{\mu}}} x_k + \sum_{\substack{i \in \ive{n};\\ m_i \notin \Delta\tup{\lambda}}} y_{m_i+1+b_i} .
\]

\end{lemma}

We are now ready to prove a simplified version of the Konvalinka recursion (which, as we recall, will be the version that we need for our proof of Theorem~\ref{thm.main}):

\begin{lemma}
\label{lem.konvalinka-bi}
We have
\begin{align*}
\left(\sum_{\substack{k \in \ive{n};\\ \ell_k \notin \Delta\tup{\mu}}} x_k + \sum_{\substack{i \in \ive{n};\\ m_i \notin \Delta\tup{\lambda}}} y_{m_i+1+b_i}\right) \mathbf{s}_\lambda\ive{\mu}
= \sum_{\mu \lessdot \nu \subseteq \lambda} \mathbf{s}_\lambda\ive{\nu} .
\end{align*}
Here, the sum on the right hand side ranges over all partitions $\nu$ that satisfy $\mu \lessdot \nu \subseteq \lambda$.
\end{lemma}

\section{\label{sec.pf-hlf}Proof of the Naruse--Pak--Postnikov formula}

\begin{convention}
For the rest of this section, we introduce the following notations:
\begin{enumerate}

    \item We fix a skew partition $\lambda / \mu$ (that is, two partitions $\lambda$ and $\mu$ satisfying $\lambda \supseteq \mu$).

    \item We remind the reader that
    \begin{align*}
    & \ell_i := \lambda_i - i,
    \qquad \qquad \qquad
    m_i := \mu_i - i, \\
    & \ell^t_i := \lambda^t_i - i
    \qquad \text{and} \qquad
    m^t_i := \mu^t_i - i
    \qquad \text{for all } i \geq 1.
    \end{align*}

    \item We fix a positive integer $n$ that is large enough to satisfy $\lambda=\left(  \lambda_{1},\lambda_{2},\ldots,\lambda_{n-1}\right)  $ (that is, $\lambda_n=0$) and
    $\mu=\left(  \mu_{1},\mu_{2},\ldots,\mu_{n-1}\right)  $ (that is, $\mu_n=0$).

    \item For each integer $i$, we set
    \begin{align*}
        w_i := \sum_{k = -n}^i z_k = z_{-n} + z_{-n+1} + \cdots + z_i.
    \end{align*}

    \item We shall use all the notations from Convention~\ref{conv.z-and-zT} as well as the algebraic hook lengths $h_{\lambda}\tup{c;z}$ defined in Theorem~\ref{thm.pak-1}.
\end{enumerate}
\end{convention}

\begin{lemma}
\label{lem.pf-hlf.w-w}
    Let $\tup{i,j} \in Y\tup{\lambda}$. Then
    \begin{align*}
        w_{\ell_i}-w_{-\ell_j^t-1} = h_{\lambda}\tup{\tup{i,j}; z}.
    \end{align*}
\end{lemma}

\begin{lemma}
\label{section14.wY}
    Let $i \geq 1$ be a positive integer. Then
    \begin{align*}
        w_{\ell_i}-w_{m_i} = \sum_{\substack{j \geq 1; \\ \tup{i,j}\in Y\tup{\lambda/\mu}}} z_{j-i}.
    \end{align*}
\end{lemma}

\begin{lemma}
    \label{lem.pf-hlf.4}
    We have
    \begin{align*}
        \sum_{\substack{k \in \ive{n}; \\ \ell_k \notin \Delta\tup{\mu}}} w_{\ell_k}
        - \sum_{\substack{i \in \ive{n}; \\ m_i \notin \Delta\tup{\lambda}}} w_{m_i}
        = \sum_{\tup{i,j} \in Y\tup{\lambda/\mu}} z_{j-i}.
    \end{align*}
\end{lemma}

\begin{lemma}
\label{rec.2}
Assume that $\mu \neq \lambda$. Then
\begin{align*}
    \sum_{E\in\mathcal{E}\left(  \lambda/\mu\right)}\ \ \prod_{c\in Y\left(  \lambda\right)  \setminus E}\dfrac{1}{h_{\lambda}\left(  c; z\right)  }
= \frac{1}{\sum_{\tup{i,j} \in Y(\lambda/\mu)}z_{j-i}}\ \ \sum_{\mu \lessdot \nu \subseteq  \lambda} \ \ \sum_{E\in\mathcal{E}\left(  \lambda/\nu\right)  }\ \ \prod_{c\in Y\left(  \lambda\right)  \setminus E}\dfrac{1}{h_{\lambda
}\left(  c; z\right)  } .
\end{align*}
\end{lemma}

Theorem~\ref{thm.main} can now be proved by a simple induction on $\abs{Y\tup{\lambda / \mu}}$, using Lemma~\ref{recursion.main} \textbf{(b)} and Lemma~\ref{rec.2} in the induction step.
(As with all the other proofs, details can be found in Section~\ref{sec.proofs}.)

\section{\label{sec.konva-forreal}Proof of the Konvalinka recursion}

Now that the primary goal of this work (the proof of Theorem~\ref{thm.main}) has been achieved,
we turn to the proof of the Konvalinka recursion (Theorem~\ref{mainKonvalinka}).
Having established a closely related result (Lemma~\ref{lem.konvalinka-bi}) already,
we only need to show that the sums of $x$'s and of $y$'s that appear in the former agree with those that appear in the latter.
This requires some combinatorial lemmas, some of which are interesting in their own right.

\subsection{More on partitions, Delta-sets and transposes}

We begin with some general properties of partitions.

\begin{lemma}
\label{lem.conj.disjoint}
Let $\lambda$ be a partition. Let $i$ and $j$ be two
positive integers. Then, $\lambda_{j}+\lambda_{i}^{t}-i-j \neq -1$.
\end{lemma}

\begin{lemma}
\label{lem.conj.tt}
Let $\lambda$ be a partition. Then, $\tup{\lambda^t}^t = \lambda$.
\end{lemma}

\begin{lemma}
    \label{prop.Delta.disjoint}
    Let $\lambda$ be any partition. Let $p \in \Delta\left(\lambda\right)$.
    Then, $-1-p \notin \Delta\left(\lambda^t\right)$.
\end{lemma}

We note in passing that the converse of Lemma~\ref{prop.Delta.disjoint} also holds
(see Proposition~\ref{prop.Delta.disjoint-conv} below). \medskip

Next, we state an analogue of Lemma~\ref{lem.Deltas.if-not-then-k}:

\begin{lemma}
\label{lem.Deltas.if-not-then-k2}
Let $n\in\mathbb{N}$ be such that $\lambda=\left(  \lambda_{1},\lambda
_{2},\ldots,\lambda_{n}\right)  $ (that is, $\lambda_{n+1}=0$) and
$\mu=\left(  \mu_{1},\mu_{2},\ldots,\mu_{n}\right)  $ (that is, $\mu_{n+1}%
=0$).
Let $i$ be a positive integer that satisfies $m_i \notin \Delta\left(  \lambda\right)  $.
Then, $i\in\left[  n\right]  $.
\end{lemma}

\subsection{A bijection}

The next two lemmas develop the claim of Lemma~\ref{lem.Deltas.mi+1+bi}
further. They will help us transform a sum that appears in
Lemma~\ref{lem.konvalinka-bi} into a sum that appears in
Theorem~\ref{mainKonvalinka}.

\begin{lemma}
\label{lem.Deltas.ybij}
We follow Convention~\ref{conv.further.convs}.
Let $\mathbf{b}=\left(  b_{1},b_{2},b_{3},\ldots\right)  $ be the flagging
induced by $\lambda/\mu$.
Then, the map
\begin{align*}
\left\{  i\geq1\ \mid\ m_{i}\notin\Delta\left(  \lambda\right)  \right\}   &
\rightarrow\left\{  p\geq1\ \mid\ \ell_{p}^{t}\notin\Delta\left(  \mu
^{t}\right)  \right\}  ,\\
i  &  \mapsto m_{i}+1+b_{i}
\end{align*}
is well-defined and is a bijection.
\end{lemma}

\begin{lemma}
\label{lem.Deltas.y}
We follow Convention~\ref{conv.further.convs}.
Let $\mathbf{b}=\left(  b_{1},b_{2},b_{3},\ldots\right)  $ be the flagging
induced by $\lambda/\mu$.

Let $n\in\mathbb{N}$ be such that $\lambda=\left(  \lambda_{1},\lambda
_{2},\ldots,\lambda_{n}\right)  $ (that is, $\lambda_{n+1}=0$) and
$\mu=\left(  \mu_{1},\mu_{2},\ldots,\mu_{n}\right)  $ (that is, $\mu_{n+1}%
=0$). Then,%
\[
\sum_{\substack{i\in\left[  n\right]  ;\\m_{i}\notin\Delta\left(
\lambda\right)  }}y_{m_{i}+1+b_{i}}=\sum_{\substack{k\geq1;\\\ell_{k}%
^{t}\notin\Delta\left(  \mu^{t}\right)  }}y_{k}.
\]

\end{lemma}

\subsection{Proof of the Konvalinka recursion}

It is now easy to derive the original Konvalinka recursion
(Theorem \ref{mainKonvalinka}) from Lemma~\ref{lem.konvalinka-bi}.

\section{\label{sec.hints}Hints}

The following hints are given for the less straightforward proofs.
\medskip

\textit{Hint to Lemma~\ref{lem.ssyt.neighbors}.} Observe that if $e \in Y\tup{\mu}$ is any box, then the box $e_{+T}$ lies on the same diagonal as $e$.
However, the condition of each of the parts \textbf{(a)}--\textbf{(d)} of the lemma entails that the boxes $c_{+T}$ and $d_{+T}$ lie either on the same or on adjacent diagonals.
Hence, the same holds for the boxes $c$ and $d$.
Consequently, the box $d$ lies either weakly southeast or weakly northwest of $c$ (where ``weakly'' means that the boxes can share a row or column).
If the boxes $c$ and $d$ were more than one unit apart, or the entries $T\tup{c}$ and $T\tup{d}$ differed by more than $1$, then Lemma~\ref{lem.ssyt.geq} would thus place the boxes $c_{+T}$ and $d_{+T}$ too far apart from each other.
A sufficiently precise version of this reasoning leads to the claims of parts \textbf{(a)}--\textbf{(d)}.
\medskip

\textit{Hint to Lemma~\ref{transition_wd}.} \textbf{(a)} It is clear that the diagram $\DD\tup{T}$ can be obtained from $Y\tup{\mu}$ by moving each box $\tup{i,j}$ down into the $T\tup{i,j}$-th row by a sequence of $T\tup{i,j}-i$ southeastern moves (i.e., moves of the form $c \mapsto c_{\searrow}$).
What needs to be proved is that these moves can be arranged in an order in which they are excited moves (i.e., are not blocked by the presence of $c_{\to}$ or $c_{\downarrow}$ or $c_{\searrow}$ in the diagram).
One way to do this is by first moving the southeasternmost boxes of $Y\tup{\mu}$ to their intended positions first, then proceed recursively.
An easier way is to induct on the sum of all entries of $T$, using Lemma~\ref{lem.DS-vs-DT}.
\medskip

\textit{Hint to Lemma~\ref{transition_formula}.} The well-definedness follows from Lemma~\ref{transition_wd} \textbf{(a)}.
The hard part is surjectivity, i.e., to show that each excitation $E$ of $Y\tup{\mu}$ has the form $\DD\tup{T}$ for some $T \in \SSYT\tup{\mu}$.
One way is to induct on the number of excited moves needed to construct $E$.
By the induction hypothesis, the diagram obtained just before the last step has the form $\DD\tup{S}$ for some $S \in \SSYT\tup{\mu}$.
Argue that incrementing one specific entry of $S$ will cause $\DD\tup{S}$ to become $E$ instead (while $S$ still remains semistandard).
Injectivity follows from a similar argument.
\medskip

\textit{Hint to Lemma~\ref{1hprop}.} Break up the sum in the definition of $h\tup{a,b,c}$ into the part with $i_a = b$ and the part with $i_a < b$.
\medskip

\textit{Hint to Lemma~\ref{3hprop}.} Induct on $a+b$.
In the induction step, use Lemma~\ref{1hprop} three times (once for $a,b,c$; once for $a-1,b,c$; and once for $a,b,c-1$) and use the induction hypothesis twice (once for $a-1,b,c$, and once for $a,b-1,c$).
The five linear equations in the various $h\tup{i,j,k}$'s obtained can be combined to obtain the claim.
\medskip

\textit{Hint to Corollary~\ref{2hprop}.} Combine Lemma~\ref{1hprop} with Lemma~\ref{3hprop}.
\medskip

\textit{Hint to Lemma~\ref{lem.flagJT.fail}.} \textbf{(a)} Find an $i \in \set{1,2,\ldots,n-1}$ such that $\sigma\tup{i} > \sigma\tup{i+1}$, and show that the last box in the $\tup{i+1}$-th row of $P\tup{\sigma}$ is an outer failure.
\medskip

\textit{Hint to Lemma~\ref{lem.flagJT.flip1}.} Parts \textbf{(b)}, \textbf{(c)}, \textbf{(d)} and \textbf{(e)} are easy, but \textbf{(a)} and \textbf{(f)} are not to be underestimated.
Let $c = \tup{i,j}$ be as in Definition~\ref{def.flagJT.flip}.

For part \textbf{(a)}, it needs to first be checked that the swapping of the top and bottom floors results in an actual filling of $Y\tup{\sigma^\prime}$ (as opposed to a filling with gaps, or a filling of the wrong diagram).
Start by observing that the $i$-th row of $T$ has at least $j$ entries (obviously since $\tup{i,j} \in Y\tup{\sigma}$)
and that the $\tup{i-1}$-st row of $T$ has at least $j-1$ entries (here you need to use the fact that $c$ is a \textbf{leftmost} failure).
This shows that $T'$ has no gaps between entries in a row.
Then show that the rows of $T'$ have the right lengths.
Finally show that the entries in each row of $T'$ increase weakly from left to right
(this requires slightly different arguments when $c$ is an inner failure and when $c$ is an outer failure).
This handles part \textbf{(a)}.

Part \textbf{(f)} requires showing that the last entry of the $k$-th row of $T'$ is $\leq b_{\sigma'\tup{k}}$ (beware: not $\leq b_{\sigma\tup{k}}$ any more!) for each $k \in \ive{n}$.
This is clear for $k \notin \set{i, i-1}$, as in this case the $k$-th row does not change from $T$ to $T'$ and the value $\sigma\tup{k}$ does not change from $\sigma$ to $\sigma'$.
It remains to check the inequality for $k = i$ and for $k = i-1$.
This is easy when the top floor and the bottom floor both are nonempty (i.e., contain at least one entry), because in this case the 
last entries of the $\tup{i-1}$-st and $i$-th row are swapped along with the values $\sigma\tup{i}$ and $\sigma\tup{i-1}$.
The tricky cases are when one of the two floors is empty.
If the bottom floor is empty, then we need to show that $T\tup{i-1, j-1} \leq b_{\sigma\tup{i}}$, but this follows from $T\tup{i-1, j-1} < T\tup{i, j-1} \leq b_{\sigma\tup{i}}$.
If the top floor is empty, then $c$ is an outer failure, and in this case we need to show that $b_{\sigma\tup{i}} \leq b_{\sigma\tup{i-1}}$, but this follows from $\sigma\tup{i} > \sigma\tup{i-1}$, which in turn follows (after a bit of work) from the fact that the $t$-th row of $Y\tup{\sigma}$ is longer than the $\tup{i-1}$-st row.
In either case, the claim follows.
\medskip

\textit{Hint to Lemma~\ref{determinant.sum}.} Laplace expansion helps.
\medskip

\textit{Hint to Lemma~\ref{lem.det-u-cor}.} Apply  Lemma~\ref{determinant.sum} to $P = \tup{u_{i,j}}_{i,j\in\ive{n}}$ and $Q = \tup{u_{i,j+1}}_{i,j\in\ive{n}}$.
\medskip

\textit{Hint to Lemma~\ref{lem.det-uc}.} Expand the determinant on the left hand side along the $k$-th row and thus rewrite it as
$\det\tup{u_{i,j+\ive{k=i}}}_{i,j\in \ive{n}} - p_i \det\tup{u_{i,j}}_{i,j\in \ive{n}}$.
Now recall Lemma~\ref{lem.det-u-cor}.
\medskip

\textit{Hint to Lemma~\ref{lem.Deltas.x}.} Apply Lemma~\ref{lem.Deltas.f=g} to $f\tup{i} = x_{b_i}$ and $g\tup{i} = x_j$.
\medskip

\textit{Hint to Lemma~\ref{lem.Deltas.mi+1+bi}.} Set $k = m_i + 1 + b_i$.
Note that $m_i \neq \ell_{b_i}$ (since $m_i \notin \Delta\tup{\lambda}$).
Use Lemma~\ref{lem.flagging-of-lm.uniprop} and this fact
to prove that $\lambda_{b_i} \geq k$ but $\lambda_{b_i+1} < k$.
From this, conclude that $\lambda^t_k = b_i$. The rest is easy.
\medskip

\textit{Hint to Lemma~\ref{lem.ER-bb}.} Set $\ell_0 = \infty$.
Observe that $b_j = \max\set{k \geq 0 \mid \ell_k \geq m_j}$
and $b_{j-1} = \max\set{k \geq 0 \mid \ell_k \geq m_{j-1}}$.
Draw consequences.
\medskip

\textit{Hint to Lemma~\ref{lem.b-vs-bk}.} Rather similar to the proof of Lemma~\ref{lem.ER-bb}.
\medskip

\textit{Hint to Lemma~\ref{lem.mu+k.sdet}.}
If $k \in \operatorname{ER}\tup{\mu}$, then both parts of this lemma follow easily from \eqref{eq.snu=.det} and Lemma~\ref{lem.b-vs-bk}.
In the remaining case, both right hand sides are $0$, since the matrices have two equal rows.
\medskip

\textit{Hint to Lemma~\ref{lem.k-th-det-1}.} Use Lemma~\ref{lem.mu+k.sdet} as well as
Corollary~\ref{2hprop} (in the case when $m_k \in \Delta\tup{\lambda}$) and
Corollary~\ref{3hprop2} (in the case when $m_k \notin \Delta\tup{\lambda}$).
Don't forget that $x_i = y_i = 0$ for $i < 0$.
\medskip

\textit{Hint to Lemma~\ref{lem.k-th-det-2}.} Apply Lemma~\ref{lem.k-th-det-1}, then simplify using Lemma~\ref{lem.det-uc}.
\medskip

\textit{Hint to Lemma~\ref{lem.k-th-det-3}.} Use Lemma~\ref{lem.det.lastrow=01}, then simplify using Lemma~\ref{lem.h1bc}.
\medskip

\textit{Hint to Lemma~\ref{lem.k-th-det-4}.} The only nontrivial ingredient is Lemma~\ref{lem.Deltas.x}.
\medskip

\textit{Hint to Lemma~\ref{lem.konvalinka-bi}.} Combine Lemmas~\ref{lem.konval-RHS-as-sum},~\ref{lem.k-th-det-2},~\ref{lem.k-th-det-3} and~\ref{lem.k-th-det-4}.
\medskip

\textit{Hint to Lemma~\ref{lem.pf-hlf.4}.} Sum Lemma \ref{section14.wY} over all $i \in \ive{n}$ and recall Lemma~\ref{lem.Deltas.f=g}.
\medskip

\textit{Hint to Lemma~\ref{rec.2}.} Set $x_i := w_{\ell_i}$ and $y_i := - w_{-\ell^t_i - 1}$ for all $i \geq 1$, and apply Lemma \ref{lem.konvalinka-bi}. Use Lemmas \ref{lem.Deltas.mi+1+bi} and \ref{lem.pf-hlf.4} to simplify the first factor on the left hand side.
Observe that Lemma \ref{lem.pf-hlf.w-w} yields $x_i + y_j = h_\lambda \tup{\tup{i,j};z}$ for each $\tup{i,j} \in Y\tup{\lambda}$; use this to simplify the $\mathbf{s}_\lambda\ive{\nu}$ and $\mathbf{s}_\lambda\ive{\mu}$ expressions.
\medskip

\textit{Hint to Lemma~\ref{prop.Delta.disjoint}.} Use Lemma \ref{lem.conj.disjoint}.
\medskip


\textit{Hint to Lemma~\ref{lem.Deltas.ybij}.}
Let us denote this map by $\Phi$.
First, show that this map is well-defined by using Lemmas~\ref{lem.Deltas.mi+1+bi} and~\ref{prop.Delta.disjoint}.
Then, show that $\Phi$ is injective (argue that $\Phi\tup{u} = \Phi\tup{v}$ entails $m_u = m_v$ using Lemma~\ref{lem.Deltas.mi+1+bi}, and then conclude $u = v$).
It remains to show that $\Phi$ is surjective.
For that, it suffices to argue that its domain and its target have the same size (since they are easily seen to be finite sets).
The injectivity of $\Phi$ yields that its domain is at most as large as its target;
but the same argument, with $\lambda$ and $\mu$ replaced by $\mu^t$ and $\lambda^t$, yields that its target is at most as large as its domain.

\medskip
\section{\label{sec.proofs}Proofs}

\subsection{To Section \ref{sec.nots}}

\begin{proof}[Proof of Lemma \ref{lem.exc.straight}.]
The diagram $Y\left(  \varnothing\right)  $ is the empty set $\varnothing$.
Hence, no excited moves can be applied to it. In other words, its only
excitation is $\varnothing$ itself. This excitation $\varnothing$ of course
satisfies $\varnothing\subseteq Y\left(  \lambda\right)  $. Therefore,
$\mathcal{E}\left(  \lambda/\varnothing\right)  =\left\{  \varnothing\right\}
$ (by Definition~\ref{def.Elamu}).
\end{proof}

\begin{proof}[Proof of Lemma \ref{lem.exc.empty}.]
We don't have $\lambda \supseteq \mu$.
In other words, we don't have $Y\tup{\lambda} \supseteq Y\tup{\mu}$.
In other words, there exists some box $c \in Y\tup{\mu}$ that does not belong to $Y\tup{\lambda}$.
This box must clearly retain this property (of not belonging to $Y\tup{\lambda}$) under any excited move (since any excited move can only move this box further southeast).
Thus, any excitation of $Y\tup{\mu}$ contains a box that does not belong to $Y\tup{\lambda}$.
Hence, there exists no excitation $E$ of $Y\tup{\mu}$ that satisfies $E \subseteq Y\tup{\lambda}$.
In other words, the set of all such excitations is empty.
In other words, $\mathcal{E}\tup{\lambda/\mu}$ is empty (since $\mathcal{E}\tup{\lambda/\mu}$ was defined to be this set).
\end{proof}

\begin{proof}[Proof of Lemma \ref{lem.exc.equal}.]
The only excitation $E$ of the diagram $Y\left(  \lambda\right)  $ that satisfies $E \subseteq Y\tup{\lambda}$ is $Y\tup{\lambda}$ itself
(since any excited move would cause a box to move out of $Y\tup{\lambda}$, and thus break the $E \subseteq Y\tup{\lambda}$ condition).
This proves the lemma.
\end{proof}

\subsection{To Section \ref{sec.recz}}

\begin{proof}[Proof of Lemma \ref{recus.z}.]
Recall that $Y\left(  \lambda/ \mu\right)  =
Y\left(  \lambda\right)  \setminus Y\left(  \mu\right) $.

Let $c=\left(  p,q\right)  $ be the box in $Y\left(  \lambda/\mu\right)  $
that contains the entry $1$ in $T$ (that is, that satisfies $T\left(
c\right)  =1$). If the western neighbor $\left(  p,q-1\right)  $ of this box
$c$ lied in $Y\left(  \lambda/\mu\right)  $, then the entry of $T$ in this
western neighbor would be smaller than $1$ (since the entries of $T$ increase
left-to-right in each row), which is impossible (since the entries of $T$ are
positive integers). Thus, the western neighbor $\left(  p,q-1\right)  $ of $c$
does not lie in $Y\left(  \lambda/\mu\right)  $. Hence, this western neighbor
$\left(  p,q-1\right)  $ lies in $Y\left(  \mu\right)  $ unless $q=1$ (since
it definitely lies in $Y\left(  \lambda\right)  $ unless $q=1$).

Similarly, we can see that the northern neighbor $\left(  p-1,q\right)  $ of
$c$ lies in $Y\left(  \mu\right)  $ unless $p=1$. In other words, $\mu
_{p-1}\geq q$ unless $p=1$.

Recall that the box $\left(  p,q-1\right)  $ lies in $Y\left(
\mu\right)  $ unless $q=1$. This box must be the easternmost box in the
$p$-th row of $Y\left(  \mu\right)  $ (since its eastern neighbor $\left(
p,q\right)  =c$ lies in $Y\left(  \lambda/\mu\right)  =Y\left(  \lambda
\right)  \setminus Y\left(  \mu\right)  $ and thus does not lie in $Y\left(
\mu\right)  $). Thus, the $p$-th row of $Y\left(  \mu\right)  $ has $q-1$
boxes in total. In other words, $\mu_{p}=q-1$, so that $\mu_{p}+1=q$.

If we increment the $p$-th entry of the
partition $\mu$ by $1$, then we obtain a new sequence
$\nu = \left(\nu_1, \nu_2, \nu_3, \ldots\right)$. Consider this
$\nu$. Explicitly, it is given by
\begin{align*}
\nu_{i}  & =\mu_{i}\ \ \ \ \ \ \ \ \ \ \text{for all }i\neq
p,\ \ \ \ \ \ \ \ \ \ \text{and}\\
\nu_{p}  & =\mu_{p}+1.
\end{align*}
In particular, if $p\neq 1$, then $\nu_{p-1} = \mu_{p-1} \geq q$ (since we know that 
 $\mu_{p-1} \geq q$ unless $p=1$), and thus $\nu_{p-1} \geq q = \mu_p+1 = \nu_p$.
Thus, it is easy to see that the sequence $\nu$ is a partition.
[\textit{Proof:} The sequence $\mu$ is weakly decreasing (since $\mu$ is a partition).
But the sequence $\nu$ is obtained from $\mu$ by incrementing the $p$-th entry by $1$.
Hence, $\nu$ is also weakly decreasing, unless its (incremented) $p$-th entry has
become larger than the previous (i.e., the $\tup{p-1}$-th) entry.
In other words, $\nu$ is also weakly decreasing, unless $p \neq 1$ and $\nu_p > \nu_{p-1}$.
But the latter case cannot happen (because if $p \neq 1$, then $\nu_{p-1} \geq \nu_p$).
Thus, we conclude that $\nu$ is weakly decreasing, i.e., is a partition
(since it is clear that $\nu_i = \mu_i = 0$ for all sufficiently large $i$).]

Recall that the partition $\nu$ is obtained from $\mu$ by
incrementing the $p$-th entry by $1$. Hence,
\begin{equation}
Y\left(  \nu\right)  =Y\left(  \mu\right)  \cup\left\{  \left(  p,\mu
_{p}+1\right)  \right\}  \label{eq.pf.recus.z.1}%
\end{equation}
(because by incrementing $\mu_{p}$, we add a new box $\left(  p,\mu
_{p}+1\right)  $ to the Young diagram $Y\left(  \mu\right)  $) and
$\mu\lessdot\nu$.

Let us now recall that $\mu_{p}+1=q$. Hence, $\tup{p,\mu_p+1} = \tup{p,q} = c$.
Thus, we can rewrite \eqref{eq.pf.recus.z.1} as
\begin{equation}
Y\left(  \nu\right)  =Y\left(  \mu\right)  \cup\left\{  c\right\}
.\label{eq.pf.recus.z.2}%
\end{equation}

Note furthermore that $Y\left(  \mu\right)  \subseteq Y\left(  \lambda\right)
$ (since $\lambda/\mu$ is a skew partition) and $\left\{  c\right\}  \subseteq
Y\left(  \lambda\right)  $ (since $c\in Y\left(  \lambda/\mu\right)  \subseteq
Y\left(  \lambda\right)  $). Hence, $Y\left(  \mu\right)  \cup\left\{
c\right\}  \subseteq Y\left(  \lambda\right)  $ (since a union of two
subsets of $Y\tup{\lambda}$ must again be a subset of $Y\tup{\lambda}$).
In view of \eqref{eq.pf.recus.z.2}, this rewrites as
$Y\left(  \nu\right)  \subseteq Y\left(  \lambda\right)  $. In other words,
$\nu\subseteq\lambda$. Hence, $\mu \lessdot \nu \subseteq \lambda$.

Furthermore,%
\begin{align*}
Y\left(  \lambda/\nu\right)    & =Y\left(  \lambda\right)  \setminus Y\left(
\nu\right)  =Y\left(  \lambda\right)  \setminus\left(  Y\left(  \mu\right)
\cup\left\{  c\right\}  \right)  \ \ \ \ \ \ \ \ \ \ \left(  \text{by
\eqref{eq.pf.recus.z.2}}\right)  \\
& =\underbrace{\left(  Y\left(  \lambda\right)  \setminus Y\left(  \mu\right)
\right)  }_{=Y\left(  \lambda/\mu\right)  }\setminus\left\{  c\right\}
=Y\left(  \lambda/\mu\right)  \setminus\left\{  c\right\}  .
\end{align*}

Now, remove the box $c$ with the entry $1$ from $T$, and subtract $1$ from
all remaining entries. Let $T^{\prime}$ be the filling that remains. Then,
$T^{\prime}$ is a filling of $Y\left(  \lambda/\mu\right)
\setminus\left\{  c \right\}  = Y\left(  \lambda/ \nu\right) $. The entries of
this filling $T^{\prime}$ are $1,2,\ldots,n-1$ (since they are obtained by
subtracting $1$ from each of $2,3,\ldots,n$). Furthermore, these entries
increase left-to-right in each row (since this is true for $T$, and clearly
remains true as we subtract $1$ from each entry), and increase top-to-bottom
in each column (similarly). Hence, $T^{\prime}$ is a standard tableau of shape
$\lambda/ \nu$.

It remains to prove \eqref{eq.recus_z.eq}. To that purpose, we observe that
the construction of $T^{\prime}$ yields $c_{T^{\prime}}(k) = c_{T}(k+1)$ for
each $k \in\left\{  1, 2, \ldots, n-1 \right\} $
(since each entry $k$ of $T^{\prime}$ originated as an entry $k+1$ in $T$). Thus,
\begin{align*}
\prod_{k=1}^{n-1}\left(  z_{c_{T^{\prime}}\left(  k\right)  } +z_{c_{T^{\prime
}}\left(  k+1\right)  }+\cdots+z_{c_{T^{\prime}}\left(  n-1\right)  }\right)
& = \prod_{k=1}^{n-1}\left(  z_{c_{T}\left(  k+1\right)  }+z_{c_{T}\left(
k+2\right)  }+\cdots+z_{c_{T}\left(  n\right)  }\right) \\
& = \prod_{k=2}^{n}\left(  z_{c_{T}\left(  k\right)  }+z_{c_{T}\left(
k+1\right)  }+\cdots+z_{c_{T}\left(  n\right)  }\right)
\end{align*}
(here, we have substituted $k$ for $k+1$ in the product).
On the other hand,
\[
\sum_{\left(  i, j \right)  \in Y\left(  \lambda/ \mu\right) } z_{j-i} =
z_{c_{T}\left(  1\right)  }+z_{c_{T}\left(  2\right)  }+\cdots+z_{c_{T}\left(
n\right)  } ,
\]
since the numbers $c_T\tup{1},\ c_T\tup{2},\ \ldots,\ c_T\tup{n}$ are precisely the numbers $j-i$ for all $\tup{i,j} \in Y\tup{\lambda/\mu}$
(because each box of $Y\left(  \lambda/ \mu\right) $ is occupied by exactly one
of the numbers $1, 2, \ldots, n$ in $T$).
Multiplying this equality by the preceding one, we find
\begin{align*}
&  \left(  \sum_{\left(  i, j \right)  \in Y\left(  \lambda/ \mu\right) }
z_{j-i} \right)  \cdot\prod_{k=1}^{n-1}\left(  z_{c_{T^{\prime}}\left(
k\right)  } +z_{c_{T^{\prime}}\left(  k+1\right)  }+\cdots+z_{c_{T^{\prime}%
}\left(  n-1\right)  }\right) \\
= &  \left(  z_{c_{T}\left(  1\right)  }+z_{c_{T}\left(  2\right)  }%
+\cdots+z_{c_{T}\left(  n\right)  } \right)  \cdot\prod_{k=2}^{n}\left(
z_{c_{T}\left(  k\right)  }+z_{c_{T}\left(  k+1\right)  }+\cdots
+z_{c_{T}\left(  n\right)  }\right) \\
= &  \prod_{k=1}^{n}\left(  z_{c_{T}\left(  k\right)  }+z_{c_{T}\left(
k+1\right)  }+\cdots+z_{c_{T}\left(  n\right)  }\right)  .
\end{align*}
In other words,
\begin{align*}
&  \dfrac{1}{\prod_{k=1}^{n}\left(  z_{c_{T}\left(  k\right)  }+z_{c_{T}%
\left(  k+1\right)  }+\cdots+z_{c_{T}\left(  n\right)  }\right) }\\
=\  &  \dfrac{1}{\sum_{\left(  i, j \right)  \in Y\left(  \lambda/ \mu\right)
} z_{j-i}} \cdot\dfrac{1}{\prod_{k=1}^{n-1}\left(  z_{c_{T^{\prime}}\left(
k\right)  } +z_{c_{T^{\prime}}\left(  k+1\right)  }+\cdots+z_{c_{T^{\prime}%
}\left(  n-1\right)  }\right) } .
\end{align*}
But the fraction on the left hand side of this equality is $\mathbf{z}_{T}$,
while the second fraction on the right hand side is $\mathbf{z}_{T^{\prime}}$.
Thus, this equality is precisely \eqref{eq.recus_z.eq}.
\end{proof}

\begin{proof}
[Proof of Lemma \ref{recursion.main}.]
\textbf{(a)} Assume that $\lambda=\mu$.
Then, $Y\left(  \lambda/\mu\right)  =\varnothing$. Hence, there is only one
$T\in\operatorname*{SYT}\left(  \lambda/\mu\right)  $, namely the empty
tableau $T_{\varnothing}$ (with no entries). Therefore, $\sum_{T\in
\operatorname{SYT}\left(  \lambda/\mu\right)  }\mathbf{z}_{T}=\mathbf{z}%
_{T_{\varnothing}}=1$ (since an empty product is $1$ by definition).
\medskip

\textbf{(b)} Assume that $\lambda\neq\mu$. Thus, the diagram $Y\left(
\lambda/\mu\right)  $ has at least one box. Hence, any standard tableau
$T\in\operatorname*{SYT}\left(  \lambda/\mu\right)  $ must contain the entry
$1$ somewhere.

We shall use the following notations:

\begin{itemize}
\item If $\nu$ is a partition satisfying $\mu\lessdot\nu$, then
$\operatorname*{box}\left(  \nu/\mu\right)  $ shall denote the unique box in
$Y\left(  \nu\right)  \setminus Y\left(  \mu\right)  =Y\left(  \nu/\mu\right)
$. (This box is unique, since $\mu\lessdot\nu$.)

\item If $T\in\operatorname*{SYT}\left(  \lambda/\mu\right)  $ is a standard
tableau, then $B_{1}\left(  T\right)  $ shall denote the unique box of $T$
that contains the entry $1$. (This box exists, since any standard tableau
$T\in\operatorname*{SYT}\left(  \lambda/\mu\right)  $ must contain the entry
$1$ somewhere. It is unique, since a standard tableau cannot have repeated entries.)
\end{itemize}

Let $T\in\operatorname*{SYT}\left(  \lambda/\mu\right)  $ be a standard
tableau. Lemma~\ref{recus.z} shows that if we remove the box $B_{1}\left(
T\right)  $ (that is, the box that contains the entry $1$) from $T$, and if we
subtract $1$ from all remaining entries, then we obtain a new standard tableau
$T^{\prime}$, which has shape $\lambda/\nu$ for some partition $\nu$
satisfying $\mu\lessdot\nu\subseteq\lambda$. This latter partition $\nu$ has
the property that the unique box of $Y\left(  \nu/\mu\right)  $ is
$B_{1}\left(  T\right)  $ (since this was the box removed from $T$ to obtain
$T^{\prime}$). In other words, it satisfies $B_{1}\left(  T\right)
=\operatorname*{box}\left(  \nu/\mu\right)  $. Furthermore, this property (in
combination with $\mu\lessdot\nu\subseteq\lambda$) uniquely determines $\nu$
(since $B_{1}\left(  T\right)  =\operatorname*{box}\left(  \nu/\mu\right)  $
means $Y\left(  \nu/\mu\right)  =\left\{  B_{1}\left(  T\right)  \right\}  $
and thus $Y\left(  \nu\right)  =Y\left(  \mu\right)  \cup\underbrace{Y\left(
\nu/\mu\right)  }_{=\left\{  B_{1}\left(  T\right)  \right\}  }=Y\left(
\mu\right)  \cup\left\{  B_{1}\left(  T\right)  \right\}  $, but this uniquely
determines $\nu$). Thus, there is exactly one partition $\nu$ satisfying
$\mu\lessdot\nu\subseteq\lambda$ and $B_{1}\left(  T\right)
=\operatorname*{box}\left(  \nu/\mu\right)  $.

Hence, the sum $\sum\limits_{\substack{\mu\lessdot\nu\subseteq\lambda
;\\B_{1}\left(  T\right)  =\operatorname*{box}\left(  \nu/\mu\right)
}}\mathbf{z}_{T}$ has exactly one addend, and thus equals $\mathbf{z}_{T}$. In
other words,%
\[
\mathbf{z}_{T}=\sum\limits_{\substack{\mu\lessdot\nu\subseteq\lambda
;\\B_{1}\left(  T\right)  =\operatorname*{box}\left(  \nu/\mu\right)
}}\mathbf{z}_{T}.
\]

Forget that we fixed $T$. We thus have shown that%
\begin{equation}
\mathbf{z}_{T}=\sum\limits_{\substack{\mu\lessdot\nu\subseteq\lambda
;\\B_{1}\left(  T\right)  =\operatorname*{box}\left(  \nu/\mu\right)
}}\mathbf{z}_{T} \label{pf.recursion.main.b.0}%
\end{equation}
for any standard tableau $T\in\operatorname*{SYT}\left(  \lambda/\mu\right)  $.

Now, summing the equality \eqref{pf.recursion.main.b.0} over all
$T\in\operatorname{SYT}\left(  \lambda/\mu\right)  $, we obtain
\begin{align}
\sum_{T\in\operatorname{SYT}\left(  \lambda/\mu\right)  }\mathbf{z}_{T}  &
=\underbrace{\sum_{T\in\operatorname{SYT}\left(  \lambda/\mu\right)  }%
\ \ \sum\limits_{\substack{\mu\lessdot\nu\subseteq\lambda;\\B_{1}\left(
T\right)  =\operatorname*{box}\left(  \nu/\mu\right)  }}}_{=\sum
\limits_{\mu\lessdot\nu\subseteq\lambda}\ \ \sum_{\substack{T\in
\operatorname{SYT}\left(  \lambda/\mu\right)  ;\\B_{1}\left(  T\right)
=\operatorname*{box}\left(  \nu/\mu\right)  }}}\mathbf{z}_{T}\nonumber\\
&  =\sum\limits_{\mu\lessdot\nu\subseteq\lambda}\ \ \sum_{\substack{T\in
\operatorname{SYT}\left(  \lambda/\mu\right)  ;\\B_{1}\left(  T\right)
=\operatorname*{box}\left(  \nu/\mu\right)  }}\mathbf{z}_{T}.
\label{pf.recursion.main.b.1}%
\end{align}

Now, fix a partition $\nu$ satisfying $\mu\lessdot\nu\subseteq\lambda$. Thus,
$Y\left(  \mu\right)  \subseteq Y\left(  \nu\right)  \subseteq Y\left(
\lambda\right)  $.

Let $b=\operatorname*{box}\left(  \nu/\mu\right)  $ (this is well-defined,
since $\mu\lessdot\nu$). Then, $Y\left(  \nu/\mu\right)  =\left\{  b\right\}
$. Hence,%
\begin{align*}
\underbrace{Y\left(  \lambda/\mu\right)  }_{=Y\left(  \lambda\right)
\setminus Y\left(  \mu\right)  }\setminus\underbrace{\left\{  b\right\}
}_{\substack{=Y\left(  \nu/\mu\right)  \\=Y\left(  \nu\right)  \setminus
Y\left(  \mu\right)  }}  & =\left(  Y\left(  \lambda\right)  \setminus
Y\left(  \mu\right)  \right)  \setminus\left(  Y\left(  \nu\right)  \setminus
Y\left(  \mu\right)  \right)  \\
& =Y\left(  \lambda\right)  \setminus Y\left(  \nu\right)
\ \ \ \ \ \ \ \ \ \ \left(  \text{since }Y\left(  \mu\right)  \subseteq
Y\left(  \nu\right)  \subseteq Y\left(  \lambda\right)  \right)  \\
& =Y\left(  \lambda/\nu\right)  .
\end{align*}

Consider a standard tableau $T\in\operatorname{SYT}\left(  \lambda/\mu\right)
$ satisfying $B_{1}\left(  T\right)  =b$. Then, the standard tableau $T$ has
its entry $1$ in this box $b$ (by the definition of $B_{1}\left(  T\right)
$). If we remove this box $b$ from $T$, and subtract $1$ from all remaining
entries, then we obtain a new standard tableau $T^{\prime}$ of shape
$\lambda/\nu$ (indeed, we clearly obtain a filling of the diagram $Y\left(
\lambda/\mu\right)  \setminus\left\{  b\right\}  =Y\left(  \lambda/\nu\right)
$, and furthermore this filling is a standard tableau because the subtraction
of $1$ from each entry of $T$ did not overturn the inequalities between the
entries of $T$). Thus, for any standard tableau $T\in\operatorname*{SYT}%
\left(  \lambda/\mu\right)  $ satisfying $B_{1}\left(  T\right)  =b$, we have
defined a standard tableau $T^{\prime}$ of shape $\lambda/\nu$. In other
words, we have defined a map%
\begin{align}
\left\{  T\in\operatorname{SYT}\left(  \lambda/\mu\right)  \ \mid
\ B_{1}\left(  T\right)  =b\right\}   &  \rightarrow\operatorname*{SYT}\left(
\lambda/\nu\right)  ,\nonumber\\
T &  \mapsto T^{\prime}.\label{pf.recursion.main.b.bij}%
\end{align}
This map is easily seen to be injective (since we can recover $T$ from $T^{\prime}$ by incrementing all entries of $T^\prime$ by $1$ and placing a $1$ into the extra box $b$) and surjective (because given any standard tableau $S$ of
shape $\lambda/\nu$, we can increment all entries of $S$ by $1$ and fill the
additional box $b$ with the entry $1$, thus obtaining a standard tableau
$T\in\operatorname{SYT}\left(  \lambda/\mu\right)  $ that satisfies
$B_{1}\left(  T\right)  =b$ and $T^{\prime}=S$). Hence, it is bijective.

Now, recall that $\operatorname*{box}\left(  \nu/\mu\right)  =b$. Thus,%
\begin{align}
\sum_{\substack{T\in\operatorname{SYT}\left(  \lambda/\mu\right)
;\\B_{1}\left(  T\right)  =\operatorname*{box}\left(  \nu/\mu\right)
}}\mathbf{z}_{T} &  =\sum_{\substack{T\in\operatorname{SYT}\left(  \lambda
/\mu\right)  ;\\B_{1}\left(  T\right)  =b}}\mathbf{z}_{T}\nonumber\\
&  =\sum_{\substack{T\in\operatorname{SYT}\left(  \lambda/\mu\right)
;\\B_{1}\left(  T\right)  =b}}\frac{1}{\sum\limits_{(i,j)\in Y\left(
\lambda/\mu\right)  }z_{j-i}}\cdot\mathbf{z}_{T^{\prime}}\qquad\left(
\text{by (\ref{eq.recus_z.eq})}\right)  \nonumber\\
&  =\sum_{T\in\operatorname{SYT}\left(  \lambda/\nu\right)  }\frac{1}%
{\sum\limits_{(i,j)\in Y\left(  \lambda/\mu\right)  }z_{j-i}}\cdot
\mathbf{z}_{T}\qquad\left(
\begin{array}
[c]{c}%
\text{here, we substituted }T
\text{ for }T^{\prime}\\
\text{in the outer sum, since}\\
\text{the map \eqref{pf.recursion.main.b.bij} is bijective}%
\end{array}
\right)  \nonumber\\
&  =\frac{1}{\sum\limits_{(i,j)\in Y\left(  \lambda/\mu\right)  }z_{j-i}}%
\cdot\sum_{T\in\operatorname{SYT}\left(  \lambda/\nu\right)  }\mathbf{z}%
_{T}.\label{pf.recursion.main.b.6}%
\end{align}

Now, forget that we fixed $\nu$. We have thus proved
(\ref{pf.recursion.main.b.6}) for each partition $\nu$ satisfying $\mu
\lessdot\nu\subseteq\lambda$. Now,
\begin{align*}
\sum_{T\in\operatorname{SYT}\left(  \lambda/\mu\right)  }\mathbf{z}_{T} &
=\sum\limits_{\mu\lessdot\nu\subseteq\lambda}\ \ \sum_{\substack{T\in
\operatorname{SYT}\left(  \lambda/\mu\right)  ;\\B_{1}\left(  T\right)
=\operatorname*{box}\left(  \nu/\mu\right)  }}\mathbf{z}_{T}\qquad\left(
\text{by (\ref{pf.recursion.main.b.1})}\right)  \\
&  =\sum\limits_{\mu\lessdot\nu\subseteq\lambda}\frac{1}{\sum\limits_{(i,j)\in
Y\left(  \lambda/\mu\right)  }z_{j-i}}\cdot\sum_{T\in\operatorname{SYT}\left(
\lambda/\nu\right)  }\mathbf{z}_{T}\qquad\left(  \text{by
(\ref{pf.recursion.main.b.6})}\right)  \\
&  =\frac{1}{\sum\limits_{(i,j)\in Y\left(  \lambda/\mu\right)  }z_{j-i}}%
\cdot\sum\limits_{\mu\lessdot\nu\subseteq\lambda}\ \ \sum_{T\in
\operatorname{SYT}\left(  \lambda/\nu\right)  }\mathbf{z}_{T}.
\end{align*}
This proves Lemma \ref{recursion.main} \textbf{(b)}.
\end{proof}

\subsection{To Section \ref{sec.konva}}

\begin{proof}[Proof of Lemma \ref{lem.conj.uniprop}.]
    We have the chain of equivalences
    \begin{align*}
    \tup{\lambda^t_i \geq j}
    \ \Longleftrightarrow & \ %
    \tup{\tup{i,j} \in Y\tup{\lambda^t}}
    \qquad \qquad \tup{\text{by the definition of $Y\tup{\lambda^t}$}}
    \\
    \Longleftrightarrow & \ %
    \tup{\tup{j,i} \in Y\tup{\lambda}}
    \qquad \qquad \tup{\text{by \eqref{eq.def.lambdat.1}}}
    \\
    \Longleftrightarrow & \ %
    \tup{\lambda_j \geq i}
    \qquad \qquad \tup{\text{by the definition of $Y\tup{\lambda}$}}.
    \end{align*}
    Thus, Lemma \ref{lem.conj.uniprop} follows.
\end{proof}

\subsection{To Section \ref{sec.flag}}

\begin{proof}[Proof of Lemma \ref{lem.ssyt.geq}.]
\textbf{(b)} From $\left(  u,v\right)  \in
Y\left(  \mu\right)  $, we obtain $v\leq\mu_{u}$, so that $j\leq v\leq\mu_{u}$
and therefore $\left(  u,j\right)  \in Y\left(  \mu\right)  $. Hence, the
entry $T\left(  u,j\right)  $ is well-defined.

The entries of $T$ weakly increase left-to-right in each row (since $T$ is a
semistandard tableau). Hence, from $j\leq v$, we obtain $T\left(  u,j\right)
\leq T\left(  u,v\right)  $.

Next, recall that $\mu$ is a partition, thus weakly decreasing. Hence, for
each $k\in\left\{  1,2,\ldots,u\right\}  $, we have $\mu_u \leq \mu_k$ (since $u \geq k$) and thus $j\leq\mu_u\leq\mu_{k}$, so that $\left(  k,j\right)
\in Y\left(  \mu\right)  $. Hence, $T\left(  k,j\right)  $ is well-defined for
each $k\in\left\{  1,2,\ldots,u\right\}  $. We set%
\[
p_{k}:=T\left(  k,j\right)  -k\ \ \ \ \ \ \ \ \ \ \text{for each }k\in\left\{
1,2,\ldots,u\right\}  .
\]

We shall now show that the sequence $\left(  p_{i},p_{i+1},\ldots
,p_{u}\right)  $ is weakly increasing.

[\textit{Proof:} It suffices to show that $p_{s}\leq p_{s+1}$ for each
$s\in\left\{  i,i+1,\ldots,u-1\right\}  $. So let us fix $s\in\left\{
i,i+1,\ldots,u-1\right\}  $ and show this. The entries of $T$ strictly
increase top-to-bottom in each column (since $T$ is a semistandard tableau).
Hence, $T\left(  s,j\right)  <T\left(  s+1,j\right)  $. Thus, $T\left(
s,j\right)  \leq T\left(  s+1,j\right)  -1$ (since all entries of $T$ are
integers). However, the definition of $p_{k}$ yields $p_{s}=T\left(
s,j\right)  -s$ and $p_{s+1}=T\left(  s+1,j\right)  -\left(  s+1\right)  $.
Thus,%
\[
p_{s}=\underbrace{T\left(  s,j\right)  }_{\leq T\left(  s+1,j\right)
-1} - \, s\leq T\left(  s+1,j\right)  -1-s=T\left(  s+1,j\right)  -\left(
s+1\right)  =p_{s+1},
\]
as we desired to prove.]

Thus, we have shown that the sequence $\left(  p_{i},p_{i+1},\ldots
,p_{u}\right)  $ is weakly increasing. Hence,
\begin{align*}
p_{i}  &  \leq p_{u}=T\left(  u,j\right)  -u\ \ \ \ \ \ \ \ \ \ \left(
\text{by the definition of }p_{u}\right) \\
&  \leq T\left(  u,v\right)  -u\ \ \ \ \ \ \ \ \ \ \left(  \text{since
}T\left(  u,j\right)  \leq T\left(  u,v\right)  \right)  .
\end{align*}
Therefore,%
\[
T\left(  u,v\right)  -u\geq p_{i}=T\left(  i,j\right)
-i\ \ \ \ \ \ \ \ \ \ \left(  \text{by the definition of }p_{i}\right)  .
\]
This proves Lemma \ref{lem.ssyt.geq} \textbf{(b)}. \medskip

\textbf{(a)} Let $\left(  i,j\right)  \in Y\left(  \mu\right)  $. Thus,
$j\leq\mu_{i}$. But $\mu$ is a partition, thus weakly decreasing. Hence,
$\mu_{i}\leq\mu_{1}$ (since $i\geq1$). Thus, $j\leq\mu_{i}\leq\mu_{1}$, so
that $\left(  1,j\right)  \in Y\left(  \mu\right)  $. Since we also have
$1\leq i$ and $j\leq j$, we can thus apply Lemma \ref{lem.ssyt.geq}
\textbf{(b)} to $\left(  1,j\right)  $ and $\left(  i,j\right)  $ instead of
$\left(  i,j\right)  $ and $\left(  u,v\right)  $. We thus obtain
\[
T\left(  i,j\right)  -i\geq T\left(  1,j\right)
-1 \geq 0\qquad \qquad \left(\text{since } T\tup{1,j} \geq 1\right).
\]
In other words, $T\left(  i,j\right)  \geq i$. This proves Lemma
\ref{lem.ssyt.geq} \textbf{(a)}.
\end{proof}

\begin{proof}[Proof of Lemma \ref{lem.flagging-of-lm.wd}.]
    Let $i \geq 1$ be an integer. Every sufficiently large $k$ satisfies both $k > i - \mu_i$ and $\lambda_k = 0$, so that it satisfies $\underbrace{\lambda_k}_{=0} - \underbrace{k}_{> i - \mu_i} < 0 - \tup{i - \mu_i} = \mu_i - i$.
    Hence, any sufficiently large integer $k \geq 0$ will fail to satisfy $\lambda_k - k \geq \mu_i - i$. Hence, the set of all integers $k \geq 0$ that satisfy $\lambda_k - k \geq \mu_i - i$ is finite. Since this set is furthermore nonempty (because $\lambda_0 - 0 = \lambda_0 = \infty \geq \mu_i - i$ shows that $0$ belongs to this set), this set must thus have a maximum
    (because any nonempty finite set of integers has a maximum). In other words, $\max\set{k \geq 0 \mid \lambda_k - k \geq \mu_i - i }$ is well-defined.
\end{proof}

\begin{proof}[Proof of Lemma \ref{lem.flagging-of-lm.uniprop}.]
The \textquotedblleft$\Longleftarrow$\textquotedblright\ direction is easy:
Since $b_{i}$ is defined as the largest $k\geq0$ that satisfies $\lambda
_{k}-k\geq\mu_{i}-i$, it is clear that every integer $k\geq0$ that satisfies
$\lambda_{k}-k\geq\mu_{i}-i$ must be $\leq b_{i}$. Hence, if $\lambda
_{j}-j\geq\mu_{i}-i$, then $j\leq b_{i}$. This proves the \textquotedblleft%
$\Longleftarrow$\textquotedblright\ direction of the lemma.

It remains to prove the \textquotedblleft$\Longrightarrow$\textquotedblright%
\ direction. Thus, we assume that $j\leq b_{i}$, and set out to prove that
$\lambda_{j}-j\geq\mu_{i}-i$.

From $j\leq b_{i}$, we obtain $\lambda_{j}\geq\lambda_{b_{i}}$ (since
$\lambda_{0}\geq\lambda_{1}\geq\lambda_{2}\geq\lambda_{3}\geq\cdots$) and thus
$\lambda_{j}-j\geq\lambda_{b_{i}}-j\geq\lambda_{b_{i}}-b_{i}$ (since $j\leq
b_{i}$). However, $b_{i}$ is defined as the largest $k\geq0$ that satisfies
$\lambda_{k}-k\geq\mu_{i}-i$. Thus, in particular, $b_{i}$ itself is such a
$k$. In other words, $\lambda_{b_{i}}-b_{i}\geq\mu_{i}-i$. Therefore,
$\lambda_{j}-j\geq\lambda_{b_{i}}-b_{i}\geq\mu_{i}-i$. This completes the
proof of the \textquotedblleft$\Longrightarrow$\textquotedblright\ direction.
Hence, Lemma \ref{lem.flagging-of-lm.uniprop} is proved.
\end{proof}

\begin{proof}[Proof of Lemma \ref{lem.flagging-of-lm.inc}.]
We must show that $b_{i-1}\leq b_{i}$ for each $i\in\left\{
2,3,4,\ldots\right\}  $ (where the notations are those of Definition
\ref{def.flagging-of-lm}). Let us do this.

Let $i\in\left\{  2,3,4,\ldots\right\}  $. Then, $b_{i-1}$ is defined as
the largest $k\geq0$ that satisfies $\lambda_{k}-k\geq\mu_{i-1}-\left(
i-1\right)  $. Thus, in particular, $b_{i-1}$ itself is such a $k$. In other
words, $\lambda_{b_{i-1}}-b_{i-1}\geq\mu_{i-1}-\left(  i-1\right)  $. However,
from $\mu_{1}\geq\mu_{2}\geq\mu_{3}\geq\cdots$, we obtain $\mu_{i-1}-\left(
i-1\right)  \geq\mu_{i}-\left(  i-1\right)  \geq\mu_{i}-i$ (since $i-1\leq i$).
Therefore,%
\[
\lambda_{b_{i-1}}-b_{i-1}\geq\mu_{i-1}-\left(  i-1\right)  \geq\mu_{i}-i.
\]

Now, Lemma \ref{lem.flagging-of-lm.uniprop} (applied to $j=b_{i-1}$) yields
the equivalence%
\[
\left(  b_{i-1}\leq b_{i}\right)  \ \Longleftrightarrow\ \left(
\lambda_{b_{i-1}}-b_{i-1}\geq\mu_{i}-i\right)  .
\]
Hence, we have $b_{i-1}\leq b_{i}$ (since we have $\lambda_{b_{i-1}}%
-b_{i-1}\geq\mu_{i}-i$). This completes the proof of the lemma.
\end{proof}

\begin{proof}[Proof of Lemma \ref{lem.flagging-of-lm.0}.]
From $\mu_i = 0$, we obtain $\mu_i - i = -i$.

Let $j > i$ be an integer. Then, $j \geq i+1$ (since $j > i$). Since $\lambda_1 \geq \lambda_2 \geq \lambda_3 \geq \cdots$, we thus have $\lambda_j \leq \lambda_{i+1} = 0$ and therefore $\lambda_j = 0$. Hence, $\lambda_j - j = -j < -i$ (since $j>i$). In other words, $\lambda_j - j < \mu_i - i$ (since $\mu_i - i = -i$).

Forget that we fixed $j$. We thus have shown that for every integer $j > i$, we have $\lambda_j - j < \mu_i - i$.

We defined $b_i$ to be the maximum of the set $\set{k \geq 0 \mid \lambda_k - k \geq \mu_i - i}$. But this set clearly contains $i$ (since $\lambda_i - i \geq -i = \mu_i - i$), but does not contain any integer $j > i$ (because for every integer $j > i$, we have $\lambda_j - j < \mu_i - i$). Thus, its maximum is $i$. In other words, $b_i = i$.
This proves Lemma~\ref{lem.flagging-of-lm.0}.
\end{proof}

\begin{proof}
[Proof of Lemma \ref{lem.ssyt.neighbors}.]
Let us write the two boxes $c$ and
$d$ as $c=\left(  i,j\right)  $ and $d=\left(  u,v\right)  $ for some positive
integers $i$, $j$, $u$ and $v$. From $c=\left(  i,j\right)  $, we obtain%
\begin{equation}
c_{+T}=\left(  T\left(  i,j\right)  ,\ T\left(  i,j\right)  +j-i\right)
.\label{pf.lem.ssyt.neighbors.c+T=}%
\end{equation}
Similarly,%
\begin{equation}
d_{+T}=\left(  T\left(  u,v\right)  ,\ T\left(  u,v\right)  +v-u\right)
.\label{pf.lem.ssyt.neighbors.d+T=}%
\end{equation}
\medskip

\textbf{(a)} Assume that $d_{+T}=c_{+T}$. In view of
(\ref{pf.lem.ssyt.neighbors.c+T=}) and (\ref{pf.lem.ssyt.neighbors.d+T=}), we
can rewrite this as%
\[
\left(  T\left(  u,v\right)  ,\ T\left(  u,v\right)  +v-u\right)  =\left(
T\left(  i,j\right)  ,\ T\left(  i,j\right)  +j-i\right)  .
\]
In other words, the two equalities $T\left(  u,v\right)  =T\left(  i,j\right)
$ and $T\left(  u,v\right)  +v-u=T\left(  i,j\right)  +j-i$ hold. Subtracting
the former equality from the latter, we obtain $v-u=j-i$. In other words,
$v+i=u+j$, so that $v-j=u-i$.

We are in one of the following two cases:

\textit{Case 1:} We have $u\geq i$.

\textit{Case 2:} We have $u<i$.

Let us first consider Case 1. In this case, we have $u\geq i$. Hence,
$u-i\geq0$, so that $v\geq j$ (since $v-j=u-i\geq0$). In other words, $j\leq
v$. Also, $i\leq u$ (since $u\geq i$). Thus, Lemma \ref{lem.ssyt.geq}
\textbf{(b)} yields $T\left(  u,v\right)  -u\geq T\left(  i,j\right)  -i$.
Subtracting the equality $T\left(  u,v\right)  =T\left(  i,j\right)  $ from
this inequality, we find $-u\geq-i$, so that $i\geq u$. Combining this with
$i\leq u$, we obtain $i=u$. Hence, $u=i$, so that $u-i=0$. Thus, $v=j$ (since
$v-j=u-i=0$). Combining $u=i$ with $v=j$, we obtain $\left(  u,v\right)
=\left(  i,j\right)  $. In other words, $d=c$ (since $d=\left(  u,v\right)  $
and $c=\left(  i,j\right)  $). Thus, Lemma \ref{lem.ssyt.neighbors}
\textbf{(a)} is proved in Case 1.

Let us now consider Case 2. In this case, we have $u<i$. Hence, $u-i<0$, so
that $v<j$ (since $v-j=u-i<0$). Now, Lemma \ref{lem.ssyt.geq} \textbf{(b)}
(applied to $\left(  u,v\right)  $ and $\left(  i,j\right)  $ instead of
$\left(  i,j\right)  $ and $\left(  u,v\right)  $) yields $T\left(
i,j\right)  -i\geq T\left(  u,v\right)  -u$ (since $u<i$ and $v<j$). In other
words, $T\left(  u,v\right)  -u\leq T\left(  i,j\right)  -i$. Subtracting the
equality $T\left(  u,v\right)  =T\left(  i,j\right)  $ from this inequality,
we find $-u\leq-i$, so that $u\geq i$. But this contradicts $u<i$. Thus, we
have found a contradiction in Case 2, which shows that Case 2 is impossible.

Consequently, the only possible case is Case 1. Since we have proved Lemma
\ref{lem.ssyt.neighbors} \textbf{(a)} in this Case 1, we thus conclude that
Lemma \ref{lem.ssyt.neighbors} \textbf{(a)} is proved. \medskip

\textbf{(b)} Assume that $d_{+T}=\left(  c_{+T}\right)  _{\rightarrow}$. In
view of (\ref{pf.lem.ssyt.neighbors.c+T=}) and
(\ref{pf.lem.ssyt.neighbors.d+T=}), we can rewrite this as%
\begin{align*}
\left(  T\left(  u,v\right)  ,\ T\left(  u,v\right)  +v-u\right)    & =\left(
T\left(  i,j\right)  ,\ T\left(  i,j\right)  +j-i\right)  _{\rightarrow}\\
& =\left(  T\left(  i,j\right)  ,\ T\left(  i,j\right)  +j-i+1\right)
\end{align*}
(by the definition of $\left(  T\left(  i,j\right)  ,\ T\left(  i,j\right)
+j-i\right)  _{\rightarrow}$). In other words, the two equalities $T\left(
u,v\right)  =T\left(  i,j\right)  $ and $T\left(  u,v\right)  +v-u=T\left(
i,j\right)  +j-i+1$ hold. Subtracting the former equality from the latter, we
obtain $v-u=j-i+1$. In other words, $v+i=u+j+1$, so that $v-j=u-i+1$.

We are in one of the following two cases:

\textit{Case 1:} We have $u\geq i$.

\textit{Case 2:} We have $u<i$.

Let us first consider Case 1. In this case, we have $u\geq i$. Hence,
$u-i\geq0$, and therefore $v>j$ (since $v-j=u-i+1>u-i\geq0$). In other words,
$j<v$. Also, $i\leq u$ (since $u\geq i$). Thus, Lemma \ref{lem.ssyt.geq}
\textbf{(b)} yields $T\left(  u,v\right)  -u\geq T\left(  i,j\right)  -i$.
Subtracting the equality $T\left(  u,v\right)  =T\left(  i,j\right)  $ from
this inequality, we find $-u\geq-i$, so that $i\geq u$. Combining this with
$i\leq u$, we obtain $i=u$. Hence, $u=i$, so that $u-i+1=i-i+1=1$. Thus,
$v=j+1$ (since $v-j=u-i+1=1$). Combining $u=i$ with $v=j+1$, we obtain
$\left(  u,v\right)  =\left(  i,j+1\right)  =\left(  i,j\right)
_{\rightarrow}$. In other words, $d=c_{\rightarrow}$ (since $d=\left(
u,v\right)  $ and $c=\left(  i,j\right)  $). Moreover, we have $T\left(
u,v\right)  =T\left(  i,j\right)  $. In other words, $T\left(  d\right)
=T\left(  c\right)  $ (since $d=\left(  u,v\right)  $ and $c=\left(
i,j\right)  $). Thus, Lemma \ref{lem.ssyt.neighbors} \textbf{(b)} is proved in
Case 1.

Let us now consider Case 2. In this case, we have $u<i$. Hence, $u-i<0$, so
that $u-i\leq-1$ (since $u-i$ is an integer). In other words, $u-i+1\leq0$.
Therefore, $v\leq j$ (since $v-j=u-i+1\leq0$). Now, Lemma \ref{lem.ssyt.geq}
\textbf{(b)} (applied to $\left(  u,v\right)  $ and $\left(  i,j\right)  $
instead of $\left(  i,j\right)  $ and $\left(  u,v\right)  $) yields $T\left(
i,j\right)  -i\geq T\left(  u,v\right)  -u$ (since $u<i$ and $v\leq j$). In
other words, $T\left(  u,v\right)  -u\leq T\left(  i,j\right)  -i$.
Subtracting the equality $T\left(  u,v\right)  =T\left(  i,j\right)  $ from
this inequality, we find $-u\leq-i$, so that $u\geq i$. But this contradicts
$u<i$. Thus, we have found a contradiction in Case 2, which shows that Case 2
is impossible.

Consequently, the only possible case is Case 1. Since we have proved Lemma
\ref{lem.ssyt.neighbors} \textbf{(b)} in this Case 1, we thus conclude that
Lemma \ref{lem.ssyt.neighbors} \textbf{(b)} is proved. \medskip

\textbf{(c)} Assume that $d_{+T}=\left(  c_{+T}\right)  _{\downarrow}$. In
view of (\ref{pf.lem.ssyt.neighbors.c+T=}) and
(\ref{pf.lem.ssyt.neighbors.d+T=}), we can rewrite this as%
\begin{align*}
\left(  T\left(  u,v\right)  ,\ T\left(  u,v\right)  +v-u\right)    & =\left(
T\left(  i,j\right)  ,\ T\left(  i,j\right)  +j-i\right)  _{\downarrow}\\
& =\left(  T\left(  i,j\right)  +1,\ T\left(  i,j\right)  +j-i\right)
\end{align*}
(by the definition of $\left(  T\left(  i,j\right)  ,\ T\left(  i,j\right)
+j-i\right)  _{\downarrow}$). In other words, the two equalities $T\left(
u,v\right)  =T\left(  i,j\right)  +1$ and $T\left(  u,v\right)  +v-u=T\left(
i,j\right)  +j-i$ hold. Subtracting the former equality from the latter, we
obtain $v-u=j-i-1$. In other words, $v+i=u+j-1$, so that $v-j=u-i-1$.

We are in one of the following two cases:

\textit{Case 1:} We have $u>i$.

\textit{Case 2:} We have $u\leq i$.

Let us first consider Case 1. In this case, we have $u>i$. Hence, $u\geq i+1$
(since $u$ and $i$ are integers), so that $u-i\geq1$. In other words,
$u-i-1\geq0$. Therefore, $v\geq j$ (since $v-j=u-i-1\geq0$). In other words,
$j\leq v$. Also, $i\leq u$ (since $u>i$). Thus, Lemma \ref{lem.ssyt.geq}
\textbf{(b)} yields $T\left(  u,v\right)  -u\geq T\left(  i,j\right)  -i$.
Subtracting the equality $T\left(  u,v\right)  =T\left(  i,j\right)  +1$ from
this inequality, we find $-u\geq-i-1$, so that $i+1\geq u$. Combining this
with $u\geq i+1$, we obtain $u=i+1$. In other words, $u-i-1=0$. Therefore,
$v=j$ (since $v-j=u-i-1=0$). Combining $u=i+1$ with $v=j$, we obtain $\left(
u,v\right)  =\left(  i+1,j\right)  =\left(  i,j\right)  _{\downarrow}$. In
other words, $d=c_{\downarrow}$ (since $d=\left(  u,v\right)  $ and $c=\left(
i,j\right)  $). Moreover, we have $T\left(  u,v\right)  =T\left(  i,j\right)
+1$. In other words, $T\left(  d\right)  =T\left(  c\right)  +1$ (since
$d=\left(  u,v\right)  $ and $c=\left(  i,j\right)  $). Thus, Lemma
\ref{lem.ssyt.neighbors} \textbf{(c)} is proved in Case 1.

Let us now consider Case 2. In this case, we have $u\leq i$. Hence, $u-i\leq
0$, so that $u-i-1<u-i\leq0$. Therefore, $v<j$ (since $v-j=u-i-1<0$). Now,
Lemma \ref{lem.ssyt.geq} \textbf{(b)} (applied to $\left(  u,v\right)  $ and
$\left(  i,j\right)  $ instead of $\left(  i,j\right)  $ and $\left(
u,v\right)  $) yields $T\left(  i,j\right)  -i\geq T\left(  u,v\right)  -u$
(since $u\leq i$ and $v<j$). In other words, $T\left(  u,v\right)  -u\leq
T\left(  i,j\right)  -i$. Subtracting the equality $T\left(  u,v\right)
=T\left(  i,j\right)  +1$ from this inequality, we find $-u\leq-i-1<-i$. In
other words, $i<u$. But this contradicts $u\leq i$. Thus, we have found a
contradiction in Case 2, which shows that Case 2 is impossible.

Consequently, the only possible case is Case 1. Since we have proved Lemma
\ref{lem.ssyt.neighbors} \textbf{(c)} in this Case 1, we thus conclude that
Lemma \ref{lem.ssyt.neighbors} \textbf{(c)} is proved. \medskip

\textbf{(d)} Assume that $d_{+T}=\left(  c_{+T}\right)  _{\searrow}$. In view
of (\ref{pf.lem.ssyt.neighbors.c+T=}) and (\ref{pf.lem.ssyt.neighbors.d+T=}),
we can rewrite this as%
\begin{align*}
\left(  T\left(  u,v\right)  ,\ T\left(  u,v\right)  +v-u\right)    & =\left(
T\left(  i,j\right)  ,\ T\left(  i,j\right)  +j-i\right)  _{\searrow}\\
& =\left(  T\left(  i,j\right)  +1,\ T\left(  i,j\right)  +j-i+1\right)
\end{align*}
(by the definition of $\left(  T\left(  i,j\right)  ,\ T\left(  i,j\right)
+j-i\right)  _{\searrow}$). In other words, the two equalities $T\left(
u,v\right)  =T\left(  i,j\right)  +1$ and $T\left(  u,v\right)  +v-u=T\left(
i,j\right)  +j-i+1$ hold. Subtracting the former equality from the latter, we
obtain $v-u=j-i$. In other words, $v+i=u+j$, so that $v-j=u-i$.

We are in one of the following two cases:

\textit{Case 1:} We have $u>i$.

\textit{Case 2:} We have $u\leq i$.

Let us first consider Case 1. In this case, we have $u>i$.  Therefore, $u\geq
i+1$ (since $u$ and $i$ are integers), so that $u-i\geq1$. Therefore, $v>j$
(since $v-j=u-i\geq1>0$). In other words, $j<v$. Also, $i<u$ (since $u>i$).
Thus, Lemma \ref{lem.ssyt.geq} \textbf{(b)} yields $T\left(  u,v\right)
-u\geq T\left(  i,j\right)  -i$. Subtracting the equality $T\left(
u,v\right)  =T\left(  i,j\right)  +1$ from this inequality, we find
$-u\geq-i-1$, so that $i+1\geq u$. Combining this with $u\geq i+1$, we obtain
$u=i+1$. In other words, $u-i=1$. Hence, $v=j+1$ (since $v-j=u-i=1$).

Combining $u=i+1$ with $v=j+1$, we obtain $\left(  u,v\right)  =\left(
i+1,j+1\right)  =\left(  i,j\right)  _{\searrow}$. In other words,
$d=c_{\searrow}$ (since $d=\left(  u,v\right)  $ and $c=\left(  i,j\right)
$). Moreover, we have $T\left(  u,v\right)  =T\left(  i,j\right)  +1$. In
other words, $T\left(  d\right)  =T\left(  c\right)  +1$ (since $d=\left(
u,v\right)  $ and $c=\left(  i,j\right)  $). Thus, Lemma
\ref{lem.ssyt.neighbors} \textbf{(d)} is proved in Case 1.

Let us now consider Case 2. In this case, we have $u\leq i$. Hence, $u-i\leq
0$, so that $v\leq j$ (since $v-j=u-i\leq0$). Now, Lemma \ref{lem.ssyt.geq}
\textbf{(b)} (applied to $\left(  u,v\right)  $ and $\left(  i,j\right)  $
instead of $\left(  i,j\right)  $ and $\left(  u,v\right)  $) yields $T\left(
i,j\right)  -i\geq T\left(  u,v\right)  -u$ (since $u\leq i$ and $v\leq j$).
In other words, $T\left(  u,v\right)  -u\leq T\left(  i,j\right)  -i$.
Subtracting the equality $T\left(  u,v\right)  =T\left(  i,j\right)  +1$ from
this inequality, we find $-u\leq-i-1<-i$. In other words, $i<u$. But this
contradicts $u\leq i$. Thus, we have found a contradiction in Case 2, which
shows that Case 2 is impossible.

Consequently, the only possible case is Case 1. Since we have proved Lemma
\ref{lem.ssyt.neighbors} \textbf{(d)} in this Case 1, we thus conclude that
Lemma \ref{lem.ssyt.neighbors} \textbf{(d)} is proved.
\end{proof}

\begin{proof}
[Proof of Lemma \ref{lem.DS-vs-DT}.]
If $c\in Y\left(  \mu\right)  $ is any box, then the box $c_{+T}$
depends only on the position of the box $c$ and on the entry $T\left(
c\right)  $ in this box. Thus, if a box $c$ has the same entry in both
tableaux $T$ and $S$ (that is, satisfies $T\left(  c\right)  =S\left(
c\right)  $), then the boxes $c_{+T}$ and $c_{+S}$ will also be identical.
Therefore, from (\ref{eq.lem.DS-vs-DT.ass}), we obtain%
\begin{equation}
c_{+T}=c_{+S}\ \ \ \ \ \ \ \ \ \ \text{for all }c\in Y\left(  \mu\right)
\text{ distinct from }\left(  i,j\right)  .\label{pf.lem.DS-vs-DT.1}%
\end{equation}

Now, define the set $A:=\left\{  c_{+T}\ \mid\ c\in Y\left(  \mu\right)
\text{ such that }c\neq\left(  i,j\right)  \right\}  $. Then,%
\begin{align*}
A &  =\left\{  c_{+T}\ \mid\ c\in Y\left(  \mu\right)  \text{ such that }%
c\neq\left(  i,j\right)  \right\}  \\
&  =\left\{  c_{+S}\ \mid\ c\in Y\left(  \mu\right)  \text{ such that }%
c\neq\left(  i,j\right)  \right\}
\end{align*}
(since (\ref{pf.lem.DS-vs-DT.1}) allows us to rewrite $c_{+T}$ as $c_{+S}$ here).

Lemma \ref{lem.ssyt.neighbors} \textbf{(a)} shows that if two boxes $d,c\in
Y\left(  \mu\right)  $ satisfy $d_{+T}=c_{+T}$, then $d=c$. In other words,
the boxes $c_{+T}$ for all $c\in Y\left(  \mu\right)  $ are distinct. Hence,
in particular, for any box $c\in Y\left(  \mu\right)  $ that satisfies
$c\neq\left(  i,j\right)  $, we have $c_{T}\neq\left(  i,j\right)  _{+T}$. In
other words, any element of $A$ is distinct from $\left(  i,j\right)  _{+T}$
(since the elements of $A$ are exactly the boxes of the form $c_{+T}$, where
$c\in Y\left(  \mu\right)  $ is a box that satisfies $c\neq\left(
i,j\right)  $). In other words, $\left(  i,j\right)  _{+T}\notin A$.

However, the definition of $\DD\tup{T}$ yields
\begin{align*}
\mathbf{D}\left(  T\right)   &  =\left\{  c_{+T}\ \mid\ c\in Y\left(
\mu\right)  \right\}  \\
&  =\left\{  c_{+T}\ \mid\ c\in Y\left(  \mu\right)  \text{ such that }%
c\neq\left(  i,j\right)  \right\}  \cup\left\{  \left(  i,j\right)
_{+T}\right\}
\end{align*}
(here, we have split off the element for $c=\left(  i,j\right)  $ from the
set, since $\left(  i,j\right)  \in Y\left(  \mu\right)  $). In other words,%
\begin{equation}
\mathbf{D}\left(  T\right)  =A\cup\left\{  \left(  i,j\right)  _{+T}\right\}
\label{pf.lem.DS-vs-DT.DT=Acup}
\end{equation}
(since $A=\left\{  c_{+T}\ \mid\ c\in Y\left(  \mu\right)  \text{ such that
}c\neq\left(  i,j\right)  \right\}  $). Hence,%
\[
\mathbf{D}\left(  T\right)  \setminus\left\{  \left(  i,j\right)
_{+T}\right\}  =\left(  A\cup\left\{  \left(  i,j\right)  _{+T}\right\}
\right)  \setminus\left\{  \left(  i,j\right)  _{+T}\right\}  =A
\]
(since $\left(  i,j\right)  _{+T}\notin A$). The same argument (applied to $S$
instead of $T$) yields%
\[
\mathbf{D}\left(  S\right)  \setminus\left\{  \left(  i,j\right)
_{+S}\right\}  =A
\]
(since $A=\left\{  c_{+S}\ \mid\ c\in Y\left(  \mu\right)  \text{ such that
}c\neq\left(  i,j\right)  \right\}  $).

Now, if we replace the box $\left(  i,j\right)  _{+S}$ in the set
$\mathbf{D}\left(  S\right)  $ by the box $\left(  i,j\right)  _{+T}$, then
we obtain the set%
\[
\underbrace{\left(  \mathbf{D}\left(  S\right)  \setminus\left\{  \left(
i,j\right)  _{+S}\right\}  \right)  }_{=A}\cup\left\{  \left(  i,j\right)
_{+T}\right\}  =A\cup\left\{  \left(  i,j\right)  _{+T}\right\}
=\mathbf{D}\left(  T\right)
\qquad \qquad \left(\text{by \eqref{pf.lem.DS-vs-DT.DT=Acup}}\right) .
\]
In other words, the diagram $\mathbf{D}\left(  T\right)  $
can be obtained from $\mathbf{D}\left(  S\right)  $ by replacing the box
$\left(  i,j\right)  _{+S}$ by the box  $\left(  i,j\right)  _{+T}$. This
proves Lemma \ref{lem.DS-vs-DT}.
\end{proof}

\begin{proof}
[Proof of Lemma \ref{transition_wd}.]We begin with part \textbf{(b)} of the lemma.

\textbf{(b)} The map
\begin{align*}
Y\left(  \mu\right)   &  \rightarrow\mathbf{D}\left(  T\right)  ,\\
c  &  \mapsto c_{+T}%
\end{align*}
is well-defined (by (\ref{eq.def.DD.DDT=})) and surjective (again by
(\ref{eq.def.DD.DDT=})). Furthermore, this map is injective (since Lemma
\ref{lem.ssyt.neighbors} \textbf{(a)} shows that the boxes $c_{+T}$ for all
$c\in Y\left(  \mu\right)  $ are distinct). Hence, this map is bijective. In
other words, the map
\begin{align*}
Y\left(  \mu\right)   &  \rightarrow\mathbf{D}\left(  T\right)  ,\\
\left(  i,j\right)   &  \mapsto\left(  T\left(  i,j\right)  ,\ T\left(
i,j\right)  +j-i\right)
\end{align*}
is bijective (indeed, this is the same map as the one discussed in the
preceding sentence, since $c_{+T}$ is defined to be $\left(  T\left(
i,j\right)  ,\ T\left(  i,j\right)  +j-i\right)  $ for every box $c=\left(
i,j\right)  $). Hence, we can substitute $\left(  T\left(  i,j\right)
,\ T\left(  i,j\right)  +j-i\right)  $ for $\left(  i,j\right)  $ in the
product $\prod_{\left(  i,j\right)  \in\mathbf{D}\left(  T\right)  }\left(
x_{i}+y_{j}\right)  $. We thus obtain
\[
\prod_{\left(  i,j\right)  \in\mathbf{D}\left(  T\right)  }\left(  x_{i}%
+y_{j}\right)  =\prod_{\left(  i,j\right)  \in Y\left(  \mu\right)  }\left(
x_{T\left(  i,j\right)  }+y_{T\left(  i,j\right)  +j-i}\right)  .
\]
This proves Lemma \ref{transition_wd} \textbf{(b)}. \medskip

\textbf{(a)} Forget that we fixed $T$.

We define the \emph{load} of a semistandard tableau $T \in \SSYT\tup{\mu}$ to be the nonnegative
integer $\sum_{c\in Y\left(  \mu\right)  }T\left(  c\right)  $ (that is, the
sum of all entries of $T$). Note that this load is positive whenever $\mu
\neq\varnothing$ (since $\mu\neq\varnothing$ means that $Y\left(  \mu\right)
$ contains at least one box $c$, and of course the entry $T\left(  c\right)
$ in this box $c$ must be positive).

We shall prove Lemma \ref{transition_wd} \textbf{(a)} by induction on the load
of $T$.

\textit{Base case:} The load of $T$ can be $0$ only if $\mu=\varnothing$
(since the load of $T$ is positive whenever $\mu\neq\varnothing$). Of course,
$\mathbf{D}\left(  T\right)  =\varnothing$ in this case, and this renders the
claim of Lemma \ref{transition_wd} \textbf{(a)} trivial (since the empty
diagram $\varnothing$ is clearly an excitation of $Y\left(  \mu\right)
=\varnothing$). Thus, Lemma \ref{transition_wd} \textbf{(a)} is proved in the
case when the load of $T$ is $0$.

\textit{Induction step:} Fix a positive integer $n$. Assume (as the induction
hypothesis) that Lemma \ref{transition_wd} \textbf{(a)} is proved for all
tableaux $T\in\operatorname*{SSYT}\left(  \mu\right)  $ of load $n-1$. We now
fix a semistandard tableau $T\in\operatorname*{SSYT}\left(  \mu\right)  $ of
load $n$. Our goal is to prove Lemma \ref{transition_wd} \textbf{(a)} for this
$T$.

Lemma \ref{lem.ssyt.geq} \textbf{(a)} shows that the inequality $T\left(
i,j\right)  \geq i$ holds for each $\left(  i,j\right)  \in Y\left(
\mu\right)  $. If this inequality is always an equality, then we have
$\mathbf{D}\left(  T\right)  =Y\left(  \mu\right)  $ (this is easy to
see\footnote{\textit{Proof.} Assume that the inequality $T\left(  i,j\right)
\geq i$ is always an equality. In other words, each $\left(  i,j\right)  \in
Y\left(  \mu\right)  $ satisfies $T\left(  i,j\right)  =i$. Hence, each
$\left(  i,j\right)  \in Y\left(  \mu\right)  $ satisfies
\begin{align*}
\left(  i,j\right)  _{+T}  &  =\left(  T\left(  i,j\right)  ,\ T\left(
i,j\right)  +j-i\right)  \ \ \ \ \ \ \ \ \ \ \left(  \text{by the definition
of }\left(  i,j\right)  _{+T}\right) \\
&  =\left(  i,\ i+j-i\right)  \ \ \ \ \ \ \ \ \ \ \left(  \text{since
}T\left(  i,j\right)  =i\right) \\
&  =\left(  i,j\right)  .
\end{align*}
In other words, each $c\in Y\left(  \mu\right)  $ satisfies $c_{+T}=c$. Hence,
(\ref{eq.def.DD.DDT=}) can be rewritten as $\mathbf{D}\left(  T\right)
=\left\{  c\ \mid\ c\in Y\left(  \mu\right)  \right\}  =Y\left(  \mu\right)
$, qed.}), and thus the claim of Lemma \ref{transition_wd} \textbf{(a)} holds
(since $Y\left(  \mu\right)  $ itself is clearly an excitation of $Y\left(
\mu\right)  $, obtained by making a sequence of $0$ excited moves). Thus, for
the rest of this induction step, we WLOG assume that the inequality $T\left(
i,j\right)  \geq i$ is \textbf{not} always an equality. Hence, there exists a
box $\left(  i,j\right)  \in Y\left(  \mu\right)  $ such that $T\left(
i,j\right)  \neq i$. We call such a box \emph{interesting}.

Now, let $\left(  i,j\right)  $ be an interesting box with smallest possible
$i+j$. Then, a box $\left(  u,v\right)  $ with $u+v<i+j$ cannot be
interesting. In particular:

\begin{itemize}
\item The box $\left(  i-1,j\right)  $ cannot be interesting (since $\left(
i-1\right)  +j<i+j$). Hence,
\begin{equation}
\text{if }i>1\text{, then }T\left(  i-1,j\right)  =i-1
\label{pf.transition_wd.Ti-1j}%
\end{equation}
(because otherwise, the box $\left(  i-1,j\right)  $ would be interesting).

\item The box $\left(  i,j-1\right)  $ cannot be interesting (since
$i+\left(  j-1\right)  <i+j$). Hence,
\begin{equation}
\text{if }j>1\text{, then }T\left(  i,j-1\right)  =i
\label{pf.transition_wd.Tij-1}%
\end{equation}
(because otherwise, the box $\left(  i,j-1\right)  $ would be interesting).
\end{itemize}

Note that $T\left(  i,j\right)  \neq i$ (since the box $\left(  i,j\right)  $
is interesting) and therefore $T\left(  i,j\right)  >i$ (since Lemma
\ref{lem.ssyt.geq} \textbf{(a)} yields $T\left(  i,j\right)  \geq i$). Hence,
$T\left(  i,j\right)  \geq i+1$ (since $T\left(  i,j \right) $ and $i$ are
integers), so that%
\[
T\left(  i,j\right)  -1\geq i\geq1.
\]

Now, let us decrease the entry $T\left(  i,j\right)  $ of the tableau $T$ by
$1$, while leaving all other entries unchanged. The resulting filling of
$Y\left(  \mu\right)  $ will be called $\overline{T}$. Formally speaking,
$\overline{T}$ is thus the map from $Y\left(  \mu\right)  $ to $\set{1,2,3,\ldots}$ given by%
\begin{align}
\overline{T}\left(  i,j\right)   &  =T\left(  i,j\right)
-1\ \ \ \ \ \ \ \ \ \ \text{and}\label{pf.transition_wd.Tbar.1}\\
\overline{T}\left(  c\right)   &  =T\left(  c\right)
\ \ \ \ \ \ \ \ \ \ \text{for all }c\in Y\left(  \mu\right)  \text{ distinct
from }\left(  i,j\right)  . \label{pf.transition_wd.Tbar.2}%
\end{align}

It is easy to see (using \eqref{pf.transition_wd.Ti-1j} and
\eqref{pf.transition_wd.Tij-1}) that $\overline{T}$ is again a semistandard
tableau\footnote{\textit{Proof.} First, we note that the entries of
$\overline{T}$ are positive integers (since $\overline{T}\left(  i,j\right)
=T\left(  i,j\right)  -1\geq1$ shows that $\overline{T}\left(  i,j\right)  $
is a positive integer, and since all the other entries of $\overline{T}$ are
copied from $T$).
\par
Next, we claim that the entries of $\overline{T}$ weakly increase
left-to-right in each row. Indeed, by the construction of $\overline{T}$, this
will follow from the analogous property of $T$, as long as we can show that
the decreased entry $\overline{T}\left(  i,j\right)  $ is still greater or
equal to its neighboring entry $\overline{T}\left(  i,j-1\right)  $ (assuming
that $j>1$). But we can easily show this: If $j>1$, then
\begin{align*}
\overline{T}\left(  i,j-1\right)   &  =T\left(  i,j-1\right)
\ \ \ \ \ \ \ \ \ \ \left(  \text{by (\ref{pf.transition_wd.Tbar.1})}\right)
\\
&  =i\ \ \ \ \ \ \ \ \ \ \left(  \text{by (\ref{pf.transition_wd.Tij-1}%
)}\right)
\end{align*}
and thus
\[
\overline{T}\left(  i,j\right)  =T\left(  i,j\right)  -1\geq i=\overline
{T}\left(  i,j-1\right)  .
\]
Thus, we conclude that the entries of $\overline{T}$ weakly increase
left-to-right in each row.
\par
Finally, we claim that the entries of $\overline{T}$ strictly increase
top-to-bottom in each column. Indeed, by the construction of $\overline{T}$,
this will follow from the analogous property of $T$, as long as we can show
that the decreased entry $\overline{T}\left(  i,j\right)  $ is still greater
than its neighboring entry $\overline{T}\left(  i-1,j\right)  $ (assuming that
$i>1$). But we can easily show this: If $i>1$, then
\begin{align*}
\overline{T}\left(  i-1,j\right)   &  =T\left(  i-1,j\right)
\ \ \ \ \ \ \ \ \ \ \left(  \text{by (\ref{pf.transition_wd.Tbar.1})}\right)
\\
&  =i-1\ \ \ \ \ \ \ \ \ \ \left(  \text{by (\ref{pf.transition_wd.Ti-1j}%
)}\right)
\end{align*}
and thus
\[
\overline{T}\left(  i,j\right)  =T\left(  i,j\right)  -1\geq i>i-1=\overline
{T}\left(  i-1,j\right)  .
\]
Thus, we conclude that the entries of $\overline{T}$ strictly increase
top-to-bottom in each column.
\par
Altogether, we have now shown that $\overline{T}$ is a semistandard tableau.}.
Hence, $\overline{T}\in\operatorname*{SSYT}\left(  \mu\right)  $.

This tableau $\overline{T}$ is obtained from $T$ by decreasing the $\left(
i,j\right)  $-th entry by $1$. Thus, $\overline{T}$ has load $n-1$ (since $T$
has load $n$, but the load of a tableau is just the sum of its entries).
Hence, by our induction hypothesis, Lemma \ref{transition_wd} \textbf{(a)}
holds for $\overline{T}$ instead of $T$. In other words, $\mathbf{D}\left(
\overline{T}\right)  $ is an excitation of the diagram $Y\left(  \mu\right)
$. In other words, $\mathbf{D}\left(  \overline{T}\right)  $ can be obtained
from $Y\left(  \mu\right)  $ by a sequence of excited moves.

From (\ref{pf.transition_wd.Tbar.1}), we easily obtain\footnote{Recall that
$\left(  u,v\right)  _{\searrow}$ denotes the southeastern neighbor $\left(
u+1,v+1\right)  $ of a box $\left(  u,v\right)  $.}%
\begin{equation}
\left(  i,j\right)  _{+T}=\left(  \left(  i,j\right)  _{+\overline{T}}\right)
_{\searrow}\label{pf.transition_wd.+T1}%
\end{equation}
\footnote{\textit{Proof.} From (\ref{pf.transition_wd.Tbar.1}), we obtain
$\overline{T}\left(  i,j\right)  +1=T\left(  i,j\right)  $. The definition of
$\left(  i,j\right)  _{+\overline{T}}$ yields%
\[
\left(  i,j\right)  _{+\overline{T}}=\left(  \overline{T}\left(  i,j\right)
,\ \overline{T}\left(  i,j\right)  +j-i\right)  .
\]
Hence,
\begin{align*}
\left(  \left(  i,j\right)  _{+\overline{T}}\right)  _{\searrow} &  =\left(
\overline{T}\left(  i,j\right)  +1,\ \underbrace{\left(  \overline{T}\left(
i,j\right)  +j-i\right)  +1}_{=\overline{T}\left(  i,j\right)  +1+j-i}\right)
=\left(  \underbrace{\overline{T}\left(  i,j\right)  +1}_{=T\left(
i,j\right)  },\ \underbrace{\overline{T}\left(  i,j\right)  +1}_{=T\left(
i,j\right)  }+\,j-i\right) \\
&=\left(  T\left(  i,j\right)  ,\ T\left(
i,j\right)  +j-i\right)
=\left(  i,j\right)  _{+T}%
\end{align*}
(by the definition of $\left(  i,j\right)  _{+T}$).}.

We can rewrite (\ref{pf.transition_wd.Tbar.2}) as follows:%
\[
T\left(  c\right)  =\overline{T}\left(  c\right)
\ \ \ \ \ \ \ \ \ \ \text{for all }c\in Y\left(  \mu\right)  \text{ distinct
from }\left(  i,j\right)  .
\]
Hence, Lemma \ref{lem.DS-vs-DT} (applied to $S=\overline{T}$) yields that the
diagram $\mathbf{D}\left(  T\right)  $ can be obtained from $\mathbf{D}\left(
\overline{T}\right)  $ by replacing the box $\left(  i,j\right)
_{+\overline{T}}$ by the box $\left(  i,j\right)  _{+T}$.

In view of (\ref{pf.transition_wd.+T1}), we can rewrite this as follows: The
diagram $\mathbf{D}\left(  T\right)  $ can be obtained from $\mathbf{D}\left(
\overline{T}\right)  $ by replacing the box $\left(  i,j\right)
_{+\overline{T}}$ by its southeastern neighbor $\left(  \left(  i,j\right)
_{+\overline{T}}\right)  _{\searrow}$.

We shall now show that this replacement is an excited move.

Indeed, in order to show this, we need only to check that the diagram
$\mathbf{D}\left(  \overline{T}\right)  $ contains none of the three boxes
$\left(  \left(  i,j\right)  _{+\overline{T}}\right)  _{\downarrow}$, $\left(
\left(  i,j\right)  _{+\overline{T}}\right)  _{\rightarrow}$ and $\left(
\left(  i,j\right)  _{+\overline{T}}\right)  _{\searrow}$ (because this is the
condition under which the replacement of $\left(  i,j\right)  _{+\overline{T}%
}$ by $\left(  \left(  i,j\right)  _{+\overline{T}}\right)  _{\searrow}$ is
considered an excited move). To do so, we assume the contrary. Thus,
$\mathbf{D}\left(  \overline{T}\right)  $ contains one of these three boxes.
In other words, we must be in one of the following three cases:

\textit{Case 1:} The diagram $\mathbf{D}\left(  \overline{T}\right)  $
contains $\left(  \left(  i,j\right)  _{+\overline{T}}\right)  _{\downarrow}$.

\textit{Case 2:} The diagram $\mathbf{D}\left(  \overline{T}\right)  $
contains $\left(  \left(  i,j\right)  _{+\overline{T}}\right)
_{\rightarrow}$.

\textit{Case 3:} The diagram $\mathbf{D}\left(  \overline{T}\right)  $
contains $\left(  \left(  i,j\right)  _{+\overline{T}}\right)  _{\searrow}$.

\begin{enumerate}
\item
Let us first consider Case 1. In this case, the diagram $\mathbf{D}\left(
\overline{T}\right)  $ contains $\left(  \left(  i,j\right)  _{+\overline{T}%
}\right)  _{\downarrow}$. In other words,%
\[
\left(  \left(  i,j\right)  _{+\overline{T}}\right)  _{\downarrow}%
\in\mathbf{D}\left(  \overline{T}\right)  =\left\{  c_{+\overline{T}}%
\ \mid\ c\in Y\left(  \mu\right)  \right\}
\]
(by the definition of $\mathbf{D}\left(  \overline{T}\right)  $). In other
words, there exists some $c\in Y\left(  \mu\right)  $ such that $\left(
\left(  i,j\right)  _{+\overline{T}}\right)  _{\downarrow}=c_{+\overline{T}}$.
Consider this $c$.

The two boxes $c$ and $\left(  i,j\right)  $ in $Y\left(  \mu\right)  $
satisfy $c_{+\overline{T}}=\left(  \left(  i,j\right)  _{+\overline{T}%
}\right)  _{\downarrow}$. Hence, Lemma \ref{lem.ssyt.neighbors} \textbf{(c)}
(applied to $\overline{T}$, $\left(  i,j\right)  $ and $c$ instead of $T$, $c$
and $d$) yields that $c=\left(  i,j\right)  _{\downarrow}$ and $\overline
{T}\left(  c\right)  =\overline{T}\left(  i,j\right)  +1$. However,
(\ref{pf.transition_wd.Tbar.2}) yields $\overline{T}\left(  c\right)
=T\left(  c\right)  $ (since $c=\left(  i,j\right)  _{\downarrow}\neq\left(
i,j\right)  $). Thus,%
\[
T\left(  c\right)  =\overline{T}\left(  c\right)  =\overline{T}\left(
i,j\right)  +1=T\left(  i,j\right)  \ \ \ \ \ \ \ \ \ \ \left(  \text{by
(\ref{pf.transition_wd.Tbar.1})}\right)  .
\]
However, $c=\left(  i,j\right)  _{\downarrow}=\left(  i+1,j\right)  $, so that
$T\left(  c\right)  =T\left(  i+1,j\right)  >T\left(  i,j\right)  $ (since $T$
is a semistandard tableau, and thus its entries strictly increase
top-to-bottom in each column). This contradicts $T\left(  c\right)  =T\left(
i,j\right)  $. Thus, we have found a contradiction in Case 1.

\item
Let us next consider Case 2. In this case, the diagram $\mathbf{D}\left(
\overline{T}\right)  $ contains $\left(  \left(  i,j\right)  _{+\overline{T}%
}\right)  _{\rightarrow}$. In other words,%
\[
\left(  \left(  i,j\right)  _{+\overline{T}}\right)  _{\rightarrow}%
\in\mathbf{D}\left(  \overline{T}\right)  =\left\{  c_{+\overline{T}}%
\ \mid\ c\in Y\left(  \mu\right)  \right\}
\]
(by the definition of $\mathbf{D}\left(  \overline{T}\right)  $). In other
words, there exists some $c\in Y\left(  \mu\right)  $ such that $\left(
\left(  i,j\right)  _{+\overline{T}}\right)  _{\rightarrow}=c_{+\overline{T}}%
$. Consider this $c$.

The two boxes $c$ and $\left(  i,j\right)  $ in $Y\left(  \mu\right)  $
satisfy $c_{+\overline{T}}=\left(  \left(  i,j\right)  _{+\overline{T}%
}\right)  _{\rightarrow}$. Hence, Lemma \ref{lem.ssyt.neighbors} \textbf{(b)}
(applied to $\overline{T}$, $\left(  i,j\right)  $ and $c$ instead of $T$, $c$
and $d$) yields that $c=\left(  i,j\right)  _{\rightarrow}$ and $\overline
{T}\left(  c\right)  =\overline{T}\left(  i,j\right)  $. However,
(\ref{pf.transition_wd.Tbar.2}) yields $\overline{T}\left(  c\right)
=T\left(  c\right)  $ (since $c=\left(  i,j\right)  _{\rightarrow}\neq\left(
i,j\right)  $). Thus,%
\begin{align*}
T\left(  c\right)   &  =\overline{T}\left(  c\right)  =\overline{T}\left(
i,j\right)  =T\left(  i,j\right)  -1\ \ \ \ \ \ \ \ \ \ \left(  \text{by
(\ref{pf.transition_wd.Tbar.1})}\right) \\
&  <T\left(  i,j\right)  .
\end{align*}
However, $c=\left(  i,j\right)  _{\rightarrow}=\left(  i,j+1\right)  $, so
that $T\left(  c\right)  =T\left(  i,j+1\right)  \geq T\left(  i,j\right)  $
(since $T$ is a semistandard tableau, and thus its entries weakly increase
left-to-right in each row). This contradicts $T\left(  c\right)  <T\left(
i,j\right)  $. Thus, we have found a contradiction in Case 2.

\item
Let us finally consider Case 3. In this case, the diagram $\mathbf{D}\left(
\overline{T}\right)  $ contains $\left(  \left(  i,j\right)  _{+\overline{T}%
}\right)  _{\searrow}$. In other words,%
\[
\left(  \left(  i,j\right)  _{+\overline{T}}\right)  _{\searrow}\in
\mathbf{D}\left(  \overline{T}\right)  =\left\{  c_{+\overline{T}}\ \mid\ c\in
Y\left(  \mu\right)  \right\}
\]
(by the definition of $\mathbf{D}\left(  \overline{T}\right)  $). In other
words, there exists some $c\in Y\left(  \mu\right)  $ such that $\left(
\left(  i,j\right)  _{+\overline{T}}\right)  _{\searrow}=c_{+\overline{T}}$.
Consider this $c$.

The two boxes $c$ and $\left(  i,j\right)  $ in $Y\left(  \mu\right)  $
satisfy $c_{+\overline{T}}=\left(  \left(  i,j\right)  _{+\overline{T}%
}\right)  _{\searrow}$. Hence, Lemma \ref{lem.ssyt.neighbors} \textbf{(d)}
(applied to $\overline{T}$, $\left(  i,j\right)  $ and $c$ instead of $T$, $c$
and $d$) yields that $c=\left(  i,j\right)  _{\searrow}$ and $\overline
{T}\left(  c\right)  =\overline{T}\left(  i,j\right)  +1$. However,
(\ref{pf.transition_wd.Tbar.2}) yields $\overline{T}\left(  c\right)
=T\left(  c\right)  $ (since $c=\left(  i,j\right)  _{\searrow}\neq\left(
i,j\right)  $). Thus,%
\[
T\left(  c\right)  =\overline{T}\left(  c\right)  =\overline{T}\left(
i,j\right)  +1=T\left(  i,j\right)  \ \ \ \ \ \ \ \ \ \ \left(  \text{by
(\ref{pf.transition_wd.Tbar.1})}\right)  .
\]
However, $c=\left(  i,j\right)  _{\searrow}=\left(  i+1,j+1\right)  $, and
we have $i\leq i+1$ and $j\leq j+1$. Hence, Lemma \ref{lem.ssyt.geq}
\textbf{(b)} (applied to $\left(  u,v\right)  =\left(  i+1,j+1\right)  $)
yields $T\left(  i+1,j+1\right)  -\left(  j+1\right)  \geq T\left(
i,j\right)  -j$. Hence,%
\[
T\left(  i+1,j+1\right)  \geq\left(  T\left(  i,j\right)  -j\right)  +\left(
j+1\right)  =T\left(  i,j\right)  +1>T\left(  i,j\right)  .
\]
In other words, $T\left(  c\right)  >T\left(  i,j\right)  $ (since $c=\left(
i+1,j+1\right)  $). This contradicts $T\left(  c\right)  =T\left(  i,j\right)
$. Thus, we have found a contradiction in Case 3.
\end{enumerate}

We have now found a contradiction in each of our three cases. Thus, we always
obtain a contradiction. This shows that our assumption was false.

Thus, we have shown that our replacement of $\left(  i,j\right)
_{+\overline{T}}$ by $\left(  \left(  i,j\right)  _{+\overline{T}}\right)
_{\searrow}$ in $\mathbf{D}\left(  \overline{T}\right)  $ is an excited move.
Let $\mathbf{e}$ denote this excited move. Then, the diagram $\mathbf{D}%
\left(  T\right)  $ can be obtained from $\mathbf{D}\left(  \overline
{T}\right)  $ by this excited move $\mathbf{e}$ (since we know that the
diagram $\mathbf{D}\left(  T\right)  $ can be obtained from $\mathbf{D}\left(
\overline{T}\right)  $ by replacing the box $\left(  i,j\right)
_{+\overline{T}}$ by its southeastern neighbor $\left(  \left(  i,j\right)
_{+\overline{T}}\right)  _{\searrow}$). Since $\mathbf{D}\left(  \overline
{T}\right)  $ can, in turn, be obtained from $Y\left(  \mu\right)  $ by a
sequence of excited moves (as we know), we thus conclude that $\mathbf{D}%
\left(  T\right)  $ can be obtained from $Y\left(  \mu\right)  $ by a sequence
of excited moves (just apply the sequence of excited moves that gives
$\mathbf{D}\left(  \overline{T}\right)  $ first, and then perform the excited
move $\mathbf{e}$ to transform it further into $\mathbf{D}\left(  T\right)
$). In other words, $\mathbf{D}\left(  T\right)  $ is an excitation of
$Y\left(  \mu\right)  $. This proves Lemma \ref{transition_wd} \textbf{(a)}
for our tableau $T$. This completes the induction step, and thus Lemma
\ref{transition_wd} \textbf{(a)} is proved.

\medskip

\textbf{(c)} Let $\mathbf{b}=\left(  b_{1},b_{2},b_{3},\ldots\right)  $ be the
flagging induced by $\lambda/\mu$. Thus, $\mathcal{F}\left(  \lambda
/\mu\right)  =\operatorname{FSSYT}\left(  \mu,\mathbf{b}\right)  $ (by the
definition of $\mathcal{F}\left(  \lambda/\mu\right)  $).

We must prove the equivalence $\tup{\mathbf{D}\tup{T} \in \calE\tup{\lambda/\mu}} \Longleftrightarrow \tup{T \in \calF\tup{\lambda/\mu}}$. Let us prove the ``$\Longrightarrow$'' and ``$\Longleftarrow$'' directions of this equivalence separately:

$\Longrightarrow:$ Assume that $\mathbf{D}\left(  T\right)  \in\mathcal{E}%
\left(  \lambda/\mu\right)  $. We must prove that $T\in\mathcal{F}\left(
\lambda/\mu\right)  $.

Let $\left(  i,j\right)  \in Y\left(  \mu\right)  $. We shall prove that
$T\left(  i,j\right)  \leq b_{i}$.

Indeed, consider the box $\left(  i,\mu_{i}\right)  $. Then, $\left(
i,\mu_{i}\right)  \in Y\left(  \mu\right)  $ (since $\mu_{i}\leq\mu_{i}$), so
that $\left(  i,\mu_{i}\right)  _{+T}\in\mathbf{D}\left(  T\right)  $ (by
(\ref{eq.def.DD.DDT=})).

Set $k:=T\left(  i,\mu_{i}\right)  $ (this is well-defined since $\left(
i,\mu_{i}\right)  \in Y\left(  \mu\right)  $). Then, the definition of
$\left(  i,\mu_{i}\right)  _{+T}$ yields
\begin{align*}
\left(  i,\mu_{i}\right)  _{+T}  &  =\left(  T\left(  i,\mu_{i}\right)
,\ T\left(  i,\mu_{i}\right)  +\mu_{i}-i\right) \\
&  =\left(  k,\ k+\mu_{i}-i\right)  \ \ \ \ \ \ \ \ \ \ \left(  \text{since
}T\left(  i,\mu_{i}\right)  =k\right)  .
\end{align*}
Hence, $\left(  k,\ k+\mu_{i}-i\right)  =\left(  i,\mu_{i}\right)  _{+T}%
\in\mathbf{D}\left(  T\right)  \subseteq Y\left(  \lambda\right)  $ (since
$\mathbf{D}\left(  T\right)  \in\mathcal{E}\left(  \lambda/\mu\right)  $). In other words, $k+\mu_{i}-i\leq\lambda_{k}$. In other words, $\lambda_{k}%
-k\geq\mu_{i}-i$. By Lemma \ref{lem.flagging-of-lm.uniprop} (applied to
$j=k$), this is equivalent to $k\leq b_{i}$. Thus, we have $k\leq b_i$.
However, from $\left(  i,j\right)  \in Y\left(  \mu\right)  $, we obtain
$j\leq\mu_{i}$. Since the entries of $T$ weakly increase left-to-right in each
row, this entails $T\left(  i,j\right)  \leq T\left(  i,\mu_{i}\right)
= k\leq b_{i}$.

Forget that we fixed $\left(  i,j\right)  $. We thus have shown that $T\left(
i,j\right)  \leq b_{i}$ for all $\left(  i,j\right)  \in Y\left(  \mu\right)
$. In other words, the tableau $T$ is $\mathbf{b}$-flagged. In other words,
$T\in\operatorname{FSSYT}\left(  \mu,\mathbf{b}\right)  =\mathcal{F}\left(
\lambda/\mu\right)  $. Thus, the \textquotedblleft$\Longrightarrow
$\textquotedblright\ direction of Lemma \ref{transition_wd} \textbf{(c)} is proved.

$\Longleftarrow:$ Assume that $T\in\mathcal{F}\left(  \lambda/\mu\right)  $.
We must prove that $\mathbf{D}\left(  T\right)  \in\mathcal{E}\left(
\lambda/\mu\right)  $. In other words, we must prove that $\mathbf{D}\left(
T\right)  \subseteq Y\left(  \lambda\right)  $
(since Lemma \ref{transition_wd} \textbf{(a)} shows that $\mathbf{D}\tup{T}$ is an excitation of $Y\tup{\mu}$).

In other words, we must prove that $c_{+T}\in Y\left(  \lambda\right)  $ for
each $c\in Y\left(  \mu\right)  $ (by (\ref{eq.def.DD.DDT=})).

So let us do this. Consider any $c\in Y\left(  \mu\right)  $. We must prove
that $c_{+T}\in Y\left(  \lambda\right)  $.

Write the box $c$ as $c=\left(  i,j\right)  $. Hence, $\left(  i,j\right)
=c\in Y\left(  \mu\right)  $, so that $\mu_{i}\geq j$.

Set $k:=T\left(  i,j\right)  $. Then, $\left(  i,j\right)  _{+T}=\left(
k,\ k+j-i\right)  $ by the definition of $\left(  i,j\right)  _{+T}$.

However, $T\in\mathcal{F}\left(  \lambda/\mu\right)  =\operatorname{FSSYT}%
\left(  \mu,\mathbf{b}\right)  $, which shows that $T$ is $\mathbf{b}%
$-flagged. Therefore, $T\left(  i,j\right)  \leq b_{i}$. In other words,
$k\leq b_{i}$ (since $k=T\left(  i,j\right)  $). By Lemma
\ref{lem.flagging-of-lm.uniprop} (applied to $j=k$), this is equivalent to
$\lambda_{k}-k\geq\mu_{i}-i$. Hence, we have $\lambda_{k}-k\geq\underbrace{\mu
_{i}}_{\geq j}-i\geq j-i$. In other words, $k+j-i\leq\lambda_{k}$. In other
words, $\left(  i,j\right)  _{+T}\in Y\left(  \lambda\right)  $ (since we know
that $\left(  i,j\right)  _{+T}=\left(  k,\ k+j-i\right)  $). In other words,
$c_{+T}\in Y\left(  \lambda\right)  $ (since $c=\left(  i,j\right)  $). As
explained above, this completes the proof of the \textquotedblleft%
$\Longleftarrow$\textquotedblright\ direction of Lemma \ref{transition_wd}
\textbf{(c)}.

Thus, Lemma \ref{transition_wd} \textbf{(c)} is proved.
\end{proof}

\begin{proof}[Proof of Lemma \ref{transition_formula}.]
Lemma \ref{transition_wd}
\textbf{(a)} yields that the map
\begin{align*}
\operatorname{SSYT}\left(  \mu\right)   &  \rightarrow\left\{  \text{all
excitations of $Y\left(  \mu\right)  $}\right\}  ,\\
T  &  \mapsto\mathbf{D}\left(  T\right)
\end{align*}
is well-defined. It remains to show that this map is a bijection. Towards this
aim, we will prove the following two claims:

\begin{statement}
\textit{Claim 1:} Let $T$ and $S$ be two distinct tableaux in
$\operatorname*{SSYT}\left(  \mu\right)  $. Then, $\mathbf{D}\left(  T\right)
\neq\mathbf{D}\left(  S\right)  $.
\end{statement}

\begin{statement}
\textit{Claim 2:} Let $n\in\mathbb{N}$. Let $E$ be a diagram obtained from
$Y\left(  \mu\right)  $ by a sequence of $n$ excited moves. Then, there is a
tableau $T\in\operatorname*{SSYT}\left(  \mu\right)  $ such that
$E=\mathbf{D}\left(  T\right)  $.
\end{statement}

\begin{proof}
[Proof of Claim 1.]Assume the contrary. Thus, $\mathbf{D}\left(  T\right)
=\mathbf{D}\left(  S\right)  $.

Since $T$ and $S$ are distinct, there exists a box $\left(  i,j\right)  \in
Y\left(  \mu\right)  $ satisfying $T\left(  i,j\right)  \neq S\left(
i,j\right)  $. Pick such a box $\left(  i,j\right)  $ with smallest $i$.
Thus,%
\begin{equation}
T\left(  u,v\right)  =S\left(  u,v\right)
\label{pf.transition_formula.c1.pf.1}%
\end{equation}
for every box $\left(  u,v\right)  \in Y\left(  \mu\right)  $ satisfying
$u<i$ (since $\left(  i,j\right)  $ was chosen to have smallest $i$).

From $T\left(  i,j\right)  \neq S\left(  i,j\right)  $, we see that either
$T\left(  i,j\right)  <S\left(  i,j\right)  $ or $S\left(  i,j\right)
<T\left(  i,j\right)  $. We WLOG assume that $S\left(  i,j\right)  <T\left(
i,j\right)  $ is the case (since otherwise, we can simply swap $T$ with $S$).

The definition of $\mathbf{D}\left(  S\right)  $ yields $\mathbf{D}\left(
S\right)  =\left\{  c_{+S}\ \mid\ c\in Y\left(  \mu\right)  \right\}  $.
Hence,%
\begin{align*}
\left(  i,j\right)  _{+S}  &  \in\mathbf{D}\left(  S\right)
\ \ \ \ \ \ \ \ \ \ \left(  \text{since }\left(  i,j\right)  \in Y\left(
\mu\right)  \right) \\
&  =\mathbf{D}\left(  T\right)  \ \ \ \ \ \ \ \ \ \ \left(  \text{since
}\mathbf{D}\left(  T\right)  =\mathbf{D}\left(  S\right)  \right) \\
&  =\left\{  c_{+T}\ \mid\ c\in Y\left(  \mu\right)  \right\}
\ \ \ \ \ \ \ \ \ \ \left(  \text{by the definition of }\mathbf{D}\left(
T\right)  \right)  .
\end{align*}
In other words, $\left(  i,j\right)  _{+S}=c_{+T}$ for some $c\in Y\left(
\mu\right)  $. Consider this $c$, and denote it by $\left(  u,v\right)  $.
Thus,
\[
\left(  i,j\right)  _{+S}=\left(  u,v\right)  _{+T}=\left(  T\left(
u,v\right)  ,\ T\left(  u,v\right)  +v-u\right)
\]
(by the definition of $\left(  u,v\right)  _{+T}$). Hence,%
\[
\left(  T\left(  u,v\right)  ,\ T\left(  u,v\right)  +v-u\right)  =\left(
i,j\right)  _{+S}=\left(  S\left(  i,j\right)  ,\ S\left(  i,j\right)
+j-i\right)
\]
(by the definition of $\left(  i,j\right)  _{+S}$). In other words, the two
equalities $T\left(  u,v\right)  =S\left(  i,j\right)  $ and $T\left(
u,v\right)  +v-u=S\left(  i,j\right)  +j-i$ hold. Subtracting the former
equality from the latter, we obtain $v-u=j-i$. In other words, $v+i=j+u$.
Thus, $v-j=u-i$.

We are in one of the following two cases:

\textit{Case 1:} We have $u<i$.

\textit{Case 2:} We have $u\geq i$.

Let us first consider Case 1. In this case, we have $u<i$. Hence, $u-i<0$ and
$i>u$. Moreover, $v<j$ (since $v-j=u-i<0$). Hence, Lemma \ref{lem.ssyt.geq}
\textbf{(b)} (applied to $S$, $\left(  u,v\right)  $ and $\left(  i,j\right)
$ instead of $T$, $\left(  i,j\right)  $ and $\left(  u,v\right)  $) yields
$S\left(  i,j\right)  -i\geq S\left(  u,v\right)  -u$ (since $u<i$ and $v<j$).
Adding the inequality $i>u$ to this inequality, we obtain $S\left(
i,j\right)  >S\left(  u,v\right)  $. However, from
(\ref{pf.transition_formula.c1.pf.1}), we obtain $T\left(  u,v\right)
=S\left(  u,v\right)  $ (since $u<i$). Thus, $S\left(  i,j\right)  >S\left(
u,v\right)  =T\left(  u,v\right)  =S\left(  i,j\right)  $. This is absurd.
Thus, we have obtained a contradiction in Case 1.

Let us now consider Case 2. In this case, we have $u\geq i$. In other words,
$i\leq u$. Also, $u\geq i$ entails $u-i\geq0$ and thus $v\geq j$ (since
$v-j=u-i\geq0$). Hence, $j\leq v$. Thus, Lemma \ref{lem.ssyt.geq} \textbf{(b)}
yields $T\left(  u,v\right)  -u\geq T\left(  i,j\right)  -i$. Adding the
inequality $u\geq i$ to this inequality, we obtain $T\left(  u,v\right)  \geq
T\left(  i,j\right)  $. This contradicts $T\left(  u,v\right)  =S\left(
i,j\right)  <T\left(  i,j\right)  $. Thus, we have obtained a contradiction in
Case 2.

We have now obtained a contradiction in each of the two cases. Thus, we always
have a contradiction. Hence, our assumption was false. This proves Claim 1.
\end{proof}

\begin{proof}
[Proof of Claim 2.]We proceed by induction on $n$:

\textit{Base case:} Let us prove Claim 2 for $n=0$.

Indeed, let $E$ be a diagram obtained from $Y\left(  \mu\right)  $ by a
sequence of $0$ excited moves. Thus, $E$ must be $Y\left(  \mu\right)  $
itself. Now, let $T$ be the filling of the diagram $Y\left(  \mu\right)  $
(that is, the map $Y\left(  \mu\right)  \rightarrow \set{1,2,3,\ldots}$) that is given by%
\[
T\left(  i,j\right)  =i\ \ \ \ \ \ \ \ \ \ \text{for each }i\in Y\left(
\mu\right)  .
\]
(For instance, if $\mu=\left(  5,2,2\right)  $, then
$T=\ytableaushort{11111,22,33}$.) Note that $T$ is a semistandard tableau of
shape $\mu$ (indeed, the entries of $T$ are weakly increasing left-to-right in
each row\footnote{since they are constant in each row} and strictly increasing
top-to-bottom in each column\footnote{since the entries of $T$ in any given
column are $1,2,\ldots,k$ from top to bottom (where $k$ is the length of said
column)}). In other words, $T\in\operatorname*{SSYT}\left(  \mu\right)  $.
Moreover, if $c=\left(  i,j\right)  \in Y\left(  \mu\right)  $ is any box,
then%
\begin{align*}
c_{+T}  &  =\left(  T\left(  i,j\right)  ,\ T\left(  i,j\right)  +j-i\right)
\ \ \ \ \ \ \ \ \ \ \left(  \text{by the definition of }c_{+T}\text{, since
}c=\left(  i,j\right)  \right) \\
&  =\left(  i,\ i+j-i\right)  \ \ \ \ \ \ \ \ \ \ \left(  \text{since
}T\left(  i,j\right)  =i\right) \\
&  =\left(  i,j\right)  =c.
\end{align*}
Thus, $\left\{  c_{+T}\ \mid\ c\in Y\left(  \mu\right)  \right\}  =\left\{
c\ \mid\ c\in Y\left(  \mu\right)  \right\}  =Y\left(  \mu\right)  $.
Therefore, the definition of $\mathbf{D}\left(  T\right)  $ yields
$\mathbf{D}\left(  T\right)  =\left\{  c_{+T}\ \mid\ c\in Y\left(  \mu\right)
\right\}  =Y\left(  \mu\right)  =E$ (since $E$ is $Y\left(  \mu\right)  $
itself). Hence, we have found a tableau $T\in\operatorname*{SSYT}\left(
\mu\right)  $ such that $E=\mathbf{D}\left(  T\right)  $. Thus, Claim 2 is
proved for our $E$. This completes the base case.

\textit{Induction step:} Let $n$ be a positive integer. Assume (as the
induction hypothesis) that Claim 2 is proved for $n-1$ instead of $n$. We must
now prove Claim 2 for $n$.

So let $E$ be a diagram obtained from $Y\left(  \mu\right)  $ by a sequence of
$n$ excited moves. Then, we must prove that there is a tableau $T\in
\operatorname*{SSYT}\left(  \mu\right)  $ such that $E=\mathbf{D}\left(
T\right)  $.

We assumed that $E$ is obtained from $Y\left(  \mu\right)  $ by a sequence of
$n$ excited moves. Let $\mathbf{s}$ be this sequence of $n$ excited moves. Let
$\mathbf{e}$ be the last move in this sequence, and let $F$ be the diagram
obtained from $Y\left(  \mu\right)  $ after the first $n-1$ moves of this
sequence $\mathbf{s}$ (stopping short of the last move $\mathbf{e}$). Then,
$E$ is obtained from $F$ by the excited move $\mathbf{e}$.

Moreover, the diagram $F$ is obtained from $Y\left(  \mu\right)  $ by a
sequence of $n-1$ excited moves (namely, by the first $n-1$ moves of the
sequence $\mathbf{s}$). Hence, our induction hypothesis (applied to $F$
instead of $E$) shows that there is a tableau $T\in\operatorname*{SSYT}\left(
\mu\right)  $ such that $F=\mathbf{D}\left(  T\right)  $. Consider this
tableau $T$, and denote it by $S$. Thus, $S\in\operatorname*{SSYT}\left(
\mu\right)  $ and $F=\mathbf{D}\left(  S\right)  $.

Recall that $E$ is obtained from $F$ by the excited move $\mathbf{e}$. Hence,
$\mathbf{e}$ is an excited move for $F$. In other words, $\mathbf{e}$ replaces
some box $d\in F$ by its southeastern neighbor $d_{\searrow}$, where the box
$d$ is chosen in such a way that $F$ contains none of its three neighbors
$d_{\downarrow},d_{\rightarrow},d_{\searrow}$ (by the definition of an
\textquotedblleft excited move\textquotedblright). Consider this box $d$.

We have
\begin{align*}
d  &  \in F=\mathbf{D}\left(  S\right)  =\left\{  c_{+S}\ \mid\ c\in Y\left(
\mu\right)  \right\}  \ \ \ \ \ \ \ \ \ \ \left(  \text{by the definition of
}\mathbf{D}\left(  S\right)  \right) \\
&  =\left\{  \left(  i,j\right)  _{+S}\ \mid\ \left(  i,j\right)  \in Y\left(
\mu\right)  \right\}  \ \ \ \ \ \ \ \ \ \ \left(  \text{here, we have renamed
the index }c\text{ as }\left(  i,j\right)  \right)  .
\end{align*}
In other words, $d=\left(  i,j\right)  _{+S}$ for some box $\left(
i,j\right)  \in Y\left(  \mu\right)  $. Consider this $\left(  i,j\right)  $.
Thus,%
\[
d=\left(  i,j\right)  _{+S}=\left(  S\left(  i,j\right)  ,\ S\left(
i,j\right)  +j-i\right)
\]
(by the definition of $\left(  i,j\right)  _{+S}$). Therefore,%
\[
d_{\rightarrow}=\left(  S\left(  i,j\right)  ,\ S\left(  i,j\right)
+j-i\right)  _{\rightarrow}=\left(  S\left(  i,j\right)  ,\ S\left(
i,j\right)  +j-i+1\right)
\]
and%
\[
d_{\downarrow}=\left(  S\left(  i,j\right)  ,\ S\left(  i,j\right)
+j-i\right)  _{\downarrow}=\left(  S\left(  i,j\right)  +1,\ S\left(
i,j\right)  +j-i\right)  .
\]

We observe the following two facts:

\begin{enumerate}
\item We have%
\begin{equation}
S\left(  i,j\right)  +1\leq S\left(  i,j+1\right)
\ \ \ \ \ \ \ \ \ \ \text{if }\left(  i,j+1\right)  \in Y\left(  \mu\right)  .
\label{pf.transition_formula.c2.pf.row-still-grows}%
\end{equation}

[\textit{Proof:} Assume that $\left(  i,j+1\right)  \in Y\left(  \mu\right)
$. Since the entries of $S$ weakly increase left-to-right in each row (because
$S$ is a semistandard tableau), we then have $S\left(  i,j\right)  \leq
S\left(  i,j+1\right)  $.

We must prove that $S\left(  i,j\right)  +1\leq S\left(  i,j+1\right)  $.
Assume the contrary. Thus, $S\left(  i,j\right)  +1>S\left(  i,j+1\right)  $,
so that $S\left(  i,j\right)  +1\geq S\left(  i,j+1\right)  +1$ (since
$S\left(  i,j\right)  +1$ and $S\left(  i,j+1\right)  $ are integers). In
other words, $S\left(  i,j\right)  \geq S\left(  i,j+1\right)  $. Combining
this with $S\left(  i,j\right)  \leq S\left(  i,j+1\right)  $, we obtain
$S\left(  i,j\right)  =S\left(  i,j+1\right)  $.

However, from $\left(  i,j+1\right)  \in Y\left(  \mu\right)  $, we obtain
$\left(  i,j+1\right)  _{+S}\in F$ (since we have $F=\left\{  c_{+S}\ \mid\ c\in
Y\left(  \mu\right)  \right\}  $). The definition of $\left(  i,j+1\right)
_{+S}$ yields%
\begin{align*}
\left(  i,j+1\right)  _{+S}  &  =\left(  S\left(  i,j+1\right)  ,\ S\left(
i,j+1\right)  +\left(  j+1\right)  -i\right) \\
&  =\left(  S\left(  i,j\right)  ,\ S\left(  i,j\right)  +\left(  j+1\right)
-i\right)  \ \ \ \ \ \ \ \ \ \ \left(  \text{since }S\left(  i,j+1\right)
=S\left(  i,j\right)  \right) \\
&  =\left(  S\left(  i,j\right)  ,\ S\left(  i,j\right)  +j-i+1\right)
\ \ \ \ \ \ \ \ \ \ \left(  \text{since }\left(  j+1\right)  -i=j-i+1\right)
\\
&  =d_{\rightarrow}\ \ \ \ \ \ \ \ \ \ \left(  \text{since }d_{\rightarrow
}=\left(  S\left(  i,j\right)  ,\ S\left(  i,j\right)  +j-i+1\right)  \right)
\\
&  \notin F\ \ \ \ \ \ \ \ \ \ \left(  \text{since }F\text{ contains none of
}d_{\downarrow},d_{\rightarrow},d_{\searrow}\right)  .
\end{align*}
This contradicts $\left(  i,j+1\right)  _{+S}\in F$. This contradiction shows
that our assumption was false. Hence,
(\ref{pf.transition_formula.c2.pf.row-still-grows}) is proved.]

\item We have%
\begin{equation}
S\left(  i,j\right)  +1<S\left(  i+1,j\right)  \ \ \ \ \ \ \ \ \ \ \text{if
}\left(  i+1,j\right)  \in Y\left(  \mu\right)  .
\label{pf.transition_formula.c2.pf.col-still-grows}%
\end{equation}

[\textit{Proof:} Assume that $\left(  i+1,j\right)  \in Y\left(  \mu\right)
$. Since the entries of $S$ strictly increase top-to-bottom in each column
(because $S$ is a semistandard tableau), we then have $S\left(  i,j\right)
<S\left(  i+1,j\right)  $. Thus, $S\left(  i,j\right)  \leq S\left(
i+1,j\right)  -1$ (since $S\left(  i,j\right)  $ and $S\left(  i+1,j\right)  $
are integers). In other words, $S\left(  i,j\right)  +1\leq S\left(
i+1,j\right)  $.

We must prove that $S\left(  i,j\right)  +1<S\left(  i+1,j\right)  $. Assume
the contrary. Thus, $S\left(  i,j\right)  +1\geq S\left(  i+1,j\right)  $.
Combining this with $S\left(  i,j\right)  +1\leq S\left(  i+1,j\right)  $, we
obtain $S\left(  i,j\right)  +1=S\left(  i+1,j\right)  $.

However, from $\left(  i+1,j\right)  \in Y\left(  \mu\right)  $, we obtain
$\left(  i+1,j\right)  _{+S}\in F$ (since we have $F=\left\{  c_{+S}\ \mid\ c\in
Y\left(  \mu\right)  \right\}  $). The definition of $\left(  i+1,j\right)
_{+S}$ yields%
\begin{align*}
\left(  i+1,j\right)  _{+S}  &  =\left(  S\left(  i+1,j\right)  ,\ S\left(
i+1,j\right)  +j-\left(  i+1\right)  \right) \\
&  =\left(  S\left(  i,j\right)  +1,\ S\left(  i,j\right)  +1+j-\left(
i+1\right)  \right)  \\
&\ \ \ \ \ \ \ \ \ \ \ \ \ \ \ \ \ \  \ \ \ \left(  \text{since }S\left(
i+1,j\right)  =S\left(  i,j\right)  +1\right) \\
&  =\left(  S\left(  i,j\right)  +1,\ S\left(  i,j\right)  +j-i\right)
\ \ \ \ \ \ \ \ \ \ \left(  \text{since }1+j-\left(  i+1\right)  =j-i\right)
\\
&  =d_{\downarrow}\ \ \ \ \ \ \ \ \ \ \left(  \text{since }d_{\downarrow
}=\left(  S\left(  i,j\right)  +1,\ S\left(  i,j\right)  +j-i\right)  \right)
\\
&  \notin F\ \ \ \ \ \ \ \ \ \ \left(  \text{since }F\text{ contains none of
}d_{\downarrow},d_{\rightarrow},d_{\searrow}\right)  .
\end{align*}
This contradicts $\left(  i+1,j\right)  _{+S}\in F$. This contradiction shows
that our assumption was false. Hence,
(\ref{pf.transition_formula.c2.pf.col-still-grows}) is proved.]
\end{enumerate}

Now, let us increase the entry $S\left(  i,j\right)  $ of the tableau $S$ by
$1$, while leaving all other entries unchanged. The resulting filling of
$Y\left(  \mu\right)  $ will be called $T$. Formally speaking, $T$ is thus the
map from $Y\left(  \mu\right)  $ to $\set{1,2,3,\ldots}$ given by%
\begin{align}
T\left(  i,j\right)   &  =S\left(  i,j\right)
+1\ \ \ \ \ \ \ \ \ \ \text{and}\label{pf.transition_fml.T.1}\\
T\left(  c\right)   &  =S\left(  c\right)  \ \ \ \ \ \ \ \ \ \ \text{for all
}c\in Y\left(  \mu\right)  \text{ distinct from }\left(  i,j\right)  .
\label{pf.transition_fml.T.2}%
\end{align}

It is easy to see (using (\ref{pf.transition_formula.c2.pf.row-still-grows})
and (\ref{pf.transition_formula.c2.pf.col-still-grows})) that $T$ is again a
semistandard tableau\footnote{\textit{Proof.} First, we note that the entries
of $T$ are positive integers (since $T\left(  i,j\right)  =S\left(
i,j\right)  +1\geq1$ shows that $T\left(  i,j\right)  $ is a positive integer,
and since all the other entries of $T$ are copied from $S$).
\par
Next, we claim that the entries of $T$ weakly increase left-to-right in each
row. Indeed, by the construction of $T$, this will follow from the analogous
property of $S$, as long as we can show that the increased entry $T\left(
i,j\right)  $ is still smaller or equal than its neighboring entry $T\left(
i,j+1\right)  $ (assuming that $\left(  i,j+1\right)  \in Y\left(  \mu\right)
$). But we can easily show this: If $\left(  i,j+1\right)  \in Y\left(
\mu\right)  $, then
\begin{align*}
T\left(  i,j+1\right)   &  =S\left(  i,j+1\right)  \ \ \ \ \ \ \ \ \ \ \left(
\text{by (\ref{pf.transition_fml.T.2})}\right)  \\
&  \geq S\left(  i,j\right)  +1\ \ \ \ \ \ \ \ \ \ \left(  \text{by
(\ref{pf.transition_formula.c2.pf.row-still-grows})}\right)  \\
&  =T\left(  i,j\right)  \ \ \ \ \ \ \ \ \ \ \left(  \text{by
(\ref{pf.transition_fml.T.1})}\right)
\end{align*}
and therefore $T\left(  i,j\right)  \leq T\left(  i,j+1\right)  $. Thus, we
conclude that the entries of $T$ weakly increase left-to-right in each row.
\par
Finally, we claim that the entries of $T$ strictly increase top-to-bottom in
each column. Indeed, by the construction of $T$, this will follow from the
analogous property of $S$, as long as we can show that the increased entry
$T\left(  i,j\right)  $ is still smaller than its neighboring entry $T\left(
i+1,j\right)  $ (assuming that $\left(  i+1,j\right)  \in Y\left(  \mu\right)
$). But we can easily show this: If $\left(  i+1,j\right)  \in Y\left(
\mu\right)  $, then%
\begin{align*}
T\left(  i+1,j\right)   &  =S\left(  i+1,j\right)  \ \ \ \ \ \ \ \ \ \ \left(
\text{by (\ref{pf.transition_fml.T.2})}\right)  \\
&  >S\left(  i,j\right)  +1\ \ \ \ \ \ \ \ \ \ \left(  \text{by
(\ref{pf.transition_formula.c2.pf.col-still-grows})}\right)  \\
&  =T\left(  i,j\right)  \ \ \ \ \ \ \ \ \ \ \left(  \text{by
(\ref{pf.transition_fml.T.1})}\right)
\end{align*}
and thus $T\left(  i,j\right)  <T\left(  i+1,j\right)  $. Thus, we conclude
that the entries of $T$ strictly increase top-to-bottom in each column.
\par
Altogether, we have now shown that $T$ is a semistandard tableau.}. Thus,
$T\in\operatorname*{SSYT}\left(  \mu\right)  $.

From (\ref{pf.transition_fml.T.1}), we easily see that%
\begin{equation}
\left(  i,j\right)  _{+T}=\left(  \left(  i,j\right)  _{+S}\right)
_{\searrow}\label{pf.transition_fml.+T1}%
\end{equation}
\footnote{\textit{Proof.} From (\ref{pf.transition_fml.T.1}), we obtain
$S\left(  i,j\right)  +1=T\left(  i,j\right)  $. The definition of $\left(
i,j\right)  _{+S}$ yields%
\[
\left(  i,j\right)  _{+S}=\left(  S\left(  i,j\right)  ,\ S\left(  i,j\right)
+j-i\right)  .
\]
Hence,
\begin{align*}
\left(  \left(  i,j\right)  _{+S}\right)  _{\searrow} &  =\left(  S\left(
i,j\right)  +1,\ \underbrace{\left(  S\left(  i,j\right)  +j-i\right)
+1}_{=S\left(  i,j\right)  +1+j-i}\right)  \\
&  =\left(  \underbrace{S\left(  i,j\right)  +1}_{=T\left(  i,j\right)
},\ \underbrace{S\left(  i,j\right)  +1}_{=T\left(  i,j\right)  }+\,j-i\right)
=\left(  T\left(  i,j\right)  ,\ T\left(  i,j\right)  +j-i\right)  =\left(
i,j\right)  _{+T}%
\end{align*}
(by the definition of $\left(  i,j\right)  _{+T}$).}. Since $\left(
i,j\right)  _{+S}=d$, we can rewrite this as $\left(  i,j\right)
_{+T}=d_{\searrow}$.

However, Lemma \ref{lem.DS-vs-DT} shows that the diagram $\mathbf{D}\left(
T\right)  $ can be obtained from $\mathbf{D}\left(  S\right)  $ by replacing
the box $\left(  i,j\right)  _{+S}$ by the box $\left(  i,j\right)  _{+T}$
(since (\ref{pf.transition_fml.T.2}) holds). In other words, the diagram
$\mathbf{D}\left(  T\right)  $ can be obtained from $F$ by replacing the box
$d$ by the box $d_{\searrow}$ (since $\mathbf{D}\left(  S\right)  =F$ and
$\left(  i,j\right)  _{+S}=d$ and $\left(  i,j\right)  _{+T}=d_{\searrow}$).
In other words, the diagram $\mathbf{D}\left(  T\right)  $ can be obtained
from $F$ by the excited move $\mathbf{e}$ (since the replacement of the box
$d$ by $d_{\searrow}$ in the diagram $F$ is precisely the excited move
$\mathbf{e}$ (by the definition of the box $d$)).

However, we know that the diagram $E$ is also obtained from $F$ by the excited
move $\mathbf{e}$. Thus, both diagrams $E$ and $\mathbf{D}\left(  T\right)  $
can be obtained from $F$ by the excited move $\mathbf{e}$. Since any given
excited move has only one possible outcome, we thus conclude that these
diagrams $E$ and $\mathbf{D}\left(  T\right)  $ are identical. In other words,
$E=\mathbf{D}\left(  T\right)  $.

Hence, we have found a tableau $T\in\operatorname*{SSYT}\left(  \mu\right)  $
such that $E=\mathbf{D}\left(  T\right)  $. Thus, Claim 2 is proved for our
$E$. This completes the induction step.

The induction proof of Claim 2 is now complete.
\end{proof}

Claim 2 easily yields the following:

\begin{statement}
\textit{Claim 3:} Let $E\in\left\{  \text{all excitations of $Y\left(
\mu\right)  $}\right\}  $. Then, there is a tableau $T\in\operatorname*{SSYT}%
\left(  \mu\right)  $ such that $E=\mathbf{D}\left(  T\right)  $.
\end{statement}

\begin{proof}
[Proof of Claim 3.] We have
$E\in\left\{  \text{all excitations of $Y\left(  \mu\right)  $%
}\right\}$.
Thus, $E$ is an excitation of $Y\left(  \mu\right)
$. In other words, $E$ can be obtained from $Y\left(  \mu\right)
$ by a finite sequence of excited moves. Consider this sequence, and let $n$
be its length. Thus, $E$ can be obtained from $Y\left(  \mu\right)  $ by a
sequence of $n$ excited moves. Hence, Claim 2 shows that there is a tableau
$T\in\operatorname*{SSYT}\left(  \mu\right)  $ such that $E=\mathbf{D}\left(
T\right)  $. This proves Claim 3.
\end{proof}

Let us now consider the map%
\begin{align*}
\operatorname{SSYT}\left(  \mu\right)   &  \rightarrow\left\{  \text{all
excitations of $Y\left(  \mu\right)  $}\right\}  ,\\
T &  \mapsto\mathbf{D}\left(  T\right)
\end{align*}
once again. This map is injective (by Claim 1) and surjective (by Claim 3).
Hence, it is bijective, i.e., is a bijection. This completes the proof of Lemma
\ref{transition_formula}.
\end{proof}


\begin{proof}
[Proof of Lemma \ref{transition_formula_flagged}.]
We know that $\mathcal{E}%
\left(  \lambda/\mu\right)  $ is a subset of $\left\{  \text{all excitations
of $Y\left(  \mu\right)  $}\right\}  $, whereas $\mathcal{F}\left(
\lambda/\mu\right)  $ is a subset of $\operatorname{SSYT}\left(  \mu\right)  $.

Lemma \ref{transition_formula} yields that the map%
\begin{align*}
\operatorname{SSYT}\left(  \mu\right)   &  \rightarrow\left\{  \text{all
excitations of $Y\left(  \mu\right)  $}\right\}  ,\\
T &  \mapsto\mathbf{D}\left(  T\right)
\end{align*}
is well-defined and is a bijection. According to Lemma \ref{transition_wd}
\textbf{(c)}, a semistandard tableau $T\in\operatorname{SSYT}\left(
\mu\right)  $ belongs to $\mathcal{F}\left(  \lambda/\mu\right)  $ if and only
if its image $\mathbf{D}\left(  T\right)  $ under this map belongs to
$\mathcal{E}\left(  \lambda/\mu\right)  $. Hence, the subset $\mathcal{F}%
\left(  \lambda/\mu\right)  $ of $\operatorname{SSYT}\left(  \mu\right)  $
corresponds precisely to the subset $\mathcal{E}\left(  \lambda/\mu\right)  $
of $\left\{  \text{all excitations of $Y\left(  \mu\right)  $}\right\}  $
under this map. Therefore, restricting this map to $\mathcal{F}\left(
\lambda/\mu\right)  $, we obtain a bijection\footnote{Here is the argument in more detail:
Let $\Phi$ be the map
\begin{align*}
\operatorname{SSYT}\left(  \mu\right)   &  \rightarrow\left\{  \text{all
excitations of $Y\left(  \mu\right)  $}\right\}  ,\\
T &  \mapsto\mathbf{D}\left(  T\right) .
\end{align*}
As we have seen above, this map $\Phi$ is a bijection. Thus, $\Phi$ is injective and surjective.
Moreover, each $T \in \calF\tup{\lambda/\mu}$ satisfies $\mathbf{D}\tup{T} \in \calE\tup{\lambda/\mu}$
(by Lemma \ref{transition_wd} \textbf{(c)}). Hence, the map
\begin{align*}
\Psi: \mathcal{F}\left(  \lambda/\mu\right)    & \rightarrow\mathcal{E}\left(
\lambda/\mu\right)  ,\\
T  & \mapsto\mathbf{D}\left(  T\right)
\end{align*}
is well-defined.
Consider this map $\Psi$.
Clearly, $\Psi$ is a restriction of the map $\Phi$ (since both maps are given by the same formula), and thus is injective (since $\Phi$ is injective).
Let us next show that $\Psi$ is surjective.
\par
Indeed, let $E \in \calE\tup{\lambda/\mu}$.
Then, $E \in \left\{  \text{all
excitations of $Y\left(  \mu\right)  $}\right\}  $.
Hence, there exists a $T \in \SSYT\tup{\mu}$ such that  $E = \Phi\tup{T}$ (since $\Phi$ is surjective).
Consider this $T$. Then, $E = \Phi\tup{T} = \DD\tup{T}$ (by the definition of $\Phi$), so that $\DD\tup{T} = E \in \calE\tup{\lambda/\mu}$. Hence, Lemma \ref{transition_wd} \textbf{(c)} shows that $T \in \calF\tup{\lambda/\mu}$. Thus, $\Psi\tup{T}$ is well-defined and equals $\DD\tup{T}$ (by the definition of $\Psi$). Therefore, $\Psi\tup{T} = \DD\tup{T} = E$.
This shows that $E$ is a value of the map $\Psi$.
\par
Forget that we fixed $E$. We thus have shown that each $E \in \calE\tup{\lambda/\mu}$ is a value of the map $\Psi$.
Hence, the map $\Psi$ is surjective.
Thus, $\Psi$ is a bijection (since $\Psi$ is both injective and surjective).
In other words, the map
\begin{align*}
\mathcal{F}\left(  \lambda/\mu\right)    & \rightarrow\mathcal{E}\left(
\lambda/\mu\right)  ,\\
T  & \mapsto\mathbf{D}\left(  T\right)
\end{align*}
is a bijection (since this map is $\Psi$).}
\begin{align*}
\mathcal{F}\left(  \lambda/\mu\right)    & \rightarrow\mathcal{E}\left(
\lambda/\mu\right)  ,\\
T  & \mapsto\mathbf{D}\left(  T\right)  .
\end{align*}
This proves Lemma \ref{transition_formula_flagged}.
\end{proof}

\begin{proof}
[Proof of Corollary \ref{corollaryFlaggedExc}.] The definition of
$\mathbf{s}_{\lambda}\left[  \mu\right]  $ yields%
\begin{align*}
\mathbf{s}_{\lambda}\left[  \mu\right]    & =\sum_{D\in\mathcal{E}\left(
\lambda/\mu\right)  }\ \ \prod_{\left(  i,j\right)  \in D}\left(  x_{i}%
+y_{j}\right)  \nonumber
=\sum_{T\in\mathcal{F}\left(  \lambda/\mu\right)  }\ \ \prod_{\left(
i,j\right)  \in\mathbf{D}\left(  T\right)  }\left(  x_{i}+y_{j}\right)
\end{align*}
(here, we have substituted $\mathbf{D}\left(  T\right)  $ for $D$ in the sum,
since Lemma \ref{transition_formula_flagged} shows that the map
\begin{align*}
\mathcal{F}(\lambda/\mu) &  \rightarrow\mathcal{E}(\lambda/\mu),\\
T &  \mapsto\mathbf{D}\left(  T\right)
\end{align*}
is a bijection).
Hence,
\[
\mathbf{s}_{\lambda}\left[  \mu\right]  =\sum_{T\in\mathcal{F}\left(  \lambda/\mu\right)  }\underbrace{\prod_{\left(
i,j\right)  \in\mathbf{D}\left(  T\right)  }\left(  x_{i}+y_{j}\right)}_{\substack{=\prod_{(i,j)\in
Y(\mu)}\left(  x_{T(i,j)}+y_{T(i,j)+j-i}\right) \\ \text{(by Lemma \ref{transition_wd} \textbf{(b)})}}}
=\sum_{T\in\mathcal{F}(\lambda/\mu
)}\ \ \prod_{(i,j)\in Y(\mu)}\left(  x_{T(i,j)}+y_{T(i,j)+j-i}\right)  .
\]
This proves Corollary \ref{corollaryFlaggedExc}.
\end{proof}

\subsection{To Section \ref{sec.habc}}

\begin{proof}
[Proof of Lemma \ref{lem.h1bc}.] Definition \ref{defh} \textbf{(c)} yields%
\begin{align*}
h\left(  1,b,c\right)    & =\underbrace{\sum_{\substack{\left(  i_{1}%
,i_{2},\ldots,i_{1}\right)  \in\left[  b\right]  ^{1};\\i_{1}\leq i_{2}%
\leq\cdots\leq i_{1}}}}_{\substack{=\sum_{\left(  i_{1}\right)  \in\left[
b\right]  ^{1}}\\=\sum_{i_{1}\in\left[  b\right]  }}}\ \ \underbrace{\prod
_{j=1}^{1}\left(  x_{i_{j}}+y_{i_{j}+\left(  j-1\right)  +c
}\right)  }_{\substack{=x_{i_{1}}+y_{i_{1}+\left(  1-1\right)  +c }\\=x_{i_{1}}+y_{i_{1}+c}\\\text{(since }i_{1}+\left(  1-1\right)
+c  =i_{1}+c\text{)}}}=\sum_{i_{1}\in\left[  b\right]
}\left(  x_{i_{1}}+y_{i_{1}+c}\right)  \\
& =\underbrace{\sum_{i\in\left[  b\right]  }}_{=\sum_{i=1}^{b}}\left(
x_i+y_{i+c}\right)  \ \ \ \ \ \ \ \ \ \ \left(
\begin{array}
[c]{c}%
\text{here, we have renamed the}\\
\text{summation index }i_{1}\text{ as }i
\end{array}
\right)  \\
& =\sum_{i=1}^{b}\left(  x_{i}+y_{i+c}\right)  = \sum_{i=1}^{b}x_{i} + \sum_{i=1}^{b}y_{i+c}
= \sum_{i=1}^b x_i + \sum_{j=c+1}^{c+b} y_j
\end{align*}
(here, we have substituted $j$ for $i+c$ in the second sum).
This proves Lemma~\ref{lem.h1bc}.
\end{proof}

\begin{proof}[Proof of Lemma \ref{1hprop}.]
Let $a$ and $c$ be integers, and let $b$ be a positive integer.

If $a \leq 0$, then the claim of the lemma boils down to $1 = \tup{x_b+y_{a+b+c-1}} \cdot 0 + 1$ (by \eqref{eq.h.triv.aleq0}), which is obviously true. Thus, from now on, we WLOG assume that $a > 0$.

We can rewrite the definition of $h(a,b,c)$ as follows:
\begin{align*}
h(a,b,c)
&= \sum_{\substack{(i_1,i_2,\ldots,i_a) \in [b]^a; \\ i_1 \leq i_2 \leq \cdots \leq i_a}}\ \ \prod_{j=1}^a(x_{i_j}+y_{i_j+(j-1)+c}) \\
&= \sum_{\substack{(i_1,i_2,\ldots,i_a) \in [b]^a; \\ i_1 \leq i_2 \leq \cdots \leq i_a; \\ i_a = b}}\ \ \prod_{j=1}^a(x_{i_j}+y_{i_j+(j-1)+c}) + \sum_{\substack{(i_1,i_2,\ldots,i_a) \in [b]^a; \\ i_1 \leq i_2 \leq \cdots \leq i_a; \\ i_a \neq b}}\ \ \prod_{j=1}^a(x_{i_j}+y_{i_j+(j-1)+c})
\end{align*}
(here we have broken up the sum into a part with $i_a = b$ and a part with $i_a \neq b$).

In view of
\begin{align*}
&\sum_{\substack{(i_1,i_2,\ldots,i_{a}) \in [b]^{a}; \\ i_1 \leq i_2 \leq \cdots \leq i_{a}; \\ i_a = b}}\ \ \prod_{j=1}^{a}(x_{i_j}+y_{i_j+(j-1)+c}) \\
&= \sum_{\substack{(i_1,i_2,\ldots,i_{a}) \in [b]^{a}; \\ i_1 \leq i_2 \leq \cdots \leq i_{a}; \\ i_a = b}}\ \ \underbrace{(x_{i_a}+y_{i_a+(a-1)+c})}_{\substack{= x_{i_a} + y_{a+i_a+c-1} \\ = x_b + y_{a+b+c-1}\\ \tup{\text{since } i_a=b}}}
\prod_{j=1}^{a-1}(x_{i_j}+y_{i_j+(j-1)+c}) \\
& \qquad \qquad \tup{\text{here, we have split off the $j=a$ factor from the product}} \\
&= \tup{x_b+y_{a+b+c-1}} \sum_{\substack{(i_1,i_2,\ldots,i_{a}) \in [b]^{a}; \\ i_1 \leq i_2 \leq \cdots \leq i_{a}; \\ i_a = b}}\ \ \prod_{j=1}^{a-1}(x_{i_j}+y_{i_j+(j-1)+c}) \\
&= \tup{x_b+y_{a+b+c-1}}
\sum_{\substack{(i_1,i_2,\ldots,i_{a-1}) \in [b]^{a-1}; \\ i_1 \leq i_2 \leq \cdots \leq i_{a-1} \leq b}}\ \ \prod_{j=1}^{a-1}(x_{i_j}+y_{i_j+(j-1)+c}) \\
& \qquad \qquad \left( \begin{array}{c}
\text{here, we substituted $\tup{i_1, i_2, \ldots, i_{a-1}, b}$ for $\tup{i_1, i_2, \ldots, i_a}$} \\ \text{in the sum, since the condition $i_a = b$} \\ \text{uniquely determines the entry $i_a$}
\end{array} \right) \\
&= \tup{x_b+y_{a+b+c-1}}
\underbrace{\sum_{\substack{(i_1,i_2,\ldots,i_{a-1}) \in [b]^{a-1}; \\ i_1 \leq i_2 \leq \cdots \leq i_{a-1}}}\ \ \prod_{j=1}^{a-1}(x_{i_j}+y_{i_j+(j-1)+c})}_{\substack{= h\tup{a-1,b,c} \\ \text{(by the definition of $h\tup{a-1,b,c}$)}}} \\
& \qquad \qquad \left( \begin{array}{c}
\text{here, we replaced the condition $i_1 \leq i_2 \leq \cdots \leq i_{a-1} \leq b$ under} \\
\text{the summation sign by the condition $i_1 \leq i_2 \leq \cdots \leq i_{a-1}$,} \\ \text{which is equivalent because $i_1,i_2,\ldots,i_{a-1} \in \ive{b}$}
\end{array} \right) \\
&= \tup{x_b+y_{a+b+c-1}} h(a-1,b,c)
\end{align*}
and
\begin{align*}
&\sum_{\substack{(i_1,i_2,\ldots,i_a) \in [b]^a; \\ i_1 \leq i_2 \leq \cdots \leq i_a; \\ i_a \neq b}}\ \ \prod_{j=1}^a(x_{i_j}+y_{i_j+(j-1)+c}) \\
=\ &\sum_{\substack{(i_1,i_2,\ldots,i_a) \in [b]^a; \\ i_1 \leq i_2 \leq \cdots \leq i_a; \\ i_a < b}}\ \ \prod_{j=1}^a(x_{i_j}+y_{i_j+(j-1)+c})
\qquad \qquad
\tup{\begin{array}{c} \text{since the condition $i_a \neq b$} \\ \text{is equivalent to $i_a < b$} \\ \text{when $(i_1,i_2,\ldots,i_a) \in [b]^a$}
\end{array}} \\
=\ &\sum_{\substack{(i_1,i_2,\ldots,i_a) \in [b]^a; \\ i_1 \leq i_2 \leq \cdots \leq i_a < b}}\ \ \prod_{j=1}^a(x_{i_j}+y_{i_j+(j-1)+c}) \\
=\ &\sum_{\substack{(i_1,i_2,\ldots,i_a) \in [b-1]^a; \\ i_1 \leq i_2 \leq \cdots \leq i_a }}\ \ \prod_{j=1}^a(x_{i_j}+y_{i_j+(j-1)+c})
\qquad \qquad
\tup{\begin{array}{c} \text{since an element of $[b]$} \\ \text{that is $< b$ is the same as} \\ \text{an element of $[b-1]$}
\end{array}} \\
=\ & h(a,b-1,c)
\qquad \qquad \tup{\text{by the definition of $h(a,b-1,c)$}},
\end{align*}
we can rewrite this as 
\begin{align*}
h(a,b,c) = \tup{x_b+y_{a+b+c-1}} h(a-1,b,c) + h(a,b-1,c) .
\end{align*}
This proves the lemma.
\end{proof}

\begin{proof}[Proof of Lemma \ref{3hprop}.]
First, we observe that the lemma holds for $b=0$. Indeed, for $b=0$, it is
claiming that $h\left(  a,0,c\right)  -h\left(  a,0,c-1\right)  =\left(
y_{a+c-1}-y_{c}\right)  h\left(  a-1,0,c\right)  $. But this follows by
comparing
\begin{align*}
h\left(  a,0,c\right)  -h\left(  a,0,c-1\right)    &
=[a=0]-[a=0]\ \ \ \ \ \ \ \ \ \ \left(  \text{by (\ref{eq.h.triv.b=0}%
)}\right)  \\
& =0
\end{align*}
with
\begin{align*}
\left(  y_{a+c-1}-y_{c}\right)  h\left(  a-1,0,c\right)    &
=\underbrace{(y_{a+c-1}-y_{c})}_{=0\text{ if }a=1}\underbrace{[a-1=0]}%
_{=0\text{ if }a\neq1}\ \ \ \ \ \ \ \ \ \ \left(  \text{by
(\ref{eq.h.triv.b=0})}\right)  \\
& =0.
\end{align*}
Thus, the lemma is proved for $b=0$.

Furthermore, the lemma clearly holds for $a<0$, since all three h-polynomials
are $0$ in this case.

Thus, we can WLOG assume that $a\geq0$. Hence, $a+b\geq0$ (since both $a$ and
$b$ are $\geq0$). We shall now prove the lemma by induction on $a+b$.

The \textit{base case} ($a+b=0$) is clear, since $a+b=0$ entails $b=0$ (in
light of $a\geq0$), but we already have proved the lemma for $b=0$.

For the \textit{induction step}, we fix a positive integer $N$. Assume (as the
induction hypothesis) that the lemma holds whenever $a+b=N-1$. We must now
prove that the lemma holds whenever $a+b=N$. Thus, we fix integers $a,b\geq0$
satisfying $a+b=N$. Our goal is to prove the equality \eqref{eq.3hprop.eq}.

We WLOG assume that $b\neq0$, since we already have proved the lemma for
$b=0$. Thus, $b\geq1$ (since $b$ is a nonnegative integer), so that $b-1\geq
0$. Hence, by the induction hypothesis, we can apply the lemma to $b-1$
instead of $b$ (since $a+\left(  b-1\right)  =\underbrace{a+b}_{=N}-\,1=N-1$).
As a result, we obtain%
\begin{equation}
h\left(  a,b-1,c\right)  -h\left(  a,b-1,c-1\right)  =\left(  y_{a+b+c-2}%
-y_{c}\right)  \cdot h\left(  a-1,b-1,c\right)  .\label{pf.3hprop.ind-b-1}%
\end{equation}

Furthermore, by the induction hypothesis, we can apply the lemma to $a-1$
instead of $a$ (since $\left(  a-1\right)  +b=\underbrace{a+b}_{=N}-\,1=N-1$).
As a result, we obtain%
\begin{equation}
h\left(  a-1,b,c\right)  -h\left(  a-1,b,c-1\right)  =\left(  y_{a+b+c-2}%
-y_{c}\right)  \cdot h\left(  a-2,b,c\right)  .\label{pf.3hprop.ind-a-1}%
\end{equation}

On the other hand, Lemma \ref{1hprop} yields%
\begin{equation}
h\left(  a,b,c\right)  =\left(  x_{b}+y_{a+b+c-1}\right)  \cdot h\left(
a-1,b,c\right)  +h\left(  a,b-1,c\right)  .\label{pf.3hprop.lem1}%
\end{equation}

Furthermore, Lemma \ref{1hprop} (applied to $a-1$ instead of $a$) yields%
\begin{equation}
h\left(  a-1,b,c\right)  =\left(  x_{b}+y_{a+b+c-2}\right)  \cdot h\left(
a-2,b,c\right)  +h\left(  a-1,b-1,c\right)  .\label{pf.3hprop.lem1-a-1}%
\end{equation}
Finally, Lemma \ref{1hprop} (applied to $c-1$ instead of $c$) yields%
\begin{equation}
h\left(  a,b,c-1\right)  =\left(  x_{b}+y_{a+b+c-2}\right)  \cdot h\left(
a-1,b,c-1\right)  +h\left(  a,b-1,c-1\right)  .\label{pf.3hprop.lem1-c-1}%
\end{equation}

Let us set%
\begin{align*}
t  & :=h\left(  a,b,c\right)  ,\\
p  & :=h\left(  a-1,b,c\right)  ,\ \ \ \ \ \ \ \ \ \ \ \ \ \ \ \ q:=h\left(
a,b-1,c\right)  ,\ \ \ \ \ \ \ \ \ \ \ \ \ \ \ \ r:=h\left(  a,b,c-1\right)  ,\\
u  & :=h\left(  a,b-1,c-1\right)  ,\ \ \ \ \ \ \ \ \ \ v:=h\left(
a-1,b,c-1\right)  ,\ \ \ \ \ \ \ \ \ \ w:=h\left(  a-1,b-1,c\right)  ,\\
s  & :=h\left(  a-2,b,c\right)  ,\\
x  & :=x_{b},\ \ \ \ \ \ \ \ \ \ y:=y_{c},\ \ \ \ \ \ \ \ \ \ y^{\prime
}:=y_{a+b+c-1},\ \ \ \ \ \ \ \ \ \ y^{\prime\prime}:=y_{a+b+c-2}.
\end{align*}
Then, the equalities (\ref{pf.3hprop.ind-b-1}), (\ref{pf.3hprop.ind-a-1}),
(\ref{pf.3hprop.lem1}), (\ref{pf.3hprop.lem1-a-1}) and
(\ref{pf.3hprop.lem1-c-1}) can be rewritten as follows:%
\begin{align}
q-u  & =\left(  y^{\prime\prime}-y\right)  \cdot w;\label{pf.3hprop.sys1}\\
p-v  & =\left(  y^{\prime\prime}-y\right)  \cdot s;\label{pf.3hprop.sys2}\\
t  & =\left(  x+y^{\prime}\right)  \cdot p+q;\label{pf.3hprop.sys3}\\
p  & =\left(  x+y^{\prime\prime}\right)  \cdot s+w;\label{pf.3hprop.sys4}\\
r  & =\left(  x+y^{\prime\prime}\right)  \cdot v+u.\label{pf.3hprop.sys5}%
\end{align}

Recall that our goal is to prove the equality \eqref{eq.3hprop.eq}. In view
of the notations we have just introduced, we can rewrite this equality as
\[
t-r=\left(  y^{\prime}-y\right)  \cdot p.
\]
But it is not hard to derive this equality from the five equalities
(\ref{pf.3hprop.sys1})--(\ref{pf.3hprop.sys5}): Namely, subtracting
(\ref{pf.3hprop.sys5}) from (\ref{pf.3hprop.sys3}), we obtain%
\begin{align*}
t-r  & =\left(  \left(  x+y^{\prime}\right)  \cdot p+q\right)  -\left(
\left(  x+y^{\prime\prime}\right)  \cdot v+u\right)  \\
& =\underbrace{\left(  x+y^{\prime}\right)  \cdot p}_{=\left(  y^{\prime
}-y\right)  \cdot p+\left(  x+y\right)  \cdot p}-\left(  x+y^{\prime\prime
}\right)  \cdot v+\underbrace{q-u}_{\substack{=\left(  y^{\prime\prime
}-y\right)  \cdot w\\\text{(by (\ref{pf.3hprop.sys1}))}}}\\
& =\left(  y^{\prime}-y\right)  \cdot p+\underbrace{\left(  x+y\right)  \cdot
p-\left(  x+y^{\prime\prime}\right)  \cdot v+\left(  y^{\prime\prime
}-y\right)  \cdot w}_{\substack{=\left(  x+y^{\prime\prime}\right)  \left(
p-v\right)  +\left(  y-y^{\prime\prime}\right)  \left(  p-w\right)
\\\text{(by direct computation)}}}\\
& =\left(  y^{\prime}-y\right)  \cdot p+\left(  x+y^{\prime\prime}\right)
\underbrace{\left(  p-v\right)  }_{\substack{=\left(  y^{\prime\prime
}-y\right)  \cdot s\\\text{(by (\ref{pf.3hprop.sys2}))}}}+\left(
y-y^{\prime\prime}\right)  \underbrace{\left(  p-w\right)  }%
_{\substack{=\left(  x+y^{\prime\prime}\right)  \cdot s\\\text{(by
(\ref{pf.3hprop.sys4}))}}}\\
& =\left(  y^{\prime}-y\right)  \cdot p+\underbrace{\left(  x+y^{\prime\prime
}\right)  \left(  y^{\prime\prime}-y\right)  \cdot s+\left(  y-y^{\prime
\prime}\right)  \left(  x+y^{\prime\prime}\right)  \cdot s}%
_{\substack{=\left(  \left(  y^{\prime\prime}-y\right)  +\left(
y-y^{\prime\prime}\right)  \right)  \left(  x+y^{\prime\prime}\right)  \cdot
s\\=0\left(  x+y^{\prime\prime}\right)  \cdot s=0}}\\
& =\left(  y^{\prime}-y\right)  \cdot p.
\end{align*}
Thus, the equality \eqref{eq.3hprop.eq} is proved. This completes the
induction step, and therefore Lemma \ref{3hprop} is proved.
\end{proof}

\begin{proof}[Proof of Corollary \ref{2hprop}.]
Let $a$ and $c$ be integers, and let $b$ be a positive integer.
Lemma \ref{1hprop} yields
\[
h(a,b,c) = \tup{x_b+y_{a+b+c-1}} \cdot h(a-1,b,c) + h(a,b-1,c).
\]
Lemma \ref{3hprop} yields
\[
h\tup{a, b, c} - h\tup{a, b, c-1}
= \tup{y_{a+b+c-1} - y_c}\cdot h\tup{a-1, b, c}.
\]
Subtracting the latter equality from the former, we obtain
\begin{align*}
h\tup{a, b, c-1}
&= \tup{x_b+y_{a+b+c-1}} \cdot h(a-1,b,c) + h(a,b-1,c) \\
& \qquad \qquad - \tup{y_{a+b+c-1} - y_c}\cdot h\tup{a-1, b, c} \\
&= h(a,b-1,c) + (x_b+y_c)\cdot h(a-1,b,c) .
\end{align*}
In other words,
\begin{align*}
h(a,b-1,c) = h(a,b,c-1) - (x_b+y_c)\cdot h(a-1,b,c) .
\end{align*}
This proves Corollary \ref{2hprop}.
\end{proof}

\begin{proof}[Proof of Corollary \ref{3hprop2}.]
Lemma \ref{3hprop} (applied to $a+1$ and $c+1$ instead of $a$ and $c$) yields
\begin{align*}
h\tup{a+1, b, c+1} - h\tup{a+1, b, c+1-1}
= \tup{y_{\tup{a+1}+b+\tup{c+1}-1} - y_{c+1}}\cdot h\tup{a+1-1, b, c+1}.
\end{align*}
In view of $a+1-1=a$ and $c+1-1=c$ and $\tup{a+1}+b+\tup{c+1}-1 = a+b+c+1$, we can simplify this to
\begin{align*}
h\tup{a+1, b, c+1} - h\tup{a+1, b, c}
= \tup{y_{a+b+c+1} - y_{c+1}}\cdot h\tup{a, b, c+1}.
\end{align*}
In other words,
\[
h\tup{a+1, b, c+1}
= h\tup{a+1, b, c}
+ \tup{y_{a+b+c+1} - y_{c+1}} \cdot h\tup{a, b, c+1}.
\]
This proves Corollary~\ref{3hprop2}.
\end{proof}

\subsection{To Section \ref{sec.jt}}


\begin{proof}
[Proof of Lemma \ref{lem.flagJT.det=sum}.]\textbf{(a)} Let $\sigma\in S_{n}$.
We must prove the equality (\ref{eq.lem.flagJT.det=sum.a}). This equality
easily boils down to $0=0$ when $\sigma$ is not
legitimate\footnote{\textit{Proof.} Assume that $\sigma$ is not legitimate.
Then, there are no $\sigma$-arrays (by Definition \ref{def.jt.sig-array}
\textbf{(c)}). Hence, the sum $\sum_{\substack{T\text{ is a }\mathbf{b}%
\text{-flagged}\\\sigma\text{-array}}}w\left(  T\right)  $ is an empty sum and
thus equals $0$. On the other hand, not every $i\in\left[  n\right]  $
satisfies $\mu_{\sigma\left(  i\right)  }-\sigma\left(  i\right)  +i\geq0$
(since $\sigma$ is not legitimate). In other words, there exists a
$j\in\left[  n\right]  $ such that $\mu_{\sigma\left(  j\right)  }%
-\sigma\left(  j\right)  +j<0$. Consider this $j$.
\par
Now, recall that $h_{b;\ q}\left[  d\right]  $ is defined to be $0$ when
$q<0$. Hence, $h_{b_{\sigma\left(  j\right)  };\ \mu_{\sigma\left(  j\right)
}-\sigma\left(  j\right)  +j}\left[  j\right]  =0$ (since $\mu_{\sigma\left(
j\right)  }-\sigma\left(  j\right)  +j<0$). Hence, one factor of the product
$\prod_{i=1}^{n}h_{b_{\sigma\left(  i\right)  };\ \mu_{\sigma\left(  i\right)
}-\sigma\left(  i\right)  +i}\left[  i\right]  $ is $0$ (namely, the factor
for $i=j$). Therefore, this whole product is $0$. In other words, $\prod
_{i=1}^{n}h_{b_{\sigma\left(  i\right)  };\ \mu_{\sigma\left(  i\right)
}-\sigma\left(  i\right)  +i}\left[  i\right]  =0$. Since we also know that
$\sum_{\substack{T\text{ is a }\mathbf{b}\text{-flagged}\\\sigma\text{-array}%
}}w\left(  T\right)  $ equals $0$, we thus conclude that the equality
(\ref{eq.lem.flagJT.det=sum.a}) boils down to $0=0$.}. Thus, for the rest of
this proof, we WLOG assume that $\sigma$ is legitimate.

Set%
\[
q_{i}:=\mu_{\sigma\left(  i\right)  }-\sigma\left(  i\right)
+i\ \ \ \ \ \ \ \ \ \ \text{for each }i\in\left[  n\right]  .
\]
Note that this number $q_{i}$ is $\geq0$ because $\sigma$ is legitimate.

For any $\sigma$-array $T$, we have%
\[
w\left(  T\right)  =\prod_{(i,j)\in P(\sigma)}u_{T\left(  i,j\right)  ,\ j-i}%
\]
(by the definition of $w\left(  T\right)  $). Thus,%
\begin{equation}
\sum_{\substack{T\text{ is a }\mathbf{b}\text{-flagged}\\\sigma\text{-array}%
}}w\left(  T\right)  =\sum_{\substack{T\text{ is a }\mathbf{b}\text{-flagged}%
\\\sigma\text{-array}}}\ \ \prod_{(i,j)\in P(\sigma)}u_{T\left(  i,j\right)
,\ j-i}.\label{pf.lem.flagJT.det=sum.1}%
\end{equation}

For each $i\in\left[  n\right]  $, the $i$-th row of the diagram $P\left(
\sigma\right)  $ has $\mu_{\sigma\left(  i\right)  }-\sigma\left(  i\right)
+i$ many boxes (by the definition of $P\left(  \sigma\right)  $). In other
words, for each $i\in\left[  n\right]  $, the $i$-th row of the diagram
$P\left(  \sigma\right)  $ has $q_{i}$ many boxes (since $q_{i}=\mu
_{\sigma\left(  i\right)  }-\sigma\left(  i\right)  +i$). These boxes occupy
the columns $1,2,\ldots,q_{i}$ (since the rows of $P\left(  \sigma\right)  $
are left-aligned). Hence, these boxes are%
\[
\left(  i,1\right)  ,\ \left(  i,2\right)  ,\ \ldots,\ \left(  i,q_{i}\right)
.
\]
Thus, altogether, the boxes of $P\left(  \sigma\right)  $ are
\begin{align*}
& \left(  1,1\right)  ,\ \left(  1,2\right)  ,\ \ldots,\ \left(
1,q_{1}\right)  ,\\
& \left(  2,1\right)  ,\ \left(  2,2\right)  ,\ \ldots,\ \left(
2,q_{2}\right)  ,\\
& \ldots,\\
& \left(  n,1\right)  ,\ \left(  n,2\right)  ,\ \ldots,\ \left(
n,q_{n}\right)  .
\end{align*}
Therefore, the product sign $\prod_{(i,j)\in P(\sigma)}$ can be rewritten as
$\prod_{i=1}^{n}\ \ \prod_{j=1}^{q_{i}}$. Hence, the equality
(\ref{pf.lem.flagJT.det=sum.1}) rewrites as%
\begin{equation}
\sum_{\substack{T\text{ is a }\mathbf{b}\text{-flagged}\\\sigma\text{-array}%
}}w\left(  T\right)  =\sum_{\substack{T\text{ is a }\mathbf{b}\text{-flagged}%
\\\sigma\text{-array}}}\ \ \prod_{i=1}^{n}\ \ \prod_{j=1}^{q_{i}}u_{T\left(
i,j\right)  ,\ j-i}.\label{pf.lem.flagJT.det=sum.2}%
\end{equation}

Now, recall that for each $i\in\left[  n\right]  $, the $i$-th row of the
diagram $P\left(  \sigma\right)  $ has $q_{i}$ many boxes, and these boxes
occupy the columns $1,2,\ldots,q_{i}$. In a $\sigma$-array, these $q_{i}$
boxes have to be filled with positive integers $a_{i,1},a_{i,2},\ldots
,a_{i,q_{i}}$ that satisfy $a_{i,1}\leq a_{i,2}\leq\cdots\leq a_{i,q_{i}}$
(since the entries of a $\sigma$-array must be weakly increasing along each
row). Moreover, in a $\mathbf{b}$-flagged $\sigma$-array, these entries
$a_{i,1},a_{i,2},\ldots,a_{i,q_{i}}$ must belong to the set $\left[
b_{\sigma\left(  i\right)  }\right]  $ (since every entry of $T$ in the $i$-th
row must be $\leq b_{\sigma\left(  i\right)  }$). Thus, a $\mathbf{b}$-flagged
$\sigma$-array is simply a way to fill the $i$-th row of $P\left(
\sigma\right)  $ with positive integers $a_{i,1},a_{i,2},\ldots,a_{i,q_{i}}%
\in\left[  b_{\sigma\left(  i\right)  }\right]  $ satisfying $a_{i,1}\leq
a_{i,2}\leq\cdots\leq a_{i,q_{i}}$ for each $i\in\left[  n\right]  $. If we
denote our $\sigma$-array by $T$, then the latter integers $a_{i,j}$ are
simply its respective entries $T\left(  i,j\right)  $. Thus,
\begin{align}
&  \sum_{\substack{T\text{ is a }\mathbf{b}\text{-flagged}\\\sigma
\text{-array}}}\ \ \prod_{i=1}^{n}\ \ \prod_{j=1}^{q_{i}}u_{T\left(
i,j\right)  ,\ j-i}\nonumber\\
&  =\sum_{\substack{\left(a_{i,1},a_{i,2},\ldots,a_{i,q_{i}}\right)\in\left[
b_{\sigma\left(  i\right)  }\right]^{q_i}  \text{ for each }i\in\left[  n\right]
;\\a_{i,1}\leq a_{i,2}\leq\cdots\leq a_{i,q_{i}}\text{ for each }i\in\left[
n\right]  }}\ \ \prod_{i=1}^{n}\ \ \prod_{j=1}^{q_{i}}u_{a_{i,j}%
,\ j-i}\nonumber\\
&  =\prod_{i=1}^{n}\ \ \sum_{\substack{\left(a_{i,1},a_{i,2},\ldots,a_{i,q_{i}}\right)
\in\left[  b_{\sigma\left(  i\right)  }\right]^{q_i};\\a_{i,1}\leq a_{i,2}%
\leq\cdots\leq a_{i,q_{i}}}}\ \ \prod_{j=1}^{q_{i}}u_{a_{i,j},\ j-i}%
\label{sol.lem.flagJT.det=sum.4}%
\end{align}
(by the product rule).%
\footnote{Here is a more rigorous way of deriving this equality:

Consider the map%
\[
\left\{  \mathbf{b}\text{-flagged }\sigma\text{-arrays}\right\}  \rightarrow
\prod_{i=1}^{n}\left\{  \left(  a_{i,1},a_{i,2},\ldots,a_{i,q_{i}}\right)
\in\left[  b_{\sigma\left(  i\right)  }\right]  ^{q_{i}}\ \mid\ a_{i,1}\leq
a_{i,2}\leq\cdots\leq a_{i,q_{i}}\right\}
\]
that sends each $\mathbf{b}$-flagged $\sigma$-array $T$ to the $n$-tuple
\begin{align*}
& \big(\left(  T\left(  1,1\right)  ,\ T\left(  1,2\right)  ,\ \ldots
,\ T\left(  1,q_{1}\right)  \right)  ,\\
& \ \ \ \ \left(  T\left(  2,2\right)  ,\ T\left(  2,2\right)  ,\ \ldots
,\ T\left(  2,q_{2}\right)  \right)  ,\\
& \ \ \ \ \ldots,\\
& \ \ \ \ \left(  T\left(  n,1\right)  ,\ T\left(  n,2\right)  ,\ \ldots
,\ T\left(  n,q_{n}\right)  \right)  \big).
\end{align*}
This map is well-defined (because if $T$ is any $\mathbf{b}$-flagged $\sigma
$-array, then each $i\in\left[  n\right]  $ and each $j\in\left[
q_{i}\right]  $ satisfy $T\left(  i,j\right)  \leq b_{\sigma\left(  i\right)
}$ and therefore $T\left(  i,j\right)  \in\left[  b_{\sigma\left(  i\right)
}\right]  $, and we furthermore have $T\left(  i,1\right)  \leq T\left(
i,2\right)  \leq\cdots\leq T\left(  i,q_{i}\right)  $ for each $i\in\left[
n\right]  $ since $T$ is a $\sigma$-array). Moreover, it is injective (since
each entry of $T\left(  i,j\right)  $ can be read off from its output) and
surjective (because if we fill the $i$-th row of the diagram $P\left(
\sigma\right)  $ with some elements $a_{i,1},a_{i,2},\ldots,a_{i,q_{i}}%
\in\left[  b_{\sigma\left(  i\right)  }\right]  $ satisfying $a_{i,1}\leq
a_{i,2}\leq\cdots\leq a_{i,q_{i}}$ for each $i\in\left[  n\right]  $, then we
obtain a $\mathbf{b}$-flagged $\sigma$-array). Thus, it is bijective, i.e., is
a bijection. Hence, we can substitute $a_{i,j}$ for $T\left(  i,j\right)  $ in
the sum%
\[
\sum_{\substack{T\text{ is a }\mathbf{b}\text{-flagged}\\\sigma\text{-array}%
}}\ \ \prod_{i=1}^{n}\ \ \prod_{j=1}^{q_{i}}u_{T\left(  i,j\right)  ,\ j-i},
\]
and thus obtain
\begin{align*}
&  \sum_{\substack{T\text{ is a }\mathbf{b}\text{-flagged}\\\sigma
\text{-array}}}\ \ \prod_{i=1}^{n}\ \ \prod_{j=1}^{q_{i}}u_{T\left(
i,j\right)  ,\ j-i}\\
&  =\sum_{\substack{\left(  \left(  a_{1,1},a_{1,2},\ldots,a_{1,q_{1}}\right)
,\ \ \left(  a_{2,1},a_{2,2},\ldots,a_{2,q_{2}}\right)  ,\ \ \ldots
,\ \ \left(  a_{n,1},a_{n,2},\ldots,a_{n,q_{n}}\right)  \right)  \\\in
\prod\limits_{i=1}^{n}\left\{  \left(  a_{i,1},a_{i,2},\ldots,a_{i,q_{i}%
}\right)  \in\left[  b_{\sigma\left(  i\right)  }\right]  ^{q_{i}}%
\ \mid\ a_{i,1}\leq a_{i,2}\leq\cdots\leq a_{i,q_{i}}\right\}  }%
}\ \ \prod_{i=1}^{n}\ \ \prod_{j=1}^{q_{i}}u_{a_{i,j},\ j-i}\\
&=  \prod_{i=1}^{n}\ \ \sum_{\substack{\left(a_{i,1},a_{i,2},\ldots,a_{i,q_{i}}\right)
\in\left[  b_{\sigma\left(  i\right)  }\right]^{q_i}  ;\\a_{i,1}\leq a_{i,2}%
\leq\cdots\leq a_{i,q_{i}}}}\ \ \prod_{j=1}^{q_{i}}u_{a_{i,j},\ j-i}
\qquad \qquad \tup{\text{by the product rule}}.
\end{align*}
This proves (\ref{sol.lem.flagJT.det=sum.4}).}

However, for each $i\in\left[  n\right]  $, we have%
\begin{align*}
&
\sum_{\substack{\left(  a_{i,1},a_{i,2},\ldots,a_{i,q_{i}}\right)
\in\left[  b_{\sigma\left(  i\right)  }\right]  ^{q_{i}};\\a_{i,1}\leq
a_{i,2}\leq\cdots\leq a_{i,q_{i}}}}\ \ \prod_{j=1}^{q_{i}}u_{a_{i,j},\ j-i}\\
&  =\sum_{\substack{\left(  t_{1},t_{2},\ldots,t_{q_{i}}\right)  \in\left[
b_{\sigma\left(  i\right)  }\right]  ^{q_{i}};\\t_{1}\leq t_{2}\leq\cdots\leq
t_{q_{i}}}}\ \ \prod_{j=1}^{q_{i}}u_{t_{j},\ j-i}\ \ \ \ \ \ \ \ \ \ \left(
\begin{array}
[c]{c}%
\text{here, we have renamed}\\
\text{the index }\left(  a_{i,1},a_{i,2},\ldots,a_{i,q_{i}}\right)  \\
\text{as }\left(  t_{1},t_{2},\ldots,t_{q_{i}}\right)
\end{array}
\right)  \\
&  =h_{b_{\sigma\left(  i\right)  };\ q_{i}}\left[  i\right]
\ \ \ \ \ \ \ \ \ \ \left(  \text{by the definition of }h_{b_{\sigma\left(
i\right)  };\ q_{i}}\left[  i\right]  \right)  \\
&  =h_{b_{\sigma\left(  i\right)  };\ \mu_{\sigma\left(  i\right)  }%
-\sigma\left(  i\right)  +i}\left[  i\right]  \ \ \ \ \ \ \ \ \ \ \left(
\text{since }q_{i}=\mu_{\sigma\left(  i\right)  }-\sigma\left(  i\right)
+i\right)  .
\end{align*}
Thus, we can rewrite (\ref{sol.lem.flagJT.det=sum.4}) as%
\[
\sum_{\substack{T\text{ is a }\mathbf{b}\text{-flagged}\\\sigma\text{-array}%
}}\ \ \prod_{i=1}^{n}\ \ \prod_{j=1}^{q_{i}}u_{T\left(  i,j\right)
,\ j-i}=\prod_{i=1}^{n}h_{b_{\sigma\left(  i\right)  };\ \mu_{\sigma\left(
i\right)  }-\sigma\left(  i\right)  +i}\left[  i\right]  .
\]
In view of (\ref{pf.lem.flagJT.det=sum.2}), this rewrites as%
\[
\sum_{\substack{T\text{ is a }\mathbf{b}\text{-flagged}\\\sigma\text{-array}%
}}w\left(  T\right)  =\prod_{i=1}^{n}h_{b_{\sigma\left(  i\right)  }%
;\ \mu_{\sigma\left(  i\right)  }-\sigma\left(  i\right)  +i}\left[  i\right]
.
\]
This proves Lemma \ref{lem.flagJT.det=sum} \textbf{(a)}. \medskip

\textbf{(b)} Every square matrix $A$ satisfies $\det\left(  A^{T}\right)
=\det A$. In other words, every square matrix $\left(  a_{i,j}\right)
_{i,j\in\left[  n\right]  }$ satisfies $\det\left(  a_{j,i}\right)
_{i,j\in\left[  n\right]  }=\det\left(  a_{i,j}\right)  _{i,j\in\left[
n\right]  }$. Applying this to $a_{i,j}=h_{b_{j};\ \mu_{j}-j+i}\left[
i\right]  $, we obtain%
\begin{align*}
\det\left(  h_{b_{i};\ \mu_{i}-i+j}\left[  j\right]  \right)  _{i,j\in\left[
n\right]  }  &  =\det\left(  h_{b_{j};\ \mu_{j}-j+i}\left[  i\right]  \right)
_{i,j\in\left[  n\right]  }\\
&  =\sum_{\sigma\in S_{n}}\left(  -1\right)  ^{\sigma}\underbrace{\prod
_{i=1}^{n}h_{b_{\sigma\left(  i\right)  };\ \mu_{\sigma\left(  i\right)
}-\sigma\left(  i\right)  +i}\left[  i\right]  }_{\substack{=\sum
_{\substack{T\text{ is a }\mathbf{b}\text{-flagged}\\\sigma\text{-array}%
}}w\left(  T\right)  \\\text{(by Lemma \ref{lem.flagJT.det=sum} \textbf{(a)}%
)}}}\ \ \ \ \ \ \ \ \ \ \left(  \text{by (\ref{eq.det.leibniz})}\right) \\
&  =\sum_{\sigma\in S_{n}}\left(  -1\right)  ^{\sigma}\sum_{\substack{T\text{
is a }\mathbf{b}\text{-flagged}\\\sigma\text{-array}}}w\left(  T\right)  .
\end{align*}
This proves Lemma \ref{lem.flagJT.det=sum} \textbf{(b)}.
\end{proof}


\begin{proof}
[Proof of Lemma \ref{lem.flagJT.det=sum-twisted}.] Lemma
\ref{lem.flagJT.det=sum} \textbf{(b)} yields%
\begin{align}
\det\left(  h_{b_{i};\ \mu_{i}-i+j}\left[  j\right]  \right)  _{i,j\in\left[
n\right]  } &  =\sum_{\sigma\in S_{n}}\left(  -1\right)  ^{\sigma}%
\sum_{\substack{T\text{ is a }\mathbf{b}\text{-flagged}\\\sigma\text{-array}%
}}w\left(  T\right)  \nonumber\\
&  =\sum_{\sigma\in S_{n}}\ \ \sum_{\substack{T\text{ is a }\mathbf{b}%
\text{-flagged}\\\sigma\text{-array}}}\left(  -1\right)  ^{\sigma}w\left(
T\right)  .\label{pf.lem.flagJT.det=sum-twisted.2}%
\end{align}

However, the $\mathbf{b}$-flagged twisted arrays are precisely the pairs
$\left(  \sigma,T\right)  $ in which $\sigma\in S_{n}$ and $T$ is a
$\mathbf{b}$-flagged $\sigma$-array. Thus, the summation sign $\sum
_{\substack{\left(  \sigma,T\right)  \text{ is a }\mathbf{b}\text{-flagged}%
\\\text{twisted array}}}$ can be rewritten as the nested summation
$\sum_{\sigma\in S_{n}}\ \ \sum_{\substack{T\text{ is a }\mathbf{b}%
\text{-flagged}\\\sigma\text{-array}}}$. Hence,%
\[
\sum_{\substack{\left(  \sigma,T\right)  \text{ is a }\mathbf{b}%
\text{-flagged}\\\text{twisted array}}}\left(  -1\right)  ^{\sigma}w\left(
T\right)  =\sum_{\sigma\in S_{n}}\ \ \sum_{\substack{T\text{ is a }%
\mathbf{b}\text{-flagged}\\\sigma\text{-array}}}\left(  -1\right)  ^{\sigma
}w\left(  T\right)  .
\]
Comparing this with (\ref{pf.lem.flagJT.det=sum-twisted.2}), we obtain%
\[
\det\left(  h_{b_{i};\ \mu_{i}-i+j}\left[  j\right]  \right)  _{i,j\in\left[
n\right]  }=\sum_{\substack{\left(  \sigma,T\right)  \text{ is a }%
\mathbf{b}\text{-flagged}\\\text{twisted array}}}\left(  -1\right)  ^{\sigma
}w\left(  T\right)  .
\]
Thus, Lemma \ref{lem.flagJT.det=sum-twisted} is proved.
\end{proof}


\begin{proof}
[Proof of Lemma \ref{lem.flagJT.fail}.] We know that $\left(  \sigma,T\right)
$ is a twisted array. In other words, we have $\sigma \in S_n$, and $T$ is a $\sigma$-array. \medskip

\textbf{(a)} Assume that $\sigma\neq\operatorname*{id}$. Then, there exists
some $i\in\left\{  1,2,\ldots,n-1\right\}  $ such that $\sigma\left(
i\right)  >\sigma\left(  i+1\right)  $ (since otherwise, we would have
$\sigma\left(  1\right)  \leq\sigma\left(  2\right)  \leq\cdots\leq
\sigma\left(  n\right)  $, which would readily lead to $\sigma
=\operatorname*{id}$). Consider this $i$.

We have $\mu_{1}\geq\mu_{2}\geq\mu_{3}\geq\cdots$ (since $\mu$ is a partition).

Subtracting the chain of inequalities $1<2<3<\cdots$ from $\mu_{1}\geq\mu
_{2}\geq\mu_{3}\geq\cdots$, we obtain $\mu_{1}-1>\mu_{2}-2>\mu_{3}-3>\cdots$.
In other words, if $u$ and $v$ are two positive integers satisfying $u<v$,
then $\mu_{u}-u>\mu_{v}-v$. Applying this to $u=\sigma\left(  i+1\right)  $
and $v=\sigma\left(  i\right)  $, we obtain $\mu_{\sigma\left(  i+1\right)
}-\sigma\left(  i+1\right)  >\mu_{\sigma\left(  i\right)  }-\sigma\left(
i\right)  $ (since $\sigma\left(  i+1\right)  <\sigma\left(  i\right)  $).
Hence,%
\begin{align*}
\underbrace{\mu_{\sigma\left(  i+1\right)  }-\sigma\left(  i+1\right)  }%
_{>\mu_{\sigma\left(  i\right)  }-\sigma\left(  i\right)  }+\left(
i+1\right)   &  >\mu_{\sigma\left(  i\right)  }-\sigma\left(  i\right)
+\underbrace{\left(  i+1\right)  }_{>i}
>\mu_{\sigma\left(  i\right)  }-\sigma\left(  i\right)  +i.
\end{align*}
In other words, the $\left(  i+1\right)  $-th row of the diagram $P\left(
\sigma\right)  $ is longer than the $i$-th row (since Definition \ref{def.jt.sig-array} \textbf{(b)} shows that the $\left(  i+1\right)
$-th row has length $\mu_{\sigma\left(  i+1\right)  }-\sigma\left(
i+1\right)  +\left(  i+1\right)  $, while the $i$-th has length $\mu
_{\sigma\left(  i\right)  }-\sigma\left(  i\right)  +i$). Thus, the easternmost box
of the $\left(  i+1\right)  $-th row of the diagram $P\left(  \sigma\right)  $
is an outer failure of $\left(  \sigma,T\right)  $ (because its northern
neighbor lies too far east to be contained in the $i$-th row of $P\left(
\sigma\right)  $).
Therefore, $\left(  \sigma,T\right)  $ has an outer failure.
This proves Lemma \ref{lem.flagJT.fail} \textbf{(a)}. \medskip

\textbf{(b)} Assume that $\sigma=\operatorname*{id}$ and $T\notin%
\operatorname*{SSYT}\left(  \mu\right)  $.

For each $i\in\left[  n\right]  $, we have $\sigma\left(  i\right)  =i$ (since
$\sigma=\operatorname*{id}$) and thus%
\[
\mu_{\sigma\left(  i\right)  }-\sigma\left(  i\right)  +i=\mu_{i}-i+i=\mu_{i}.
\]
Hence, the diagram $P\left(  \sigma\right)  $ is precisely the Young diagram
$Y\left(  \mu\right)  $ (just compare their definitions). Therefore, $T$ is a
filling of $Y\left(  \mu\right)  $ (since $T$ is a filling of $P\left(
\sigma\right)  $) with positive integers that weakly increase left-to-right
along each row (since $T$ is a $\sigma$-array). If the entries of $T$ also
strictly increased top-to-bottom down each column, then it would follow that
$T$ is a semistandard tableau, whence $T\in\operatorname*{SSYT}\left(
\mu\right)  $; but this would contradict  $T\notin\operatorname*{SSYT}\left(
\mu\right)  $. Hence, the entries of $T$ cannot strictly increase
top-to-bottom down each column. Thus, there must be at least one column of $T$
in which the entries do not strictly increase top-to-bottom. In other words,
there must be at least one column of $T$ that has two adjacent boxes $\left(
i-1,j\right)  $ and $\left(  i,j\right)  $ satisfying $T\left(  i-1,j\right)
\geq T\left(  i,j\right)  $. Consider these two boxes. Then, $\left(
i,j\right)  $ is an inner failure of $\left(  \sigma,T\right)  $ (by the
definition of an inner failure). Therefore, $\left(  \sigma,T\right)  $ has an
inner failure. This proves Lemma \ref{lem.flagJT.fail} \textbf{(b)}. \medskip

\textbf{(c)} Assume that $\left(  \sigma,T\right)  $ is unfailing. Thus,
$\left(  \sigma,T\right)  $ has no failures.

If we had $\sigma\neq\operatorname*{id}$, then $\left(  \sigma,T\right)  $
would have an outer failure (by Lemma \ref{lem.flagJT.fail} \textbf{(a)}),
which would contradict the fact that $\left(  \sigma,T\right)  $ has no
failures. Thus, we cannot have $\sigma\neq\operatorname*{id}$. Hence, we have
$\sigma=\operatorname*{id}$.

If we had $T\notin\operatorname*{SSYT}\left(  \mu\right)  $, then $\left(
\sigma,T\right)  $ would have an inner failure (by Lemma \ref{lem.flagJT.fail}
\textbf{(b)}), which would contradict the fact that $\left(  \sigma,T\right)
$ has no failures. Thus, we cannot have $T\notin\operatorname*{SSYT}\left(
\mu\right)  $. Hence, we have $T\in\operatorname*{SSYT}\left(  \mu\right)  $.
This completes the proof of Lemma \ref{lem.flagJT.fail} \textbf{(c)}.
\end{proof}


\begin{proof}
[Proof of Lemma \ref{lem.flagJT.failsum}.] Let $\operatorname*{id}$ denote the
identity permutation in $S_{n}$. Then, the diagram $P\left(
\operatorname*{id}\right)  $ is a left-aligned diagram whose $i$-th row has
\begin{align*}
\mu_{\operatorname*{id}\left(  i\right)  }-\operatorname*{id}\left(  i\right)
+i  & =\mu_{i}-i+i\ \ \ \ \ \ \ \ \ \ \left(  \text{since }\operatorname*{id}%
\left(  i\right)  =i\right)  \\
& =\mu_{i}%
\end{align*}
boxes for each $i\in\left[  n\right]  $. But this precisely describes the
Young diagram $Y\left(  \mu\right)  $. Hence, $P\left(  \operatorname*{id}%
\right)  =Y\left(  \mu\right)  $.

If $T\in\operatorname*{SSYT}\left(  \mu\right)  $ is a semistandard tableau,
then the pair $\left(  \operatorname*{id},T\right)  $ is clearly a twisted
array (since $T$ is a filling of the diagram $Y\left(  \mu\right)  =P\left(
\operatorname*{id}\right)  $, and its entries weakly increase left-to-right
along each row). Moreover, this twisted array $\left(  \operatorname*{id}%
,T\right)  $ has no outer failures (since its shape $P\left(
\operatorname*{id}\right)  =Y\left(  \mu\right)  $ is the Young diagram of a
partition) and no inner failures (since $T$ is semistandard, so that the
entries of $T$ strictly increase down each column). In other words, this twisted
array $\left(  \operatorname*{id},T\right)  $ is unfailing. Thus, we obtain a
map%
\begin{align*}
\Phi:\operatorname*{SSYT}\left(  \mu\right)    & \rightarrow\left\{
\text{unfailing twisted arrays}\right\}  ,\\
T  & \mapsto\left(  \operatorname*{id},T\right)  .
\end{align*}
This map $\Phi$ is injective (obviously) and surjective (by Lemma
\ref{lem.flagJT.fail} \textbf{(c)})%
\footnote{Here is the proof of the surjectivity of $\Phi$ in some more detail:
Let $\tup{\sigma, T}$ be an unfailing twisted array. Then, Lemma
\ref{lem.flagJT.fail} \textbf{(c)} shows that $\sigma = \id$ and $T \in \SSYT\tup{\mu}$.
Hence, $\Phi\tup{T} = \tup{\id, T} = \tup{\sigma, T}$ (since $\id = \sigma$).
Therefore, $\tup{\sigma, T} = \Phi\tup{T}$.
Thus, $\tup{\sigma, T}$ is an image under the map $\Phi$.
Forget that we fixed $\tup{\sigma, T}$.
We thus have shown that each unfailing twisted array $\tup{\sigma, T}$ is an image under the map $\Phi$.
In other words, $\Phi$ is surjective.}.
Hence, it is a bijection.

Moreover, a semistandard tableau $T\in\operatorname*{SSYT}\left(  \mu\right)
$ is $\mathbf{b}$-flagged if and only if the corresponding twisted array
$\Phi\left(  T\right)  =\left(  \operatorname*{id},T\right)  $ is $\mathbf{b}%
$-flagged\footnote{\textit{Proof.} Let $T\in\operatorname*{SSYT}\left(
\mu\right)  $ be a semistandard tableau. Then, the tableau $T$ is $\mathbf{b}$-flagged if
and only if it satisfies the condition%
\begin{equation}
\left(  T\left(  i,j\right)  \leq b_{i}\ \ \ \ \ \ \ \ \ \ \text{for all
}\left(  i,j\right)  \in Y\left(  \mu\right)  \right)
.\label{pf.lem.flagJT.failsum.fn1.1}%
\end{equation}
On the other hand, the twisted array $\left(  \operatorname*{id},T\right)  $
is $\mathbf{b}$-flagged if and only if the $\operatorname*{id}$-array $T$ is
$\mathbf{b}$-flagged, i.e., if and only if it satisfies the condition%
\begin{equation}
\left(  T\left(  i,j\right)  \leq b_{\operatorname*{id}\left(  i\right)
}\ \ \ \ \ \ \ \ \ \ \text{for all }\left(  i,j\right)  \in P\left(
\operatorname*{id}\right)  \right)  .\label{pf.lem.flagJT.failsum.fn1.2}%
\end{equation}
But the two conditions (\ref{pf.lem.flagJT.failsum.fn1.1}) and
(\ref{pf.lem.flagJT.failsum.fn1.2}) are equivalent (since $Y\left(
\mu\right)  =P\left(  \operatorname*{id}\right)  $ and $b_{i}%
=b_{\operatorname*{id}\left(  i\right)  }$ for each $i$). Thus, the tableau $T$ is
$\mathbf{b}$-flagged if and only if $\left(  \operatorname*{id},T\right)  $ is
$\mathbf{b}$-flagged. In other words, the tableau 
$T$ is $\mathbf{b}$-flagged if and only
if $\Phi\left(  T\right)  $ is $\mathbf{b}$-flagged (since $\Phi\left(
T\right)  =\left(  \operatorname*{id},T\right)  $).}. Hence, the bijection
$\Phi$ can be restricted to a bijection%
\begin{align*}
\left\{  \mathbf{b}\text{-flagged }T\in\operatorname*{SSYT}\left(  \mu\right)
\right\}    & \rightarrow\left\{  \mathbf{b}\text{-flagged unfailing twisted
arrays}\right\}  ,\\
T  & \mapsto\left(  \operatorname*{id},T\right)  .
\end{align*}
Thus, we can substitute $\left(  \operatorname*{id},T\right)  $ for $\left(
\sigma,T\right)  $ in the sum $\sum_{\substack{\left(  \sigma,T\right)  \text{
is an}\\\text{unfailing }\mathbf{b}\text{-flagged}\\\text{twisted array}%
}}\left(  -1\right)  ^{\sigma}w\left(  T\right)  $. We thus obtain%
\begin{align*}
& \sum_{\substack{\left(  \sigma,T\right)  \text{ is an}\\\text{unfailing
}\mathbf{b}\text{-flagged}\\\text{twisted array}}}\left(  -1\right)  ^{\sigma
}w\left(  T\right)  \\
& =\underbrace{\sum_{\substack{T\in\operatorname*{SSYT}\left(  \mu\right)
;\\T\text{ is }\mathbf{b}\text{-flagged}}}}_{\substack{=\sum_{T\in
\operatorname{FSSYT}\left(  \mu,\mathbf{b}\right)  }\\\text{(by the definition
of }\operatorname{FSSYT}\left(  \mu,\mathbf{b}\right)  \text{)}}%
}\underbrace{\left(  -1\right)  ^{\operatorname*{id}}}_{=1}%
\underbrace{w\left(  T\right)  }_{\substack{=\prod_{(i,j)\in
P(\operatorname*{id})}u_{T\left(  i,j\right)  ,\ j-i}\\\text{(by the
definition of }w\left(  T\right)  \text{)}}}\\
& =\sum_{T\in\operatorname{FSSYT}\left(  \mu,\mathbf{b}\right)  }%
\ \ \prod_{(i,j)\in P(\operatorname*{id})}u_{T\left(  i,j\right)  ,\ j-i}%
=\sum_{T\in\operatorname{FSSYT}\left(  \mu,\mathbf{b}\right)  }\ \ \prod
_{(i,j)\in Y(\mu)}u_{T\left(  i,j\right)  ,\ j-i}%
\end{align*}
(since $P\left(  \operatorname*{id}\right)  =Y\left(  \mu\right)  $). Thus,
Lemma \ref{lem.flagJT.failsum} is proved.
\end{proof}

\begin{proof}[Proof of Lemma \ref{lem.flagJT.flip1}.]
Let $c = \tup{i,j} \in P\tup{\sigma}$ be the bottommost leftmost failure of $\tup{\sigma, T}$.
(This exists, because the twisted array $\tup{\sigma, T}$ is failing.)
Thus, $\tup{\sigma^{\prime}, T^{\prime}} = \operatorname{flip}\tup{\sigma, T}$ is constructed
as described in Definition \ref{def.flagJT.flip}.

Remark \ref{rmk.flagJT.fail.over1} shows that $i>1$ (since $\left(
i,j\right)  $ is a failure of $\left(  \sigma,T\right)  $), but we also have
$i\in\left[  n\right]  $ (since each box of $P\tup{\sigma}$ has the form $\tup{p,q}$ with $p\in\ive{n}$).
Thus, $i-1\in\left[  n\right]  $.

Definition \ref{def.flagJT.flip} yields $\sigma^{\prime}=\sigma\circ s_{i-1}$.
Thus, the permutation $\sigma^{\prime}$ is obtained from $\sigma$ by swapping
the values at $i-1$ and $i$. In other words, we have%
\[
\sigma^{\prime}\left(  i-1\right)  =\sigma\left(  i\right)
\ \ \ \ \ \ \ \ \ \ \text{and}\ \ \ \ \ \ \ \ \ \ \sigma^{\prime}\left(
i\right)  =\sigma\left(  i-1\right)
\]
and%
\begin{equation}
\sigma^{\prime}\left(  k\right)  =\sigma\left(  k\right)
\ \ \ \ \ \ \ \ \ \ \text{for each }k\in\left[  n\right]  \setminus\left\{
i,i-1\right\}  . \label{pf.lem.flagJT.flip1.siprime}%
\end{equation}

We set%
\[
\rho_{k}:=\mu_{\sigma\left(  k\right)  }-\sigma\left(  k\right)
+k\ \ \ \ \ \ \ \ \ \ \text{and}\ \ \ \ \ \ \ \ \ \ \rho_{k}^{\prime}%
:=\mu_{\sigma^{\prime}\left(  k\right)  }-\sigma^{\prime}\left(  k\right)  +k
\]
for each $k\in\left[  n\right]  $. Now, we claim the following:

\begin{statement}
\textit{Claim 1:} We have $\rho_{k}^{\prime}=\rho_{k}$ for each $k\in\left[
n\right]  \setminus\left\{  i,i-1\right\}  $.
\end{statement}

\begin{statement}
\textit{Claim 2:} We have $\rho_{i}^{\prime}=\rho_{i-1}+1$.
\end{statement}

\begin{statement}
\textit{Claim 3:} We have $\rho_{i-1}^{\prime}=\rho_{i}-1$.
\end{statement}

\begin{proof}
[Proof of Claim 1.]Let $k\in\left[  n\right]  \setminus\left\{  i,i-1\right\}
$. Then, (\ref{pf.lem.flagJT.flip1.siprime}) yields $\sigma^{\prime}\left(
k\right)  =\sigma\left(  k\right)  $. Now, recall that $\rho_{k}^{\prime}%
=\mu_{\sigma^{\prime}\left(  k\right)  }-\sigma^{\prime}\left(  k\right)  +k$
and $\rho_{k}=\mu_{\sigma\left(  k\right)  }-\sigma\left(  k\right)  +k$. The
right hand sides of these two equalities are identical (since $\sigma^{\prime
}\left(  k\right)  =\sigma\left(  k\right)  $). Hence, their left hand sides
are identical as well. In other words, $\rho_{k}^{\prime}=\rho_{k}$. This
proves Claim 1.
\end{proof}

\begin{proof}
[Proof of Claim 2.]The definition of $\rho_{i-1}$ yields
\[
\rho_{i-1}=\mu_{\sigma\left(  i-1\right)  }-\sigma\left(  i-1\right)  +\left(
i-1\right)  .
\]
The definition of $\rho_{i}^{\prime}$ yields%
\begin{align*}
\rho_{i}^{\prime}  &  =\mu_{\sigma^{\prime}\left(  i\right)  }-\sigma^{\prime
}\left(  i\right)  +i\\
&  =\mu_{\sigma\left(  i-1\right)  }-\sigma\left(  i-1\right)  +\underbrace{i}%
_{=\left(  i-1\right)  +1}\ \ \ \ \ \ \ \ \ \ \left(  \text{since }%
\sigma^{\prime}\left(  i\right)  =\sigma\left(  i-1\right)  \right) \\
&  =\underbrace{\mu_{\sigma\left(  i-1\right)  }-\sigma\left(  i-1\right)
+\left(  i-1\right)  }_{=\rho_{i-1}}+\, 1\\
&  =\rho_{i-1}+1.
\end{align*}
Thus, Claim 2 is proved.
\end{proof}

\begin{proof}
[Proof of Claim 3.]Similar to the above proof of Claim 2.
\end{proof}

Now, $T$ is a $\sigma$-array (since $\left(  \sigma,T\right)  $ is a twisted
array). Hence, the permutation $\sigma$ is legitimate (since $\sigma$-arrays
only exist when $\sigma$ is legitimate). In other words, each $k\in\left[
n\right]  $ satisfies $\mu_{\sigma\left(  k\right)  }-\sigma\left(  k\right)
+k\geq0$ (by the definition of \textquotedblleft legitimate\textquotedblright%
). In other words, each $k\in\left[  n\right]  $ satisfies%
\begin{equation}
\rho_{k}\geq0 \label{pf.lem.flagJT.flip1.rhoigeq0}%
\end{equation}
(since $\rho_{k}$ was defined to be $\mu_{\sigma\left(  k\right)  }%
-\sigma\left(  k\right)  +k$).

Recall that the $k$-th row of the diagram $P\left(  \sigma\right)  $ contains
$\mu_{\sigma\left(  k\right)  }-\sigma\left(  k\right)  +k$ boxes for each
$k\in\left[  n\right]  $ (by the definition of $P\left(  \sigma\right)  $). In
other words, the $k$-th row of the diagram $P\left(  \sigma\right)  $ contains
$\rho_{k}$ boxes for each $k\in\left[  n\right]  $ (since $\rho_{k}%
=\mu_{\sigma\left(  k\right)  }-\sigma\left(  k\right)  +k$). In particular,
the $i$-th row of the diagram $P\left(  \sigma\right)  $ contains $\rho_{i}$
boxes. Since this $i$-th row contains at least one box (namely, the bottommost
leftmost failure $\left(
i,j\right)  $), we thus obtain $\rho_{i}\geq1$.

Claim 3 yields $\rho_{i-1}^{\prime}=\rho_{i}-1\geq0$ (since $\rho_{i}\geq1$).
Claim 2 yields $\rho_{i}^{\prime}=\rho_{i-1}+1\geq\rho_{i-1}\geq0$ (by
(\ref{pf.lem.flagJT.flip1.rhoigeq0})). Claim 1 yields that every $k\in\left[
n\right]  \setminus\left\{  i,i-1\right\}  $ satisfies $\rho_{k}^{\prime}%
=\rho_{k}\geq0$ (by (\ref{pf.lem.flagJT.flip1.rhoigeq0})). Combining the
preceding three sentences, we obtain
\begin{equation}
\rho_{k}^{\prime}\geq0\ \ \ \ \ \ \ \ \ \ \text{for each }k\in\left[
n\right]  . \label{pf.lem.flagJT.flip1.rhoprigeq0}%
\end{equation}
In other words, $\mu_{\sigma^{\prime}\left(  k\right)  }-\sigma^{\prime
}\left(  k\right)  +k\geq0$ for each $k\in\left[  n\right]  $ (since $\rho
_{k}^{\prime}$ was defined to be $\mu_{\sigma^{\prime}\left(  k\right)
}-\sigma^{\prime}\left(  k\right)  +k$). In other words, the permutation
$\sigma^{\prime}$ is legitimate (by the definition of \textquotedblleft
legitimate\textquotedblright). Thus, the diagram $P\left(  \sigma^{\prime
}\right)  $ is well-defined. By its definition, this diagram $P\left(
\sigma^{\prime}\right)  $ consists of the boxes $\left(  k,\ell\right)  $ with
$k\in\left[  n\right]  $ and $\ell\leq\mu_{\sigma^{\prime}\left(  k\right)
}-\sigma^{\prime}\left(  k\right)  +k$. In other words, the diagram $P\left(
\sigma^{\prime}\right)  $ consists of the boxes $\left(  k,\ell\right)  $ with
$k\in\left[  n\right]  $ and $\ell\leq\rho_{k}^{\prime}$ (since $\rho
_{k}^{\prime}$ was defined to be $\mu_{\sigma^{\prime}\left(  k\right)
}-\sigma^{\prime}\left(  k\right)  +k$).

Next we claim the following:

\begin{statement}
\textit{Claim 4:} We have $\rho_{i}\geq j$ and $\rho_{i-1}\geq j-1$.
\end{statement}

\begin{proof}
[Proof of Claim 4.]The rows of the diagram $P\left(  \sigma\right)  $ are
left-aligned. Thus, from $\left(  i,j\right)  \in P\left(  \sigma\right)  $,
it follows that the $j$ boxes $\left(  i,1\right)  ,\ \left(  i,2\right)
,\ \ldots,\ \left(  i,j\right)  $ all belong to $P\left(  \sigma\right)  $.
Hence, the $i$-th row of $P\left(  \sigma\right)  $ has at least $j$ boxes. In
other words, $\rho_{i}\geq j$ (since the $i$-th row of the diagram $P\left(
\sigma\right)  $ contains $\rho_{i}$ boxes).

It remains to prove that $\rho_{i-1}\geq j-1$. Indeed, assume the contrary.
Thus, $\rho_{i-1}<j-1$. Hence, $j-1>\rho_{i-1}\geq0$ (by
(\ref{pf.lem.flagJT.flip1.rhoigeq0})), so that $j-1\geq1$. Therefore, $\left(
i,j-1\right)  $ is one of the $j$ boxes $\left(  i,1\right)  ,\ \left(
i,2\right)  ,\ \ldots,\ \left(  i,j\right)  $. Since the latter $j$ boxes all
belong to $P\left(  \sigma\right)  $, we thus conclude that $\left(
i,j-1\right)  \in P\left(  \sigma\right)  $.

If we had $\left(  i-1,j-1\right)  \notin P\left(  \sigma\right)  $, then the
box $\left(  i,j-1\right)  $ would be an outer failure of $\left(
\sigma,T\right)  $ (because $i>1$), which would contradict the fact that
$\left(  i,j\right)  $ is a \textbf{leftmost} failure of $\left(
\sigma,T\right)  $ (since the failure $\left(  i,j-1\right)  $ would lie
further west than $\left(  i,j\right)  $). Thus, we cannot have $\left(
i-1,j-1\right)  \notin P\left(  \sigma\right)  $. In other words, we must have
$\left(  i-1,j-1\right)  \in P\left(  \sigma\right)  $.

The rows of $P\left(  \sigma\right)  $ are left-aligned. Thus, from $\left(
i-1,j-1\right)  \in P\left(  \sigma\right)  $, it follows that the $j-1$ boxes
$\left(  i-1,1\right)  ,\ \left(  i-1,2\right)  ,\ \ldots,\ \left(
i-1,j-1\right)  $ all belong to $P\left(  \sigma\right)  $. Hence, the
$\left(  i-1\right)  $-st row of $P\left(  \sigma\right)  $ has at least $j-1$
boxes. In other words, $\rho_{i-1}\geq j-1$ (since the $\left(  i-1\right)
$-st row of the diagram $P\left(  \sigma\right)  $ contains $\rho_{i-1}$ boxes
(because the $k$-th row of the diagram $P\left(  \sigma\right)  $ contains
$\rho_{k}$ boxes for each $k\in\left[  n\right]  $)). Thus, the proof of Claim
4 is complete.
\end{proof}

Let us now recall how $T^{\prime}$ was defined: Namely, we obtain $T^{\prime}$
from $T$ by swapping the top floor of $c$ with the bottom floor of $c$ (see
Definition \ref{def.flagJT.flip} for the definitions of \textquotedblleft top
floor\textquotedblright\ and \textquotedblleft bottom floor\textquotedblright%
). Thus, the entries of $T^{\prime}$ are precisely the entries of $T$, but not
all in the same positions; namely, the entries of the top floor have moved one
unit southeast whereas the entries of the bottom floor have moved one unit
northwest. We shall give names to these entries according to how they move:

\begin{itemize}
\item The entries in the top floor of $c$ will be called \emph{falling
entries}, since they move down (more precisely, southeast) to their new
positions in $T^{\prime}$.

\item The entries in the bottom floor of $c$ will be called \emph{rising
entries}, since they move up (more precisely, northwest) to their new
positions in $T^{\prime}$.

\item All other entries of $T$ will be called \emph{staying entries}, as they
keep their positions in $T^{\prime}$.
\end{itemize}

(Strictly speaking, it is the boxes, not the entries, that we should be
naming, but we hope that this language is clear enough.) \medskip

From the definition of $T^{\prime}$, it is not immediately clear that
$T^{\prime}$ is a $\sigma^{\prime}$-array, or even that the entries of
$T^{\prime}$ appear contiguously in the rows (meaning that the rows have no \textquotedblleft
holes\textquotedblright, i.e., boxes that don't contain any entries). Let us
prove this step by step, starting with some properties of $T$:

\begin{statement}
\textit{Claim 5:} The $i$-th row of $T$ contains exactly $j$ staying entries
and $\rho_{i}-j$ rising entries.
\end{statement}

\begin{proof}
[Proof of Claim 5.]We saw in the proof of Claim 4 that the $j$ boxes $\left(
i,1\right)  ,\ \left(  i,2\right)  ,\ \ldots,\ \left(  i,j\right)  $ all
belong to $P\left(  \sigma\right)  $. Hence, the $i$-th row of $T$ contains
$j$ staying entries (namely, the entries in these $j$ boxes). The remaining
entries in the $i$-th row of $T$ are rising entries (since the bottom floor of
the failure $c=\left(  i,j\right)  $ begins with the next box $\left(
i,j+1\right)  $). How many are there?

Recall that the $i$-th row of the diagram $P\left(  \sigma\right)  $ contains
$\rho_{i}$ boxes. Thus, the $i$-th row of $T$ contains $\rho_{i}$ entries.
Since it contains $j$ staying entries, we thus conclude that it contains
$\rho_{i}-j$ rising entries. This completes the proof of Claim 5.
\end{proof}

\begin{statement}
\textit{Claim 6:} The $\left(  i-1\right)  $-st row of $T$ contains exactly
$j-1$ staying entries and $\rho_{i-1}-\left(  j-1\right)  $ falling entries.
\end{statement}

\begin{proof}
[Proof of Claim 6.]Recall that the $k$-th row of the diagram $P\left(
\sigma\right)  $ contains $\rho_{k}$ boxes for each $k\in\left[  n\right]  $.
Hence, the $\left(  i-1\right)  $-st row of $P\left(  \sigma\right)  $
contains $\rho_{i-1}$ boxes. These boxes are $\left(  i-1,1\right)  ,\ \left(
i-1,2\right)  ,\ \ldots,\ \left(  i-1,\rho_{i-1}\right)  $. Since Claim 4
yields $\rho_{i-1}\geq j-1$, we know that the first $j-1$ of these boxes are
$\left(  i-1,1\right)  ,\ \left(  i-1,2\right)  ,\ \ldots,\ \left(
i-1,j-1\right)  $. Thus, the $\left(  i-1\right)  $-st row of $T$ contains
$j-1$ staying entries (namely, the entries in these $j-1$ boxes). Since it
contains $\rho_{i-1}$ entries in total (because the $\left(  i-1\right)  $-st
row of $P\left(  \sigma\right)  $ contains $\rho_{i-1}$ boxes), it follows
that the remaining $\rho_{i-1}-\left(  j-1\right)  $ entries in this row are
falling entries (since the top floor of the failure $c=\left(  i,j\right)  $
begins with the next box $\left(  i-1,j\right)  $). This completes the proof
of Claim 6.
\end{proof}

Claim 5 shows that the rightmost staying entry in the $i$-th row of $T$ is
$T\left(  i,j\right)  $. Claim 6 shows that the rightmost staying entry in the
$\left(  i-1\right)  $-st row of $T$ (if such an entry exists at all) is
$T\left(  i-1,j-1\right)  $.

\begin{statement}
\textit{Claim 7:} The entries in the $i$-th row of $T^{\prime}$ occupy
precisely the boxes $\left(  i,1\right)  ,\ \left(  i,2\right)  ,\ \ldots
,\ \left(  i,\rho_{i}^{\prime}\right)  $.
\end{statement}

\begin{proof}
[Proof of Claim 7.]The $i$-th row of $T^{\prime}$ has two kinds of entries:
the staying entries of the $i$-th row of $T$ (which remain in their places in
$T^{\prime}$), and the falling entries of the $\left(  i-1\right)  $-st row of
$T$ (which are moved to the $i$-th row by the flip operation). There are $j$
of the former (by Claim 5) and $\rho_{i-1}-\left(  j-1\right)  $ of the latter
(by Claim 6). Moreover, the former occupy the boxes $\left(  i,1\right)
,\ \left(  i,2\right)  ,\ \ldots,\ \left(  i,j\right)  $ in $T^{\prime}$,
whereas the latter occupy the boxes $\left(  i,j+1\right)  ,\ \left(
i,j+2\right)  ,\ \ldots,\ \left(  i,j+\rho_{i-1}-\left(  j-1\right)  \right)
$ in $T^{\prime}$ (by the definition of the flip operation, and because there
are $\rho_{i-1}-\left(  j-1\right)  $ of them). Thus, altogether, there are
\[
j+\left(  \rho_{i-1}-\left(  j-1\right)  \right)  =\rho_{i-1}+1=\rho
_{i}^{\prime}\ \ \ \ \ \ \ \ \ \ \left(  \text{by Claim 2}\right)
\]
of these entries, and they occupy the boxes%
\[
\underbrace{\left(  i,1\right)  ,\ \left(  i,2\right)  ,\ \ldots,\ \left(
i,j\right)  }_{\text{staying entries}},\ \underbrace{\left(  i,j+1\right)
,\ \left(  i,j+2\right)  ,\ \ldots,\ \left(  i,j+\rho_{i-1}-\left(
j-1\right)  \right)  }_{\text{falling entries}},
\]
i.e., the boxes%
\[
\underbrace{\left(  i,1\right)  ,\ \left(  i,2\right)  ,\ \ldots,\ \left(
i,j\right)  }_{\text{staying entries}},\ \underbrace{\left(  i,j+1\right)
,\ \left(  i,j+2\right)  ,\ \ldots,\ \left(  i,\rho_{i}^{\prime}\right)
}_{\text{falling entries}}%
\]
(since $j+\left(  \rho_{i-1}-\left(  j-1\right)  \right)  =\rho_{i}^{\prime}%
$). This proves Claim 7.
\end{proof}

\begin{statement}
\textit{Claim 8:} The entries in the $\left(  i-1\right)  $-st row of
$T^{\prime}$ occupy precisely the boxes $\left(  i-1,1\right)  ,\ \left(
i-1,2\right)  ,\ \ldots,\ \left(  i-1,\rho_{i-1}^{\prime}\right)  $.
\end{statement}

\begin{proof}
[Proof of Claim 8.]The $\left(  i-1\right)  $-st row of $T^{\prime}$ has two
kinds of entries: the staying entries of the $\left(  i-1\right)  $-st row of
$T$ (which remain in their places in $T^{\prime}$), and the rising entries of
the $i$-th row of $T$ (which are moved to the $\left(  i-1\right)  $-st row by
the flip operation). There are $j-1$ of the former (by Claim 6) and $\rho
_{i}-j$ of the latter (by Claim 5). Moreover, the former occupy the boxes
$\left(  i-1,1\right)  ,\ \left(  i-1,2\right)  ,\ \ldots,\ \left(
i-1,j-1\right)  $ in $T^{\prime}$, whereas the latter occupy the boxes
$\left(  i-1,j\right)  ,\ \left(  i-1,j+1\right)  ,\ \ldots,\ \left(
i-1,\left(j-1\right)+\left(  \rho_{i}-j\right)  \right)  $ in $T^{\prime}$ (by the
definition of the flip operation, and because there are $\rho_{i}-j$ of them).
Thus, altogether, there are
\[
\left(  j-1\right)  +\left(  \rho_{i}-j\right)  =\rho_{i}-1=\rho_{i-1}%
^{\prime}\ \ \ \ \ \ \ \ \ \ \left(  \text{by Claim 3}\right)
\]
of these entries, and they occupy the boxes%
\[
\underbrace{\left(  i-1,1\right)  ,\ \left(  i-1,2\right)  ,\ \ldots,\ \left(
i-1,j-1\right)  }_{\text{staying entries}},\ \underbrace{\left(  i-1,j\right)
,\ \left(  i-1,j+1\right)  ,\ \ldots,\ \left(  i-1,\left(j-1\right)+\left(  \rho_{i}-j\right)  \right)  }_{\text{rising entries}},
\]
i.e., the boxes%
\[
\underbrace{\left(  i-1,1\right)  ,\ \left(  i-1,2\right)  ,\ \ldots,\ \left(
i-1,j-1\right)  }_{\text{staying entries}},\ \underbrace{\left(  i-1,j\right)
,\ \left(  i-1,j+1\right)  ,\ \ldots,\ \left(  i-1,\rho_{i-1}^{\prime}\right)
}_{\text{rising entries}}%
\]
(since $\left(  j-1\right)  +\left(  \rho_{i}-j\right)  =\rho_{i-1}^{\prime}%
$). This proves Claim 8.
\end{proof}

\begin{statement}
\textit{Claim 9:} For each $k\in\left[  n\right]  $, the entries in the $k$-th
row of $T^{\prime}$ occupy precisely the boxes $\left(  k,1\right)  ,\ \left(
k,2\right)  ,\ \ldots,\ \left(  k,\rho_{k}^{\prime}\right)  $.
\end{statement}

\begin{proof}
[Proof of Claim 9.]Let $k\in\left[  n\right]  $. We must prove that the
entries in the $k$-th row of $T^{\prime}$ occupy precisely the boxes $\left(
k,1\right)  ,\ \left(  k,2\right)  ,\ \ldots,\ \left(  k,\rho_{k}^{\prime
}\right)  $.

For $k=i$, this follows from Claim 7. For $k=i-1$, this follows from Claim 8.
Thus, we can WLOG assume that $k$ equals neither $i$ nor $i-1$. Assume this.
Hence, the $k$-th row of $T^{\prime}$ has the same entries (in the same
positions) as the $k$-th row of $T$ (because the flip operation affects only
the $i$-th and $\left(  i-1\right)  $-st rows). Moreover, $k\in\left[
n\right]  \setminus\left\{  i,i-1\right\}  $ (since $k$ equals neither $i$ nor
$i-1$) and thus $\rho_{k}^{\prime}=\rho_{k}$ (by
Claim 1).

Recall that the $k$-th row of the diagram $P\left(  \sigma\right)  $ contains
$\rho_{k}$ boxes, and these boxes are $\left(  k,1\right)  ,\ \left(
k,2\right)  ,\ \ldots,\ \left(  k,\rho_{k}\right)  $ (since the diagram
$P\left(  \sigma\right)  $ is left-aligned). Thus, the $k$-th row of $T$ has
$\rho_{k}$ entries, and these entries occupy the boxes $\left(  k,1\right)
,\ \left(  k,2\right)  ,\ \ldots,\ \left(  k,\rho_{k}\right)  $ (since $T$ is
a filling of $P\left(  \sigma\right)  $). Therefore, the same is true for
$T^{\prime}$ instead of $T$ (since the $k$-th row of $T^{\prime}$ has the same
entries (in the same positions) as the $k$-th row of $T$). In other words, the
entries in the $k$-th row of $T^{\prime}$ occupy precisely the boxes $\left(
k,1\right)  ,\ \left(  k,2\right)  ,\ \ldots,\ \left(  k,\rho_{k}^{\prime
}\right)  $. This proves Claim 9.
\end{proof}

Claim 9 shows that the entries of $T^{\prime}$ occupy precisely the boxes
$\left(  k,1\right)  ,\ \left(  k,2\right)  ,\ \ldots,\ \left(  k,\rho
_{k}^{\prime}\right)  $ for all $k\in\left[  n\right]  $. In other words, they
occupy precisely the boxes $\left(  k,\ell\right)  $ with $k\in\left[
n\right]  $ and $\ell\leq\rho_{k}^{\prime}$. Thus, $T^{\prime}$ is a filling
of the diagram $P\left(  \sigma^{\prime}\right)  $ (since the diagram
$P\left(  \sigma^{\prime}\right)  $ consists of the boxes $\left(
k,\ell\right)  $ with $k\in\left[  n\right]  $ and $\ell\leq\rho_{k}^{\prime}$).

Next, we shall show that the entries of $T^{\prime}$ weakly increase
left-to-right along each row. Again, we prove this via several simple claims:

\begin{statement}
\textit{Claim 10:}
If $j-1 \geq 1$, then $\tup{i-1, j-1} \in P\tup{\sigma}$ and $\tup{i, j-1} \in P\tup{\sigma}$ and $T\tup{i-1, j-1} < T\tup{i, j-1}$.
\end{statement}

\begin{proof}[Proof of Claim 10.]
Assume that $j-1 \geq 1$.

Claim 6 shows that the $\left(  i-1\right)  $-st row of $T$ contains exactly
$j-1$ staying entries. Hence, in particular, it has length $\geq j-1$. Thus,
$\left(  i-1,j-1\right)  \in P\left(  \sigma\right)  $ (since $j-1\geq1$).

Also, from $\left(  i,j\right)  \in P\left(  \sigma\right)  $ and $j-1<j$, we
obtain $\left(  i,j-1\right)  \in P\left(  \sigma\right)  $ (again since
$j-1\geq1$).

Recall that $\left(  i,j\right)  $ is a \textbf{leftmost} failure of $\left(
\sigma,T\right)  $.
Thus, $\tup{\sigma, T}$ has no failures that lie further west than $\tup{i, j}$.
Hence, the box $\left(  i,j-1\right)  $ cannot be a
failure of $\left(  \sigma,T\right)  $ (since it lies further west than
$\left(  i,j\right)  $).
In particular, this box $\left(  i,j-1\right)  $
cannot be an inner failure of $\left(  \sigma,T\right)  $. In other words, we
cannot have $T\left(  i-1,j-1\right)  \geq T\left(  i,j-1\right)  $ (by the
definition of an inner failure).
Hence, we must have $T\left(  i-1,j-1\right)
<T\left(  i,j-1\right)  $.
This finishes the proof of Claim 10.
\end{proof}

\begin{statement}
\textit{Claim 11:} The entries in the $i$-th row of $T^{\prime}$ weakly
increase left-to-right.
\end{statement}

\begin{proof}
[Proof of Claim 11.]The $i$-th row of $T^{\prime}$ has two kinds of entries:
the staying entries of the $i$-th row of $T$ (which remain in their places in
$T^{\prime}$), and the falling entries of the $\left(  i-1\right)  $-st row of
$T$ (which are moved to the $i$-th row by the flip operation). Of course, the
staying entries all lie to the left of the falling entries.

But $T$ is a $\sigma$-array, and thus the entries of $T$ weakly increase
left-to-right along each row. Thus, in particular, the entries in the $i$-th
row of $T$ weakly increase left-to-right, and so do the entries in the
$\left(  i-1\right)  $-st row of $T$.

Therefore, the staying entries in the $i$-th row of $T^{\prime}$ weakly
increase left-to-right (since they are copied unchanged from $T$), and the
falling entries in the $i$-th row of $T^{\prime}$ also weakly increase
left-to-right (since they have been moved from the $\left(  i-1\right)  $-st
row of $T$, without changing their order). If we can furthermore show
that the rightmost staying entry in this row is $\leq$ to the leftmost falling
entry\footnote{We are implicitly assuming that there \textbf{is} a rightmost
staying entry and there \textbf{is} a leftmost falling entry. But this is
unproblematic, because if one of these entries does not exist, then Claim 11
follows immediately.}, then we will be able to conclude from this that
\textbf{all} entries in the $i$-th row of $T^{\prime}$ weakly increase
left-to-right. Thus, we only need to show that the rightmost staying entry in
the $i$-th row of $T^{\prime}$ is $\leq$ to the leftmost falling entry.

In other words, we need to show that $T\left(  i,j\right)  \leq T\left(
i-1,j\right)  $ (since the rightmost staying entry in the $i$-th row of
$T^{\prime}$ is $T\left(  i,j\right)  $\ \ \ \ \footnote{because the rightmost
staying entry in the $i$-th row of $T$ is $T\left(  i,j\right)  $}, whereas
the leftmost falling entry is $T\left(  i-1,j\right)  $). But this is easy:
Since $T\left(  i-1,j\right)  $ is well-defined in the first place, we must
have $\left(  i-1,j\right)  \in P\left(  \sigma\right)  $, and therefore the
failure $c=\left(  i,j\right)  $ of $\left(  \sigma,T\right)  $ cannot be an
outer failure. Hence, $c$ must be an inner failure (since $c$ is a failure).
Thus, $T\left(  i-1,j\right)  \geq T\left(  i,j\right)  $ (by the definition
of an inner failure). In other words, $T\left(  i,j\right)  \leq T\left(
i-1,j\right)  $. As we explained, this completes the proof of Claim 11.
\end{proof}

\begin{statement}
\textit{Claim 12:} The entries in the $\left(  i-1\right)  $-st row of
$T^{\prime}$ weakly increase left-to-right.
\end{statement}

\begin{proof}
[Proof of Claim 12.]The $\left(i-1\right)$-st row of $T^{\prime}$ has two kinds of entries:
the staying entries of the $\left(  i-1\right)  $-st row of $T$ (which remain
in their places in $T^{\prime}$), and the rising entries of the $i$-th row of
$T$ (which are moved to the $\left(  i-1\right)  $-st row by the flip
operation). Of course, the staying entries all lie to the left of the rising entries.

But $T$ is a $\sigma$-array, and thus the entries of $T$ weakly increase
left-to-right along each row. Thus, in particular, the entries in the $i$-th
row of $T$ weakly increase left-to-right, and so do the entries in the
$\left(  i-1\right)  $-st row of $T$.

Therefore, the staying entries in the $\left(  i-1\right)  $-st row of
$T^{\prime}$ weakly increase left-to-right (since they are copied unchanged
from $T$), and the rising entries in the $\left(  i-1\right)  $-st row of
$T^{\prime}$ also weakly increase left-to-right (since they have been moved
from the $i$-th row of $T$, without changing their order). If we can
furthermore show that the rightmost staying entry in this row is $\leq$ to the
leftmost rising entry\footnote{We are implicitly assuming that there
\textbf{is} a rightmost staying entry and there \textbf{is} a leftmost rising
entry. But this is unproblematic, because if one of these entries does not
exist, then Claim 12 follows immediately.}, then we will be able to conclude
from this that \textbf{all} entries in the $\left(  i-1\right)  $-st row of
$T^{\prime}$ weakly increase left-to-right. Thus, we only need to show that
the rightmost staying entry in the $\left(  i-1\right)  $-st row of
$T^{\prime}$ is $\leq$ to the leftmost rising entry.

In other words, we need to show that $T\left(  i-1,j-1\right)  \leq T\left(
i,j+1\right)  $ (since the rightmost staying entry in the $\left(  i-1\right)
$-st row of $T^{\prime}$ is $T\left(  i-1,j-1\right)  $%
\ \ \ \ \footnote{because the rightmost staying entry in the $\left(
i-1\right)  $-st row of $T$ is $T\left(  i-1,j-1\right)  $}, whereas the
leftmost rising entry is $T\left(  i,j+1\right)  $). But this is easy:
Since the entry $T\tup{i-1, j-1}$ is well-defined in the first place, we have $j-1 \geq 1$. Thus, Claim 10 yields
$\tup{i-1, j-1} \in P\tup{\sigma}$ and $\tup{i, j-1} \in P\tup{\sigma}$ and $T\tup{i-1, j-1} < T\tup{i, j-1}$.
Furthermore, $T\left(  i,j-1\right)  \leq T\left(  i,j+1\right)  $ (since the
entries in the $i$-th row of $T$ weakly increase left-to-right). Hence,
$T\left(  i-1,j-1\right)  <T\left(  i,j-1\right)  \leq T\left(  i,j+1\right)
$. As we explained, this completes the proof of Claim 12.
\end{proof}

\begin{statement}
\textit{Claim 13:} For each $k\in\left[  n\right]  $, the entries in the
$k$-th row of $T^{\prime}$ weakly increase left-to-right.
\end{statement}

\begin{proof}
[Proof of Claim 13.]Let $k\in\left[  n\right]  $. We must prove that the
entries in the $k$-th row of $T^{\prime}$ weakly increase left-to-right.

If $k=i$, then this follows from Claim 11. If $k=i-1$, then this follows from
Claim 12. Thus, we WLOG assume that $k$ equals neither $i$ nor $i-1$. Hence,
the $k$-th row of $T^{\prime}$ has the same entries (in the same positions) as
the $k$-th row of $T$ (because the flip operation affects only the $i$-th and
$\left(  i-1\right)  $-st rows). But the entries in the $k$-th row of $T$
weakly increase left-to-right (since $T$ is a $\sigma$-array). Hence, the
entries in the $k$-th row of $T^{\prime}$ weakly increase left-to-right (since
the $k$-th row of $T^{\prime}$ has the same entries (in the same positions) as
the $k$-th row of $T$). This proves Claim 13.
\end{proof}

Recall that $T^{\prime}$ is a filling of the diagram $P\left(  \sigma^{\prime
}\right)  $ with positive integers. Moreover, these integers weakly increase
left-to-right along each row (by Claim 13). Hence, $T^{\prime}$ is a
$\sigma^{\prime}$-array (by the definition of a $\sigma^{\prime}$-array). In
other words, $\left(  \sigma^{\prime},T^{\prime}\right)  $ is a twisted array.

Next, we claim:

\begin{statement}
\textit{Claim 14:} The box $\left(  i,j\right)  $ is a failure of the twisted
array $\left(  \sigma^{\prime},T^{\prime}\right)  $.
\end{statement}

\begin{proof}
[Proof of Claim 14.]We know that $i>1$ and $\tup{i,j} \in P\tup{\sigma^{\prime}}$ (since the entry $T\tup{i,j}$ of $T$ does not change under the flip operation).
If $\left(  i-1,j\right)  \notin
P\left(  \sigma^{\prime}\right)  $, then the box $\left(  i,j\right)  $ is therefore an
outer failure of $\left(  \sigma^{\prime},T^{\prime}\right)  $, and thus Claim
14 is proved. Hence, for the rest of this proof, we WLOG assume that $\left(
i-1,j\right)  \in P\left(  \sigma^{\prime}\right)  $. Therefore, $T^{\prime
}\left(  i-1,j\right)  $ is a rising entry and equals $T\left(  i,j+1\right)
$ (by the construction of $T^{\prime}$). Meanwhile, $T^{\prime}\left(
i,j\right)  $ is a staying entry and thus equals $T\left(  i,j\right)  $.
However, the entries in the $i$-th row of $T$ weakly increase left-to-right
(since $T$ is a $\sigma$-array). Thus, $T\left(  i,j\right)  \leq T\left(
i,j+1\right)  $. In other words, $T^{\prime}\left(  i,j\right)  \leq
T^{\prime}\left(  i-1,j\right)  $ (since $T^{\prime}\left(  i,j\right)  $
equals $T\left(  i,j\right)  $ whereas $T^{\prime}\left(  i-1,j\right)  $
equals $T\left(  i,j+1\right)  $). In other words, $T^{\prime}\left(
i-1,j\right)  \geq T^{\prime}\left(  i,j\right)  $. Hence, the box $\left(
i,j\right)  $ is an inner failure of $\left(  \sigma^{\prime},T^{\prime
}\right)  $ (by the definition of an inner failure). This proves Claim 14.
\end{proof}

Claim 14 shows that the twisted array $\left(  \sigma^{\prime},T^{\prime
}\right)  $ has a failure, i.e., is failing. Thus, Lemma
\ref{lem.flagJT.flip1} \textbf{(a)} is proved. \medskip

\textbf{(b)} We make two further claims:

\begin{statement}
\textit{Claim 15:} The twisted array $\left(  \sigma^{\prime},T^{\prime
}\right)  $ has no failures in columns $1,2,\ldots,j-1$.
\end{statement}

\begin{proof}[Proof of Claim 15.]
We can restate the definition of a failure in terms of the
entries: A failure of a twisted array $\left(  \tau,S\right)  $ is either an
inner failure (i.e., an entry of $S$ that is smaller than or equal to its
northern neighbor\footnote{Formally, of course, a failure is not an entry but
rather the box containing this entry; but we ignore this distinction to keep
the writing clearer.}) or an outer failure (i.e., an entry of $S$ that has no
northern neighbor even though it is not in the first row). Thus, whether or
not a given box $\left(  p,q\right)  $ is a failure of a twisted array
$\left(  \tau,S\right)  $ depends only on the entries in the boxes $\left(
p,q\right)  $ and $\left(  p-1,q\right)  $ of $S$. In particular, it depends
only on the entries in the $q$-th column of $S$.
Therefore, for any given integer $q \geq 1$, the failures of a twisted
array $\tup{\tau, S}$ in the $q$-th column
depend only on the entries in the $q$-th column of $S$.

Hence, if a $\tau$-array $S$ and a $\tau'$-array $S'$ have
the same entries\footnote{``The same entries'' means ``the same entries in
the same position''.} in the $q$-th column (for a given $q \geq 1$), then
the two twisted arrays $\tup{\tau, S}$ and $\tup{\tau', S'}$ have the same
failures in the $q$-th column.

Now, recall that $T^{\prime}$ is obtained from $T$ by swapping the top floor
of $c=\left(  i,j\right)  $ with the bottom floor. This swap operation does
not affect the columns $1,2,\ldots,j-1$ (since all the entries being swapped
belong to columns $j,j+1,j+2,\ldots$).
Thus, in the columns $1,2,\ldots,j-1$,
the array $T^{\prime}$ has the same entries as $T$ (and in the same positions
as well).
Therefore, in the columns $1,2,\ldots,j-1$, the twisted array
$\left(  \sigma^{\prime},T^{\prime}\right)  $ has the same failures as
$\left(  \sigma,T\right)  $ (because
if a $\tau$-array $S$ and a $\tau'$-array $S'$ have the same entries
in the $q$-th column (for a given $q \geq 1$), then
the two twisted arrays $\tup{\tau, S}$ and $\tup{\tau', S'}$ have the same
failures in the $q$-th column).
Since $\left(
\sigma,T\right)  $ has no failures in these columns (because $\left(
i,j\right)  $ is a \textbf{leftmost} failure of $\left(  \sigma,T\right)  $),
we thus conclude that $\left(  \sigma^{\prime},T^{\prime}\right)  $ has no
failures in these columns either. This proves Claim 15.
\end{proof}

\begin{statement}
\textit{Claim 16:} The twisted array $\left(  \sigma^{\prime},T^{\prime
}\right)  $ has no failures in column $j$ lying south of the box $\left(
i,j\right)  $.
\end{statement}

\begin{proof}[Proof of Claim 16.]
Recall that $\left(
i,j\right)  $ is the \textbf{bottommost} leftmost failure of $\left(
\sigma,T\right)  $.
Thus, the twisted array $\left(  \sigma,T\right)  $ has no failures in
column $j$ lying south of the box $\left(  i,j\right)  $.
In other words, the twisted array $\tup{\sigma, T}$ has no failures in
the boxes \newline $\tup{i+1,j},\ \tup{i+2,j},\ \tup{i+3,j},\ \ldots$.
(Of course, a box that does not belong to $P\tup{\sigma}$ is not
considered as a failure of $\tup{\sigma, T}$.)

Recall the following fact (which we proved in our proof of
Claim 15): Whether or not a given box $\left(  p,q\right)  $ is a failure of a
twisted array $\left(  \tau,S\right)  $ depends only on the entries in the
boxes $\left(  p,q\right)  $ and $\left(  p-1,q\right)  $ of $S$.
Therefore, the failures of a twisted array $\tup{\tau, S}$ in the boxes
$\tup{i+1,j},\ \tup{i+2,j},\ \tup{i+3,j},\ \ldots$
depend only on the entries of $S$ in the boxes
$\tup{i,j},\ \tup{i+1,j},\ \tup{i+2,j},\ \ldots$.

Hence, if a $\tau$-array $S$ and a $\tau'$-array $S'$ have
the same entries\footnote{``The same entries'' means ``the same entries in
the same position''.} in the boxes
$\tup{i,j},\ \tup{i+1,j},\ \tup{i+2,j},\ \ldots$, then
the two twisted arrays $\tup{\tau, S}$ and $\tup{\tau', S'}$ have the same
failures in the boxes $\tup{i+1,j},\ \tup{i+2,j},\ \tup{i+3,j},\ \ldots$.

Now, recall that $T^{\prime}$ is obtained from $T$ by swapping the top floor
of $c=\left(  i,j\right)  $ with the bottom floor. This swap operation only
affects the entries in the boxes%
\[
\left(  i-1,j\right)  ,\ \ \left(  i-1,j+1\right)  ,\ \ \left(
i-1,j+2\right)  ,\ \ \ldots
\]
and in the boxes%
\[
\left(  i,j+1\right)  ,\ \ \left(  i,j+2\right)  ,\ \ \left(  i,j+3\right)
,\ \ \ldots.
\]
Thus, the only box \textbf{in column }$j$ that is affected by this swap
operation is the box $\left(  i-1,j\right)  $. In particular,
the entries in boxes $\tup{i,j},\ \tup{i+1,j},\ \tup{i+2,j},\ \ldots$
are not affected by this swap operation.
Hence, the arrays $T$ and $T'$ have the same entries in these boxes
$\tup{i,j},\ \tup{i+1,j},\ \tup{i+2,j},\ \ldots$.
Consequently, the twisted arrays $\tup{\sigma, T}$ and $\tup{\sigma', T'}$
have the same failures in the boxes
$\tup{i+1,j},\ \tup{i+2,j},\ \tup{i+3,j},\ \ldots$
(because if a $\tau$-array $S$ and a $\tau'$-array $S'$ have
the same entries in the boxes \newline
$\tup{i,j},\ \tup{i+1,j},\ \tup{i+2,j},\ \ldots$, then
the two twisted arrays $\tup{\tau, S}$ and $\tup{\tau', S'}$ have the same
failures in the boxes $\tup{i+1,j},\ \tup{i+2,j},\ \tup{i+3,j},\ \ldots$).
Since we know that the twisted array $\tup{\sigma, T}$ has no failures
in these boxes,
we thus conclude that the twisted array $\tup{\sigma', T'}$ has no failures
in these boxes either.
In other words, $\tup{\sigma', T'}$ has no failures in the $j$-th column
lying south of the box $\tup{i,j}$.
This proves Claim 16.
\end{proof}

The box $c=\left(  i,j\right)  $ is a failure of the twisted array $\left(
\sigma^{\prime},T^{\prime}\right)  $ (by Claim 14). It is thus a leftmost
failure of $\left(  \sigma^{\prime},T^{\prime}\right)  $ (since Claim 15 shows
that $\left(  \sigma^{\prime},T^{\prime}\right)  $ has no failures in any
column to its west), and therefore the bottommost leftmost failure of $\left(
\sigma^{\prime},T^{\prime}\right)  $ (since Claim 16 shows that $\left(
\sigma^{\prime},T^{\prime}\right)  $ has no failures in the $j$-th column to
the south of $\left(  i,j\right)  $). This proves Lemma \ref{lem.flagJT.flip1}
\textbf{(b)}. \medskip

\textbf{(c)} Lemma \ref{lem.flagJT.flip1} \textbf{(b)} shows that the box $c$
is again the bottommost leftmost failure of $\left(  \sigma^{\prime}%
,T^{\prime}\right)  $. Hence, the twisted array $\operatorname*{flip}\left(
\sigma^{\prime},T^{\prime}\right)  $ is obtained from $\left(  \sigma^{\prime
},T^{\prime}\right)  $ by exchanging the values of $\sigma^{\prime}$ on $i-1$
and $i$ and by swapping the top floor and the bottom floor of $c$. But
$\left(  \sigma^{\prime},T^{\prime}\right)  $ was obtained from $\left(
\sigma,T\right)  $ in the exact same way -- i.e., by exchanging the values of
$\sigma$ on $i-1$ and $i$ and by swapping the top floor and the bottom floor
of $c$. The nature of these operations is such that doing them twice in
succession returns the $\sigma$ and the $T$ to their original values (because
exchanging the values of $\sigma$ on $i-1$ and $i$ twice in succession returns
the original $\sigma$, whereas swapping the top floor and the bottom floor
twice in succession returns all entries to their original places). Thus,
$\operatorname*{flip}\left(  \sigma^{\prime},T^{\prime}\right)  =\left(
\sigma,T\right)  $. This proves Lemma \ref{lem.flagJT.flip1} \textbf{(c)}.
\medskip

\textbf{(d)} The definition of $\sigma^{\prime}$ yields $\sigma^{\prime
}=\sigma\circ s_{i-1}$. But it is well-known that the simple transposition
$s_{i-1}$ has sign $\left(  -1\right)  ^{s_{i-1}}=-1$. Furthermore, it is
well-known that any two permutations $\alpha,\beta\in S_{n}$ satisfy $\left(
-1\right)  ^{\alpha\circ\beta}=\left(  -1\right)  ^{\alpha}\cdot\left(
-1\right)  ^{\beta}$. Thus, $\left(  -1\right)  ^{\sigma\circ s_{i-1}}=\left(
-1\right)  ^{\sigma}\cdot\underbrace{\left(  -1\right)  ^{s_{i-1}}}%
_{=-1}=\left(  -1\right)  ^{\sigma}\cdot\left(  -1\right)  =-\left(
-1\right)  ^{\sigma}$.
In other words,
$\left(-1\right)^{\sigma^{\prime}} = -\left(-1\right)^{\sigma}$
(since $\sigma^{\prime}=\sigma\circ s_{i-1}$).
Thus, Lemma \ref{lem.flagJT.flip1} \textbf{(d)} is
proved. \medskip

\textbf{(e)} We shall abuse notation somewhat and speak of entries when we
really mean boxes.

If $e$ is the entry of $T$ in a given box $\left(  p,q\right)  $, then we let
$d\left(  e\right)  $ denote the number $q-p$ (so that the entry $e$ belongs
to the $d\left(  e\right)  $-th diagonal). Strictly speaking, this depends on
the box $\left(  p,q\right)  $, not on the entry $e$, but we will pretend that
each entry \textquotedblleft remembers\textquotedblright\ what box it is in.
We shall refer to the number $d\left(  e\right)  $ as the \emph{diagonal
position} of $e$ in $T$.

Now, the definition of the weight $w\left(  T\right)  $ yields%
\begin{equation}
w\left(  T\right)  =\prod_{\left(  p,q\right)  \in P\left(  \sigma\right)
}u_{T\left(  p,q\right)  ,\ q-p}=\prod_{e\text{ is an entry of }%
T}u_{e,\ d\left(  e\right)  } \label{pf.lem.flagJT.flip1.e.1}%
\end{equation}
(where we still pretend that each entry \textquotedblleft
remembers\textquotedblright\ its box, so that equal entries in different
positions create different factors of the product). Similarly,%
\begin{equation}
w\left(  T^{\prime}\right)  =\prod_{e\text{ is an entry of }T^{\prime}%
}u_{e,\ d\left(  e\right)  }, \label{pf.lem.flagJT.flip1.e.2}%
\end{equation}
where $d\left(  e\right)  $ now refers to the diagonal position of $e$ in
$T^{\prime}$.

Now, recall that $T^{\prime}$ is obtained from $T$ by swapping the top floor
of $c$ with the bottom floor of $c$. During this swapping operation, some
entries of $T$ get moved (namely, the rising entries are moved by $1$ unit to
the northwest, whereas the falling entries are moved by $1$ unit to the
southeast), but their diagonal positions remain unchanged (since a move by $1$
unit to the northwest or southeast does not change the diagonal position).
Hence, the array $T^{\prime}$ contains the same entries as $T$, possibly in
different positions but in the same diagonal positions. Thus,%
\[
\prod_{e\text{ is an entry of }T^{\prime}}u_{e,\ d\left(  e\right)  }%
=\prod_{e\text{ is an entry of }T}u_{e,\ d\left(  e\right)  }.
\]
In view of (\ref{pf.lem.flagJT.flip1.e.1}) and (\ref{pf.lem.flagJT.flip1.e.2}%
), we can rewrite this as $w\left(  T^{\prime}\right)  =w\left(  T\right)  $.
This proves Lemma \ref{lem.flagJT.flip1} \textbf{(e)}. \medskip

\textbf{(f)} Assume that $\left(  \sigma,T\right)  $ is $\mathbf{b}$-flagged.
We must prove that $\left(  \sigma^{\prime},T^{\prime}\right)  $ is
$\mathbf{b}$-flagged. In other words, we must prove that each entry in the
$k$-th row of $T^{\prime}$ is $\leq b_{\sigma^{\prime}\left(  k\right)  }$ for
each $k\in\left[  n\right]  $ (by the definition of \textquotedblleft%
$\mathbf{b}$-flagged $\sigma$-array\textquotedblright).

We have assumed that $\left(  \sigma,T\right)  $ is $\mathbf{b}$-flagged. In
other words, the $\sigma$-array $T$ is $\mathbf{b}$-flagged. In other words,
\begin{equation}
\text{each entry in the }k\text{-th row of }T\text{ is }\leq b_{\sigma\left(
k\right)  } \label{pf.lem.flagJT.flip1.f.T}%
\end{equation}
for each $k\in\left[  n\right]  $ (by the definition of \textquotedblleft%
$\mathbf{b}$-flagged\textquotedblright).

Again, we proceed by a series of claims:

\begin{statement}
\textit{Claim 17:} All staying entries in the $i$-th row of $T$ are $\leq
b_{\sigma^{\prime}\left(  i\right)  }$.
\end{statement}

\begin{proof}
[Proof of Claim 17.]Let $e$ be a staying entry in the $i$-th row of $T$. We
must prove that $e\leq b_{\sigma^{\prime}\left(  i\right)  }$.

The staying entries in the $i$-th row of $T$ are
$T\tup{i,1}$, $T\tup{i,2}$, $\ldots$, $T\tup{i,j}$.
Hence, $e=T\left(  i,\ell\right)  $ for some $\ell\in\left[  j\right]  $
(since $e$ is a staying entry in the $i$-th row of $T$). Consider this $\ell$.

But $T$ is a $\sigma$-array, and thus the entries of $T$ weakly increase
left-to-right along each row. Hence, $T\left(  i,\ell\right)  \leq T\left(
i,j\right)  $ (since $\ell\leq j$). It thus remains to prove that $T\left(
i,j\right)  \leq b_{\sigma^{\prime}\left(  i\right)  }$ (because once this is
proved, it will follow that $e=T\left(  i,\ell\right)  \leq T\left(
i,j\right)  \leq b_{\sigma^{\prime}\left(  i\right)  }$, which is precisely
our goal). In other words, it remains to prove that $T\left(  i,j\right)  \leq
b_{\sigma\left(  i-1\right)  }$ (since $\sigma^{\prime}\left(  i\right)
=\sigma\left(  i-1\right)  $).

Recall that $c=\left(  i,j\right)  $ is a failure of $\left(  \sigma,T\right)
$. Hence, we are in one of the following two cases:

\textit{Case 1:} The failure $c=\left(  i,j\right)  $ is an inner failure of
$\left(  \sigma,T\right)  $.

\textit{Case 2:} The failure $c=\left(  i,j\right)  $ is an outer failure of
$\left(  \sigma,T\right)  $.

Let us consider Case 1 first. In this case, $c=\left(  i,j\right)  $ is an
inner failure of $\left(  \sigma,T\right)  $. Hence, $T\left(  i-1,j\right)
\geq T\left(  i,j\right)  $ (by the definition of an inner failure), so that
$T\left(  i,j\right)  \leq T\left(  i-1,j\right)  $. But
(\ref{pf.lem.flagJT.flip1.f.T}) shows that each entry in the $\left(
i-1\right)  $-st row of $T$ is $\leq b_{\sigma\left(  i-1\right)  }$. Hence,
in particular, $T\left(  i-1,j\right)  \leq b_{\sigma\left(  i-1\right)  }$
(since $T\left(  i-1,j\right)  $ is an entry in the $\left(  i-1\right)  $-st
row of $T$). Therefore, $T\left(  i,j\right)  \leq T\left(  i-1,j\right)  \leq
b_{\sigma\left(  i-1\right)  }$. Thus we have proved $T\left(  i,j\right)
\leq b_{\sigma\left(  i-1\right)  }$ in Case 1.

Let us now consider Case 2. In this case, $c=\left(  i,j\right)  $ is an outer
failure of $\left(  \sigma,T\right)  $. Thus, $\left(  i-1,j\right)  \notin
P\left(  \sigma\right)  $. Hence, the $\left(  i-1\right)  $-st row of
$P\left(  \sigma\right)  $ has fewer than $j$ boxes (since the rows of
$P\left(  \sigma\right)  $ are left-aligned).

Recall that the $k$-th row of the diagram $P\left(  \sigma\right)  $ contains
$\rho_{k}$ boxes for each $k\in\left[  n\right]  $. Hence, the $\left(
i-1\right)  $-st row of $P\left(  \sigma\right)  $ has $\rho_{i-1}$ boxes.
Thus, $\rho_{i-1}<j$ (since the $\left(  i-1\right)  $-st row of $P\left(
\sigma\right)  $ has fewer than $j$ boxes). But Claim 4 yields $\rho_{i}\geq
j$. Hence, $\rho_{i}\geq j>\rho_{i-1}$ (since $\rho_{i-1}<j$).

Next, we claim that $\sigma\left(  i\right)  \leq\sigma\left(  i-1\right)  $.
Indeed, assume the contrary. Thus, $\sigma\left(  i\right)  >\sigma\left(
i-1\right)  $, so that $\sigma\left(  i\right)  \geq\sigma\left(  i-1\right)
+1$. Now, recall that $\mu_{1}\geq\mu_{2}\geq\mu_{3}\geq\cdots$ (since $\mu$
is a partition). Hence, $\mu_{p}\geq\mu_{q}$ for any two positive integers $p$
and $q$ satisfying $p<q$. Applying this to $p=\sigma\left(  i-1\right)  $ and
$q=\sigma\left(  i\right)  $, we obtain $\mu_{\sigma\left(  i-1\right)  }%
\geq\mu_{\sigma\left(  i\right)  }$ (since $\sigma\left(  i-1\right)
<\sigma\left(  i\right)  $). In other words, $\mu_{\sigma\left(  i\right)
}\leq\mu_{\sigma\left(  i-1\right)  }$. Now, the definition of $\rho_{i}$
yields%
\begin{align*}
\rho_{i} &  =\underbrace{\mu_{\sigma\left(  i\right)  }}_{\leq\mu
_{\sigma\left(  i-1\right)  }}-\underbrace{\sigma\left(  i\right)  }%
_{\geq\sigma\left(  i-1\right)  +1}+\underbrace{i}_{=\left(  i-1\right)  +1}\\
&  \leq\mu_{\sigma\left(  i-1\right)  }-\left(  \sigma\left(  i-1\right)
+1\right)  +\left(  \left(  i-1\right)  +1\right)  \\
&  =\mu_{\sigma\left(  i-1\right)  }-\sigma\left(  i-1\right)  +\left(
i-1\right)  =\rho_{i-1}%
\end{align*}
(since $\rho_{i-1}$ is defined to be $\mu_{\sigma\left(  i-1\right)  }%
-\sigma\left(  i-1\right)  +\left(  i-1\right)  $). But this contradicts
$\rho_{i}>\rho_{i-1}$. This contradiction shows that our assumption was false.

Hence, $\sigma\left(  i\right)  \leq\sigma\left(  i-1\right)  $ is proved. But
the flagging $\mathbf{b}$ is weakly increasing; thus, $b_{1}\leq b_{2}\leq
b_{3}\leq\cdots$. In other words, $b_{p}\leq b_{q}$ for any two positive
integers $p$ and $q$ satisfying $p\leq q$. Applying this to $p=\sigma\left(
i\right)  $ and $q=\sigma\left(  i-1\right)  $, we obtain $b_{\sigma\left(
i\right)  }\leq b_{\sigma\left(  i-1\right)  }$
(since $\sigma\left(  i\right)  \leq\sigma\left(  i-1\right)  $).

But (\ref{pf.lem.flagJT.flip1.f.T}) shows that each entry in the $i$-th row of
$T$ is $\leq b_{\sigma\left(  i\right)  }$. Hence, in particular, $T\left(
i,j\right)  \leq b_{\sigma\left(  i\right)  }$ (since $T\left(  i,j\right)  $
is an entry in the $i$-th row of $T$). Therefore, $T\left(  i,j\right)  \leq
b_{\sigma\left(  i\right)  }\leq b_{\sigma\left(  i-1\right)  }$. Thus we have
proved $T\left(  i,j\right)  \leq b_{\sigma\left(  i-1\right)  }$ in Case 2.

We have now proved $T\left(  i,j\right)  \leq b_{\sigma\left(  i-1\right)  }$
in both Cases 1 and 2. Hence, $T\left(  i,j\right)  \leq b_{\sigma\left(
i-1\right)  }$ always holds. As explained above, this completes the proof of
Claim 17.
\end{proof}

\begin{statement}
\textit{Claim 18:} All falling entries in the $\left(  i-1\right)  $-st row of
$T$ are $\leq b_{\sigma^{\prime}\left(  i\right)  }$.
\end{statement}

\begin{proof}
[Proof of Claim 18.]All entries in the $\left(  i-1\right)  $-st row of $T$
are $\leq b_{\sigma\left(  i-1\right)  }$ (by (\ref{pf.lem.flagJT.flip1.f.T}),
applied to $k=i-1$). In other words, all entries in the $\left(  i-1\right)
$-st row of $T$ are $\leq b_{\sigma^{\prime}\left(  i\right)  }$ (since
$\sigma^{\prime}\left(  i\right)  =\sigma\left(  i-1\right)  $). Hence, in
particular, all falling entries in the $\left(  i-1\right)  $-st row of $T$
are $\leq b_{\sigma^{\prime}\left(  i\right)  }$. This proves Claim 18.
\end{proof}

\begin{statement}
\textit{Claim 19:} All staying entries in the $\left(  i-1\right)  $-st row of
$T$ are $\leq b_{\sigma^{\prime}\left(  i-1\right)  }$.
\end{statement}

\begin{proof}
[Proof of Claim 19.]Let $e$ be a staying entry in the $\left(  i-1\right)
$-st row of $T$. We must prove that $e\leq b_{\sigma^{\prime}\left(
i-1\right)  }$.

The staying entries in the $\tup{i-1}$-st row of $T$ are
$T\tup{i-1,1}$, $T\tup{i-1,2}$, $\ldots$, $T\tup{i-1,j-1}$.
Hence, $e=T\left(  i-1,\ell\right)  $ for some $\ell\in\left[  j-1\right]  $
(since $e$ is a staying entry in the $\tup{i-1}$-st
row of $T$). Consider this $\ell$.

From $\ell\in\left[  j-1\right]  $, we obtain $\ell\leq j-1$ and thus
$j-1\geq\ell\geq1$.
Hence, Claim 10 yields $\tup{i-1, j-1} \in P\tup{\sigma}$ and
$\tup{i, j-1} \in P\tup{\sigma}$ and $T\tup{i-1, j-1} < T\tup{i, j-1}$.

But $T$ is a $\sigma$-array, and thus the entries of $T$ weakly increase
left-to-right along each row. Hence, from $\ell\leq j-1$, we obtain $T\left(
i-1,\ell\right)  \leq T\left(  i-1,j-1\right)  $.

Altogether, we now have $e=T\left(  i-1,\ell\right)  \leq T\left(
i-1,j-1\right)  <T\left(  i,j-1\right)  $.

But (\ref{pf.lem.flagJT.flip1.f.T}) shows that each entry in the $i$-th row of
$T$ is $\leq b_{\sigma\left(  i\right)  }$. Hence, in particular, $T\left(
i,j-1\right)  \leq b_{\sigma\left(  i\right)  }$ (since $T\left(
i,j-1\right)  $ is an entry in the $i$-th row of $T$). Therefore, $e<T\left(
i,j-1\right)  \leq b_{\sigma\left(  i\right)  }$. This rewrites as
$e<b_{\sigma^{\prime}\left(  i-1\right)  }$ (since $\sigma^{\prime}\left(
i-1\right)  =\sigma\left(  i\right)  $). Hence, of course, $e\leq
b_{\sigma^{\prime}\left(  i-1\right)  }$. This completes the proof of Claim 19.
\end{proof}

\begin{statement}
\textit{Claim 20:} All rising entries in the $i$-th row of $T$ are $\leq
b_{\sigma^{\prime}\left(  i-1\right)  }$.
\end{statement}

\begin{proof}
[Proof of Claim 20.]All entries in the $i$-th row of $T$ are $\leq
b_{\sigma\left(  i\right)  }$ (by (\ref{pf.lem.flagJT.flip1.f.T}), applied to
$k=i$). In other words, all entries in the $i$-th row of $T$ are $\leq
b_{\sigma^{\prime}\left(  i-1\right)  }$ (since $\sigma^{\prime}\left(
i-1\right)  =\sigma\left(  i\right)  $). Hence, in particular, all rising
entries in the $i$-th row of $T$ are $\leq b_{\sigma^{\prime}\left(
i-1\right)  }$. This proves Claim 20.
\end{proof}

\begin{statement}
\textit{Claim 21:} Let $k\in\left[  n\right]  $. Then, all entries in the
$k$-th row of $T^{\prime}$ are $\leq b_{\sigma^{\prime}\left(  k\right)  }$.
\end{statement}

\begin{proof}
[Proof of Claim 21.]The $i$-th row of $T^{\prime}$ has two kinds of entries:
the staying entries of the $i$-th row of $T$ (which remain in their places in
$T^{\prime}$), and the falling entries of the $\left(  i-1\right)  $-st row of
$T$ (which are moved to the $i$-th row by the flip operation). Both kinds of
entries are $\leq b_{\sigma^{\prime}\left(  i\right)  }$ (by Claim 17 and
Claim 18, respectively). Hence, all entries in the $i$-th row of $T^{\prime}$
are $\leq b_{\sigma^{\prime}\left(  i\right)  }$. In other words, Claim 21 is
proved for $k=i$.

The $\left(  i-1\right)  $-st row of $T^{\prime}$ has two kinds of entries:
the staying entries of the $\left(  i-1\right)  $-st row of $T$ (which remain
in their places in $T^{\prime}$), and the rising entries of the $i$-th row of
$T$ (which are moved to the $\left(  i-1\right)  $-st row by the flip
operation). Both kinds of entries are $\leq b_{\sigma^{\prime}\left(
i-1\right)  }$ (by Claim 19 and Claim 20, respectively). Hence, all entries in
the $\left(  i-1\right)  $-st row of $T^{\prime}$ are $\leq b_{\sigma^{\prime
}\left(  i-1\right)  }$. In other words, Claim 21 is proved for $k=i-1$.

We have now proved Claim 21 for $k=i$ and for $k=i-1$. Hence, for the rest of
this proof of Claim 21, we WLOG assume that $k$ equals neither $i$ nor
$i-1$. Hence, the entries in the $k$-th row of $T^{\prime}$ are precisely the
entries in the $k$-th row of $T$ (because the flip operation affects only the
$i$-th and $\left(  i-1\right)  $-st rows). But all the latter entries are
$\leq b_{\sigma\left(  k\right)  }$ (by (\ref{pf.lem.flagJT.flip1.f.T})).
Hence, all the former entries are $\leq b_{\sigma\left(  k\right)  }$ as well.
So we have shown that all entries in the $k$-th row of $T^{\prime}$ are $\leq
b_{\sigma\left(  k\right)  }$.

But $k\in\left[  n\right]  \setminus\left\{  i,i-1\right\}  $ (since $k$
equals neither $i$ nor $i-1$). Hence, $\sigma^{\prime}\left(  k\right)
=\sigma\left(  k\right)  $ (by (\ref{pf.lem.flagJT.flip1.siprime})). Recall
that all entries in the $k$-th row of $T^{\prime}$ are $\leq b_{\sigma\left(
k\right)  }$. In other words, all entries in the $k$-th row of $T^{\prime}$
are $\leq b_{\sigma^{\prime}\left(  k\right)  }$ (since $\sigma^{\prime
}\left(  k\right)  =\sigma\left(  k\right)  $). This proves Claim 21.
\end{proof}

Claim 21 shows that the $\sigma^{\prime}$-array $T^{\prime}$ is $\mathbf{b}%
$-flagged (by the definition of \textquotedblleft$\mathbf{b}$%
-flagged $\sigma$-array\textquotedblright). In other words, the twisted array $\left(
\sigma^{\prime},T^{\prime}\right)  $ is $\mathbf{b}$-flagged. This proves
Lemma \ref{lem.flagJT.flip1} \textbf{(f)}.
\end{proof}

The proof of Lemma \ref{lem.flagJT.cancel} uses the flip operation (defined in
Lemma \ref{lem.flagJT.flip1}) to cancel all the unwanted addends (i.e., the
addends corresponding to failing twisted arrays $\left(  \sigma,T\right)  $)
from the sum on the left hand side. This is best formalized using the
following general cancellation rule:

\begin{lemma}
\label{lem.A.cancel}Let $A$ be a finite set. Let $R$ be a ring. For each $a\in
A$, let $m_{a}$ be an integer, and let $r_{a}$ be an element of $R$. Let
$f:A\rightarrow A$ be a bijection. Assume that each $a\in A$ satisfies%
\begin{equation}
m_{f\left(  a\right)  }=-m_{a} \label{eq.lem.A.cancel.m}%
\end{equation}
and%
\begin{equation}
r_{f\left(  a\right)  }=r_{a}. \label{eq.lem.A.cancel.r}%
\end{equation}
Then,%
\[
\sum_{a\in A}m_{a}r_{a}=0.
\]

\end{lemma}

\begin{proof}
[Proof of Lemma \ref{lem.A.cancel}.]The map $f$ is a bijection, thus injective
and surjective.

For each $a\in A$, the number $m_{a}$ is an integer, thus is either $>0$ or
$=0$ or $<0$. Hence, we can split up the sum $\sum_{a\in A}m_{a}r_{a}$ as
follows:%
\begin{align}
\sum_{a\in A}m_{a}r_{a}  &  =\sum_{\substack{a\in A;\\m_{a}>0}}m_{a}r_{a}%
+\sum_{\substack{a\in A;\\m_{a}=0}}\underbrace{m_{a}}_{=0}r_{a}+\sum
_{\substack{a\in A;\\m_{a}<0}}m_{a}r_{a}
=\sum_{\substack{a\in A;\\m_{a}>0}}m_{a}r_{a}+\underbrace{\sum
_{\substack{a\in A;\\m_{a}=0}}0r_{a}}_{=0}+\sum_{\substack{a\in A;\\m_{a}%
<0}}m_{a}r_{a}\nonumber\\
&  =\sum_{\substack{a\in A;\\m_{a}>0}}m_{a}r_{a}+\sum_{\substack{a\in
A;\\m_{a}<0}}m_{a}r_{a}. \label{pf.lem.A.cancel.1}%
\end{align}

Now, let us define two subsets%
\[
A_{+}:=\left\{  a\in A\ \mid\ m_{a}>0\right\}  \ \ \ \ \ \ \ \ \ \ \text{and}%
\ \ \ \ \ \ \ \ \ \ A_{-}:=\left\{  a\in A\ \mid\ m_{a}<0\right\}
\]
of $A$. Then, the summation signs $\sum_{\substack{a\in A;\\m_{a}>0}}$ and
$\sum_{\substack{a\in A;\\m_{a}<0}}$ can be rewritten as $\sum_{a\in A_{+}}$
and $\sum_{a\in A_{-}}$, respectively. Therefore, we can rewrite
(\ref{pf.lem.A.cancel.1}) as%
\begin{equation}
\sum_{a\in A}m_{a}r_{a}=\sum_{a\in A_{+}}m_{a}r_{a}+\sum_{a\in A_{-}}%
m_{a}r_{a}. \label{pf.lem.A.cancel.2}%
\end{equation}

Each $a\in A_{+}$ satisfies $f\left(  a\right)  \in A_{-}$%
\ \ \ \ \footnote{\textit{Proof.} Let $a\in A_{+}$. Thus, $a\in A$ and
$m_{a}>0$ (by the definition of $A_{+}$). However, (\ref{eq.lem.A.cancel.m})
yields $m_{f\left(  a\right)  }=-m_{a}<0$ (since $m_{a}>0$). Hence, $f\left(
a\right)  \in A$ and $m_{f\left(  a\right)  }<0$. In other words, $f\left(
a\right)  \in A_{-}$ (by the definition of $A_{-}$), qed.}. Hence, the map%
\begin{align*}
g:A_{+}  &  \rightarrow A_{-},\\
a  &  \mapsto f\left(  a\right)
\end{align*}
is well-defined. This map $g$ is a restriction of the injective map $f$, and
thus is injective itself. Furthermore, $g$ is
surjective\footnote{\textit{Proof.} Let $b\in A_{-}$. Then, there exists some
$a\in A$ such that $b=f\left(  a\right)  $ (since $f$ is surjective). Consider
this $a$. From $b\in A_{-}$, we obtain $b\in A$ and $m_{b}<0$ (by the
definition of $A_{-}$). However, from $b=f\left(  a\right)  $, we obtain
$m_{b}=m_{f\left(  a\right)  }=-m_{a}$ (by (\ref{eq.lem.A.cancel.m})). Hence,
$m_{a}=-m_{b}>0$ (since $m_{b}<0$) and therefore $a\in A_{+}$ (by the
definition of $A_{+}$). The definition of $g$ thus yields $g\left(  a\right)
=f\left(  a\right)  =b$. Thus, $b=g\left(  a\right)  \in g\left(
A_{+}\right)  $ (since $a\in A_{+}$).
\par
Forget that we fixed $b$. We thus have shown that $b\in g\left(  A_{+}\right)
$ for each $b\in A_{-}$. In other words, $A_{-}\subseteq g\left(
A_{+}\right)  $. Thus, the map $g$ is surjective.}. Hence, $g$ is a bijection
(since $g$ is both injective and surjective).

Therefore, we can substitute $g\left(  a\right)  $ for $a$ in the sum
$\sum_{a\in A_{-}}m_{a}r_{a}$. We thus obtain%
\begin{align*}
\sum_{a\in A_{-}}m_{a}r_{a}  &  =\sum_{a\in A_{+}}\underbrace{m_{g\left(
a\right)  }r_{g\left(  a\right)  }}_{\substack{=m_{f\left(  a\right)
}r_{f\left(  a\right)  }\\\text{(since the definition of }g\\\text{yields
}g\left(  a\right)  =f\left(  a\right)  \text{)}}}=\sum_{a\in A_{+}%
}\underbrace{m_{f\left(  a\right)  }}_{\substack{=-m_{a}\\\text{(by
(\ref{eq.lem.A.cancel.m}))}}}\ \ \underbrace{r_{f\left(  a\right)  }%
}_{\substack{=r_{a}\\\text{(by (\ref{eq.lem.A.cancel.r}))}}}\\
&  =\sum_{a\in A_{+}}\left(  -m_{a}\right)  r_{a}=-\sum_{a\in A_{+}}m_{a}%
r_{a}.
\end{align*}
Therefore,
\[
\sum_{a\in A_{+}}m_{a}r_{a}+\sum_{a\in A_{-}}m_{a}r_{a}=0.
\]
In light of this, we can rewrite (\ref{pf.lem.A.cancel.2}) as $\sum_{a\in
A}m_{a}r_{a}=0$. This proves Lemma \ref{lem.A.cancel}.
\end{proof}

\begin{proof}
[Proof of Lemma \ref{lem.flagJT.cancel}.]Each twisted array $\left(
\sigma,T\right)  $ is either failing or unfailing. Hence,%
\begin{align}
&  \sum_{\substack{\left(  \sigma,T\right)  \text{ is a }\mathbf{b}%
\text{-flagged}\\\text{twisted array}}}\left(  -1\right)  ^{\sigma}w\left(
T\right)  \nonumber\\
&  =\sum_{\substack{\left(  \sigma,T\right)  \text{ is a}\\\text{failing
}\mathbf{b}\text{-flagged}\\\text{twisted array}}}\left(  -1\right)  ^{\sigma
}w\left(  T\right)  +\sum_{\substack{\left(  \sigma,T\right)  \text{ is
an}\\\text{unfailing }\mathbf{b}\text{-flagged}\\\text{twisted array}}}\left(
-1\right)  ^{\sigma}w\left(  T\right)  .\label{pf.lem.flagJT.cancel.2}%
\end{align}
Our next goal is to show that the first sum on the right hand side is $0$.

Let $A$ be the set of all failing $\mathbf{b}$-flagged twisted arrays. This
set $A$ is finite\footnote{\textit{Proof.} Let $b_{\max}:=\max\left\{
b_{1},b_{2},\ldots,b_{n}\right\}  $ (this should be understood to mean $0$
when $n=0$). If $\left(  \sigma,T\right)  $ is a $\mathbf{b}$-flagged twisted
array, then $\sigma$ is a permutation in $S_{n}$, whereas $T$ is a filling of
the diagram $P\left(  \sigma\right)  $ with entries from the set $\left\{
1,2,\ldots,b_{\max}\right\}  $ (since the \textquotedblleft$\mathbf{b}%
$-flagged\textquotedblright\ condition forces each entry of $T$ to be $\leq
b_{i}$ for an appropriate $i\in\left[  n\right]  $, and thus to be $\leq
b_{\max}$). Clearly, this leaves only finitely many options for $\sigma$ and
only finitely many options for $T$. Thus, there are finitely many $\mathbf{b}%
$-flagged twisted arrays. Hence, a fortiori, there are only finitely many
failing $\mathbf{b}$-flagged twisted arrays. In other words, the set $A$ is
finite.}. Moreover, the following follows easily from Lemma
\ref{lem.flagJT.flip1}:

\begin{statement}
\textit{Claim 1:} For each $\left(  \sigma,T\right)  \in A$, we have
$\operatorname*{flip}\left(  \sigma,T\right)  \in A$.
\end{statement}

\begin{proof}
[Proof of Claim 1.] Let $\left(  \sigma,T\right)  \in A$. Thus, $\left(
\sigma,T\right)  $ is a failing $\mathbf{b}$-flagged twisted array (by the
definition of $A$).
Let $\left(  \sigma^{\prime},T^{\prime}\right)  $ be the pair
$\operatorname*{flip}\left(  \sigma,T\right)  $. Then, Lemma
\ref{lem.flagJT.flip1} \textbf{(a)} shows that the pair $\left(
\sigma^{\prime},T^{\prime}\right)  $ is again a failing twisted array, and
furthermore, Lemma \ref{lem.flagJT.flip1} \textbf{(f)} shows that this twisted
array $\left(  \sigma^{\prime},T^{\prime}\right)  $ is again $\mathbf{b}%
$-flagged. Thus, $\left(  \sigma^{\prime},T^{\prime}\right)  $ is a failing
$\mathbf{b}$-flagged twisted array. In other words, $\left(  \sigma^{\prime
},T^{\prime}\right)  \in A$ (by the definition of $A$). In other words,
$\operatorname*{flip}\left(  \sigma,T\right)  \in A$ (since $\left(
\sigma^{\prime},T^{\prime}\right)  $ is the pair $\operatorname*{flip}\left(
\sigma,T\right)  $). This proves Claim 1.
\end{proof}

Thanks to Claim 1, we can define a map%
\begin{align*}
\operatorname*{flip}:A  & \rightarrow A,\\
\left(  \sigma,T\right)    & \mapsto\operatorname*{flip}\left(  \sigma
,T\right)  .
\end{align*}
Lemma \ref{lem.flagJT.flip1} \textbf{(c)} shows that this map is inverse to
itself (i.e., if we apply it twice in succession to some twisted array
$\left(  \sigma,T\right)  $, then we obtain the original $\left(
\sigma,T\right)  $ back). Hence, this map is invertible, i.e., a bijection.

For each element $a=\left(  \sigma,T\right)  $ of $A$, we define the integer
$m_{a}:=\left(  -1\right)  ^{\sigma}$ and the element $r_{a}:=w\left(
T\right)  $ of $R$ (where $R$ is as in Definition \ref{defh} \textbf{(a)}).
Then, we claim the following:

\begin{statement}
\textit{Claim 2:} Each $a\in A$ satisfies $m_{\operatorname*{flip}\left(
a\right)  }=-m_{a}$ and $r_{\operatorname*{flip}\left(  a\right)  }=r_{a}$.
\end{statement}

\begin{proof}
[Proof of Claim 2.] Let $a\in A$. Write $a$ as $\left(  \sigma,T\right)  $.
Thus, $\left(  \sigma,T\right)  =a\in A$. Hence, $\left(  \sigma,T\right)  $
is a failing $\mathbf{b}$-flagged twisted array (by the definition of $A$).
Let $\left(
\sigma^{\prime},T^{\prime}\right)  $ be the pair $\operatorname*{flip}\left(
\sigma,T\right)  $. Then, $\left(  \sigma^{\prime},T^{\prime}\right)
=\operatorname*{flip}\left(  \sigma,T\right)  =\operatorname*{flip}\left(
a\right)  $ (since $\left(  \sigma,T\right)  =a$).

From $a=\left(  \sigma,T\right)  $, we obtain $m_{a}=m_{\left(  \sigma
,T\right)  }=\left(  -1\right)  ^{\sigma}$ (by the definition of $m_{\left(
\sigma,T\right)  }$) and $r_{a}=r_{\left(  \sigma,T\right)  }=w\left(
T\right)  $ (by the definition of $r_{\left(  \sigma,T\right)  }$). Similarly,
from $\operatorname*{flip}\left(  a\right)  =\left(  \sigma^{\prime}%
,T^{\prime}\right)  $, we obtain $m_{\operatorname*{flip}\left(  a\right)
}=\left(  -1\right)  ^{\sigma^{\prime}}$ and $r_{\operatorname*{flip}\left(
a\right)  }=w\left(  T^{\prime}\right)  $.

However, Lemma \ref{lem.flagJT.flip1} \textbf{(d)} yields $\left(  -1\right)
^{\sigma^{\prime}}=-\left(  -1\right)  ^{\sigma}$. In other words,
$m_{\operatorname*{flip}\left(  a\right)  }=-m_{a}$ (since
$m_{\operatorname*{flip}\left(  a\right)  }=\left(  -1\right)  ^{\sigma
^{\prime}}$ and $m_{a}=\left(  -1\right)  ^{\sigma}$).

Furthermore, Lemma \ref{lem.flagJT.flip1} \textbf{(e)} yields that $w\left(
T^{\prime}\right)  =w\left(  T\right)  $. In other words,
$r_{\operatorname*{flip}\left(  a\right)  }=r_{a}$ (since
$r_{\operatorname*{flip}\left(  a\right)  }=w\left(  T^{\prime}\right)  $ and
$r_{a}=w\left(  T\right)  $). This completes the proof of Claim 2.
\end{proof}

Thus we know that $\operatorname*{flip}:A\rightarrow A$ is a bijection and
satisfies Claim 2. Hence, we can apply Lemma \ref{lem.A.cancel} to
$f=\operatorname*{flip}$. As a result, we obtain%
\[
\sum_{a\in A}m_{a}r_{a}=0.
\]
In view of%
\begin{align*}
\sum_{a\in A}m_{a}r_{a}  & =\sum_{\left(  \sigma,T\right)  \in A}%
\underbrace{m_{\left(  \sigma,T\right)  }}_{\substack{=\left(  -1\right)
^{\sigma}\\\text{(by the definition}\\\text{of }m_{\left(  \sigma,T\right)
}\text{)}}}\ \ \underbrace{r_{\left(  \sigma,T\right)  }}_{\substack{=w\left(
T\right)  \\\text{(by the definition}\\\text{of }r_{\left(  \sigma,T\right)
}\text{)}}}\ \ \ \ \ \ \ \ \ \ \left(
\begin{array}
[c]{c}%
\text{here, we have renamed the}\\
\text{summation index }a\text{ as }\left(  \sigma,T\right)
\end{array}
\right)  \\
& =\sum_{\left(  \sigma,T\right)  \in A}\left(  -1\right)  ^{\sigma}w\left(
T\right)  \\
& =\sum_{\substack{\left(  \sigma,T\right)  \text{ is a}\\\text{failing
}\mathbf{b}\text{-flagged}\\\text{twisted array}}}\left(  -1\right)  ^{\sigma
}w\left(  T\right)  \ \ \ \ \ \ \ \ \ \ \left(
\begin{array}
[c]{c}%
\text{since }A\text{ is the set of all}\\
\text{failing }\mathbf{b}\text{-flagged twisted arrays}%
\end{array}
\right)  ,
\end{align*}
we can rewrite this as%
\[
\sum_{\substack{\left(  \sigma,T\right)  \text{ is a}\\\text{failing
}\mathbf{b}\text{-flagged}\\\text{twisted array}}}\left(  -1\right)  ^{\sigma
}w\left(  T\right)  =0.
\]

Thus, (\ref{pf.lem.flagJT.cancel.2}) becomes%
\begin{align*}
\sum_{\substack{\left(  \sigma,T\right)  \text{ is a }\mathbf{b}%
\text{-flagged}\\\text{twisted array}}}\left(  -1\right)  ^{\sigma}w\left(
T\right)    & =\underbrace{\sum_{\substack{\left(  \sigma,T\right)  \text{ is
a}\\\text{failing }\mathbf{b}\text{-flagged}\\\text{twisted array}}}\left(
-1\right)  ^{\sigma}w\left(  T\right)  }_{=0}+\sum_{\substack{\left(
\sigma,T\right)  \text{ is an}\\\text{unfailing }\mathbf{b}\text{-flagged}%
\\\text{twisted array}}}\left(  -1\right)  ^{\sigma}w\left(  T\right)  \\
& =\sum_{\substack{\left(  \sigma,T\right)  \text{ is an}\\\text{unfailing
}\mathbf{b}\text{-flagged}\\\text{twisted array}}}\left(  -1\right)  ^{\sigma
}w\left(  T\right)  .
\end{align*}
This proves Lemma \ref{lem.flagJT.cancel}.
\end{proof}

\begin{proof}
[Proof of Theorem \ref{thm.flagJT.gen}.] Lemma
\ref{lem.flagJT.det=sum-twisted} yields%
\begin{align*}
\det\left(  h_{b_{i};\ \mu_{i}-i+j}\left[  j\right]  \right)  _{i,j\in\left[
n\right]  }  & =\sum_{\substack{\left(  \sigma,T\right)  \text{ is a
}\mathbf{b}\text{-flagged}\\\text{twisted array}}}\left(  -1\right)  ^{\sigma
}w\left(  T\right)  \\
& =\sum_{\substack{\left(  \sigma,T\right)  \text{ is an}\\\text{unfailing
}\mathbf{b}\text{-flagged}\\\text{twisted array}}}\left(  -1\right)  ^{\sigma
}w\left(  T\right)  \ \ \ \ \ \ \ \ \ \ \left(  \text{by Lemma
\ref{lem.flagJT.cancel}}\right)  \\
& =\sum_{T\in\operatorname{FSSYT}\left(  \mu,\mathbf{b}\right)  }%
\ \ \prod_{(i,j)\in Y(\mu)}u_{T\left(  i,j\right)  ,\ j-i}%
\ \ \ \ \ \ \ \ \ \ \left(  \text{by Lemma \ref{lem.flagJT.failsum}}\right)  .
\end{align*}
This proves Theorem \ref{thm.flagJT.gen}.
\end{proof}

\begin{proof}[Proof of Proposition \ref{prop.flagJT.f}.]
We set
\[
u_{i,j}:=x_{i}+y_{i+j}\ \ \ \ \ \ \ \ \ \ \text{for each }\left(  i,j\right)
\in\mathbb{Z}\times\mathbb{Z}.
\]
Then, for any $b\in\mathbb{N}$ and $q,d\in\mathbb{Z}$, the elements
$h_{b;\ q}\left[  d\right]  $ defined in Theorem \ref{thm.flagJT.gen} are
given by
\begin{align}
h_{b;\ q}\left[  d\right]    & =\sum_{\substack{\left(  i_{1},i_{2}%
,\ldots,i_{q}\right)  \in\left[  b\right]  ^{q};\\i_{1}\leq i_{2}\leq
\cdots\leq i_{q}}}\ \ \prod_{j=1}^{q}\underbrace{u_{i_{j},\ j-d}%
}_{\substack{=x_{i_{j}}+y_{i_{j}+\left(  j-d\right)  }\\=x_{i_{j}}%
+y_{i_{j}+\left(  j-1\right)  +\left(  1-d\right)  }\\\text{(since
}j-d=\left(  j-1\right)  +\left(  1-d\right)  \text{)}}}\nonumber\\
& =\sum_{\substack{\left(  i_{1},i_{2},\ldots,i_{q}\right)  \in\left[
b\right]  ^q;\\i_{1}\leq i_{2}\leq\cdots\leq i_{q}}}\ \ \prod_{j=1}%
^{q}\left(  x_{i_{j}}+y_{i_{j}+\left(  j-1\right)  +\left(  1-d\right)
}\right)  \nonumber\\
& =h\left(  q,\ \ b,\ \ 1-d\right)
\label{pf.prop.flagJT.f.hbqd=}
\end{align}
(by Definition \ref{defh} \textbf{(c)}).

Furthermore, the definition of the $u_{i,j}$ yields%
\begin{align*}
& \sum_{T\in\operatorname{FSSYT}\left(  \mu,\mathbf{b}\right)  }%
\ \ \prod_{(i,j)\in Y(\mu)}\underbrace{u_{T\left(  i,j\right)  ,\ j-i}%
}_{=x_{T\left(  i,j\right)  }+y_{T\left(  i,j\right)  +j-i}}\\
& =\sum_{T\in\operatorname{FSSYT}(\mu,\mathbf{b})}\ \ \prod_{\left(
i,j\right)  \in Y\left(  \mu\right)  }\left(  x_{T\left(  i,j\right)
}+y_{T\left(  i,j\right)  +j-i}\right)  .
\end{align*}
Hence,%
\begin{align*}
& \sum_{T\in\operatorname{FSSYT}(\mu,\mathbf{b})}\ \ \prod_{\left(
i,j\right)  \in Y\left(  \mu\right)  }\left(  x_{T\left(  i,j\right)
}+y_{T\left(  i,j\right)  +j-i}\right)  \\
& =\sum_{T\in\operatorname{FSSYT}\left(  \mu,\mathbf{b}\right)  }%
\ \ \prod_{(i,j)\in Y(\mu)}u_{T\left(  i,j\right)  ,\ j-i}\\
& =\det\left(  \underbrace{h_{b_{i};\ \mu_{i}-i+j}
\left[  j\right] }_{\substack{=
h\left(  \mu_{i}-i+j,\ \ b_{i},\ \ 1-j\right)
\\ \text{(by \eqref{pf.prop.flagJT.f.hbqd=})}}} \right)
_{i,j\in\left[  n\right]  }\ \ \ \ \ \ \ \ \ \ \left(  \text{by Theorem
\ref{thm.flagJT.gen}}\right)  \\
& =\det\left(  h\left(  \mu_{i}-i+j,\ \ b_{i},\ \ 1-j\right)  \right)
_{i,j\in\left[  n\right]  }.
\end{align*}
This proves Proposition \ref{prop.flagJT.f}.
\end{proof}

\begin{proof}
[Proof of Corollary \ref{cor.slam.det}.] The definition of $\mathcal{F}%
(\lambda/\mu)$ yields $\mathcal{F}(\lambda/\mu)=\operatorname{FSSYT}\left(
\mu,\mathbf{b}\right)  $ (since $\mathbf{b}$ is the flagging
induced by $\lambda/\mu$). Furthermore, Lemma \ref{lem.flagging-of-lm.inc}
shows that the flagging $\mathbf{b}$ is weakly increasing.

Corollary \ref{corollaryFlaggedExc} yields
\begin{align*}
\mathbf{s}_{\lambda}\left[  \mu\right]   &  =\sum_{T\in\mathcal{F}\left(
\lambda/\mu\right)  }\ \ \prod_{\left(  i,j\right)  \in Y\left(  \mu\right)
}\left(  x_{T\left(  i,j\right)  }+y_{T\left(  i,j\right)  +j-i}\right)  \\
&  =\sum_{T\in\operatorname{FSSYT}(\mu,\mathbf{b})}\ \ \prod_{\left(
i,j\right)  \in Y\left(  \mu\right)  }(x_{T\left(  i,j\right)  }+y_{T\left(
i,j\right)  +j-i})\ \ \ \ \ \ \ \ \ \ \left(  \text{since }\mathcal{F}%
(\lambda/\mu)=\operatorname{FSSYT}\left(  \mu,\mathbf{b}\right)
\right)  \\
&  =\det\left(  h\left(  \mu_{i}-i+j,\ \ b_{i},\ \ 1-j\right)  \right)
_{i,j\in\left[  n\right]  }
\end{align*}
(by Proposition \ref{prop.flagJT.f}). This proves
Corollary \ref{cor.slam.det}.
\end{proof}

\subsection{To Section \ref{sec.det}}

\begin{proof}[Proof of Lemma \ref{determinant.sum}.]
Let $k\in\left[  n\right]  $.
Applying \eqref{pf.lem.det-u.lap} to $A=P\underset{\operatorname*{row}}{\overset{k}{\leftarrow}} Q$, we obtain
\begin{align}
\det\left(  P\underset{\operatorname*{row}}{\overset{k}{\leftarrow}}Q\right)
& =\sum_{\ell=1}^{n}\left(  -1\right)  ^{k+\ell}\underbrace{\left(
P\underset{\operatorname*{row}}{\overset{k}{\leftarrow}}Q\right)  _{k,\ell}%
}_{\substack{=Q_{k,\ell}\\\text{(since the matrix }%
P\underset{\operatorname*{row}}{\overset{k}{\leftarrow}}Q\\\text{takes its
}k\text{-th row from }Q\text{)}}}\det\underbrace{\left(  \left(
P\underset{\operatorname*{row}}{\overset{k}{\leftarrow}}Q\right)  _{\sim
k,\sim\ell}\right)  }_{\substack{=P_{\sim k,\sim\ell}\\\text{(since the matrix
}P\underset{\operatorname*{row}}{\overset{k}{\leftarrow}}Q\\\text{differs from
}P\text{ only}\\\text{in its }k\text{-th row)}}}\nonumber\\
& =\sum_{\ell=1}^{n}\left(  -1\right)  ^{k+\ell}Q_{k,\ell}\det\left(  P_{\sim
k,\sim\ell}\right)  .\label{pf.determinant.sum.2}%
\end{align}
Forget that we fixed $k$. We thus have proved the equality
\eqref{pf.determinant.sum.2} for each $k\in\left[  n\right]  $. Summing this
equality over all $k\in\left[  n\right]  $, we find%
\[
\sum_{k=1}^{n}\det\left(  P\underset{\operatorname*{row}%
}{\overset{k}{\leftarrow}}Q\right)  =\sum_{k=1}^{n}\ \ \sum_{\ell=1}%
^{n}\left(  -1\right)  ^{k+\ell}Q_{k,\ell}\det\left(  P_{\sim k,\sim\ell
}\right)  .
\]

An analogous argument (using Laplace expansion along columns rather than rows)
shows that%
\begin{align*}
\sum_{k=1}^{n}\det\left(  P\underset{\operatorname*{col}%
}{\overset{k}{\leftarrow}}Q\right)    & =\sum_{k=1}^{n}\ \ \sum_{\ell=1}%
^{n}\left(  -1\right)  ^{\ell+k}Q_{\ell,k}\det\left(  P_{\sim\ell,\sim
k}\right)  \\
& =\sum_{\ell=1}^{n}\ \ \sum_{k=1}^{n}\left(  -1\right)  ^{\ell+k}Q_{\ell
,k}\det\left(  P_{\sim\ell,\sim k}\right)  \\
& =\sum_{k=1}^{n}\ \ \sum_{\ell=1}^{n}\left(  -1\right)  ^{k+\ell}Q_{k,\ell
}\det\left(  P_{\sim k,\sim\ell}\right)
\end{align*}
(here, we have renamed the summation indices $\ell$ and $k$ as $k$ and $\ell
$). Comparing these two equalities, we obtain%
\[
\sum_{k=1}^{n}\det\left(  P\underset{\operatorname*{row}%
}{\overset{k}{\leftarrow}}Q\right)  =\sum_{k=1}^{n}\det\left(
P\underset{\operatorname*{col}}{\overset{k}{\leftarrow}}Q\right)  .
\]
This proves Lemma \ref{determinant.sum}.
\end{proof}

\begin{proof}
[Proof of Lemma~\ref{lem.det-u-cor}.]
Let $P$ be the $n\times n$-matrix
$\left(  u_{i,j}\right)  _{i,j\in\left[  n\right]  }$, and let $Q$ be the
$n\times n$-matrix $\left(  u_{i,j+1}\right)  _{i,j\in\left[  n\right]  }$.
Consider the matrices $P\underset{\operatorname*{row}}{\overset{k}{\leftarrow
}}Q$ and $P\underset{\operatorname*{col}}{\overset{k}{\leftarrow}}Q$ defined
in Lemma \ref{determinant.sum}. It is easy to see the following:

\begin{statement}
\textit{Claim 1:} Every $k\in\left[  n\right]  $ satisfies
\begin{equation}
P\underset{\operatorname*{row}}{\overset{k}{\leftarrow}}Q = \left(
u_{i,j+\left[  k=i\right]  }\right)  _{i,j\in\left[  n\right]  }%
\label{pf.lem.det-u-cor.row}%
\end{equation}
and%
\begin{equation}
P\underset{\operatorname*{col}}{\overset{k}{\leftarrow}}Q = \left(
u_{i,j+\left[  k=j\right]  }\right)  _{i,j\in\left[  n\right]  }%
.\label{pf.lem.det-u-cor.col}%
\end{equation}

\end{statement}

\begin{proof}
[Proof of Claim 1.] Let $k\in\left[  n\right]  $. As we recall, the matrix
$P\underset{\operatorname*{row}}{\overset{k}{\leftarrow}}Q$ is obtained from
$P$ by replacing the $k$-th row by the $k$-th row of $Q$. Hence, its $\left(
i,j\right)  $-th entry is given by%
\begin{equation}
\left(  P\underset{\operatorname*{row}}{\overset{k}{\leftarrow}}Q\right)
_{i,j}=%
\begin{cases}
P_{i,j}, & \text{if }i\neq k;\\
Q_{i,j}, & \text{if }i=k
\end{cases}
\label{pf.lem.det-u-cor.row.pf.1}%
\end{equation}
for all $i,j\in\left[  n\right]  $.

Now, let $i,j\in\left[  n\right]  $. Then, the element $u_{i,j+\left[
k=i\right]  }$ equals $u_{i,j}$ when $i\neq k$ (because in this case, we have
$k\neq i$ and thus $\left[  k=i\right]  =0$ and therefore $j+\left[
k=i\right]  =j+0=j$ and thus
$u_{i,j+\left[
k=i\right]  } = u_{i,j}$), but equals $u_{i,j+1}$ when $i=k$ (because in this case,
we have $k=i$ and thus $\left[  k=i\right]  =1$  and thus
$u_{i,j+\left[
k=i\right]  } = u_{i,j+1}$). Hence,%
\[
u_{i,j+\left[  k=i\right]  }=%
\begin{cases}
u_{i,j}, & \text{if }i\neq k;\\
u_{i,j+1}, & \text{if }i=k.
\end{cases}
\]
On the other hand, (\ref{pf.lem.det-u-cor.row.pf.1}) yields
\[
\left(  P\underset{\operatorname*{row}}{\overset{k}{\leftarrow}}Q\right)
_{i,j}=%
\begin{cases}
P_{i,j}, & \text{if }i\neq k;\\
Q_{i,j}, & \text{if }i=k
\end{cases}
=%
\begin{cases}
u_{i,j}, & \text{if }i\neq k;\\
u_{i,j+1}, & \text{if }i=k
\end{cases}
\]
(since $P_{i,j}=u_{i,j}$ (by the definition of $P$) and $Q_{i,j}=u_{i,j+1}$
(by the definition of $Q$)). Comparing these two equalities, we obtain%
\[
\left(  P\underset{\operatorname*{row}}{\overset{k}{\leftarrow}}Q\right)
_{i,j}=u_{i,j+\left[  k=i\right]  }.
\]

Forget that we fixed $i,j$. We thus have proved the equality $\left(
P\underset{\operatorname*{row}}{\overset{k}{\leftarrow}}Q\right)
_{i,j}=u_{i,j+\left[  k=i\right]  }$ for all $i,j\in\left[  n\right]  $. In
other words, $P\underset{\operatorname*{row}}{\overset{k}{\leftarrow}}%
Q = \left(  u_{i,j+\left[  k=i\right]  }\right)  _{i,j\in\left[  n\right]
}$. This proves (\ref{pf.lem.det-u-cor.row}). Similarly,
(\ref{pf.lem.det-u-cor.col}) can be shown. Thus, Claim 1 follows.
\end{proof}

However, Lemma \ref{determinant.sum} says that
\[
\sum_{k=1}^{n}\det\left(  P\underset{\operatorname*{row}%
}{\overset{k}{\leftarrow}}Q\right)  =\sum_{k=1}^{n}\det\left(
P\underset{\operatorname*{col}}{\overset{k}{\leftarrow}}Q\right)  .
\]
In view of (\ref{pf.lem.det-u-cor.row}) and (\ref{pf.lem.det-u-cor.col}), we
can rewrite this as%
\begin{equation}
\sum_{k=1}^{n}\det\left(  u_{i,j+\left[  k=i\right]  }\right)  _{i,j\in\left[
n\right]  }=\sum_{k=1}^{n}\det\left(  u_{i,j+\left[  k=j\right]  }\right)
_{i,j\in\left[  n\right]  }.
\label{pf.lem.det-u-cor.2}
\end{equation}

However, the right hand side of this equality can be simplified thanks to the following claim:

\begin{statement}
    \textit{Claim 2:} Let $k \in \ive{n-1}$.
    Then,
    \begin{equation}
    \det\left(  u_{i,j+\left[  k=j\right]  }\right)  _{i,j\in\left[  n\right]}=0.
    \label{pf.lem.det-u-cor.3}%
    \end{equation}
\end{statement}

\begin{proof}[Proof of Claim 2.]
We have $\ive{k=k} = 1$ and $\ive{k=k+1} = 0$.
Thus, for each $i \in \ive{n}$, we have
$u_{i,k+\left[  k=k\right]  }  =u_{i,k+1}$
and
$u_{i,\left(  k+1\right)  +\left[  k=k+1\right]  }  =u_{i,\left(  k+1\right) +0}=u_{i,k+1}$.
Comparing these two equalities, we see that
\[
u_{i,k+\left[  k=k\right]  } = u_{i,\left(  k+1\right)  +\left[ k=k+1\right]  }
\qquad \text{for each } i \in \ive{n}.
\]
In other words, the $k$-th and $\tup{k+1}$-st columns of the matrix $\left(
u_{i,j+\left[  k=j\right]  }\right)  _{i,j\in\left[  n\right]  }$
are equal. Hence, this matrix has two equal columns, and thus its determinant vanishes.
This proves Claim 2.
\end{proof}

Now, (\ref{pf.lem.det-u-cor.2}) becomes%
\begin{align*}
\sum_{k=1}^{n}\det\left(  u_{i,j+\left[  k=i\right]  }\right)  _{i,j\in\left[
n\right]  }  & =\sum_{k=1}^{n}\det\left(  u_{i,j+\left[  k=j\right]  }\right)
_{i,j\in\left[  n\right]  }\\
& =\sum_{k=1}^{n-1}\underbrace{\det\left(  u_{i,j+\left[  k=j\right]
}\right)  _{i,j\in\left[  n\right]  }}_{\substack{=0\\\text{(by
(\ref{pf.lem.det-u-cor.3}))}}}+\det\left(  u_{i,j+\left[  n=j\right]
}\right)  _{i,j\in\left[  n\right]  }\\
& =\det\left(  u_{i,j+\left[  n=j\right]  }\right)  _{i,j\in\left[  n\right]
}.
\end{align*}
This proves Lemma \ref{lem.det-u-cor}.
\end{proof}

\begin{proof}[Proof of Lemma~\ref{lem.det-uc}.]
Let $U$ denote the $n\times n$-matrix $\left(  u_{i,j}\right)  _{i,j\in\left[
n\right]  }$.

Let $k\in\left[  n\right]  $.

Define an $n\times n$-matrix
\begin{equation}
V_{k}:=\left(  u_{i,j+\left[  k=i\right]  }\right)  _{i,j\in\left[  n\right]
}.\label{pf.lem.det-u.Vk=}%
\end{equation}
This matrix $V_{k}$ differs from $U$ only in its $k$-th row (because if $i\neq
k$, then the $\left(  i,j\right)  $-th entry of the matrix $V_{k}$ equals%
\begin{align*}
u_{i,j+\left[  k=i\right]  }  & =u_{i,j}%
\ \ \ \ \ \ \ \ \ \ \left(  \text{since }\left[  k=i\right]  =0\text{
(because }k\neq i\text{)}\right),
\end{align*}
which is precisely the corresponding entry of $U$). In other words, the
matrices $V_{k}$ and $U$ become equal if we remove their $k$-th rows.
Therefore,
\begin{equation}
\left(  V_{k}\right)  _{\sim k,\sim\ell}=U_{\sim k,\sim\ell}%
\ \ \ \ \ \ \ \ \ \ \text{for each }\ell\in\left[  n\right]
.\label{pf.lem.det-u.4}%
\end{equation}

Moreover, for each $\ell\in\left[  n\right]  $, the $\left(  k,\ell\right)
$-th entry of $V_{k}$ is%
\begin{align}
\left(  V_{k}\right)  _{k,\ell}  & =u_{k,\ell+\left[  k=k\right]
}\ \ \ \ \ \ \ \ \ \ \left(  \text{by the definition of }V_{k}\right)
\nonumber\\
& =u_{k,\ell+1}\ \ \ \ \ \ \ \ \ \ \left(  \text{since }\left[  k=k\right]
=1\right)  .\label{pf.lem.det-u.5}%
\end{align}

Now, applying (\ref{pf.lem.det-u.lap}) to $A=V_{k}$, we obtain%
\begin{align}
\det\left( V_{k}\right)    & =\sum_{\ell=1}^{n}\left(  -1\right)  ^{k+\ell
}\underbrace{\left(  V_{k}\right)  _{k,\ell}}_{\substack{=u_{k,\ell
+1}\\\text{(by (\ref{pf.lem.det-u.5}))}}}\det\left(  \underbrace{\left(
V_{k}\right)  _{\sim k,\sim\ell}}_{\substack{=U_{\sim k,\sim\ell}\\\text{(by
(\ref{pf.lem.det-u.4}))}}}\right)  \nonumber\\
& =\sum_{\ell=1}^{n}\left(  -1\right)  ^{k+\ell}u_{k,\ell+1}\det\left(
U_{\sim k,\sim\ell}\right)  .\label{pf.lem.det-u.6}%
\end{align}

Now, we define a further $n\times n$-matrix
\begin{equation}
U_{k}:=\left(  u_{i,j+\left[  k=i\right]  }-p_{i}u_{i,j}\left[  k=i\right]
\right)  _{i,j\in\left[  n\right]  }.\label{pf.lem.det-u.Uk=}%
\end{equation}
Note that this matrix $U_{k}$ differs from $U$ only in its $k$-th row (because
if $i\neq k$, then the $\left(  i,j\right)  $-th entry of the matrix $U_{k}$
equals%
\begin{align*}
u_{i,j+\left[  k=i\right]  }-p_{i}u_{i,j}\left[  k=i\right]
& =\underbrace{u_{i,j+0}}_{=u_{i,j}}-\underbrace{p_{i}u_{i,j}0}_{=0}%
\ \ \ \ \ \ \ \ \ \ \left(  \text{since }\left[  k=i\right]  =0\text{ (because
}k\neq i\text{)}\right)  \\
& =u_{i,j},
\end{align*}
which is precisely the corresponding entry of $U$). In other words, the
matrices $U_{k}$ and $U$ become equal if we remove their $k$-th rows.
Therefore,
\begin{equation}
\left(  U_{k}\right)  _{\sim k,\sim\ell}=U_{\sim k,\sim\ell}%
\ \ \ \ \ \ \ \ \ \ \text{for each }\ell\in\left[  n\right]
.\label{pf.lem.det-u.1}%
\end{equation}

Moreover, for each $\ell\in\left[  n\right]  $, the $\left(  k,\ell\right)
$-th entry of $U_{k}$ is%
\begin{align}
\left(  U_{k}\right)  _{k,\ell}  & =u_{k,\ell+\left[  k=k\right]  }-p_{k
}u_{k,\ell}\left[  k=k\right]  \nonumber\\
& =u_{k,\ell+1}-p_{k}u_{k,\ell}1\ \ \ \ \ \ \ \ \ \ \left(  \text{since
}\left[  k=k\right]  =1\right)  \nonumber\\
& =u_{k,\ell+1}-p_{k}u_{k,\ell}.\label{pf.lem.det-u.2}%
\end{align}

Now, applying (\ref{pf.lem.det-u.lap}) to $A=U_{k}$, we obtain%
\begin{align}
\det\left(  U_{k}\right)    & =\sum_{\ell=1}^{n}\left(  -1\right)  ^{k+\ell
}\underbrace{\left(  U_{k}\right)  _{k,\ell}}_{\substack{=u_{k,\ell+1}%
-p_{k}u_{k,\ell}\\\text{(by (\ref{pf.lem.det-u.2}))}}}\det\left(
\underbrace{\left(  U_{k}\right)  _{\sim k,\sim\ell}}_{\substack{=U_{\sim
k,\sim\ell}\\\text{(by (\ref{pf.lem.det-u.1}))}}}\right)  \nonumber\\
& =\sum_{\ell=1}^{n}\left(  -1\right)  ^{k+\ell}\left(  u_{k,\ell+1}-p_{k
}u_{k,\ell}\right)  \det\left(  U_{\sim k,\sim\ell}\right)  \nonumber\\
& =\underbrace{\sum_{\ell=1}^{n}\left(  -1\right)  ^{k+\ell}u_{k,\ell+1}%
\det\left(  U_{\sim k,\sim\ell}\right)  }_{\substack{=\det\left(
V_{k}\right)  \\\text{(by (\ref{pf.lem.det-u.6}))}}}-\sum_{\ell=1}^{n}\left(
-1\right)  ^{k+\ell}p_{k}u_{k,\ell}\det\left(  U_{\sim k,\sim\ell}\right)
\nonumber\\
& =\det\left(  V_{k}\right)  -\sum_{\ell=1}^{n}\left(  -1\right)  ^{k+\ell
}p_{k}u_{k,\ell}\det\left(  U_{\sim k,\sim\ell}\right)
.\label{pf.lem.det-u.3}%
\end{align}

Forget that we fixed $k$. Thus, for each $k\in\left[  n\right]  $, we have
defined two matrices $V_{k}$ and $U_{k}$ and proved the relation
(\ref{pf.lem.det-u.3}) between their determinants.

Now, summing the equality (\ref{pf.lem.det-u.3}) for all $k\in\left[
n\right]  $, we obtain%
\begin{align*}
\sum_{k=1}^{n}\det\left(  U_{k}\right)    & =\sum_{k=1}^{n}\left(  \det\left(
V_{k}\right)  -\sum_{\ell=1}^{n}\left(  -1\right)  ^{k+\ell}p_{k}u_{k,\ell
}\det\left(  U_{\sim k,\sim\ell}\right)  \right)  \\
& =\sum_{k=1}^{n}\det\left(  V_{k}\right)  -\sum_{k=1}^{n}\ \ \sum_{\ell
=1}^{n}\left(  -1\right)  ^{k+\ell}p_{k}u_{k,\ell}\det\left(  U_{\sim
k,\sim\ell}\right)  .
\end{align*}
In view of%
\begin{align*}
\sum_{k=1}^{n}\det\underbrace{\left(  V_{k}\right)  }_{\substack{=\left(
u_{i,j+\left[  k=i\right]  }\right)  _{i,j\in\left[  n\right]  }\\\text{(by
the definition of }V_{k}\text{)}}}  & =\sum_{k=1}^{n}\det\left(
u_{i,j+\left[  k=i\right]  }\right)  _{i,j\in\left[  n\right]  }\\
& =\det\left(  u_{i,j+\left[  n=j\right]  }\right)  _{i,j\in\left[  n\right]
}\ \ \ \ \ \ \ \ \ \ \left(  \text{by Lemma \ref{lem.det-u-cor}}\right)
\end{align*}
and%
\begin{align}
& \sum_{k=1}^{n}\ \ \sum_{\ell=1}^{n}\left(  -1\right)  ^{k+\ell}p_{k}\underbrace{u_{k,\ell}}_{\substack{=U_{k,\ell}\\\text{(by the definition of
}U\text{)}}}\det\left(  U_{\sim k,\sim\ell}\right)  \nonumber\\
&=\sum_{k=1}^{n}p_{k}\underbrace{\sum_{\ell=1}^{n}\left(  -1\right)
^{k+\ell}U_{k,\ell}\det\left(  U_{\sim k,\sim\ell}\right)  }_{\substack{=\det
U\\\text{(by (\ref{pf.lem.det-u.lap}), applied to }A=U\text{)}}}\nonumber\\
& =\sum_{k=1}^{n}p_{k}\det U=\left(
\sum_{k=1}^{n}p_{k}\right)  \det\underbrace{U}_{=\left(  u_{i,j}\right)
_{i,j\in\left[  n\right]  }}\nonumber\\
& =\left(  \sum_{k=1}^{n}p_{k}\right)  \det\left(  u_{i,j}\right)
_{i,j\in\left[  n\right]  } ,
\label{pf.lem.det-u.term1}
\end{align}
this becomes
\begin{align*}
\sum_{k=1}^{n}\det\left(  U_{k}\right)
&=
\underbrace{\sum_{k=1}^{n}\det\left(  V_{k}\right)}_{= \det\left(  u_{i,j+\left[  n=j\right]  }\right)  _{i,j\in\left[  n\right]
}}
-
\underbrace{\sum_{k=1}^{n}\ \ \sum_{\ell
=1}^{n}\left(  -1\right)  ^{k+\ell}p_{k}u_{k,\ell}\det\left(  U_{\sim
k,\sim\ell}\right)}_{= \left(  \sum_{k=1}^{n}p_{k}\right)  \det\left(  u_{i,j}\right)
_{i,j\in\left[  n\right]  }}
\\
&=\det\left(  u_{i,j+\left[
n=j\right]  }\right)  _{i,j\in\left[  n\right]  }-\left(  \sum_{k=1}^{n}%
p_{k}\right)  \det\left(  u_{i,j}\right)  _{i,j\in\left[  n\right]  }.
\end{align*}
In view of (\ref{pf.lem.det-u.Uk=}), we can rewrite this as%
\begin{align*}
& \sum_{k=1}^{n}\det\left(  u_{i,j+\left[  k=i\right]  }-p_{i}u_{i,j}\left[
k=i\right]  \right)  _{i,j\in\left[  n\right]  }\\
& =\det\left(  u_{i,j+\left[  n=j\right]  }\right)  _{i,j\in\left[  n\right]
}-\left(  \sum_{k=1}^{n}p_{k}\right)  \det\left(  u_{i,j}\right)
_{i,j\in\left[  n\right]  }.
\end{align*}
This proves Lemma \ref{lem.det-uc}.
\end{proof}

\begin{proof}[Proof of Lemma \ref{lem.det.lastrow=01}.]
Let $A$ be the matrix $\left(  a_{i,j}\right)  _{i,j\in\left[  n\right]  }$.
Then,  $A_{\sim n,\sim n}=\left(  a_{i,j}\right)  _{i,j\in\left[  n-1\right]
}$. Now, (\ref{pf.lem.det-u.lap}) (applied to $k=n$) yields
\begin{align*}
\det A &  =\sum_{\ell=1}^{n}\left(  -1\right)  ^{n+\ell}\underbrace{A_{n,\ell
}}_{\substack{=a_{n,\ell}\\\text{(by the definition of }A\text{)}}}\det\left(
A_{\sim n,\sim\ell}\right)  \\
&= \sum_{\ell=1}^{n}\left(  -1\right)  ^{n+\ell}
a_{n,\ell}
\det\left(
A_{\sim n,\sim\ell}\right)  \\
&  =\sum_{\ell=1}^{n-1}\left(  -1\right)  ^{n+\ell}\underbrace{a_{n,\ell}%
}_{\substack{=0\\\text{(by (\ref{eq.lem.det.lastrow=01.ass}))}}}\det\left(
A_{\sim n,\sim\ell}\right)  +\underbrace{\left(  -1\right)  ^{n+n}}%
_{=1}a_{n,n}\det\underbrace{\left(  A_{\sim n,\sim n}\right)  }_{=\left(
a_{i,j}\right)  _{i,j\in\left[  n-1\right]  }}\\
&  =\underbrace{\sum_{\ell=1}^{n-1}\left(  -1\right)  ^{n+\ell}0\det\left(
A_{\sim n,\sim\ell}\right)  }_{=0}+\, a_{n,n}\cdot
\det\left(  a_{i,j}\right)
_{i,j\in\left[  n-1\right]  }=a_{n,n}\cdot\det\left(  a_{i,j}\right)
_{i,j\in\left[  n-1\right]  }.
\end{align*}
This proves Lemma \ref{lem.det.lastrow=01}.
\end{proof}

\subsection{To Section \ref{sec.further}}

\begin{proof}[Proof of Lemma \ref{lem.l-decrease}.]
Let $i$ be a positive integer.
Then, $\ell_i = \lambda_i - i$ (by the definition of $\ell_i$) and 
$\ell_{i+1} = \lambda_{i+1} - \tup{i+1}$ (by the definition of $\ell_{i+1}$).

However, $\lambda_i \geq \lambda_{i+1}$ (since $\lambda$ is a partition).
Hence, $\lambda_i - i \geq \lambda_{i+1} - i > \lambda_{i+1} - i - 1 = \lambda_{i+1} - \tup{i+1}$.
In view of $\ell_i = \lambda_i - i$ and $\ell_{i+1} = \lambda_{i+1} - \tup{i+1}$, we can rewrite this as $\ell_i > \ell_{i+1}$.

Forget that we fixed $i$. We thus have shown that $\ell_i > \ell_{i+1}$ for every positive integer $i$.
In other words, $\ell_1 > \ell_2 > \ell_3 > \cdots$.
Similarly, we can show that
$m_1 > m_2 > m_3 > \cdots$ and
$\ell^t_1 > \ell^t_2 > \ell^t_3 > \cdots$ and
$m^t_1 > m^t_2 > m^t_3 > \cdots$.
Thus, Lemma~\ref{lem.l-decrease} is proved.
\end{proof}
\begin{proof}[Proof of Lemma \ref{lem.nu=mu+k}.]
If $\nu$ is a partition that satisfies $\mu\lessdot\nu$, then $\nu$ can be
obtained from $\mu$ by incrementing exactly one entry of $\mu$ by $1$. In
other words, $\nu=\mu^{+k}$ for some positive integer $k$.
This $k$ must furthermore satisfy $k=1$ or $\mu_{k}\neq\mu_{k-1}$ (since
otherwise, $\mu^{+k}$ is not a partition). In other words, this $k$ must
belong to $\operatorname*{ER}\left(  \mu\right)  $.

Thus, we have shown that every partition $\nu$ that satisfies $\mu\lessdot\nu$
has the form $\mu^{+k}$ for some $k\in\operatorname*{ER}\left(  \mu\right)  $.
Conversely, it is clear that any partition $\mu^{+k}$ with $k\in
\operatorname*{ER}\left(  \mu\right)  $ is a partition $\nu$ that satisfies
$\mu\lessdot\nu$. Combining these two facts, we see that the partitions $\nu$
that satisfy $\mu\lessdot\nu$ are precisely the partitions $\mu^{+k}$ for the
elements $k\in\operatorname*{ER}\left(  \mu\right)  $.
Hence, Lemma \ref{lem.nu=mu+k} is proved.
\end{proof}

\begin{proof}[Proof of Lemma \ref{lem.mu+ki}.]
This follows immediately from the definition of $\mu^{+k}$.
\end{proof}

\begin{proof}[Proof of Lemma \ref{lem.Deltas.if-k-then-i}.]
    We must prove that $k \in \left[  n\right]$. Assume the contrary.
	Thus, $k > n$. Hence, $\mu_k = 0$ (since
	$\mu = \left(  \mu_{1}, \mu_{2}, \ldots, \mu_{n}\right)  $),
    so that the definition of $m_k$ yields
	$m_k = \underbrace{\mu_k}_{=0} -\, k = -k < -n$ (since $k > n$).
    However, the definition of $\ell_j$ yields
	$\ell_j = \underbrace{\lambda_j}_{\geq 0} -\, j \geq -j \geq -n$
	(since $j \leq n$).
    This contradicts $\ell_j = m_k < -n$. This
    contradiction shows that our assumption was false.
    Hence, Lemma~\ref{lem.Deltas.if-k-then-i} is proved.
\end{proof}

%

\begin{proof}[Proof of Lemma \ref{lem.Deltas.sum1i}.]
We have $\ell_j \in\Delta\left(  \mu \right)  =\left\{  m_{1}, m_2, m_3, \ldots \right\}  $.
In other words, $\ell_j = m_{k}$ for some positive integer
$k$. Consider this $k$.
Lemma~\ref{lem.Deltas.if-k-then-i}
yields $k \in \ive{n}$ (since $j \in \ive{n}$).
Thus, there exists at least one  $i \in\left[
n\right]  $ satisfying $\ell_j = m_i$ (namely, $i = k$).

Lemma~\ref{lem.l-decrease} yields $m_{1} > m_{2} > m_{3} > \cdots$. Hence, the numbers $m_1, m_2, m_3, \ldots$ are
distinct. Thus, there exists at most one $i \in\left[  n\right]  $ satisfying
$\ell_j = m_i$. Since we also know that there exists at least one such $i$,
we thus conclude that there exists exactly one such $i$. Hence, the sum
$\sum_{\substack{i\in\left[  n\right]  ;\\\ell_{j} = m_i}}1$ contains exactly
one addend, so that it simplifies to $1$. This proves Lemma
\ref{lem.Deltas.sum1i}.
\end{proof}

\begin{proof}[Proof of Lemma \ref{lem.Deltas.if-not-then-k}.]
Assume the contrary. Thus, $k>n$. Hence, $\lambda_{k}=0$ (since $\lambda
=\left(  \lambda_{1},\lambda_{2},\ldots,\lambda_{n}\right)  $). Thus, the
definition of $\ell_{k}$ yields $\ell_{k}=\underbrace{\lambda_{k}}_{=0}-\,k=-k$.
Similarly, $m_{k}=-k$. Comparing these equalities, we find $m_k = \ell_{k}\notin\Delta\left(
\mu\right)  =\left\{  m_{1},m_{2},m_{3},\ldots\right\}  $, which contradicts
$m_{k}\in\left\{  m_{1},m_{2},m_{3},\ldots\right\}  $. This
contradiction shows that our assumption was false. Hence, Lemma
\ref{lem.Deltas.if-not-then-k} is proved.
\end{proof}

\begin{proof}[Proof of Lemma \ref{lem.Deltas.f=g}.]
We have
\begin{align}
\sum_{\substack{i,j\in\left[  n\right]  ;\\\ell_{j}=m_{i}}}g\left(  j\right)
&  =\sum_{j\in\left[  n\right]  }\ \ \sum_{\substack{i\in\left[  n\right]
;\\\ell_{j}=m_{i}}}g\left(  j\right)  \nonumber\\
&  =\sum_{\substack{j\in\left[  n\right]  ;\\\ell_{j}\in\Delta\left(
\mu\right)  }}\ \ \underbrace{\sum_{\substack{i\in\left[  n\right]
;\\\ell_{j}=m_{i}}}g\left(  j\right)  }_{\substack{=g\left(  j\right)
\cdot\sum_{\substack{i\in\left[  n\right]  ;\\\ell_{j}=m_{i}}}1}%
}+\sum_{\substack{j\in\left[  n\right]  ;\\\ell_{j}\notin\Delta\left(
\mu\right)  }}\ \ \underbrace{\sum_{\substack{i\in\left[  n\right]
;\\\ell_{j}=m_{i}}}g\left(  j\right)  }_{\substack{=\left(  \text{empty
sum}\right)  \\\text{(since }\ell_{j}\notin\Delta\left(  \mu\right)  =\left\{
m_{1},m_{2},m_{3},\ldots\right\}  \\\text{shows that there exists no }%
i\in\left[  n\right]  \\\text{satisfying }\ell_{j}=m_{i}\text{)}}}\nonumber\\
&  =\sum_{\substack{j\in\left[  n\right]  ;\\\ell_{j}\in\Delta\left(
\mu\right)  }}g\left(  j\right)  \cdot\underbrace{\sum_{\substack{i\in\left[
n\right]  ;\\\ell_{j}=m_{i}}}1}_{\substack{=1\\\text{(by Lemma
\ref{lem.Deltas.sum1i})}}}+\sum_{\substack{j\in\left[  n\right]  ;\\\ell
_{j}\notin\Delta\left(  \mu\right)  }}\underbrace{\left(  \text{empty
sum}\right)  }_{=0}\nonumber\\
&  =\sum_{\substack{j\in\left[  n\right]  ;\\\ell_{j}\in\Delta\left(
\mu\right)  }}g\left(  j\right)  +\sum_{\substack{j\in\left[  n\right]
;\\\ell_{j}\notin\Delta\left(  \mu\right)  }}0\nonumber\\
&  =\sum_{\substack{j\in\left[  n\right]  ;\\\ell_{j}\in\Delta\left(
\mu\right)  }}g\left(  j\right)  .\label{pf.lem.Deltas.f=g.1}%
\end{align}
An analogous argument (with the labels $i,j,\lambda,\mu,m_{k},\ell_{k},g$
replaced by $j,i,\mu,\lambda,\ell_{k},m_{k},f$) shows that%
\begin{equation}
\sum_{\substack{j,i\in\left[  n\right]  ;\\m_{i}=\ell_{j}}}f\left(  i\right)
=\sum_{\substack{i\in\left[  n\right]  ;\\m_{i}\in\Delta\left(  \lambda
\right)  }}f\left(  i\right)  .\label{pf.lem.Deltas.f=g.2}%
\end{equation}

However, the summation sign $\sum_{\substack{i,j\in\left[  n\right]
;\\\ell_{j}=m_{i}}}$ can be rewritten as $\sum_{\substack{j,i\in\left[
n\right]  ;\\m_{i}=\ell_{j}}}$ (since $\ell_{j}=m_{i}$ is equivalent to
$m_{i}=\ell_{j}$). Thus,%
\[
\sum_{\substack{i,j\in\left[  n\right]  ;\\\ell_{j}=m_{i}}}g\left(  j\right)
= \sum_{\substack{j,i\in\left[  n\right]  ;\\m_{i}=\ell_{j}}%
}\ \ \underbrace{g\left(  j\right)  }_{\substack{=f\left(  i\right)  \\\text{(by
(\ref{eq.lem.Deltas.f=g.as}))}}}=\sum_{\substack{j,i\in\left[  n\right]
;\\m_{i}=\ell_{j}}}f\left(  i\right)  .
\]
In other words, the left hand sides of the equalities
(\ref{pf.lem.Deltas.f=g.1}) and (\ref{pf.lem.Deltas.f=g.2}) are equal. Hence,
their right hand sides are equal as well. In other words, we have%
\[
\sum_{\substack{j\in\left[  n\right]  ;\\\ell_{j}\in\Delta\left(  \mu\right)
}}g\left(  j\right)  =\sum_{\substack{i\in\left[  n\right]  ;\\m_{i}\in
\Delta\left(  \lambda\right)  }}f\left(  i\right)  .
\]
This proves Lemma \ref{lem.Deltas.f=g}.
\end{proof}

\begin{proof}[Proof of Lemma \ref{lem.bj=i}.]
The definition of $b_i$ yields
\begin{equation}
b_i = \max\left\{  k\geq0\mid\lambda_{k}-k\geq\mu_i-i\right\} .
\label{pf.lem.bj=i.1}
\end{equation}

Lemma~\ref{lem.l-decrease} yields $\ell_{1}>\ell_{2}>\ell_{3}>\cdots$.
Hence, the numbers
$\ell_{1},\ell_{2},\ldots,\ell_j$ are $\geq\ell_j$, whereas the numbers
$\ell_{j+1},\ell_{j+2},\ell_{j+3},\ldots$ are not. Therefore, $\max\left\{
k\geq0\mid\ell_{k}\geq\ell_j\right\}  =j$. In view of
\begin{align*}
\ell_j  = m_i = \mu_i - i
\ \ \ \ \ \ \ \ \ \ \left(  \text{by the definition of }m_i\right)
\end{align*}
and%
\[
\ell_{k}=\lambda_{k}-k\ \ \ \ \ \ \ \ \ \ \left(  \text{by the definition of
}\ell_{k}\right)  ,
\]
we can rewrite this as
$\max\left\{  k\geq0 \mid \lambda_{k}-k\geq\mu_i - i\right\}  =j$.
Hence, (\ref{pf.lem.bj=i.1}) can be rewritten as
$b_i = j$. This proves Lemma \ref{lem.bj=i}.
\end{proof}

\begin{proof}[Proof of Lemma \ref{lem.Deltas.x}.]
If $i$ and $j$ are two positive integers
satisfying $m_{i}=\ell_{j}$, then $x_{b_{i}}=x_{j}$
(since Lemma \ref{lem.bj=i} yields $b_{i}=j$).

Hence, Lemma \ref{lem.Deltas.f=g} (applied to $f\left(  i\right)  =x_{b_{i}}$
and $g\left(  j\right)  =x_{j}$) yields%
\begin{align}
\sum_{\substack{i\in\left[  n\right]  ;\\m_{i}\in\Delta\left(  \lambda\right) }}x_{b_{i}}
&  =\sum_{\substack{j\in\left[  n\right]  ;\\\ell_{j}\in
\Delta\left(  \mu\right)  }}x_{j}=\underbrace{\sum_{j\in\left[  n\right]
}x_{j}}_{\substack{=x_{1}+x_{2}+\cdots+x_{n}\\=\sum_{i=1}^{n}x_{i}}%
}-\sum_{\substack{j\in\left[  n\right]  ;\\\ell_{j}\notin\Delta\left(
\mu\right)  }}x_{j}
 =\sum_{i=1}^{n}x_{i}
- \sum_{\substack{j\in\left[  n\right]  ;\\\ell
_{j}\notin\Delta\left(  \mu\right)  }}x_{j}. \nonumber
\end{align}

In other words,
\[
\sum_{i=1}^{n}x_{i}-\sum_{\substack{i\in\left[  n\right]  ;\\m_{i}\in
\Delta\left(  \lambda\right)  }}x_{b_{i}}
= \sum_{\substack{j\in\left[  n\right]  ;\\\ell
_{j}\notin\Delta\left(  \mu\right)  }}x_{j}
= \sum_{\substack{k\in\left[  n\right]  ;\\\ell
_k\notin\Delta\left(  \mu\right)  }}x_{k}
\]
(here, we have renamed the summation index $j$ as $k$).
This proves Lemma \ref{lem.Deltas.x}.
\end{proof}

\begin{proof}[Proof of Lemma \ref{lem.Deltas.mi+1+bi}.]
Set $k=m_{i}+1+b_{i}$.

We make the following two observations:

\begin{itemize}
\item We have
\begin{equation}
    \lambda_{b_{i}}\geq k \text{ if } b_{i}\geq1 .
\label{pf.lem.Deltas.mi+1+bi.obs1}
\end{equation}

[\textit{Proof:} Assume that $b_{i}\geq1$. Then, $b_{i}$ is a positive integer. Hence, Lemma \ref{lem.flagging-of-lm.uniprop}
(applied to $j=b_{i}$) yields that $\lambda_{b_{i}}-b_{i}\geq\mu_{i}-i$ (since
$b_{i}\leq b_{i}$). In view of $m_{i}=\mu_{i}-i$, we can rewrite this as
$\lambda_{b_{i}}-b_{i}\geq m_{i}$. In other words,
$m_{i}\leq\lambda_{b_{i}}-b_{i}$.

Furthermore, $m_{i}\notin\Delta\left(  \lambda\right)  =\left\{  \ell_{1}%
,\ell_{2},\ell_{3},\ldots\right\}  $. In other words, $m_{i}\neq\ell_{k}$ for
every positive integer $k$. In other words, $m_{i}\neq
\lambda_{k}-k$ for every positive integer $k$ (since
$\ell_{k}$ is defined to be $\lambda_{k}-k$). Applying this to $k=b_{i}$, we
obtain $m_{i}\neq\lambda_{b_{i}}-b_{i}$. Combined with $m_{i}\leq
\lambda_{b_{i}}-b_{i}$, this yields $m_{i}<\lambda_{b_{i}}-b_{i}$. In other
words, $m_{i}+b_{i}<\lambda_{b_{i}}$. Equivalently, $\lambda_{b_{i}}%
>m_{i}+b_{i}$. Since both sides of this equality are integers, we thus obtain
$\lambda_{b_{i}}\geq m_{i}+b_{i}+1=m_{i}+1+b_{i}=k$. This proves \eqref{pf.lem.Deltas.mi+1+bi.obs1}.]

\item We have
\begin{equation}
    \lambda_{j}<k \text{ for each integer } j>b_{i} .
\label{pf.lem.Deltas.mi+1+bi.obs2}
\end{equation}

[\textit{Proof:} We don't have $b_{i}+1\leq b_{i}$. Hence, Lemma
\ref{lem.flagging-of-lm.uniprop} (applied to $j=b_{i}+1$) yields that we don't
have $\lambda_{b_{i}+1}-\left(  b_{i}+1\right)  \geq\mu_{i}-i$ either. In
other words, we have $\lambda_{b_{i}+1}-\left(  b_{i}+1\right)  <\mu_{i}-i$.
Thus,%
\[
\lambda_{b_{i}+1}<\underbrace{\mu_{i}-i}_{\substack{=m_{i}\\\text{(by the
definition of }m_{i}\text{)}}}+\left(  b_{i}+1\right)  =m_{i}+\left(
b_{i}+1\right)  =m_{i}+1+b_{i}=k.
\]
Now, for each integer $j > b_i$, we have $j\geq b_{i}+1$ and thus
$\lambda_{j}\leq\lambda_{b_{i}+1}$ (since $\lambda_{1}\geq\lambda_{2}\geq\lambda_{3}\geq\cdots$)
and therefore $\lambda
_{j}\leq\lambda_{b_{i}+1}<k$.
This proves \eqref{pf.lem.Deltas.mi+1+bi.obs2}.]
\end{itemize}

Combining \eqref{pf.lem.Deltas.mi+1+bi.obs1} with \eqref{pf.lem.Deltas.mi+1+bi.obs2}, we see that $b_{i}$ is the largest $j\geq1$
satisfying $\lambda_{j}\geq k$
(where we agree that if no such $j$ exists, then we consider the largest such $j$ to be $0$).
In other words,
\begin{equation}
b_{i}=\max\left\{  j\geq1\mid\lambda_{j}\geq k\right\}  .
\label{pf.lem.Deltas.ybij.a.0}
\end{equation}

Also, \eqref{pf.lem.Deltas.mi+1+bi.obs2}
(applied to $j = b_i+1$) yields $\lambda_{b_i+1} < k$ (since $b_i+1 > b_i$),
so that $k > \lambda_{b_i+1} \geq 0$.
Thus, $k$ is a positive integer. In
other words, $k\geq1$.

From (\ref{eq.def.lambdat.3}), we thus obtain
\[
\lambda_{k}^{t}=\max\left\{  j\geq1\mid\lambda_{j}\geq k\right\}  =b_{i}%
\]
(by (\ref{pf.lem.Deltas.ybij.a.0})). Now, the definition of $\ell_{k}^{t}$
yields%
\[
\ell_{k}^{t}=\underbrace{\lambda_{k}^{t}}_{=b_{i}}-\underbrace{k}%
_{=m_{i}+1+b_{i}}=b_{i}-\left(  m_{i}+1+b_{i}\right)  =-1-m_{i}.
\]

We have thus shown that $k\geq1$ and $\ell_{k}^{t}=-1-m_{i}$.
In other words, $m_i + 1 + b_i \geq 1$ and
$\ell_{m_{i}+1+b_{i}}^{t}=-1-m_{i}$ (since $k=m_{i}+1+b_{i}$). This proves
Lemma \ref{lem.Deltas.mi+1+bi}.
\end{proof}

\begin{proof}[Proof of Lemma \ref{lem.ER-bb}.]
Recall that $\lambda_0 = \infty$ by convention.
We also set $\ell_0 := \infty$.

We defined $\ell_i$ by the equality $\ell_i = \lambda_i - i$ for all $i \geq 1$. This equality holds for $i = 0$ as well (since $\ell_0 = \infty$ and $\lambda_0 = \infty$), and thus it holds for all $i \in \NN$.
In other words, $\ell_k = \lambda_k - k$ for all $k \in \NN$.

The definition of $b_{j}$ yields%
\begin{align}
b_{j} &= \max\left\{  k\geq0\mid\lambda_{k}-k\geq\mu_{j}-j\right\} \nonumber \\
&= \max\left\{  k\geq0 \mid \ell_k\geq m_j\right\}
\label{pf.lem.ER-bb.j}
\end{align}
(since $\ell_k = \lambda_k - k$ and $m_j = \mu_j - j$).
Similarly, from the definition of $b_{j-1}$, we obtain
\begin{align}
b_{j-1} &= \max\left\{  k\geq0 \mid \ell_k\geq m_{j-1}\right\} .
\label{pf.lem.ER-bb.j-1.1}
\end{align}
However, $\mu_{j-1}=\mu_{j}$, so that
$\mu_{j-1}-\left(  j-1\right)  =\mu_{j}-\left(  j-1\right)  =\mu_{j}-j+1$.
In other words, $m_{j-1} = m_j + 1$
(since $m_{j-1} = \mu_{j-1}-\left(  j-1\right)$ and $m_j = \mu_j - j$).
Thus, the inequality $\ell_k \geq m_{j-1}$ (for any given $k\geq0$) is
equivalent to $\ell_k \geq m_j + 1$, which in turn is equivalent to
$\ell_k > m_j$ (since $m_j$ is not $\infty$).
Thus, we can rewrite (\ref{pf.lem.ER-bb.j-1.1}) as
\begin{equation}
b_{j-1} = \max\left\{  k\geq0 \mid \ell_k > m_j\right\}
.
\label{pf.lem.ER-bb.j-1.2}%
\end{equation}

Lemma~\ref{lem.l-decrease} yields $\ell_{1} > \ell_{2} > \ell_{3} > \cdots$. Since
$\ell_{0} > \ell_1$ (because $\ell_0 = \infty$ while $\ell_1$ is finite), we can extend this to
$\ell_{0}>\ell_{1}>\ell_{2}>\ell_{3}>\cdots$. \medskip

\textbf{(a)} Assume that $m_j \notin\Delta\left(  \lambda\right)  $.
Thus, $m_j \notin\Delta\left(  \lambda\right)  =\left\{  \ell_{1}%
,\ell_{2},\ell_{3},\ldots\right\}  $. In other words, for every $k\geq1$, we
have $m_j \neq \ell_k$ (by the definition of $\ell_{k}$).
This also holds for $k=0$ (since $\ell_0 = \infty$), and thus holds for
each $k\geq0$. In other words, we have $\ell_k \neq m_j$ for each
$k\geq0$.

But this, in turn, entails that the weak inequality $\ell_k \geq m_j$ is equivalent to the strict inequality $\ell_k > m_j$ for
any $k\geq0$. Hence, the right hand sides of the equalities
(\ref{pf.lem.ER-bb.j}) and (\ref{pf.lem.ER-bb.j-1.2}) are equal. Therefore,
the left hand sides of these equalities are equal as well. In other words,
$b_{j}=b_{j-1}$. This proves Lemma \ref{lem.ER-bb} \textbf{(a)}. \medskip

\textbf{(b)} Assume that $m_j \in\Delta\left(  \lambda\right)  $. Thus,
$m_j \in\Delta\left(  \lambda\right)  =\left\{  \ell_{1},\ell_{2}%
,\ell_{3},\ldots\right\}  $. In other words, there exists some $i\geq1$ such
that $m_j = \ell_{i}$. Consider this $i$.

We can rewrite (\ref{pf.lem.ER-bb.j}) as%
\begin{equation}
b_{j} = \max\left\{  k\geq0 \mid \ell_{k}\geq\ell_{i}\right\}
\label{pf.lem.ER-bb.b.j}
\end{equation}
(since $m_j=\ell_{i}$).
For the same reasons, we can rewrite (\ref{pf.lem.ER-bb.j-1.2}) as
\begin{equation}
b_{j-1} = \max\left\{  k\geq0 \mid \ell_{k}>\ell_{i}\right\}
.\label{pf.lem.ER-bb.b.j-1}
\end{equation}

But we have $\ell_{0}>\ell_{1}>\ell_{2}>\ell_{3}>\cdots$. Hence, the numbers
$\ell_{0},\ell_{1},\ldots,\ell_{i}$ are $\geq\ell_{i}$, whereas the numbers
$\ell_{i+1},\ell_{i+2},\ell_{i+3},\ldots$ are not. Therefore, $\max\left\{
k\geq0 \mid \ell_{k}\geq\ell_{i}\right\}  =i$. In view of this, we can rewrite
(\ref{pf.lem.ER-bb.b.j}) as $b_{j}=i$.

We have $\ell_{0}>\ell_{1}>\ell_{2}>\ell_{3}>\cdots$. Hence, the
numbers $\ell_{0},\ell_{1},\ldots,\ell_{i-1}$ are $>\ell_{i}$, whereas the
numbers $\ell_{i},\ell_{i+1},\ell_{i+2},\ldots$ are not. Therefore,
$\max\left\{  k\geq0 \mid \ell_{k}>\ell_{i}\right\}  =i-1$. In view of this, we
can rewrite (\ref{pf.lem.ER-bb.b.j-1}) as $b_{j-1}=i-1$. Hence, $i = b_{j-1} + 1$.

Therefore, $b_{j} = i = b_{j-1}+1$.
This proves Lemma \ref{lem.ER-bb} \textbf{(b)}.
\end{proof}

\begin{proof}
[Proof of Lemma \ref{lem.b-vs-bk}.]Recall that $\lambda_{0} = \infty$ by
convention. Let us also set $\ell_{0} := \infty$. We thus have
\[
\ell_{p} = \lambda_{p} - p \qquad\text{ for all } p \geq0
\]
(indeed, this is clear for $p = 0$, and is the definition of $\ell_{p}$ for $p
> 0$). Also, $\ell_{0}>\ell_{1}>\ell_{2}>\ell_{3}>\cdots$ (this is proved as
in the proof of Lemma \ref{lem.ER-bb} above).

Let $i$ be a positive integer. The definition of $m_{i}$ yields $m_{i}=\mu
_{i}-i$. Also, Lemma \ref{lem.mu+ki} yields $\left(  \mu^{+k}\right)  _{i}%
=\mu_{i}+\left[  k=i\right]  $. Subtracting $i$ from both sides of this
equality, we obtain%
\begin{align}
\left(  \mu^{+k}\right)  _{i}-i  & =\mu_{k}+\left[  k=i\right]
-i=\underbrace{\mu_{i}-i}_{=m_{i}}+\left[  k=i\right]  \nonumber\\
& =m_{i}+\left[  k=i\right]  .\label{pf.lem.b-vs-bk.mu-to-m}%
\end{align}

The definition of $b_{i}$ yields
\begin{align}
b_{i}= &  \ \max\left\{  k\geq0\mid\lambda_{k}-k\geq\mu_{i}-i\right\}
\nonumber\\
= &  \ \max\left\{  p\geq0\mid\lambda_{p}-p\geq\mu_{i}-i\right\}  \nonumber\\
= &  \ \max\left\{  p\geq0\mid\ell_{p}\geq\mu_{i}-i\right\}
\label{pf.lem.b-vs-bk.1a}%
\end{align}
(since each $p\geq0$ satisfies $\ell_{p}=\lambda_{p}-p$). The same argument
(applied to $\mu^{+k}$ and $b_{i}^{\ast}$ instead of $\mu$ and $b_{i}$) yields%
\[
b_{i}^{\ast}=\max\left\{  p\geq0\mid\ell_{p}\geq\left(  \mu^{+k}\right)
_{i}-i\right\}  .
\]
Using (\ref{pf.lem.b-vs-bk.mu-to-m}), we can rewrite this as
\begin{equation}
b_{i}^{\ast}=\max\left\{  p\geq0\mid\ell_{p}\geq m_{i}+\left[  k=i\right]
\right\}  .\label{pf.lem.b-vs-bk.2}%
\end{equation}

Also, using $\mu_{i}-i=m_{i}$, we can rewrite (\ref{pf.lem.b-vs-bk.1a}) as%
\begin{equation}
b_{i}=\max\left\{  p\geq0\mid\ell_{p}\geq m_{i}\right\}
.\label{pf.lem.b-vs-bk.1}%
\end{equation}
\medskip

\textbf{(c)} Assume that $i\neq k$. Then, $k\neq i$, so that $\left[
k=i\right]  =0$ and thus $m_{i}+\left[  k=i\right]  =m_{i}$. Hence, the right
hand sides of the equalities (\ref{pf.lem.b-vs-bk.2}) and
(\ref{pf.lem.b-vs-bk.1}) are equal. Therefore, their left hand sides are equal
as well. In other words, $b_{i}^{\ast}=b_{i}$. This proves Lemma
\ref{lem.b-vs-bk} \textbf{(c)}. \medskip

\textbf{(a)} Assume that $m_{k} \notin\Delta\left(  \lambda\right)  $. We must
prove that $b_{i}^{\ast}=b_{i}$.

If $i\neq k$, then this follows from part \textbf{(c)}.

Thus, we WLOG assume that $i=k$. Hence, $k=i$, so that $\left[  k=i\right]
=1$. Also, recall that $m_{k}\notin\Delta\left(  \lambda\right)  $. In other
words, $m_{i}\notin\Delta\left(  \lambda\right)  $ (since $k=i$).

We shall now prove the following:

\begin{statement}
\textit{Claim 1:} Let $p\geq0$ be an integer. Then, the statements
\textquotedblleft$\ell_{p}\geq m_{i}+\left[  k=i\right]  $\textquotedblright%
\ and \textquotedblleft$\ell_{p}\geq m_{i}$\textquotedblright\ are equivalent.
\end{statement}

\begin{proof}
[Proof of Claim 1.] This equivalence is obvious if $p=0$ (because in this
case, we have $\ell_{p}=\ell_{0}=\infty$, and thus both statements
\textquotedblleft$\ell_{p}\geq m_{i}+\left[  k=i\right]  $\textquotedblright%
\ and \textquotedblleft$\ell_{p}\geq m_{i}$\textquotedblright\ are true).
Hence, for the rest of this proof, we WLOG assume that $p\neq0$. Hence,
$p\geq1$ (since $p$ is an integer), and thus $\ell_{p}$ is an integer.

If we had $m_{i}=\ell_{p}$, then we would have $m_{i}=\ell_{p}\in\left\{
\ell_{1},\ell_{2},\ell_{3},\ldots\right\}  $ (since $p\geq1$ is an integer),
which would contradict $m_{i}\notin\Delta\left(  \lambda\right)  =\left\{
\ell_{1},\ell_{2},\ell_{3},\ldots\right\}  $. Hence, we cannot have
$m_{i}=\ell_{p}$. Thus, we have $m_{i}\neq\ell_{p}$. In other words, $\ell
_{p}\neq m_{i}$.

Now, we have the following chain of logical equivalences:%
\begin{align*}
&  \ \left(  \ell_{p}\geq m_{i}+\left[  k=i\right]  \right)  \\
&  \Longleftrightarrow\ \left(  \ell_{p}\geq m_{i}+1\right)
\ \ \ \ \ \ \ \ \ \ \left(  \text{since }\left[  k=i\right]  =1\right)  \\
&  \Longleftrightarrow\ \left(  \ell_{p}>m_{i}\right)
\ \ \ \ \ \ \ \ \ \ \left(  \text{since }\ell_{p}\text{ and }m_{i}\text{ are
integers}\right)  \\
&  \Longleftrightarrow\ \left(  \ell_{p}\geq m_{i}\right)
\ \ \ \ \ \ \ \ \ \ \left(  \text{since we know that }\ell_{p}\neq
m_{i}\right)  .
\end{align*}
In other words, the statements \textquotedblleft$\ell_{p}\geq m_{i}+\left[
k=i\right]  $\textquotedblright\ and \textquotedblleft$\ell_{p}\geq m_{i}%
$\textquotedblright\ are equivalent. This proves Claim 1.
\end{proof}

Now, the right hand sides of the equalities (\ref{pf.lem.b-vs-bk.2}) and
(\ref{pf.lem.b-vs-bk.1}) are equal (because Claim 1 shows that the statements
\textquotedblleft$\ell_{p}\geq m_{i}+\left[  k=i\right]  $\textquotedblright%
\ and \textquotedblleft$\ell_{p}\geq m_{i}$\textquotedblright\ are
equivalent). Therefore, their left hand sides are equal as well. In other
words, $b_{i}^{\ast}=b_{i}$. This proves Lemma \ref{lem.b-vs-bk} \textbf{(a)}.
\medskip

\textbf{(b)} Assume that $m_{k} \in\Delta\left(  \lambda\right)  $. We must
prove that $b_{i}^{\ast}=b_{i}-\left[  k=i\right]  $.

If $i\neq k$, then this follows from part \textbf{(c)} (since $i\neq k$
entails $k\neq i$ and thus $\left[  k=i\right]  =0$, so that $b_i - \ive{k=i} = b_i - 0 = b_i$,
but part \textbf{(c)} yields $b_{i}^{\ast} = b_{i} = b_i - \ive{k=i}$).
Thus, we WLOG assume that $i=k$. Hence, $k=i$,
so that $\left[  k=i\right]  =1$.

We have $i=k$, so that $m_{i}=m_{k}\in\Delta\left(  \lambda\right)  =\left\{
\ell_{1},\ell_{2},\ell_{3},\ldots\right\}  $. In other words, $m_{i}=\ell_{j}$
for some positive integer $j$. Consider this $j$.

But we have $\ell_{0}>\ell_{1}>\ell_{2}>\ell_{3}>\cdots$. Hence, the numbers
$\ell_{0},\ell_{1},\ldots,\ell_{j}$ are $\geq\ell_{j}$, whereas the numbers
$\ell_{j+1},\ell_{j+2},\ell_{j+3},\ldots$ are not. Therefore, $\max\left\{
p\geq0\mid\ell_{p}\geq\ell_{j}\right\}  =j$. In view of $m_{i}=\ell_{j}$, we
can rewrite this as%
\[
\max\left\{  p\geq0\mid\ell_{p}\geq m_{i}\right\}  =j.
\]
This allows us to rewrite (\ref{pf.lem.b-vs-bk.1}) as $b_{i}=j$.

Again, recall that $\ell_{0}>\ell_{1}>\ell_{2}>\ell_{3}>\cdots$. Hence, the
numbers $\ell_{0},\ell_{1},\ldots,\ell_{j-1}$ are $>\ell_{j}$, whereas the
numbers $\ell_{j},\ell_{j+1},\ell_{j+2},\ldots$ are not. Therefore,
$\max\left\{  p\geq0\mid\ell_{p}>\ell_{j}\right\}  =j-1$. In other words,
\[
\max\left\{  p\geq0\mid\ell_{p}\geq\ell_{j}+1\right\}  =j-1
\]
(since the inequality $\ell_{p}>\ell_{j}$ is equivalent to $\ell_{p}\geq
\ell_{j}+1$). We can rewrite this further as%
\[
\max\left\{  p\geq0\mid\ell_{p}\geq m_{i}+\left[  k=i\right]  \right\}  =j-1
\]
(since $m_{i}+\underbrace{\left[  k=i\right]  }_{=1}=\underbrace{m_{i}}%
_{=\ell_{j}}+1=\ell_{j}+1$). This allows us to rewrite (\ref{pf.lem.b-vs-bk.2}%
) as $b_{i}^{\ast}=j-1$. Comparing this with $\underbrace{b_{i}}%
_{=j}-\underbrace{\left[  k=i\right]  }_{=1}=j-1$, we obtain $b_{i}^{\ast
}=b_{i}-\left[  k=i\right]  $. This proves Lemma \ref{lem.b-vs-bk}
\textbf{(b)}.
\end{proof}

\subsection{To Section \ref{sec.kpf}}

\begin{proof}[Proof of Lemma \ref{lem.nu-sub-lambda-irrelevant}.]
It clearly suffices to prove that $\mathbf{s}_{\lambda}\left[  \nu\right]  =0$
for any partition $\nu$ that does not satisfy $\nu\subseteq\lambda$. But this
is obvious, because if $\nu$ does not satisfy $\nu\subseteq\lambda$, then the
set $\mathcal{E}\left(  \lambda/\nu\right)  $ is empty (by Lemma
\ref{lem.exc.empty}, applied to $\mu=\nu$), and thus we have%
\[
\mathbf{s}_{\lambda}\left[  \nu\right]  =\sum_{D\in\mathcal{E}\left(
\lambda/\nu\right)  }\ \ \prod_{\left(  i,j\right)  \in D}\left(  x_{i}%
+y_{j}\right)  =\left(  \text{empty sum}\right)  =0.
\]

\end{proof}

\begin{proof}[Proof of Lemma \ref{lem.konval-RHS-as-sum}.]
Lemma \ref{lem.nu=mu+k} yields that the partitions $\nu$ that satisfy
$\mu\lessdot\nu$ are precisely the partitions $\mu^{+k}$ for the elements
$k\in\operatorname*{ER}\left(  \mu\right)  $. Hence,
\begin{equation}
\sum_{\mu\lessdot\nu}\mathbf{s}_{\lambda}\left[  \nu\right]  =\sum
_{k\in\operatorname*{ER}\left(  \mu\right)  }\mathbf{s}_{\lambda}\left[
\mu^{+k}\right]  .
\label{pf.lem.konval-RHS-as-sum.1}
\end{equation}
(since the partitions $\mu^{+k}$ for different $k \in \operatorname{ER}\tup{\mu}$ are furthermore distinct).

Recall that $\mu_n = 0$, so that $\mu_j = 0$ for all $j \geq n$.
Hence, every integer $k>n$ satisfies $\mu_{k-1}=\mu_{k}$ (since
$\mu_{k-1}=0$ and $\mu_{k}=0$) and thus $k\notin\operatorname*{ER}\left(
\mu\right)  $. In other words, every integer in $\operatorname*{ER}\left(
\mu\right)  $ is $\leq n$. In other words, $\operatorname*{ER}\left(
\mu\right)  \subseteq\left[  n\right]  $. However,
\begin{align*}
\sum_{k=1}^{n}\mathbf{s}_{\lambda}\left[  \mu^{+k}\right]    &
=\underbrace{\sum_{\substack{k\in\left[  n\right]  ;\\k\in\operatorname*{ER}%
\left(  \mu\right)  }}}_{\substack{=\sum_{k\in\operatorname*{ER}\left(
\mu\right)  }\\\text{(since }\operatorname*{ER}\left(  \mu\right)
\subseteq\left[  n\right]  \text{)}}}\mathbf{s}_{\lambda}\left[  \mu
^{+k}\right]  +\sum_{\substack{k\in\left[  n\right]  ;\\k\notin%
\operatorname*{ER}\left(  \mu\right)  }}\underbrace{\mathbf{s}_{\lambda
}\left[  \mu^{+k}\right]  }_{\substack{=0\\\text{(by the definition of
}\mathbf{s}_{\lambda}\left[  \mu^{+k}\right]  \text{,}\\\text{since }%
k\notin\operatorname*{ER}\left(  \mu\right)  \text{ shows that }\mu
^{+k}\\\text{is not a partition)}}}\\
& =\sum_{k\in\operatorname*{ER}\left(  \mu\right)  }\mathbf{s}_{\lambda
}\left[  \mu^{+k}\right]  .
\end{align*}
Comparing this with (\ref{pf.lem.konval-RHS-as-sum.1}), we find%
\[
\sum_{k=1}^{n}\mathbf{s}_{\lambda}\left[  \mu^{+k}\right]  =\sum_{\mu
\lessdot\nu}\mathbf{s}_{\lambda}\left[  \nu\right]  =\sum_{\mu\lessdot
\nu\subseteq\lambda}\mathbf{s}_{\lambda}\left[  \nu\right]
\]
(by Lemma \ref{lem.nu-sub-lambda-irrelevant}). This proves Lemma
\ref{lem.konval-RHS-as-sum}.
\end{proof}

\begin{proof}
[Proof of Lemma \ref{lem.mu+k.sdet}.]Let $\mathbf{b}^{\ast}=\left(
b_{1}^{\ast},b_{2}^{\ast},b_{3}^{\ast},\ldots\right)  $ be the flagging
induced by $\lambda/\mu^{+k}$. \medskip

\textbf{(a)} Assume that $m_{k}\notin\Delta\left(  \lambda\right)  $.

We are in one of the following two cases:

\textit{Case 1:} We have $k\in\operatorname*{ER}\left(  \mu\right)  $.

\textit{Case 2:} We have $k\notin\operatorname*{ER}\left(  \mu\right)  $.

Let us first consider Case 1. In this case, we have $k\in\operatorname*{ER}%
\left(  \mu\right)  $. Hence, $\mu^{+k}$ is a partition.

For each $i\geq1$, we have $b_{i}^{\ast}=b_{i}$ (by Lemma \ref{lem.b-vs-bk}
\textbf{(a)}, since $m_{k}\notin\Delta\left(  \lambda\right)  $) and $\left(
\mu^{+k}\right)  _{i}=\mu_{i}+\left[  k=i\right]  $ (by Lemma \ref{lem.mu+ki}).

However, $\mu^{+k}$ is a partition, and $\mathbf{b}^{\ast}=\left(  b_{1}%
^{\ast},b_{2}^{\ast},b_{3}^{\ast},\ldots\right)  $ is the flagging induced by
$\lambda/\mu^{+k}$. Thus, (\ref{eq.snu=.det}) (applied to $\mu^{+k}$, $\mathbf{b}^{\ast}$ and
$b_{i}^{\ast}$ instead of $\mu$, $\mathbf{b}$ and $b_{i}$) yields%
\begin{align*}
\mathbf{s}_{\lambda}\left[  \mu^{+k}\right]   &  =\det\left(  h\left(
\underbrace{\left(  \mu^{+k}\right)  _{i}}_{=\mu_{i}+\left[  k=i\right]
}-i+j,\ \ \underbrace{b_{i}^{\ast}}_{=b_{i}},\ \ 1-j\right)  \right)
_{i,j\in\left[  n\right]  }\\
&  =\det\left(  h\left(  \mu_{i}-i+j+\left[  k=i\right]  ,\ \ b_{i}%
,\ \ 1-j\right)  \right)  _{i,j\in\left[  n\right]  }.
\end{align*}
Thus, Lemma \ref{lem.mu+k.sdet} \textbf{(a)} is proved in Case 1.

Let us now consider Case 2. In this case, we have $k\notin\operatorname*{ER}%
\left(  \mu\right)  $. Hence, $k\neq1$ and $\mu_{k-1}=\mu_{k}$. Thus, Lemma
\ref{lem.ER-bb} \textbf{(a)} (applied to $j=k$) yields $b_{k}=b_{k-1}$ (since
$m_{k}\notin\Delta\left(  \lambda\right)  $). In other words, $b_{k-1}=b_{k}$.

Now, consider the matrix%
\[
\left(  h\left(  \mu_{i}-i+j+\left[  k=i\right]  ,\ \ b_{i},\ \ 1-j\right)
\right)  _{i,j\in\left[  n\right]  }.
\]
The $\left(  k-1\right)  $-st and $k$-th rows of this matrix are
identical\footnote{\textit{Proof.} For each $j\in\left[  n\right]  $, the
$j$-th entry of the $\left(  k-1\right)  $-st row of this matrix is%
\begin{align*}
& h\left(  \underbrace{\mu_{k-1}}_{=\mu_{k}}-\left(  k-1\right)
+j+\underbrace{\left[  k=k-1\right]  }_{=0},\ \ \underbrace{b_{k-1}}_{=b_{k}%
},\ \ 1-j\right)  \\
& =h\left(  \underbrace{\mu_{k}-\left(  k-1\right)  +j+0}_{=\mu_{k}%
-k+j+1},\ \ b_{k},\ \ 1-j\right)  =h\left(  \mu_{k}-k+j+1,\ \ b_{k}%
,\ \ 1-j\right)  ,
\end{align*}
whereas the $j$-th entry of the $k$-th row of this matrix is%
\begin{align*}
h\left(  \mu_{k}-k+j+\underbrace{\left[  k=k\right]  }_{=1},\ \ b_{k}%
,\ \ 1-j\right)
 =h\left(  \mu_{k}-k+j+1,\ \ b_{k},\ \ 1-j\right)  .
\end{align*}
These two entries are clearly equal. Thus, the $\left(  k-1\right)  $-st and
$k$-th rows of this matrix are identical.}. Hence, this matrix has two equal
rows, so that its determinant vanishes. In other words,%
\[
\det\left(  h\left(  \mu_{i}-i+j+\left[  k=i\right]  ,\ \ b_{i}%
,\ \ 1-j\right)  \right)  _{i,j\in\left[  n\right]  }=0.
\]

On the other hand, $\mu^{+k}$ is not a partition (since $k\notin%
\operatorname*{ER}\left(  \mu\right)  $), so that $\mathbf{s}_{\lambda}\left[
\mu^{+k}\right]  =0$ (by definition of $\mathbf{s}_{\lambda}\left[  \mu
^{+k}\right]  $). Comparing this with the preceding equality, we obtain%
\[
\mathbf{s}_{\lambda}\left[  \mu^{+k}\right]  =\det\left(  h\left(  \mu
_{i}-i+j+\left[  k=i\right]  ,\ \ b_{i},\ \ 1-j\right)  \right)
_{i,j\in\left[  n\right]  }.
\]
Thus, Lemma \ref{lem.mu+k.sdet} \textbf{(a)} is proved in Case 2.

We have now proved Lemma \ref{lem.mu+k.sdet} \textbf{(a)} in both Cases 1 and
2, so that its proof is complete. \medskip

\textbf{(b)} Assume that $m_{k}\in\Delta\left(  \lambda\right)  $.

We are in one of the following two cases:

\textit{Case 1:} We have $k\in\operatorname*{ER}\left(  \mu\right)  $.

\textit{Case 2:} We have $k\notin\operatorname*{ER}\left(  \mu\right)  $.

Let us first consider Case 1. In this case, we have $k\in\operatorname*{ER}%
\left(  \mu\right)  $. Hence, $\mu^{+k}$ is a partition.

For each $i\geq1$, we have $b_{i}^{\ast}=b_{i}-\left[  k=i\right]  $ (by Lemma
\ref{lem.b-vs-bk} \textbf{(b)}, since $m_{k}\in\Delta\left(  \lambda\right)
$) and $\left(  \mu^{+k}\right)  _{i}=\mu_{i}+\left[  k=i\right]  $ (by Lemma
\ref{lem.mu+ki}).

However, $\mu^{+k}$ is a partition, and $\mathbf{b}^{\ast}=\left(  b_{1}%
^{\ast},b_{2}^{\ast},b_{3}^{\ast},\ldots\right)  $ is the flagging induced by
$\lambda/\mu^{+k}$. Thus, (\ref{eq.snu=.det}) (applied to $\mu^{+k}$, $\mathbf{b}^{\ast}$ and
$b_{i}^{\ast}$ instead of $\mu$, $\mathbf{b}$ and $b_{i}$) yields%
\begin{align*}
\mathbf{s}_{\lambda}\left[  \mu^{+k}\right]   &  =\det\left(  h\left(
\underbrace{\left(  \mu^{+k}\right)  _{i}}_{=\mu_{i}+\left[  k=i\right]
}-i+j,\ \ \underbrace{b_{i}^{\ast}}_{=b_{i}-\left[  k=i\right]  }%
,\ \ 1-j\right)  \right)  _{i,j\in\left[  n\right]  }\\
&  =\det\left(  h\left(  \mu_{i}-i+j+\left[  k=i\right]  ,\ \ b_{i}-\left[
k=i\right]  ,\ \ 1-j\right)  \right)  _{i,j\in\left[  n\right]  }.
\end{align*}
Thus, Lemma \ref{lem.mu+k.sdet} \textbf{(b)} is proved in Case 1.

Let us now consider Case 2. In this case, we have $k\notin\operatorname*{ER}%
\left(  \mu\right)  $. Hence, $k\neq1$ and $\mu_{k-1}=\mu_{k}$. Thus, Lemma
\ref{lem.ER-bb} \textbf{(b)} (applied to $j=k$) yields $b_{k}=b_{k-1}+1$
(since $m_{k}\in\Delta\left(  \lambda\right)  $). In other words,
$b_{k-1}=b_{k}-1$.

Now, consider the matrix%
\[
\left(  h\left(  \mu_{i}-i+j+\left[  k=i\right]  ,\ \ b_{i}-\left[
k=i\right]  ,\ \ 1-j\right)  \right)  _{i,j\in\left[  n\right]  }.
\]
The $\left(  k-1\right)  $-st and $k$-th rows of this matrix are
identical\footnote{\textit{Proof.} For each $j\in\left[  n\right]  $, the
$j$-th entry of the $\left(  k-1\right)  $-st row of this matrix is%
\begin{align*}
& h\left(  \underbrace{\mu_{k-1}}_{=\mu_{k}}-\left(  k-1\right)
+j+\underbrace{\left[  k=k-1\right]  }_{=0},\ \ \underbrace{b_{k-1}}%
_{=b_{k}-1}-\underbrace{\left[  k=k-1\right]  }_{=0},\ \ 1-j\right)  \\
& =h\left(  \underbrace{\mu_{k}-\left(  k-1\right)  +j}_{=\mu_{k}%
-k+j+1},\ \ b_{k}-1,\ \ 1-j\right)  =h\left(  \mu_{k}-k+j+1,\ \ b_{k}%
-1,\ \ 1-j\right)  ,
\end{align*}
whereas the $j$-th entry of the $k$-th row of this matrix is%
\begin{align*}
h\left(  \mu_{k}-k+j+\underbrace{\left[  k=k\right]  }_{=1},\ \ b_{k}%
-\underbrace{\left[  k=k\right]  }_{=1},\ \ 1-j\right)
= h\left(  \mu_{k}-k+j+1,\ \ b_{k}-1,\ \ 1-j\right)  .
\end{align*}
These two entries are clearly equal. Thus, the $\left(  k-1\right)  $-st and
$k$-th rows of this matrix are identical.}. Hence, this matrix has two equal
rows, so that its determinant vanishes. In other words,%
\[
\det\left(  h\left(  \mu_{i}-i+j+\left[  k=i\right]  ,\ \ b_{i}-\left[
k=i\right]  ,\ \ 1-j\right)  \right)  _{i,j\in\left[  n\right]  }=0.
\]

On the other hand, $\mu^{+k}$ is not a partition (since $k\notin%
\operatorname*{ER}\left(  \mu\right)  $), so that $\mathbf{s}_{\lambda}\left[
\mu^{+k}\right]  =0$ (by definition of $\mathbf{s}_{\lambda}\left[  \mu
^{+k}\right]  $). Comparing this with the preceding equality, we obtain%
\[
\mathbf{s}_{\lambda}\left[  \mu^{+k}\right]  =\det\left(  h\left(  \mu
_{i}-i+j+\left[  k=i\right]  ,\ \ b_{i}-\left[  k=i\right]  ,\ \ 1-j\right)
\right)  _{i,j\in\left[  n\right]  }.
\]
Thus, Lemma \ref{lem.mu+k.sdet} \textbf{(b)} is proved in Case 2.

We have now proved Lemma \ref{lem.mu+k.sdet} \textbf{(b)} in both Cases 1 and
2, so that its proof is complete.
\end{proof}

\begin{proof}[Proof of Lemma \ref{lem.k-th-det-0}.]
Recall that $\mathbf{b}=\left(
b_{1},b_{2},b_{3},\ldots\right)  $ is the flagging induced by $\lambda/\mu$,
and that we have $\mu=\left(  \mu_{1},\mu_{2},\ldots,\mu_{n}\right)  $. Hence,
Corollary \ref{cor.slam.det} yields
\[
\mathbf{s}_{\lambda}\left[  \mu\right]  =\det\left(  \underbrace{h\left(
\mu_{i}-i+j,\ \ b_{i},\ \ 1-j\right)  }_{\substack{=u_{i,j}\\\text{(by the
definition of }u_{i,j}\text{)}}}\right)  _{i,j\in\left[  n\right]  }%
=\det\left(  u_{i,j}\right)  _{i,j\in\left[  n\right]  }.
\]
This proves (\ref{eq.lem.k-th-det-0.0}). However, the same
argument can be applied to $n-1$ instead of $n$ (since $\mu=\left(  \mu
_{1},\mu_{2},\ldots,\mu_{n-1}\right)  $), and thus we obtain
(\ref{eq.lem.k-th-det-0.-1}).
\end{proof}

\begin{proof}
[Proof of Lemma \ref{lem.k-th-det-1}.]
We first recall that every $i\in\left[  n\right]  $ satisfies%
\[
\left[  k=i\right]  =%
\begin{cases}
0, & \text{if }i\neq k;\\
1, & \text{if }i=k.
\end{cases}
\]
Hence, every $i,j\in\left[  n\right]  $ satisfy%
\begin{align}
& u_{i,j+\left[  k=i\right]  }-p_{i}u_{i,j}\left[  k=i\right]  \nonumber\\
& =%
\begin{cases}
u_{i,j+0}-p_{i}u_{i,j}0, & \text{if }i\neq k;\\
u_{i,j+1}-p_{i}u_{i,j}1, & \text{if }i=k
\end{cases}
\nonumber\\
& =%
\begin{cases}
u_{i,j}, & \text{if }i\neq k;\\
u_{i,j+1}-p_{i}u_{i,j}, & \text{if }i=k
\end{cases}
\label{pf.lem.k-th-det-1.prepare}%
\end{align}
(since $u_{i,j+0}-p_{i}u_{i,j}0=u_{i,j+0}=u_{i,j}$ and $u_{i,j}1=u_{i,j}$).

We shall use the notation $A_{i,j}$ for
the $\left(  i,j\right)  $-th entry of a matrix $A$.

We are in one of the following two cases:

\textit{Case 1:} We have $m_{k}\in\Delta\left(  \lambda\right)  $.

\textit{Case 2:} We have $m_{k}\notin\Delta\left(  \lambda\right)  $.

Let us first consider Case 1. In this case, we have $m_{k}\in\Delta\left(
\lambda\right)  $. Thus, Lemma \ref{lem.mu+k.sdet} \textbf{(b)} yields
\begin{equation}
\mathbf{s}_{\lambda}\left[  \mu^{+k}\right]  =\det\left(  h\left(  \mu
_{i}-i+j+\left[  k=i\right]  ,\ \ b_{i}-\left[  k=i\right]  ,\ \ 1-j\right)
\right)  _{i,j\in\left[  n\right]  }.\label{pf.lem.k-th-det-1.c1.1}%
\end{equation}
Let $A$ denote the matrix on the right hand side of this equality. Thus,
(\ref{pf.lem.k-th-det-1.c1.1}) rewrites as%
\begin{equation}
\mathbf{s}_{\lambda}\left[  \mu^{+k}\right]  =\det
A.\label{pf.lem.k-th-det-1.c1.1W}%
\end{equation}

We shall now compute the $\left(  i,j\right)  $-th entry
\[
A_{i,j}=h\left(  \mu_{i}-i+j+\left[  k=i\right]  ,\ \ b_{i}-\left[
k=i\right]  ,\ \ 1-j\right)
\]
of the matrix $A$ for any two numbers $i,j\in\left[  n\right]  $:

\begin{itemize}
\item When $i\neq k$, we have $\left[  k=i\right]  =0$ and thus%
\begin{align}
A_{i,j}  & =h\left(  \mu_{i}-i+j+\underbrace{\left[  k=i\right]  }%
_{=0},\ \ b_{i}-\underbrace{\left[  k=i\right]  }_{=0},\ \ 1-j\right)
\nonumber\\
& =h\left(  \mu_{i}-i+j,\ \ b_{i},\ \ 1-j\right)  =u_{i,j}%
\label{pf.lem.k-th-det-1.c1.Aij1}%
\end{align}
(by the definition of $u_{i,j}$).

\item When $i=k$, we have $\left[  k=i\right]  =1$ and $b_{i}>0$%
\ \ \ \ \footnote{\textit{Proof.} We have $i=k$, thus $m_{i}=m_{k}\in
\Delta\left(  \lambda\right)  =\left\{  \ell_{1},\ell_{2},\ell_{3}%
,\ldots\right\}  $. Hence, there exists some $p\geq1$ such that $m_{i}%
=\ell_{p}$. Consider this $p$. Thus, Lemma \ref{lem.bj=i} (applied to $p$ instead of $j$)
yields $b_{i}=p\geq1>0$.}, and thus%
\begin{align}
A_{i,j}  & =h\left(  \mu_{i}-i+j+\underbrace{\left[  k=i\right]  }%
_{=1},\ \ b_{i}-\underbrace{\left[  k=i\right]  }_{=1},\ \ 1-j\right)
\nonumber\\
& =h\left(  \mu_{i}-i+j+1,\ \ b_{i}-1,\ \ 1-j\right)  \nonumber\\
& =\underbrace{h\left(  \mu_{i}-i+j+1,\ \ b_{i},\ \ -j\right)  }%
_{\substack{=u_{i,j+1}\\\text{(by the definition of }u_{i,j+1}\text{)}%
}}\nonumber\\
& \ \ \ \ \ \ \ \ \ \ -\left(  x_{b_{i}}+\underbrace{y_{1-j}}%
_{\substack{=0\\\text{(since }1-j\leq0\text{)}}}\right)  \cdot
\underbrace{h\left(  \mu_{i}-i+j,\ \ b_{i},\ \ 1-j\right)  }%
_{\substack{=u_{i,j}\\\text{(by the definition of }u_{i,j}\text{)}%
}}\nonumber\\
& \ \ \ \ \ \ \ \ \ \ \ \ \ \ \ \ \ \ \ \ \left(
\begin{array}
[c]{c}%
\text{by Corollary \ref{2hprop}, applied to }a=\mu_{i}-i+j+1\\
\text{and }b=b_{i}\text{ and }c=1-j\text{ (since }b_i>0\text{)}%
\end{array}
\right)  \nonumber\\
& =u_{i,j+1}-\underbrace{x_{b_{i}}}_{\substack{=p_{i}\\\text{(by the
definition of }p_{i}\text{,}\\\text{since }i=k\text{ entails }m_{i}=m_{k}%
\in\Delta\left(  \lambda\right)  \text{)}}}u_{i,j}\nonumber\\
& =u_{i,j+1}-p_{i}u_{i,j}.\label{pf.lem.k-th-det-1.c1.Aij2}%
\end{align}

\end{itemize}

Combining the formulas (\ref{pf.lem.k-th-det-1.c1.Aij1}) and
(\ref{pf.lem.k-th-det-1.c1.Aij2}), we conclude that the $\left(  i,j\right)
$-th entry of $A$ always equals%
\[
A_{i,j}
=
\begin{cases}
u_{i,j}, & \text{if }i\neq k;\\
u_{i,j+1}-p_{i}u_{i,j}, & \text{if }i=k
\end{cases}
\quad =
u_{i,j+\left[  k=i\right]  }-p_{i}u_{i,j}\left[  k=i\right]
\]
(by \eqref{pf.lem.k-th-det-1.prepare}).
Hence,%
\[
A=\left(  u_{i,j+\left[  k=i\right]  }-p_{i}u_{i,j}\left[  k=i\right]
\right)  _{i,j\in\left[  n\right]  }.
\]
Therefore, (\ref{pf.lem.k-th-det-1.c1.1W}) rewrites as%
\[
\mathbf{s}_{\lambda}\left[  \mu^{+k}\right]  =\det\left(  u_{i,j+\left[
k=i\right]  }-p_{i}u_{i,j}\left[  k=i\right]  \right)  _{i,j\in\left[
n\right]  }.
\]
Thus, Lemma \ref{lem.k-th-det-1} is proved in Case 1.

Let us now consider Case 2. In this case, we have $m_{k}\notin\Delta\left(
\lambda\right)  $. Thus, Lemma \ref{lem.mu+k.sdet} \textbf{(a)} yields
\begin{equation}
\mathbf{s}_{\lambda}\left[  \mu^{+k}\right]  =\det\left(  h\left(  \mu
_{i}-i+j+\left[  k=i\right]  ,\ \ b_{i},\ \ 1-j\right)  \right)
_{i,j\in\left[  n\right]  }.\label{pf.lem.k-th-det-1.c1.2}%
\end{equation}
Let $A$ denote the matrix on the right hand side of this equality. Thus,
(\ref{pf.lem.k-th-det-1.c1.2}) rewrites as%
\begin{equation}
\mathbf{s}_{\lambda}\left[  \mu^{+k}\right]  =\det
A.\label{pf.lem.k-th-det-1.c1.2W}%
\end{equation}

We shall now compute the $\left(  i,j\right)  $-th entry
\[
A_{i,j}=h\left(  \mu_{i}-i+j+\left[  k=i\right]  ,\ \ b_{i},\ \ 1-j\right)
\]
of the matrix $A$ for any two numbers $i,j\in\left[  n\right]  $:

\begin{itemize}
\item When $i\neq k$, we have $\left[  k=i\right]  =0$ and thus%
\begin{align}
A_{i,j}  & =h\left(  \mu_{i}-i+j+\underbrace{\left[  k=i\right]  }%
_{=0},\ \ b_{i},\ \ 1-j\right)  \nonumber\\
& =h\left(  \mu_{i}-i+j,\ \ b_{i},\ \ 1-j\right)  =u_{i,j}%
\label{pf.lem.k-th-det-1.c2.Aij1}%
\end{align}
(by the definition of $u_{i,j}$).

\item When $i=k$, we have $\left[  k=i\right]  =1$ and thus
\begin{align}
A_{i,j}  & =h\left(  \mu_{i}-i+j+\underbrace{\left[  k=i\right]  }%
_{=1},\ \ b_{i},\ \ \underbrace{1-j}_{=-j+1}\right)  \nonumber\\
& =h\left(  \mu_{i}-i+j+1,\ \ b_{i},\ \ -j+1\right)  \nonumber\\
& =\underbrace{h\left(  \mu_{i}-i+j+1,\ \ b_{i},\ \ -j\right)  }%
_{\substack{=u_{i,j+1}\\\text{(by the definition of }u_{i,j+1}\text{)}%
}}\nonumber\\
& \ \ \ \ \ \ \ \ \ \ +\left(  y_{\left(  \mu_{i}-i+j\right)  +b_{i}+\left(
-j\right)  +1}-\underbrace{y_{-j+1}}_{\substack{=0\\\text{(since }%
-j+1\leq0\text{)}}}\right)  \cdot\underbrace{h\left(  \mu_{i}-i+j,\ \ b_{i}%
,\ \ -j+1\right)  }_{\substack{=u_{i,j}\\\text{(by the definition of }%
u_{i,j}\text{)}}}\nonumber\\
& \ \ \ \ \ \ \ \ \ \ \ \ \ \ \ \ \ \ \ \ \left(
\begin{array}
[c]{c}%
\text{by Corollary \ref{3hprop2},} \\ \text{applied to }a=\mu_{i}-i+j\text{ and }b=b_{i}\text{ and }c=-j
\end{array}
\right)  \nonumber\\
& =u_{i,j+1}+\underbrace{y_{\left(  \mu_{i}-i+j\right)  +b_{i}+\left(
-j\right)  +1}}_{\substack{=y_{m_{i}+1+b_{i}}\\\text{(since }\left(  \mu
_{i}-i+j\right)  +b_{i}+\left(  -j\right)  +1=\left(  \mu_{i}-i\right)
+1+b_{i}=m_{i}+1+b_{i}\\\text{(because }\mu_{i}-i=m_{i}\text{))}}}\cdot\,
u_{i,j}\nonumber\\
& =u_{i,j+1}+y_{m_{i}+1+b_{i}}\cdot u_{i,j}=u_{i,j+1}-\underbrace{\left(
-y_{m_{i}+1+b_{i}}\right)  }_{\substack{=p_{i}\\\text{(by the definition of
}p_{i}\text{,}\\\text{since }i=k\text{ entails }m_{i}=m_{k}\notin\Delta\left(
\lambda\right)  \text{)}}}\cdot\, u_{i,j}\nonumber\\
& =u_{i,j+1}-p_{i}u_{i,j}.\label{pf.lem.k-th-det-1.c2.Aij2}%
\end{align}

\end{itemize}

Combining the formulas (\ref{pf.lem.k-th-det-1.c2.Aij1}) and
(\ref{pf.lem.k-th-det-1.c2.Aij2}), we conclude that the $\left(  i,j\right)
$-th entry of $A$ always equals%
\[
A_{i,j}
=
\begin{cases}
u_{i,j}, & \text{if }i\neq k;\\
u_{i,j+1}-p_{i}u_{i,j}, & \text{if }i=k
\end{cases}
\quad =
u_{i,j+\left[  k=i\right]  }-p_{i}u_{i,j}\left[  k=i\right]
\]
(by \eqref{pf.lem.k-th-det-1.prepare}).
Hence,%
\[
A=\left(  u_{i,j+\left[  k=i\right]  }-p_{i}u_{i,j}\left[  k=i\right]
\right)  _{i,j\in\left[  n\right]  }.
\]

Therefore, (\ref{pf.lem.k-th-det-1.c1.2W}) rewrites as%
\[
\mathbf{s}_{\lambda}\left[  \mu^{+k}\right]  =\det\left(  u_{i,j+\left[
k=i\right]  }-p_{i}u_{i,j}\left[  k=i\right]  \right)  _{i,j\in\left[
n\right]  }.
\]
Thus, Lemma \ref{lem.k-th-det-1} is proved in Case 2.

We have now proved Lemma \ref{lem.k-th-det-1} in both Cases 1 and 2.
Consequently, Lemma \ref{lem.k-th-det-1} always holds.
\end{proof}

\begin{proof}[Proof of Lemma \ref{lem.k-th-det-2}.]
We have%
\begin{align*}
& \sum_{k=1}^{n}\mathbf{s}_{\lambda}\left[  \mu^{+k}\right]  \\
& =\sum_{k=1}^{n}\det\left(  u_{i,j+\left[  k=i\right]  }-p_{i}u_{i,j}\left[
k=i\right]  \right)  _{i,j\in\left[  n\right]  }\ \ \ \ \ \ \ \ \ \ \left(
\text{by Lemma \ref{lem.k-th-det-1}}\right)  \\
& =\det\left(  u_{i,j+\left[  n=j\right]  }\right)  _{i,j\in\left[  n\right]
}-\left(  \sum_{k=1}^{n}p_{k}\right)  \underbrace{\det\left(  u_{i,j}\right)
_{i,j\in\left[  n\right]  }}_{\substack{=\mathbf{s}_{\lambda}\left[
\mu\right]  \\\text{(by (\ref{eq.lem.k-th-det-0.0}))}}%
}\ \ \ \ \ \ \ \ \ \ \left(  \text{by Lemma \ref{lem.det-uc}}\right)  \\
& =\det\left(  u_{i,j+\left[  n=j\right]  }\right)  _{i,j\in\left[  n\right]
}-\left(  \sum_{k=1}^{n}p_{k}\right)  \mathbf{s}_{\lambda}\left[  \mu\right]
.
\end{align*}
This proves the lemma.
\end{proof}

\begin{proof}
[Proof of Lemma \ref{lem.k-th-det-3}.]For each $\ell\in\left[  n-1\right]  $,
we have $\underbrace{\mu_{n}}_{=0}-\,n+\ell=-n+\ell<0$ (since $\ell\leq n-1<n$)
and $\left[  n=\ell\right]  =0$ (since $\ell\leq n-1<n$ and thus $n\neq\ell$).
Hence, for each $\ell\in\left[  n-1\right]  $, we have
\begin{align*}
u_{n,\ell+\left[  n=\ell\right]  }  & =u_{n,\ell}\ \ \ \ \ \ \ \ \ \ \left(
\text{since }\ell+\underbrace{\left[  n=\ell\right]  }_{=0}=\ell\right)  \\
& =h\left(  \mu_{n}-n+\ell,\ b_{n},\ 1-\ell\right)
\ \ \ \ \ \ \ \ \ \ \left(  \text{by the definition of }u_{n,\ell}\right)  \\
& =0\ \ \ \ \ \ \ \ \ \ \left( \text{by  the definition of $h\tup{a,b,c}$, since }\mu_{n}-n+\ell<0\right)  .
\end{align*}
Therefore, we can apply Lemma \ref{lem.det.lastrow=01} to $a_{i,j}%
=u_{i,j+\left[  n=j\right]  }$. We thus obtain%
\begin{equation}
\det\left(  u_{i,j+\left[  n=j\right]  }\right)  _{i,j\in\left[  n\right]
}=u_{n,n+\left[  n=n\right]  }\cdot\det\left(  u_{i,j+\left[  n=j\right]
}\right)  _{i,j\in\left[  n-1\right]  }.\label{pf.lem.k-th-det-3.1}%
\end{equation}

However, $\left[  n=n\right]  =1$, so that
\begin{align}
u_{n,n+\left[  n=n\right]  }  & =u_{n,n+1}\nonumber\\
& =h\left(  \underbrace{\mu_{n}}_{=0}-\,n+\left(  n+1\right)
,\ \underbrace{b_{n}}_{=n},\ \underbrace{1-\left(  n+1\right)  }_{=-n}\right)
\ \ \ \ \ \ \ \ \ \ \left(  \text{by the definition of }u_{n,n+1}\right)
\nonumber\\
& =h\left(  \underbrace{-n+\left(  n+1\right)  }_{=1},\ n,\ -n\right)
= h\left(  1,\ n,\ -n\right) \nonumber \\
&= \sum_{i=1}^n x_i + \sum_{j=-n+1}^{-n+n} \underbrace{y_j}%
_{\substack{=0\\\text{(since }j \leq -n+n = 0\text{, and}\\\text{since }y_i = 0 \text{
for all }i \leq 0\text{)}}}
\qquad \qquad \left(\text{by Lemma \ref{lem.h1bc}}\right) \nonumber \\
&= \sum_{i=1}^n x_i + \underbrace{\sum_{j=-n+1}^{-n+n} 0}_{=0}
=\sum_{i=1}^{n}x_{i}%
.\label{pf.lem.k-th-det-3.2}%
\end{align}
Moreover, for every $i,j\in\left[  n-1\right]  $, we have $j\neq n$ (since
$j\leq n-1<n$) and thus $\left[  j=n\right]  =0$, so that
\[
u_{i,j+\left[  n=j\right]  }=u_{i,j+0}=u_{i,j}.
\]
Hence,%
\begin{equation}
\left(  u_{i,j+\left[  n=j\right]  }\right)  _{i,j\in\left[  n-1\right]
}=\left(  u_{i,j}\right)  _{i,j\in\left[  n-1\right]  }%
.\label{pf.lem.k-th-det-3.3}%
\end{equation}

Using (\ref{pf.lem.k-th-det-3.2}) and (\ref{pf.lem.k-th-det-3.3}), we can
rewrite (\ref{pf.lem.k-th-det-3.1}) as%
\[
\det\left(  u_{i,j+\left[  n=j\right]  }\right)  _{i,j\in\left[  n\right]
}=\left(  \sum_{i=1}^{n}x_{i}\right)  \cdot\underbrace{\det\left(
u_{i,j}\right)  _{i,j\in\left[  n-1\right]  }}_{\substack{=\mathbf{s}%
_{\lambda}\left[  \mu\right]  \\\text{(by (\ref{eq.lem.k-th-det-0.-1}))}%
}}=\left(  \sum_{i=1}^{n}x_{i}\right)  \cdot\mathbf{s}_{\lambda}\left[
\mu\right]  .
\]
This proves Lemma \ref{lem.k-th-det-3}.
\end{proof}

\begin{proof}
[Proof of Lemma \ref{lem.k-th-det-4}.]From $\lambda_n = 0$, we obtain $\lambda_{n+1} = 0$.
Similarly, $\mu_{n+1} = 0$. Thus, Lemma \ref{lem.Deltas.x} can be applied.

We have
\begin{align*}
\sum_{k=1}^{n}p_{k} &  =\sum_{i=1}^{n}p_{i}\\
&  =\sum_{i=1}^{n}%
\begin{cases}
x_{b_{i}}, & \text{if }m_{i}\in\Delta\left(  \lambda\right)  ;\\
-y_{m_{i}+1+b_{i}}, & \text{if }m_{i}\notin\Delta\left(  \lambda\right)
\end{cases}
\ \ \ \ \ \ \ \ \ \ \left(  \text{by the definition of }p_{i}\right)  \\
&  =\sum_{\substack{i\in\left[  n\right]  ;\\m_{i}\in\Delta\left(
\lambda\right)  }}x_{b_{i}}-\sum_{\substack{i\in\left[  n\right]
;\\m_{i}\notin\Delta\left(  \lambda\right)  }}y_{m_{i}+1+b_{i}}.
\end{align*}
Thus,
\begin{align*}
\sum_{i=1}^{n}x_{i}-\sum_{k=1}^{n}p_{k} &  =\sum_{i=1}^{n}x_{i}-\left(
\sum_{\substack{i\in\left[  n\right]  ;\\m_{i}\in\Delta\left(  \lambda\right)
}}x_{b_{i}}-\sum_{\substack{i\in\left[  n\right]  ;\\m_{i}\notin\Delta\left(
\lambda\right)  }}y_{m_{i}+1+b_{i}}\right)  \\
&  =\underbrace{\sum_{i=1}^{n}x_{i}-\sum_{\substack{i\in\left[  n\right]
;\\m_{i}\in\Delta\left(  \lambda\right)  }}x_{b_{i}}}_{\substack{=\sum
_{\substack{k\in\left[  n\right]  ;\\\ell_{k}\notin\Delta\left(  \mu\right)
}}x_{k}\\\text{(by Lemma \ref{lem.Deltas.x})}}}+\sum_{\substack{i\in\left[
n\right]  ;\\m_{i}\notin\Delta\left(  \lambda\right)  }}y_{m_{i}+1+b_{i}}\\
&  =\sum_{\substack{k\in\left[  n\right]  ;\\\ell_{k}\notin\Delta\left(
\mu\right)  }}x_{k}+\sum_{\substack{i\in\left[  n\right]  ;\\m_{i}\notin%
\Delta\left(  \lambda\right)  }}y_{m_{i}+1+b_{i}}.
\end{align*}
This proves Lemma \ref{lem.k-th-det-4}.
\end{proof}

\begin{proof}
[Proof of Lemma \ref{lem.konvalinka-bi}.]Define $u_{i,j}$ and $p_{i}$ as in
Convention~\ref{conv.u-and-p}. Lemma \ref{lem.konval-RHS-as-sum} yields%
\begin{align*}
\sum_{\mu\lessdot\nu\subseteq\lambda}\mathbf{s}_{\lambda}\left[  \nu\right]
&  =\sum_{k=1}^{n}\mathbf{s}_{\lambda}\left[  \mu^{+k}\right]  \\
&  =\underbrace{\det\left(  u_{i,j+\left[  n=j\right]  }\right)
_{i,j\in\left[  n\right]  }}_{\substack{=\left(  \sum_{i=1}^{n}x_{i}\right)
\cdot\mathbf{s}_{\lambda}\left[  \mu\right]  \\\text{(by Lemma
\ref{lem.k-th-det-3})}}}-\left(  \sum_{k=1}^{n}p_{k}\right)  \mathbf{s}%
_{\lambda}\left[  \mu\right]  \qquad\qquad\left(  \text{by Lemma
\ref{lem.k-th-det-2}}\right)  \\
&  =\left(  \sum_{i=1}^{n}x_{i}\right)  \cdot\mathbf{s}_{\lambda}\left[
\mu\right]  -\left(  \sum_{k=1}^{n}p_{k}\right)  \mathbf{s}_{\lambda}\left[
\mu\right]  \\
&  =\underbrace{\left(  \sum_{i=1}^{n}x_{i}-\sum_{k=1}^{n}p_{k}\right)
}_{\substack{=\sum_{\substack{k\in\left[  n\right]  ;\\\ell_{k}\notin%
\Delta\left(  \mu\right)  }}x_{k}+\sum_{\substack{i\in\left[  n\right]
;\\m_{i}\notin\Delta\left(  \lambda\right)  }}y_{m_{i}+1+b_{i}}\\\text{(by
Lemma \ref{lem.k-th-det-4})}}}\cdot\,\mathbf{s}_{\lambda}\left[  \mu\right]
\\
&  =\left(  \sum_{\substack{k\in\left[  n\right]  ;\\\ell_{k}\notin%
\Delta\left(  \mu\right)  }}x_{k}+\sum_{\substack{i\in\left[  n\right]
;\\m_{i}\notin\Delta\left(  \lambda\right)  }}y_{m_{i}+1+b_{i}}\right)
\mathbf{s}_{\lambda}\left[  \mu\right]  .
\end{align*}
This proves Lemma \ref{lem.konvalinka-bi}.
\end{proof}

\subsection{To Section \ref{sec.pf-hlf} and Theorem \ref{thm.main}}

\begin{proof}[Proof of Lemma \ref{lem.pf-hlf.w-w}.]
From $\tup{i,j} \in Y\tup{\lambda}$, we obtain
$j\leq\lambda_{i}$, so that $\lambda_{i}\geq j$. Hence, Lemma
\ref{lem.conj.uniprop} (applied to $i$ and $j$ instead of $j$ and $i$) yields
$\lambda_{j}^{t}\geq i$. Now, the definition of $\ell_{i}$ yields $\ell
_{i}=\underbrace{\lambda_{i}}_{\geq j}-\,i\geq j-i$.
Furthermore, the definition
of $\ell^t_j$ yields
$\ell^t_j = \lambda^t_j - j$, so that
$-\ell^t_j = -\tup{\lambda^t_j - j} = j - \lambda^t_j$.
In other words, $j - \lambda^t_j = -\ell^t_j$.
Moreover, we thus obtain $-\ell^t_j = j - \underbrace{\lambda^t_j}_{\geq i} \leq j-i$.
Thus,
$-\ell_{j}^{t}\leq j-i\leq\ell_{i}$ (since $\ell_{i}\geq j-i$).

Next, we shall show that $n\geq\ell_{j}^{t}$.
Indeed, Lemma~\ref{lem.conj.uniprop} (applied to $j$ and $n$ instead of $i$ and $j$) shows that we have the equivalence $\tup{\lambda^t_j \geq n} \ \Longleftrightarrow\ \tup{\lambda_n \geq j}$.
Since the statement $\tup{\lambda_n \geq j}$ is false (because $\lambda_n = 0 < j$), we thus conclude that the statement $\tup{\lambda^t_j \geq n}$ is false as well.
In other words, we have $\lambda^t_j < n$.
However, the definition of $\ell_{j}^{t}$ yields
$\ell_{j}^{t}=\lambda_{j}^{t}-\underbrace{j}_{>0}<\lambda_{j}^{t} < n$.
Hence, $n \geq \ell_{j}^{t}$, so that $-n\leq-\ell_{j}^{t}$.

Now, the definition of $w_{k}$ yields $w_{\ell_{i}}=\sum_{k=-n}^{\ell_{i}%
}z_{k}$ and $w_{-\ell_{j}^{t}-1}=\sum_{k=-n}^{-\ell_{j}^{t}-1}z_{k}$. Hence,%
\begin{align*}
w_{\ell_{i}}  & =\sum_{k=-n}^{\ell_{i}}z_{k}=\underbrace{\sum_{k=-n}%
^{-\ell_{j}^{t}-1}z_{k}}_{=w_{-\ell_{j}^{t}-1}}+\sum_{k=-\ell_{j}^{t}}%
^{\ell_{i}}z_{k}\ \ \ \ \ \ \ \ \ \ \left(  \text{since }-n\leq-\ell_{j}%
^{t}\leq\ell_{i}\right)  \\
& =w_{-\ell_{j}^{t}-1}+\sum_{k=-\ell_{j}^{t}}^{\ell_{i}}z_{k}.
\end{align*}
In other words,
\begin{equation}
w_{\ell_{i}}-w_{-\ell_{j}^{t}-1}=\sum_{k=-\ell_{j}^{t}}^{\ell_{i}}%
z_{k}.\label{pf.lem.pf-hlf.w-w.1}%
\end{equation}

Meanwhile, the equality \eqref{eq.def.hlcz} (with $c$, $i$ and $j$ renames as $\tup{i,j}$, $k$ and $p$) says that
\begin{equation}
h_{\lambda} \tup{\tup{i,j}; z}
= \sum_{\left(  k,p\right)  \in H_{\lambda}\left(
(i,j)\right)  }z_{p-k}.
\label{pf.lem.pf-hlf.w-w.2}
\end{equation}
Now, let us study the boxes $\left(  k,p\right)  $ that belong to the hook
$H_{\lambda}\left(  (i,j)\right)  $. These are the boxes of $Y\left(
\lambda\right)  $ that lie either in the $i$-th row to the east of $\left(
i,j\right)  $ (including $\left(  i,j\right)  $ itself), or in the $j$-th
column to the south of $\left(  i,j\right)  $. Thus, these boxes come in two types:

\begin{enumerate}
\item The boxes $\left(  k,p\right)  \in H_{\lambda}\left(  (i,j)\right)  $
that belong to the $i$-th row (i.e., that satisfy $k=i$) are the boxes
\[
\left(  i,j\right)  ,\ \left(  i,j+1\right)  ,\ \left(  i,j+2\right)
,\ \ldots,\ \left(  i,\lambda_{i}\right)
\]
(since the $i$-th row of $Y\left(  \lambda\right)  $ has $\lambda_{i}$ boxes).

\item The boxes $\left(  k,p\right)  \in H_{\lambda}\left(  (i,j)\right)  $
that do not belong to the $i$-th row (i.e., that satisfy $k\neq i$) are the
boxes
\[
\left(  i+1,j\right)  ,\ \left(  i+2,j\right)  ,\ \left(  i+3,j\right)
,\ \ldots,\ \left(  \lambda_{j}^{t},j\right)
\]
(since the $j$-th column of $Y\left(  \lambda\right)  $ has $\lambda_{j}^{t}$ boxes
(by \eqref{eq.def.lambdat.5}, applied to $k=j$)).
\end{enumerate}

Altogether, the boxes in $H_{\lambda}\left(  (i,j)\right)  $ are therefore the
boxes%
\begin{align*}
& \left(  i,j\right)  ,\ \left(  i,j+1\right)  ,\ \left(  i,j+2\right)
,\ \ldots,\ \left(  i,\lambda_{i}\right)  ,\\
& \left(  i+1,j\right)  ,\ \left(  i+2,j\right)  ,\ \left(  i+3,j\right)
,\ \ldots,\ \left(  \lambda_{j}^{t},j\right)
\end{align*}
(and clearly, all these boxes are distinct). Hence,%
\begin{align*}
\sum_{\left(  k,p\right)  \in H_{\lambda}\left(  (i,j)\right)  }z_{p-k}  &
=\underbrace{\left(  z_{j-i}+z_{\left(  j+1\right)  -i}+z_{\left(  j+2\right)
-i}+\cdots+z_{\lambda_{i}-i}\right)  }_{\substack{=z_{j-i}+z_{\left(
j-i\right)  +1}+z_{\left(  j-i\right)  +2}+\cdots+z_{\lambda_{i}-i}%
\\=\sum_{k=j-i}^{\lambda_{i}-i}z_{k}}}\\
& \ \ \ \ \ \ \ \ \ \ +\underbrace{\left(  z_{j-\left(  i+1\right)
}+z_{j-\left(  i+2\right)  }+z_{j-\left(  i+3\right)  }+\cdots+z_{j-\lambda
_{j}^{t}}\right)  }_{\substack{=z_{\left(  j-i\right)  -1}+z_{\left(
j-i\right)  -2}+\cdots+z_{j-\lambda_{j}^{t}}\\=\sum_{k=j-\lambda_{j}^{t}%
}^{j-i-1}z_{k}}}\\
& =\sum_{k=j-i}^{\lambda_{i}-i}z_{k}+\sum_{k=j-\lambda_{j}^{t}}^{j-i-1}%
z_{k}\\
& =\sum_{k=j-i}^{\ell_{i}}z_{k}+\sum_{k=-\ell_{j}^{t}}^{j-i-1}z_{k}
\ \ \ \ \ \ \ \ \ \ \left(  \text{since }\lambda_{i}-i=\ell_{i}\text{ and
}j-\lambda_{j}^{t}=-\ell_{j}^{t}\right)  \\
&
= \sum_{k=-\ell_{j}^{t}}^{j-i-1}z_{k} + \sum_{k=j-i}^{\ell_{i}}z_{k}
=\sum_{k=-\ell_{j}^{t}}^{\ell_{i}}z_{k}\ \ \ \ \ \ \ \ \ \ \left(
\text{since }-\ell_{j}^{t}\leq j-i\leq\ell_{i}\right)  \\
& =w_{\ell_{i}}-w_{-\ell_{j}^{t}-1}\ \ \ \ \ \ \ \ \ \ \left(  \text{by
(\ref{pf.lem.pf-hlf.w-w.1})}\right)  .
\end{align*}
Hence, we can rewrite (\ref{pf.lem.pf-hlf.w-w.2}) as%
\[
h_{\lambda} \tup{\tup{i,j}; z}
= w_{\ell_{i}}-w_{-\ell_{j}^{t}-1}.
\]
This proves Lemma \ref{lem.pf-hlf.w-w}.
\end{proof}

\begin{proof}[Proof of Lemma \ref{section14.wY}.]
We are in one of the following two cases:

\textit{Case 1:} We have $i\leq n$.

\textit{Case 2:} We have $i>n$.

Let us first consider Case 1. In this case, $i\leq n$. Thus, $m_{i}%
=\underbrace{\mu_{i}}_{\geq0}-\underbrace{i}_{\leq n}\geq-n$.

Since $\lambda\supseteq\mu$, we have $\lambda_{i}\geq\mu_{i}$, and thus
$\ell_{i}\geq m_{i}$.

However, the definition of the $w_{k}$ yields $w_{\ell_{i}}=\sum_{k=-n}%
^{\ell_{i}}z_{k}$ and $w_{m_{i}}=\sum_{k=-n}^{m_{i}}z_{k}$. Hence,
\begin{align}
w_{\ell_{i}}-w_{m_{i}}  & =\sum_{k=-n}^{\ell_{i}}z_{k}-\sum_{k=-n}^{m_{i}%
}z_{k}\nonumber\\
& =\sum_{k=m_{i}+1}^{\ell_{i}}z_{k}\ \ \ \ \ \ \ \ \ \ \left(  \text{since
}\ell_{i}\geq m_{i}\geq-n\right)  .\label{pf.section14.wY.2}%
\end{align}

However, the integers $j\geq1$ satisfying $\left(  i,j\right)  \in Y\left(
\lambda/\mu\right)  $ are exactly the integers $j$ such that $\mu_{i}%
<j\leq\lambda_{i}$. Hence,
\begin{align*}
\sum_{\substack{j\geq1;\\\left(  i,j\right)  \in Y\left(  \lambda/\mu\right)
}}z_{j-i}  & =\sum_{j=\mu_{i}+1}^{\lambda_{i}}z_{j-i}\\
& =\sum_{k=\mu_{i}-i+1}^{\lambda_{i}-i}z_{k}\ \ \ \ \ \ \ \ \ \ \left(
\begin{array}
[c]{c}%
\text{here, we have substituted }k\\
\text{for }j-i\text{ in the sum}%
\end{array}
\right)  \\
& =\sum_{k=m_{i}+1}^{\ell_{i}}z_{k}\ \ \ \ \ \ \ \ \ \ \left(  \text{since
}\mu_{i}-i=m_{i}\text{ and }\lambda_{i}-i=\ell_{i}\right)  .
\end{align*}
Comparing this with \eqref{pf.section14.wY.2}, we obtain $w_{\ell_{i}%
}-w_{m_{i}}=\sum_{\substack{j\geq1;\\\left(  i,j\right)  \in Y\left(
\lambda/\mu\right)  }}z_{j-i}$. Hence, Lemma \ref{section14.wY} is proved in
Case 1.

Let us now consider Case 2. In this case, $i>n$. Thus, both $\lambda_{i}$ and
$\mu_{i}$ equal $0$ (by the definition of $n$), so that we have $\lambda_i = \mu_i$,
and therefore we have
$\ell_{i}=m_{i}$ (since $\ell_i = \lambda_i - i$ and $m_i = \mu_i - i$). Hence, $w_{\ell_{i}}=w_{m_{i}}$, so that $w_{\ell_{i}%
}-w_{m_{i}}=0$. On the other hand, the diagram $Y\left(  \lambda/\mu\right)  $
has no boxes in the $i$-th row (since $\lambda_{i}=0$), so that the sum
$\sum_{\substack{j\geq1;\\\left(  i,j\right)  \in Y\left(  \lambda/\mu\right)
}}z_{j-i}$ is empty and thus equals $0$. Comparing this with $w_{\ell_{i}%
}-w_{m_{i}}=0$, we obtain $w_{\ell_{i}}-w_{m_{i}}=\sum_{\substack{j\geq
1;\\\left(  i,j\right)  \in Y\left(  \lambda/\mu\right)  }}z_{j-i}$. Thus,
Lemma \ref{section14.wY} is proved in Case 2.

We have now proved Lemma \ref{section14.wY} in both possible cases.
\end{proof}

\begin{proof}[Proof of Lemma \ref{lem.pf-hlf.4}.]
Lemma \ref{lem.Deltas.f=g} (applied to
$f\left(  i\right)  =w_{m_{i}}$ and $g\left(  j\right)  =w_{\ell_{j}}$)
yields
\begin{equation}
\sum_{\substack{i\in\lbrack n];\\m_{i}\in\Delta(\lambda)}}w_{m_{i}}%
=\sum_{\substack{j\in\lbrack n];\\\ell_{j}\in\Delta(\mu)}}w_{\ell_{j}%
}\label{pf.lem.pf-hlf.4.1}%
\end{equation}
(since $w_{m_{i}}=w_{\ell_{j}}$ whenever $m_{i}=\ell_{j}$).

From $\lambda_{n}=0$, we see that every $\left(  i,j\right)  \in Y\left(
\lambda/\mu\right)  $ satisfies $i<n$ and thus $i\leq n$ and therefore $i \in \ive{n}$. Hence,%
\begin{align*}
\sum_{\left(  i,j\right)  \in Y\left(  \lambda/\mu\right)  }z_{j-i}  &
=\sum_{i\in\left[  n\right]  }\ \ \underbrace{\sum_{\substack{j\geq1;\\\left(
i,j\right)  \in Y\left(  \lambda/\mu\right)  }}z_{j-i}}_{\substack{=w_{\ell
_{i}}-w_{m_{i}}\\\text{(by Lemma \ref{section14.wY})}}}=\sum_{i\in\left[
n\right]  }\left(  w_{\ell_{i}}-w_{m_{i}}\right)  \\
& =\underbrace{\sum_{i\in\left[  n\right]  }w_{\ell_{i}}}_{=\sum
_{\substack{i\in\left[  n\right]  ;\\\ell_{i}\in\Delta\left(  \mu\right)
}}w_{\ell_{i}}+\sum_{\substack{i\in\left[  n\right]  ;\\\ell_{i}\notin%
\Delta\left(  \mu\right)  }}w_{\ell_{i}}}-\underbrace{\sum_{i\in\left[
n\right]  }w_{m_{i}}}_{=\sum_{\substack{i\in\left[  n\right]  ;\\m_{i}%
\in\Delta\left(  \lambda\right)  }}w_{m_{i}}+\sum_{\substack{i\in\left[
n\right]  ;\\m_{i}\notin\Delta\left(  \lambda\right)  }}w_{m_{i}}}\\
& =\left(  \sum_{\substack{i\in\left[  n\right]  ;\\\ell_{i}\in\Delta\left(
\mu\right)  }}w_{\ell_{i}}+\sum_{\substack{i\in\left[  n\right]  ;\\\ell
_{i}\notin\Delta\left(  \mu\right)  }}w_{\ell_{i}}\right)  -\left(
\sum_{\substack{i\in\left[  n\right]  ;\\m_{i}\in\Delta\left(  \lambda\right)
}}w_{m_{i}}+\sum_{\substack{i\in\left[  n\right]  ;\\m_{i}\notin\Delta\left(
\lambda\right)  }}w_{m_{i}}\right)  \\
& =\underbrace{\sum_{\substack{i\in\left[  n\right]  ;\\\ell_{i}\in
\Delta\left(  \mu\right)  }}w_{\ell_{i}}}_{=\sum_{\substack{j\in\lbrack
n];\\\ell_{j}\in\Delta(\mu)}}w_{\ell_{j}}}+\underbrace{\sum_{\substack{i\in
\left[  n\right]  ;\\\ell_{i}\notin\Delta\left(  \mu\right)  }}w_{\ell_{i}}%
}_{=\sum_{\substack{k\in\left[  n\right]  ;\\\ell_{k}\notin\Delta\left(
\mu\right)  }}w_{\ell_{k}}}-\underbrace{\sum_{\substack{i\in\left[  n\right]
;\\m_{i}\in\Delta\left(  \lambda\right)  }}w_{m_{i}}}_{\substack{=\sum
_{\substack{j\in\lbrack n];\\\ell_{j}\in\Delta(\mu)}}w_{\ell_{j}}\\\text{(by
(\ref{pf.lem.pf-hlf.4.1}))}}}-\sum_{\substack{i\in\left[  n\right]
;\\m_{i}\notin\Delta\left(  \lambda\right)  }}w_{m_{i}}\\
& =\sum_{\substack{j\in\lbrack n];\\\ell_{j}\in\Delta(\mu)}}w_{\ell_{j}}%
+\sum_{\substack{k\in\left[  n\right]  ;\\\ell_{k}\notin\Delta\left(
\mu\right)  }}w_{\ell_{k}}-\sum_{\substack{j\in\lbrack n];\\\ell_{j}\in
\Delta(\mu)}}w_{\ell_{j}}-\sum_{\substack{i\in\left[  n\right]  ;\\m_{i}%
\notin\Delta\left(  \lambda\right)  }}w_{m_{i}}\\
& =\sum_{\substack{k\in\left[  n\right]  ;\\\ell_{k}\notin\Delta\left(
\mu\right)  }}w_{\ell_{k}}-\sum_{\substack{i\in\left[  n\right]
;\\m_{i}\notin\Delta\left(  \lambda\right)  }}w_{m_{i}}.
\end{align*}
This proves Lemma \ref{lem.pf-hlf.4}.
\end{proof}

\begin{proof}[Proof of Lemma \ref{rec.2}.]
Let $\mathbf{b}=\left(  b_{1},b_{2},b_{3}%
,\ldots\right)  $ be the flagging induced by $\lambda/\mu$.
(This was introduced in Definition~\ref{def.flagging-of-lm}.)

Recall that $x_1, x_2, x_3, \ldots$ and $y_1, y_2, y_3, \ldots$ have been indeterminates so far. But now let us instead set
\[
x_{i}:=w_{\ell_{i}}\ \ \ \ \ \ \ \ \ \ \text{and}\ \ \ \ \ \ \ \ \ \ y_{i}%
:=-w_{-\ell_{i}^{t}-1}\ \ \ \ \ \ \ \ \ \ \text{for each }i\geq1.
\]
Thus, for any $\left(  i,j\right)  \in Y\left(  \lambda\right)  $, we have%
\begin{align}
\underbrace{x_{i}}_{=w_{\ell_{i}}}+\underbrace{y_{j}}_{=-w_{-\ell_{j}^{t}-1}}
& =w_{\ell_{i}}+\left(  -w_{-\ell_{j}^{t}-1}\right)  =w_{\ell_{i}}%
-w_{-\ell_{j}^{t}-1}\nonumber\\
& =h_{\lambda}\left(  \left(  i,j\right)  ;z\right)  \label{pf.rec.2.x+y=h}
\end{align}
(by Lemma \ref{lem.pf-hlf.w-w}).

Now, Lemma \ref{lem.konvalinka-bi} yields%
\begin{align}
& \left(  \sum_{\substack{k\in\left[  n\right]  ;\\\ell_{k}\notin\Delta\left(
\mu\right)  }}x_{k}+\sum_{\substack{i\in\left[  n\right]  ;\\m_{i}\notin%
\Delta\left(  \lambda\right)  }}y_{m_{i}+1+b_{i}}\right)  \mathbf{s}_{\lambda
}\left[  \mu\right]  \nonumber\\
& =\sum_{\mu\lessdot\nu\subseteq\lambda}\mathbf{s}_{\lambda}\left[
\nu\right]  \label{pf.rec.2.1}%
\end{align}
(where the sum on the right hand side ranges over all partitions $\nu$ that
satisfy $\mu\lessdot\nu\subseteq\lambda$). However, for each $k\in\left[
n\right]  $ satisfying $\ell_{k}\notin\Delta\left(  \mu\right)  $, we have
$x_{k}=w_{\ell_{k}}$ (by the definition of $x_{k}$). Thus,%
\begin{equation}
\sum_{\substack{k\in\left[  n\right]  ;\\\ell_{k}\notin\Delta\left(
\mu\right)  }}x_{k}=\sum_{\substack{k\in\left[  n\right]  ;\\\ell_{k}%
\notin\Delta\left(  \mu\right)  }}w_{\ell_{k}}.\label{pf.rec2.2a}%
\end{equation}

Furthermore, for each $i\in\left[  n\right]  $ satisfying $m_{i}\notin%
\Delta\left(  \lambda\right)  $, we have $m_{i}+1+b_{i}\geq1$ and $\ell
_{m_{i}+1+b_{i}}^{t}=-1-m_{i}$ (both by Lemma \ref{lem.Deltas.mi+1+bi}) and%
\begin{align*}
y_{m_{i}+1+b_{i}}  & =-w_{-\ell_{m_{i}+1+b_{i}}^{t}-1}%
\ \ \ \ \ \ \ \ \ \ \left(  \text{by the definition of }y_{m_{i}+1+b_{i}%
}\right)  \\
& =-w_{-\left(  -1-m_{i}\right)  -1}\ \ \ \ \ \ \ \ \ \ \left(  \text{since
}\ell_{m_{i}+1+b_{i}}^{t}=-1-m_{i}\right)  \\
& =-w_{m_{i}}\ \ \ \ \ \ \ \ \ \ \left(  \text{since }-\left(  -1-m_{i}%
\right)  -1=m_{i}\right)  .
\end{align*}
Hence,%
\[
\sum_{\substack{i\in\left[  n\right]  ;\\m_{i}\notin\Delta\left(
\lambda\right)  }}y_{m_{i}+1+b_{i}}=\sum_{\substack{i\in\left[  n\right]
;\\m_{i}\notin\Delta\left(  \lambda\right)  }}\left(  -w_{m_{i}}\right)
=-\sum_{\substack{i\in\left[  n\right]  ;\\m_{i}\notin\Delta\left(
\lambda\right)  }}w_{m_{i}}.
\]

Adding this equality to the equality (\ref{pf.rec2.2a}), we find%
\begin{align}
\sum_{\substack{k\in\left[  n\right]  ;\\\ell_{k}\notin\Delta\left(
\mu\right)  }}x_{k}+\sum_{\substack{i\in\left[  n\right]  ;\\m_{i}\notin%
\Delta\left(  \lambda\right)  }}y_{m_{i}+1+b_{i}}  & =\sum_{\substack{k\in
\left[  n\right]  ;\\\ell_{k}\notin\Delta\left(  \mu\right)  }}w_{\ell_{k}%
}+\left(  -\sum_{\substack{i\in\left[  n\right]  ;\\m_{i}\notin\Delta\left(
\lambda\right)  }}w_{m_{i}}\right)  \nonumber\\
& =\sum_{\substack{k\in\left[  n\right]  ;\\\ell_{k}\notin\Delta\left(
\mu\right)  }}w_{\ell_{k}}-\sum_{\substack{i\in\left[  n\right]
;\\m_{i}\notin\Delta\left(  \lambda\right)  }}w_{m_{i}}\nonumber\\
& =\sum_{\left(  i,j\right)  \in Y\left(  \lambda/\mu\right)  }z_{j-i}%
\label{pf.rec2.2e}%
\end{align}
(by Lemma \ref{lem.pf-hlf.4}).

Furthermore, if $\nu$ is any partition, then the definition of
$\mathbf{s}_{\lambda}\left[  \nu\right]  $ yields
\begin{align}
\mathbf{s}_{\lambda}\left[  \nu\right]    & = \sum_{D\in\mathcal{E}\left(
\lambda/\nu\right)  }\ \ \prod_{\left(  i,j\right)  \in D}\underbrace{\left(
x_{i}+y_{j}\right)  }_{\substack{=h_{\lambda}\left(  \left(  i,j\right)
;z\right)  \\\text{(by (\ref{pf.rec.2.x+y=h}),}\\\text{since }\left(
i,j\right)  \in D\subseteq Y\left(  \lambda\right)  \text{)}}}
= \sum_{D\in\mathcal{E}\left(
\lambda/\nu\right)  }\ \ \prod_{\left(  i,j\right)  \in D} h_{\lambda}\left(  \left(  i,j\right)
;z\right)
\nonumber\\
& =\sum_{E\in\mathcal{E}\left(  \lambda/\nu\right)  }\ \ \prod_{c\in
E}h_{\lambda}\left(  c;z\right)  \ \ \ \ \ \ \ \ \ \ \left(
\begin{array}
[c]{c}%
\text{here, we have renamed the}\\
\text{indices }D\text{ and }\left(  i,j\right)  \text{ as }E\text{ and }c
\end{array}
\right)  \nonumber\\
& =\sum_{E\in\mathcal{E}\left(  \lambda/\nu\right)  }\dfrac{\prod_{c\in
Y\left(  \lambda\right)  }h_{\lambda}\left(  c;z\right)  }{\prod_{c\in
Y\left(  \lambda\right)  \setminus E}h_{\lambda}\left(  c;z\right)
}\label{pf.rec2.2c}%
\end{align}
(since $\prod_{c\in E}h_{\lambda}\left(  c;z\right)  =\dfrac{\prod_{c\in
Y\left(  \lambda\right)  }h_{\lambda}\left(  c;z\right)  }{\prod_{c\in
Y\left(  \lambda\right)  \setminus E}h_{\lambda}\left(  c;z\right)  }$ for
every subset $E$ of $Y\left(  \lambda\right)  $).

In particular, applying this to $\nu=\mu$, we obtain%
\begin{equation}
\mathbf{s}_{\lambda}\left[  \mu\right]  =\sum_{E\in\mathcal{E}\left(
\lambda/\mu\right)  }\dfrac{\prod_{c\in Y\left(  \lambda\right)  }h_{\lambda
}\left(  c;z\right)  }{\prod_{c\in Y\left(  \lambda\right)  \setminus
E}h_{\lambda}\left(  c;z\right)  }.\label{pf.rec2.2d}%
\end{equation}

Now, using the equalities (\ref{pf.rec2.2e}), (\ref{pf.rec2.2c}) and
(\ref{pf.rec2.2d}), we can rewrite (\ref{pf.rec.2.1}) as%
\begin{align*}
\left(  \sum_{\left(  i,j\right)  \in Y\left(  \lambda/\mu\right)  }%
z_{j-i}\right)  \sum_{E\in\mathcal{E}\left(  \lambda/\mu\right)  }\dfrac
{\prod_{c\in Y\left(  \lambda\right)  }h_{\lambda}\left(  c;z\right)  }%
{\prod_{c\in Y\left(  \lambda\right)  \setminus E}h_{\lambda}\left(
c;z\right)  }
 =\sum_{\mu\lessdot\nu\subseteq\lambda}\mathbf{\ \ }\sum_{E\in\mathcal{E}%
\left(  \lambda/\nu\right)  }\dfrac{\prod_{c\in Y\left(  \lambda\right)
}h_{\lambda}\left(  c;z\right)  }{\prod_{c\in Y\left(  \lambda\right)
\setminus E}h_{\lambda}\left(  c;z\right)  }.
\end{align*}
Dividing both sides of this equality by $\left(  \sum_{\left(  i,j\right)  \in
Y\left(  \lambda/\mu\right)  }z_{j-i}\right)  \prod_{c \in
Y\left(  \lambda\right)  }h_{\lambda}\left( c;z\right)  $
(this is allowed, since $\mu \neq \lambda$ ensures that $ \sum_{\left(  i,j\right)  \in
Y\left(  \lambda/\mu\right)  }z_{j-i} \neq 0$),
we obtain%
\begin{align*}
\sum_{E\in\mathcal{E}\left(  \lambda/\mu\right)  }\dfrac{1}{\prod_{c\in
Y\left(  \lambda\right)  \setminus E}h_{\lambda}\left(  c;z\right)  }
=\dfrac{1}{\sum_{\left(  i,j\right)  \in Y\left(  \lambda/\mu\right)
}z_{j-i}}\cdot\sum_{\mu\lessdot\nu\subseteq\lambda}\mathbf{\ \ }\sum
_{E\in\mathcal{E}\left(  \lambda/\nu\right)  }\dfrac{1}{\prod_{c\in Y\left(
\lambda\right)  \setminus E}h_{\lambda}\left(  c;z\right)  }.
\end{align*}
We can rewrite this as%
\[
\sum_{E\in\mathcal{E}\left(  \lambda/\mu\right)  }\ \ \prod_{c\in Y\left(
\lambda\right)  \setminus E}\dfrac{1}{h_{\lambda}\left(  c;z\right)  }%
=\frac{1}{\sum_{(i,j)\in Y(\lambda/\mu)}z_{j-i}}\ \ \sum_{\mu\lessdot
\nu\subseteq\lambda}\ \ \sum_{E\in\mathcal{E}\left(  \lambda/\nu\right)
}\ \ \prod_{c\in Y\left(  \lambda\right)  \setminus E}\dfrac{1}{h_{\lambda
}\left(  c;z\right)  }%
\]
(since $\dfrac{1}{\prod_{c\in Y\left(  \lambda\right)  \setminus E}h_{\lambda
}\left(  c;z\right)  }=\prod_{c\in Y\left(  \lambda\right)  \setminus E}%
\dfrac{1}{h_{\lambda}\left(  c;z\right)  }$ for every diagram $E$). Thus,
Lemma \ref{rec.2} is proved.
\end{proof}

\begin{proof}[Proof of Theorem~\ref{thm.main}.]
We induct on $\left\vert Y\left(
\lambda/\mu\right)  \right\vert $:

\textit{Base case:} Assume that $\left\vert Y\left(  \lambda/\mu\right)
\right\vert =0$. Thus, $\mu=\lambda$ (since $\lambda\supseteq\mu$ by
assumption), so that the set $\operatorname*{SYT}\left(  \lambda/\mu\right)  $
consists of a single tableau $T$ (with no entries), and this tableau $T$
satisfies $\zz_{T}=1$. Hence, $\sum_{T\in\operatorname*{SYT}\left(  \lambda
/\mu\right)  }\zz_{T}=1$. On the other hand, from $\mu=\lambda$, we obtain
\[
\mathcal{E}\left(  \lambda/\mu\right)  =\mathcal{E}\left(  \lambda
/\lambda\right)  =\left\{  Y\left(  \lambda\right)  \right\}
\ \ \ \ \ \ \ \ \ \ \left(  \text{by Lemma \ref{lem.exc.equal}}\right)  ,
\]
and thus%
\begin{align*}
& \sum_{E\in\mathcal{E}\left(  \lambda/\mu\right)  }\ \ \prod_{c\in Y\left(
\lambda\right)  \setminus E}\dfrac{1}{h_{\lambda}\left(  c;z\right)  }
=\prod_{c\in Y\left(  \lambda\right)  \setminus Y\left(  \lambda\right)
}\dfrac{1}{h_{\lambda}\left(  c;z\right)  }=\left(  \text{empty product}%
\right)  =1.
\end{align*}
Comparing this with $\sum_{T\in\operatorname*{SYT}\left(  \lambda/\mu\right)
}\zz_{T}=1$, we find%
\[
\sum_{T\in\operatorname*{SYT}\left(  \lambda/\mu\right)  }\zz_{T}=\sum
_{E\in\mathcal{E}\left(  \lambda/\mu\right)  }\ \ \prod_{c\in Y\left(
\lambda\right)  \setminus E}\dfrac{1}{h_{\lambda}\left(  c;z\right)  }.
\]
Hence, Theorem~\ref{thm.main} is proved in the case when $\left\vert Y\left(
\lambda/\mu\right)  \right\vert =0$. The base case is thus finished.

\textit{Induction step:} Fix a positive integer $N$. Assume (as the induction
hypothesis) that Theorem~\ref{thm.main} holds for $\left\vert Y\left(
\lambda/\mu\right)  \right\vert =N-1$. We now fix a skew partition
$\lambda/\mu$ with $\left\vert Y\left(  \lambda/\mu\right)  \right\vert =N$.
Our goal is to prove that Theorem~\ref{thm.main} holds for this $\lambda/\mu$.

From $\left\vert Y\left(  \lambda/\mu\right)  \right\vert =N>0$, we obtain
$\mu\neq\lambda$. Hence, Lemma~\ref{rec.2} yields%
\begin{align}
& \sum_{E\in\mathcal{E}\left(  \lambda/\mu\right)  }\ \ \prod_{c\in Y\left(
\lambda\right)  \setminus E}\dfrac{1}{h_{\lambda}\left(  c;z\right)
}\nonumber\\
& =\frac{1}{\sum_{(i,j)\in Y(\lambda/\mu)}z_{j-i}}\ \ \sum_{\mu\lessdot
\nu\subseteq\lambda}\ \ \sum_{E\in\mathcal{E}\left(  \lambda/\nu\right)
}\ \ \prod_{c\in Y\left(  \lambda\right)  \setminus E}\dfrac{1}{h_{\lambda
}\left(  c;z\right)  }.\label{pf.thm.main.1}%
\end{align}

On the other hand, Lemma \ref{recursion.main} \textbf{(b)} shows that%
\begin{equation}
\sum_{T\in\operatorname*{SYT}\left(  \lambda/\mu\right)  }\zz_{T}=\frac{1}%
{\sum\limits_{(i,j)\in Y\left(  \lambda/\mu\right)  }z_{j-i}}\cdot
\sum\limits_{\mu\lessdot\nu\subseteq\lambda}\ \ \sum_{T\in\operatorname*{SYT}%
\left(  \lambda/\nu\right)  }\zz_{T}.\label{pf.thm.main.2}%
\end{equation}

Now, if $\nu$ is any partition satisfying $\mu\lessdot\nu\subseteq\lambda$,
then $\left\vert Y\left(  \lambda/\nu\right)  \right\vert =\left\vert Y\left(
\lambda/\mu\right)  \right\vert -1$ (since $\mu\lessdot\nu$ entails that the
diagram $Y\left(  \nu\right)  $ has exactly one more box than $Y\left(
\mu\right)  $, and thus $Y\left(  \lambda/\nu\right)  $ has one fewer box
than $Y\left(  \lambda/\mu\right)  $), and thus $\left\vert Y\left(
\lambda/\nu\right)  \right\vert =\underbrace{\left\vert Y\left(  \lambda
/\mu\right)  \right\vert }_{=N}-1=N-1$, which allows us to apply our induction
hypothesis to $\nu$ instead of $\mu$. As a consequence, we see that any such
partition $\nu$ satisfies%
\[
\sum_{T\in\operatorname*{SYT}\left(  \lambda/\nu\right)  }\zz_{T}=\sum
_{E\in\mathcal{E}\left(  \lambda/\nu\right)  }\ \ \prod_{c\in Y\left(
\lambda\right)  \setminus E}\dfrac{1}{h_{\lambda}\left(  c;z\right)  }.
\]
Hence, we can rewrite (\ref{pf.thm.main.2}) as
\[
\sum_{T\in\operatorname*{SYT}\left(  \lambda/\mu\right)  }\zz_{T}=\frac{1}%
{\sum\limits_{(i,j)\in Y\left(  \lambda/\mu\right)  }z_{j-i}}\cdot
\sum\limits_{\mu\lessdot\nu\subseteq\lambda}\ \ \sum_{E\in\mathcal{E}\left(
\lambda/\nu\right)  }\ \ \prod_{c\in Y\left(  \lambda\right)  \setminus
E}\dfrac{1}{h_{\lambda}\left(  c;z\right)  }.
\]
Comparing this with (\ref{pf.thm.main.1}), we obtain%
\[
\sum_{T\in\operatorname*{SYT}\left(  \lambda/\mu\right)  }\zz_{T}=\sum
_{E\in\mathcal{E}\left(  \lambda/\mu\right)  }\ \ \prod_{c\in Y\left(
\lambda\right)  \setminus E}\dfrac{1}{h_{\lambda}\left(  c;z\right)  }.
\]
In other words, Theorem~\ref{thm.main} holds for our $\lambda/\mu$. This
completes the induction step, and thus Theorem~\ref{thm.main} is proved.
\end{proof}

\subsection{To Section \ref{sec.konva-forreal} and Theorem \ref{mainKonvalinka}}

\begin{proof}[Proof of Lemma \ref{lem.conj.disjoint}.]
We are in one of the following two cases:

\textit{Case 1:} We have $\lambda_{j}\geq i$.

\textit{Case 2:} We have $\lambda_{j}<i$.

Let us first consider Case 1. In this case, we have $\lambda_{j}\geq i$. Thus,
Lemma \ref{lem.conj.uniprop} yields $\lambda_{i}^{t}\geq j$. Hence,
$\underbrace{\lambda_{j}}_{\geq i}+\underbrace{\lambda_{i}^{t}}_{\geq
j}-\,i-j\geq i+j-i-j=0>-1$, so that $\lambda_{j}+\lambda_{i}^{t}-i-j\neq-1$.
This proves Lemma \ref{lem.conj.disjoint} in Case 1.

Let us now consider Case 2. In this case, we have $\lambda_{j}<i$. Hence, we
don't have $\lambda_{j}\geq i$. According to Lemma \ref{lem.conj.uniprop},
this shows that we don't have $\lambda_{i}^{t}\geq j$ either. Hence, we have
$\lambda_{i}^{t}<j$, so that $\lambda_{i}^{t}\leq j-1$ (since $\lambda_{i}%
^{t}$ and $j$ are integers). Now, $\underbrace{\lambda_{j}}_{<i}%
+\underbrace{\lambda_{i}^{t}}_{\leq j-1}-\,i-j<i+j-1-i-j=-1$, so that
$\lambda_{j}+\lambda_{i}^{t}-i-j\neq-1$. This proves Lemma
\ref{lem.conj.disjoint} in Case 2.

Lemma \ref{lem.conj.disjoint} has now been proved in both cases, and thus
always holds.
\end{proof}

\begin{proof}[Proof of Lemma \ref{lem.conj.tt}.]
Well-known and easy consequence of the definition of $\lambda^t$.
\end{proof}

\begin{proof}[Proof of Lemma \ref{prop.Delta.disjoint}.]
    Assume the contrary. Thus, $-1-p \in \Delta\left(\lambda^t\right)$.
    
	We have $p \in \Delta\left(\lambda\right) = \set{\lambda_i - i \ \mid \ i \geq 1}$ (by the definition of $\Delta\left(\lambda\right)$). In other words, $p = \lambda_j - j$ for some $j \geq 1$. Consider this $j$.
    
    Moreover, $-1-p \in \Delta\left(\lambda^t\right) = \set{\lambda^t_i - i \ \mid \ i \geq 1 \ldots}$ (by the definition of $\Delta\left(\lambda^t\right)$). In other words, $-1-p = \lambda^t_i - i$ for some $i \geq 1$. Consider this $i$.

    Adding the equalities $p = \lambda_j - j$ and $-1-p = \lambda^t_i - i$ together, we obtain
    \[
    p + \tup{-1-p} = \tup{\lambda_j - j} + \tup{\lambda^t_i - i}
    = \lambda_j + \lambda^t_i - i - j \neq -1
    \]
    (by Lemma~\ref{lem.conj.disjoint}). But this contradicts the obvious equality $p + \tup{-1-p} = -1$.
    Hence, our assumption was wrong, and Lemma~\ref{prop.Delta.disjoint} is proved.
\end{proof}

\begin{proof}[Proof of Lemma \ref{lem.Deltas.if-not-then-k2}.]
This is just Lemma \ref{lem.Deltas.if-not-then-k}, with the letters $\lambda$, $\mu$, $\ell_j$ and $k$ renamed as $\mu$, $\lambda$, $m_{j}$ and $i$.
\end{proof}

\begin{proof}[Proof of Lemma \ref{lem.Deltas.ybij}.]
We shall proceed in multiple steps.

\begin{enumerate}
\item The set $\left\{  i\geq1\ \mid\ m_{i}\notin\Delta\left(  \lambda\right)
\right\}  $ is finite.

[\textit{Proof:} Pick an $n\in\mathbb{N}$ that is large enough that
$\lambda=\left(  \lambda_{1},\lambda_{2},\ldots,\lambda_{n}\right)  $ (that
is, $\lambda_{n+1}=0$) and $\mu=\left(  \mu_{1},\mu_{2},\ldots,\mu_{n}\right)
$ (that is, $\mu_{n+1}=0$). (Such an $n$ exists, since $\lambda$ and $\mu$ are
partitions.) Then, Lemma
\ref{lem.Deltas.if-not-then-k2} shows that every positive integer $i$ that satisfies
$m_i \notin\Delta\left(  \lambda\right)  $ must satisfy $i\in\left[  n\right]
$. In other words, $\left\{  i\geq1\ \mid\ m_{i}%
\notin\Delta\left(  \lambda\right)  \right\}  \subseteq\left[  n\right]  $.
Hence, the set $\left\{  i\geq1\ \mid
\ m_{i}\notin\Delta\left(  \lambda\right)  \right\}  $ is finite (since
$\left[  n\right]  $ is finite).]

\item If $i$ is a positive integer satisfying $m_{i}\notin\Delta\left(
\lambda\right)  $, then $m_{i}+1+b_{i}$ is a positive integer $p$ satisfying
$\ell_{p}^{t}\notin\Delta\left(  \mu^{t}\right)  $.

[\textit{Proof:} Let $i$ be a positive integer satisfying $m_{i}\notin%
\Delta\left(  \lambda\right)  $. Lemma \ref{lem.Deltas.mi+1+bi} then
shows that $m_{i}+1+b_{i}\geq1$ and $\ell_{m_{i}+1+b_{i}}^{t}=-1-m_{i}$. From
$m_{i}+1+b_{i}\geq1$, we see that $m_{i}+1+b_{i}$ is a positive integer.
Moreover, recall that $\Delta\left(
\mu\right)  = \left\{  m_{1}, m_{2}, m_{3}, \ldots\right\} $, so that $m_{i}
\in\Delta\left(  \mu\right) $. Hence, Lemma~\ref{prop.Delta.disjoint}
(applied to $\mu$ and $m_{i}$ instead of $\lambda$ and $p$) yields that
$-1-m_{i} \notin\Delta\left(  \mu^{t} \right) $. Therefore, $\ell
_{m_{i}+1+b_{i}}^{t}=-1-m_{i}\notin\Delta\left(  \mu^{t}\right)  $. Hence,
$m_{i}+1+b_{i}$ is a positive integer $p$ satisfying $\ell_{p}^{t}\notin%
\Delta\left(  \mu^{t}\right)  $.]

\item The map
\begin{align*}
\Phi:\left\{  i\geq1\ \mid\ m_{i}\notin\Delta\left(  \lambda\right)  \right\}
&  \rightarrow\left\{  p\geq1\ \mid\ \ell_{p}^{t}\notin\Delta\left(  \mu
^{t}\right)  \right\}  ,\\
i  &  \mapsto m_{i}+1+b_{i}%
\end{align*}
is well-defined.

[\textit{Proof:} This is saying that if $i$ is a positive integer satisfying
$m_{i}\notin\Delta\left(  \lambda\right)  $, then $m_{i}+1+b_{i}$ is a
positive integer $p$ satisfying $\ell_{p}^{t}\notin\Delta\left(  \mu
^{t}\right)  $. But we just showed this in the previous step.]

\item This map $\Phi$ is furthermore injective.

[\textit{Proof:} Let $u$ and $v$ be two distinct elements of $\left\{
i\geq1\ \mid\ m_{i}\notin\Delta\left(  \lambda\right)  \right\}  $ that
satisfy $\Phi\left(  u\right)  =\Phi\left(  v\right)  $. We must show that
$u=v$.

The definition of $\Phi$ yields $\Phi\left(  u\right)  =m_{u}+1+b_{u}$.
However, we have $u\in\left\{  i\geq1\ \mid\ m_{i}\notin\Delta\left(
\lambda\right)  \right\}  $, so that $u\geq1$ and $m_{u}\notin\Delta\left(
\lambda\right)  $. Hence, Lemma \ref{lem.Deltas.mi+1+bi} (applied to
$i=u$) shows that $m_{u}+1+b_{u}\geq1$ and $\ell_{m_{u}+1+b_{u}}^{t}=-1-m_{u}%
$. In view of $\Phi\left(  u\right)  =m_{u}+1+b_{u}$, we can rewrite these
facts as $\Phi\left(  u\right)  \geq1$ and $\ell_{\Phi\left(  u\right)  }%
^{t}=-1-m_{u}$. Similarly, we find $\Phi\left(  v\right)  \geq1$ and
$\ell_{\Phi\left(  v\right)  }^{t}=-1-m_{v}$.

Now, consider the two equalities $\ell_{\Phi\left(  u\right)  }^{t}=-1-m_{u}$
and $\ell_{\Phi\left(  v\right)  }^{t}=-1-m_{v}$. Their left hand sides are
equal, since $\Phi\left(  u\right)  =\Phi\left(  v\right)  $. Thus, their
right hand sides are equal as well. In other words, $-1-m_{u}=-1-m_{v}$, so
that $m_{u}=m_{v}$.

However, Lemma~\ref{lem.l-decrease} yields $m_{1}%
>m_{2}>m_{3}>\cdots$. In particular, the numbers $m_{1},m_{2},m_{3},\ldots$ are
distinct. Hence, from $m_{u}=m_{v}$, we obtain $u=v$.

This completes the proof of the injectivity of $\Phi$.]

\item We have $\left\vert \left\{  i\geq1\ \mid\ m_{i}\notin\Delta\left(
\lambda\right)  \right\}  \right\vert \leq\left\vert \left\{  p\geq
1\ \mid\ \ell_{p}^{t}\notin\Delta\left(  \mu^{t}\right)  \right\}  \right\vert
$.

[\textit{Proof:} We have just shown that there is an injective map (namely,
$\Phi$) from the set $\left\{  i\geq1\ \mid\ m_{i}\notin\Delta\left(
\lambda\right)  \right\}  $ to the set $\left\{  p\geq1\ \mid\ \ell_{p}%
^{t}\notin\Delta\left(  \mu^{t}\right)  \right\}  $. Thus, the size of the
former set is at most as large as the size of the latter. In other words,
$\left\vert \left\{  i\geq1\ \mid\ m_{i}\notin\Delta\left(  \lambda\right)
\right\}  \right\vert \leq\left\vert \left\{  p\geq1\ \mid\ \ell_{p}^{t}%
\notin\Delta\left(  \mu^{t}\right)  \right\}  \right\vert $.]

\item The sets $\left\{  i\geq1\ \mid\ m_{i}\notin\Delta\left(  \lambda
\right)  \right\}  $ and $\left\{  p\geq1\ \mid\ \ell_{p}^{t}\notin%
\Delta\left(  \mu^{t}\right)  \right\}  $ are finite and have the same size.

[\textit{Proof:} We have just shown that
\begin{align}
\left\vert \left\{  i\geq1\ \mid\ m_{i}\notin\Delta\left(  \lambda\right)
\right\}  \right\vert  & \leq\left\vert \left\{  p\geq1\ \mid\ \ell_{p}%
^{t}\notin\Delta\left(  \mu^{t}\right)  \right\}  \right\vert \nonumber\\
& =\left\vert \left\{  i\geq1\ \mid\ \ell_{i}^{t}\notin\Delta\left(  \mu
^{t}\right)  \right\}  \right\vert \label{pf.lem.Deltas.ybij.b.ineq1}%
\end{align}
(here, we renamed the index $p$ as $i$). But we can apply the same reasoning
to the partitions $\mu^{t}$ and $\lambda^{t}$ instead of $\lambda$ and $\mu$.
As a result, we obtain%
\[
\left\vert \left\{  i\geq1\ \mid\ \ell_{i}^{t}\notin\Delta\left(  \mu\right)
\right\}  \right\vert \leq\left\vert \left\{  i\geq1\ \mid\ m_{i}^{tt}%
\notin\Delta\left(  \left(  \lambda^{t}\right)  ^{t}\right)  \right\}
\right\vert ,
\]
where we set $m_{i}^{tt}=\left(  \mu^{t}\right)  _{i}^{t}-i$ for each $i\geq
1$. This can be simplified to
\[
\left\vert \left\{  i\geq1\ \mid\ \ell_{i}^{t}\notin\Delta\left(  \mu\right)
\right\}  \right\vert \leq\left\vert \left\{  i\geq1\ \mid\ m_{i}\notin%
\Delta\left(  \lambda\right)  \right\}  \right\vert
\]
(since Lemma \ref{lem.conj.tt} yields $\left(  \lambda^{t}\right)
^{t}=\lambda$ and $\left(  \mu^{t}\right)  ^{t}=\mu$, so that $m_{i}%
^{tt}=\underbrace{\left(  \mu^{t}\right)  _{i}^{t}}_{\substack{=\mu
_{i}\\\text{(since }\left(  \mu^{t}\right)  ^{t}=\mu\text{)}}}-i=\mu
_{i}-i=m_{i}$ for every $i\geq1$). Combining this inequality with the
inequality (\ref{pf.lem.Deltas.ybij.b.ineq1}), we obtain%
\begin{align*}
\left\vert \left\{  i\geq1\ \mid\ m_{i}\notin\Delta\left(  \lambda\right)
\right\}  \right\vert  & =\left\vert \left\{  i\geq1\ \mid\ \ell_{i}^{t}%
\notin\Delta\left(  \mu^{t}\right)  \right\}  \right\vert \\
& =\left\vert \left\{  p\geq1\ \mid\ \ell_{p}^{t}\notin\Delta\left(  \mu
^{t}\right)  \right\}  \right\vert .
\end{align*}
Thus, the sets $\left\{  i\geq1\ \mid\ m_{i}\notin\Delta\left(  \lambda
\right)  \right\}  $ and $\left\{  p\geq1\ \mid\ \ell_{p}^{t}\notin%
\Delta\left(  \mu^{t}\right)  \right\}  $ have the same size. Since the first
of them is finite (by Step 1 above), we conclude that they are both are finite.]

\item The map $\Phi$ is bijective.

[\textit{Proof:} The sets $\left\{  i\geq1\ \mid\ m_{i}\notin\Delta\left(
\lambda\right)  \right\}  $ and $\left\{  p\geq1\ \mid\ \ell_{p}^{t}%
\notin\Delta\left(  \mu^{t}\right)  \right\}  $ are finite and have the same
size (as we have just seen), and the map $\Phi$ between these sets is
injective (by Step 4 above). It is known that any injective map between two
finite sets of the same size must be bijective. Applying this to the map
$\Phi$, we thus conclude that $\Phi$ is bijective.]

\end{enumerate}

Thus, the map $\Phi$ is a bijection. This proves Lemma \ref{lem.Deltas.ybij}
(since the map $\Phi$ is exactly the map described in Lemma \ref{lem.Deltas.ybij}).
\end{proof}

\begin{proof}[Proof of Lemma \ref{lem.Deltas.y}.]
Lemma \ref{lem.Deltas.if-not-then-k2} shows that every positive integer $i$ that satisfies
$m_{i}\notin\Delta\left(  \lambda\right)  $ must satisfy $i\in\left[
n\right]  $. Hence, the summation sign $\sum_{\substack{i\geq1;\\m_{i}%
\notin\Delta\left(  \lambda\right)  }}$ is equivalent to $\sum_{\substack{i\in
\left[  n\right]  ;\\m_{i}\notin\Delta\left(  \lambda\right)  }}$. Therefore,%
\begin{equation}
\sum_{\substack{i\geq1;\\m_{i}\notin\Delta\left(  \lambda\right)  }%
}y_{m_{i}+1+b_{i}}=\sum_{\substack{i\in\left[  n\right]  ;\\m_{i}\notin%
\Delta\left(  \lambda\right)  }}y_{m_{i}+1+b_{i}}.\label{pf.lem.Deltas.y}%
\end{equation}

On the other hand, Lemma \ref{lem.Deltas.ybij} shows that the map%
\begin{align*}
\left\{  i\geq1\ \mid\ m_{i}\notin\Delta\left(  \lambda\right)  \right\}   &
\rightarrow\left\{  p\geq1\ \mid\ \ell_{p}^{t}\notin\Delta\left(  \mu
^{t}\right)  \right\}  ,\\
i &  \mapsto m_{i}+1+b_{i}%
\end{align*}
is a bijection. Hence, we can substitute $k$ for $m_{i}+1+b_{i}$ in the sum
$\sum_{\substack{i\geq1;\\m_{i}\notin\Delta\left(  \lambda\right)  }%
}y_{m_{i}+1+b_{i}}$. We thus obtain%
\[
\sum_{\substack{i\geq1;\\m_{i}\notin\Delta\left(  \lambda\right)  }%
}y_{m_{i}+1+b_{i}}
=\sum_{\substack{k\geq1;\\\ell_{k}^{t}\notin%
\Delta\left(  \mu^{t}\right)  }}y_{k} .
\]
Comparing this with
(\ref{pf.lem.Deltas.y}), we obtain%
\[
\sum_{\substack{i\in\left[  n\right]  ;\\m_{i}\notin\Delta\left(
\lambda\right)  }}y_{m_{i}+1+b_{i}}=\sum_{\substack{k\geq1;\\\ell_{k}%
^{t}\notin\Delta\left(  \mu^{t}\right)  }}y_{k}.
\]
This proves Lemma \ref{lem.Deltas.y}.
\end{proof}

\begin{proof}[Proof of Theorem \ref{mainKonvalinka}.]
If  $k$ is a positive integer that
satisfies $\ell_{k}\notin\Delta\left(  \mu\right)  $, then $k\in\left[
n\right]  $ (by Lemma \ref{lem.Deltas.if-not-then-k}). Thus, the summation
sign \textquotedblleft$\sum_{\substack{k\in\left[  n\right]  ;\\\ell_{k}%
\notin\Delta\left(  \mu\right)  }}$\textquotedblright\ is equivalent to the
summation sign \textquotedblleft$\sum_{\substack{k\geq1;\\\ell_{k}\notin%
\Delta\left(  \mu\right)  }}$\textquotedblright. Hence,%
\begin{equation}
\sum_{\substack{k\in\left[  n\right]  ;\\\ell_{k}\notin\Delta\left(
\mu\right)  }}x_{k}=\sum_{\substack{k\geq1;\\\ell_{k}\notin\Delta\left(
\mu\right)  }}x_{k}.\label{pf.mainKonvalinka.1}%
\end{equation}
Lemma \ref{lem.Deltas.y} yields%
\begin{equation}
\sum_{\substack{i\in\left[  n\right]  ;\\m_{i}\notin\Delta\left(
\lambda\right)  }}y_{m_{i}+1+b_{i}}=\sum_{\substack{k\geq1;\\\ell_{k}%
^{t}\notin\Delta\left(  \mu^{t}\right)  }}y_{k}.\label{pf.mainKonvalinka.2}%
\end{equation}

Now, Lemma \ref{lem.konvalinka-bi} yields%
\[
\left(  \sum_{\substack{k\in\left[  n\right]  ;\\\ell_{k}\notin\Delta\left(
\mu\right)  }}x_{k}+\sum_{\substack{i\in\left[  n\right]  ;\\m_{i}\notin%
\Delta\left(  \lambda\right)  }}y_{m_{i}+1+b_{i}}\right)  \mathbf{s}_{\lambda
}\left[  \mu\right]  =\sum_{\mu\lessdot\nu\subseteq\lambda}\mathbf{s}%
_{\lambda}\left[  \nu\right]  .
\]
In view of (\ref{pf.mainKonvalinka.1}) and (\ref{pf.mainKonvalinka.2}), we can
rewrite this as%
\[
\left(  \sum_{\substack{k\geq1;\\\ell_{k}\notin\Delta\left(  \mu\right)
}}x_{k}+\sum_{\substack{k\geq1;\\\ell_{k}^{t}\notin\Delta\left(  \mu
^{t}\right)  }}y_{k}\right)  \mathbf{s}_{\lambda}\left[  \mu\right]
=\sum_{\mu\lessdot\nu\subseteq\lambda}\mathbf{s}_{\lambda}\left[  \nu\right]
.
\]
This proves Theorem \ref{mainKonvalinka}.
\end{proof}

\section{\label{sec.odds}Appendix: Odds and ends}

In this short section, we shall discuss how a few of the auxiliary results
shown above can be extended or generalized.

\subsection{To Section \ref{sec.jt}}

The most general result in Section \ref{sec.jt} is Theorem
\ref{thm.flagJT.gen}. We can generalize it further by replacing the partition
$\mu$ by a skew partition $\mu/\nu$ and adding a second flagging $\mathbf{a}$.
Here are the relevant definitions:

\begin{definition}
Let $\mu/\nu$ be a skew partition.

\begin{enumerate}
\item[\textbf{(a)}] A \emph{semistandard tableau} of shape $\mu/\nu$ is
defined just like a semistandard tableau of shape $\mu$ (Definition
\ref{SSYT_def}), but with each \textquotedblleft$\mu$\textquotedblright%
\ replaced by \textquotedblleft$\mu/\nu$\textquotedblright.

\item[\textbf{(b)}] Let $\mathbf{a}=\left(  a_{1},a_{2},a_{3},\ldots\right)  $
and $\mathbf{b}=\left(  b_{1},b_{2},b_{3},\ldots\right)  $ be two flaggings (i.e., sequences of positive integers). A
semistandard tableau $T$ of shape $\mu/\nu$ is said to be $\mathbf{b}%
/\mathbf{a}$\emph{-flagged} if and only if it satisfies
\[
a_{i}\leq T\left(  i,j\right)  \leq b_{i}\qquad\text{for all }\left(
i,j\right)  \in Y\left(  \mu/\nu\right)
\]
(that is, all entries in row $i$ are $\geq a_{i}$ and $\leq b_{i}$).

We let $\operatorname{FSSYT}\left(  \mu/\nu,\ \mathbf{b}/\mathbf{a}\right)  $
be the set of all $\mathbf{b}/\mathbf{a}$-flagged semistandard tableaux of
shape $\mu/\nu$.
\end{enumerate}
\end{definition}

Now we can generalize Theorem \ref{thm.flagJT.gen} as follows:

\begin{theorem}
\label{thm.flagJT.gen-skew}
Let $R$ be a commutative ring. Let $u_{i,j}$ be an
element of $R$ for each pair $\left(  i,j\right)  \in\mathbb{Z}\times
\mathbb{Z}$. For each $a,b\in\mathbb{N}$ and $r,q,d\in\mathbb{Z}$, we define
an element $h_{a,\ b;\ r,\ q}\left[  d\right]  \in R$ by%
\[
h_{a,\ b;\ r,\ q}\left[  d\right]  :=\sum_{\substack{\left(  i_{r+1}%
,i_{r+2},\ldots,i_{q}\right)  \in\left[  a,b\right]  ^{q-r};\\i_{r+1}\leq
i_{r+2}\leq\cdots\leq i_{q}}}\ \ \prod_{j=r+1}^{q}u_{i_{j},\ j-d}%
\]
(where $\left[  a,b\right]  :=\left\{  a,a+1,a+2,\ldots,b\right\}  $). This
sum is understood to be $0$ if $q<r$, and to be $1$ if $q=r$.

Let $\mu=\left(  \mu_{1},\mu_{2},\ldots,\mu_{n}\right)  $ and $\nu=\left(
\nu_{1},\nu_{2},\ldots,\nu_{n}\right)  $ be two partitions such that
$\nu\subseteq\mu$. Let $\mathbf{a}=\left(  a_{1},a_{2},a_{3},\ldots\right)  $
and $\mathbf{b}=\left(  b_{1},b_{2},b_{3},\ldots\right)  $ be two weakly
increasing flaggings. Then,
\begin{equation}
\sum_{T\in\operatorname{FSSYT}\left(  \mu/\nu,\ \mathbf{b}/\mathbf{a}\right)
}\ \ \prod_{\left(  i,j\right)  \in Y\left(  \mu/\nu\right)  }u_{T\left(
i,j\right)  ,\ j-i}=\det\left(  h_{a_{j},\ b_{i};\ \nu_{j},\ \mu_{i}%
-i+j}\left[  j\right]  \right)  _{i,j\in\left[  n\right]  }.\nonumber
\end{equation}

\end{theorem}

The proof of Theorem \ref{thm.flagJT.gen-skew} is mostly analogous to our
proof of Theorem \ref{thm.flagJT.gen}, with some occasional complications due
to the $a_{i}\leq T\left(  i,j\right)  $ conditions and due to the presence of
$\nu$. Some changes need to be made in Definition \ref{def.jt.sig-array}: In
the definition of a legitimate permutation $\sigma\in S_{n}$, the inequality
$\mu_{\sigma\left(  i\right)  }-\sigma\left(  i\right)  +i\geq0$ must be
replaced by $\mu_{\sigma\left(  i\right)  }-\sigma\left(  i\right)  +i\geq
\nu_{i}$. In the definition of $P\left(  \sigma\right)  $, the inequality
$j\leq\mu_{\sigma\left(  i\right)  }-\sigma\left(  i\right)  +i$ must be
replaced by $\nu_{i}<j\leq\mu_{\sigma\left(  i\right)  }-\sigma\left(
i\right)  +i$. The $\mathbf{b}$-flagged $\sigma$-arrays should be replaced by
\textquotedblleft$\mathbf{b}/\mathbf{a}$-flagged $\sigma$%
-arrays\textquotedblright, which are defined similarly but require that every
entry of $T$ in the $i$-th row is $\leq b_{\sigma\left(  i\right)  }$ and
$\geq a_{i}$ (note the different subscripts!). The definition of an outer
failure (Definition \ref{def.flagJT.fail} \textbf{(a)}) should be adapted by
replacing \textquotedblleft$\left(  i-1,j\right)  \notin P\left(
\sigma\right)  $\textquotedblright\ by \textquotedblleft$\left(  i-1,j\right)
\notin P\left(  \sigma\right)  \cup Y\left(  \nu\right)  $\textquotedblright%
\ (so that $\left(  i,j\right)  $ does \textbf{not} count as a failure if
$\nu_{i-1}\geq j$). We leave it to the reader to check that the argument (most
importantly, the proof of Lemma \ref{lem.flagJT.flip1}) survives all these
changes (and some others that are forced by these).

\subsection{To Section \ref{sec.det}}

As we said, Lemma \ref{determinant.sum} is just the $r=1$ case of
\cite[\S 319]{MuiMet60}. Here is the general case:

\begin{lemma}
\label{determinant.sumr}

Let $P$ and $Q$ be two $n\times n$-matrices over some commutative ring.

For each subset $K$ of $\left[  n\right]  $, we let
$P\underset{\operatorname*{row}}{\overset{K}{\leftarrow}}Q$ denote the
$n\times n$-matrix that is obtained from $P$ by replacing the $k$-th row by
the $k$-th row of $Q$ for all $k\in K$. (That is, the $\left(  i,j\right)
$-th entry of this matrix is $%
\begin{cases}
P_{i,j}, & \text{ if }i\notin K;\\
Q_{i,j}, & \text{ if }i\in K
\end{cases}
$ for every $i,j\in\left[  n\right]  $.)

For each subset $K$ of $\left[  n\right]  $, we let
$P\underset{\operatorname*{col}}{\overset{K}{\leftarrow}}Q$ denote the
$n\times n$-matrix that is obtained from $P$ by replacing the $k$-th column by
the $k$-th column of $Q$ for all $k\in K$. (That is, the $\left(  i,j\right)
$-th entry of this matrix is $%
\begin{cases}
P_{i,j}, & \text{ if }j\notin K;\\
Q_{i,j}, & \text{ if }j\in K
\end{cases}
$ for every $i,j\in\left[  n\right]  $.)

Let $r\in\left\{  0,1,\ldots,n\right\}  $. Then,
\[
\sum_{\substack{K\subseteq\left[  n\right]  ;\\\left\vert K\right\vert
=r}}\det\left(  P\underset{\operatorname*{row}}{\overset{K}{\leftarrow}%
}Q\right)  =\sum_{\substack{K\subseteq\left[  n\right]  ;\\\left\vert
K\right\vert =r}}\det\left(  P\underset{\operatorname*{col}%
}{\overset{K}{\leftarrow}}Q\right)  .
\]

\end{lemma}

This can be proved in a similar way as we proved Lemma \ref{determinant.sum},
but using Laplace expansion along multiple rows/columns (see, e.g.,
\cite[Theorem 6.156]{detnotes}).
\medskip

Finally, we note that Lemma \ref{lem.det-uc} has an analogue in which $p_{i}$
is replaced by $p_{j}$:

\begin{lemma}
\label{lem.det-u}
Let $n$ be a positive integer. Let $R$ be a commutative ring.

Let $u_{i,j}$ be an element of $R$ for each $i\in\left[  n\right]  $ and each
$j\in\left[  n+1\right]  $. Let $p_{1},p_{2},\ldots,p_{n}$ be $n$ further
elements of $R$.

Then,%
\begin{align*}
&  \sum_{k=1}^{n}\det\left(  u_{i,j+\left[  k=i\right]  }-p_{j}u_{i,j}\left[
k=i\right]  \right)  _{i,j\in\left[  n\right]  }\\
&  =\det\left(  u_{i,j+\left[  n=j\right]  }\right)  _{i,j\in\left[  n\right]
}-\left(  \sum_{k=1}^{n}p_{k}\right)  \det\left(  u_{i,j}\right)
_{i,j\in\left[  n\right]  }.
\end{align*}

\end{lemma}



\begin{proof}[Proof of Lemma \ref{lem.det-u}.]
Proceed exactly as in the above proof of Lemma \ref{lem.det-uc}, replacing the $p_j$ by $p_i$. The only nontrivial change is replacing the computation \eqref{pf.lem.det-u.term1} by
\begin{align}
& \sum_{k=1}^{n}\ \ \sum_{\ell=1}^{n}\left(  -1\right)  ^{k+\ell}p_{\ell
}u_{k,\ell}\det\left(  U_{\sim k,\sim\ell}\right)  \nonumber\\
& =\sum_{\ell=1}^{n}p_{\ell}\sum_{k=1}^{n}\left(  -1\right)  ^{k+\ell
}\underbrace{u_{k,\ell}}_{\substack{=U_{k,\ell}\\\text{(by the definition of
}U\text{)}}}\det\left(  U_{\sim k,\sim\ell}\right)  \nonumber\\
& =\sum_{\ell=1}^{n}p_{\ell}\underbrace{\sum_{k=1}^{n}\left(  -1\right)
^{k+\ell}U_{k,\ell}\det\left(  U_{\sim k,\sim\ell}\right)  }_{\substack{=\det
U\\\text{(by (\ref{pf.lem.det-u.lapc}), applied to }A=U\text{)}}}
=\sum_{\ell=1}^{n}p_{\ell}\det U=\sum_{k=1}^{n}p_{k}\det U
\nonumber\\
&=\left(
\sum_{k=1}^{n}p_{k}\right)  \det\underbrace{U}_{=\left(  u_{i,j}\right)
_{i,j\in\left[  n\right]  }}
=\left(  \sum_{k=1}^{n}p_{k}\right)  \det\left(  u_{i,j}\right)
_{i,j\in\left[  n\right]  }.
\nonumber
\end{align}
\end{proof}

\subsection{To Section \ref{sec.konva-forreal}}

Finally, we note that Lemma \ref{prop.Delta.disjoint} has a converse:

\begin{proposition}
\label{prop.Delta.disjoint-conv} Let $\lambda$ be any partition. Let
$p\in\mathbb{Z}$. Then, $p\in\Delta\left(  \lambda\right)  $ if and only if
$-1-p\notin\Delta\left(  \lambda^{t}\right)  $.
\end{proposition}

\begin{proof}
[Proof of Proposition \ref{prop.Delta.disjoint-conv}.]The \textquotedblleft
only if\textquotedblright\ direction was Lemma \ref{prop.Delta.disjoint}. It
remains to prove the \textquotedblleft if\textquotedblright\ direction. Thus,
we assume that $-1-p\notin\Delta\left(  \lambda^{t}\right)  $. Our goal is to
prove that $p\in\Delta\left(  \lambda\right)  $.

We will use Convention \ref{conv.further.convs}. Set $\lambda_{0}:=\infty$ and $\ell_{0}:=\infty$.
Thus, the equality $\ell_{i}=\lambda_{i}-i$ holds not only for all $i\geq1$
(for which it follows from Convention \ref{conv.further.convs}), but also for
all $i=0$. In other words, $\ell_{i}=\lambda_{i}-i$ for each $i\in\mathbb{N}$.

Also, $\ell_{0}>\ell_{1}>\ell_{2}>\ell_{3}>\cdots$ (this is proved as in the
proof of Lemma \ref{lem.ER-bb} above). Hence, there are only finitely many
$i\in\mathbb{N}$ that satisfy $\ell_{i}\geq p$. Moreover, there exists at
least one such $i$ (namely, $i=0$, since $\ell_{0}=\infty\geq p$). Consider
the \textbf{largest} such $i$. Then, $\ell_{i}\geq p$ but $\ell_{i+1}<p$.

We shall show that $\ell_{i}=p$.

Indeed, assume the contrary. Thus, $\ell_{i}\neq p$. Combining this with
$\ell_{i}\geq p$, we obtain $\ell_{i}>p$. In other words, $\lambda_{i}-i>p$
(since $\ell_{i}=\lambda_{i}-i$). In other words, $\lambda_{i}>i+p$. Hence,
$\lambda_{i}\geq i+p+1$ (since $\lambda_{i}$ and $i+p$ are integers or
$\infty$).

Also, we have $\ell_{i+1}=\lambda_{i+1}-\left(  i+1\right)  =\lambda
_{i+1}-i-1$ and thus $\lambda_{i+1}-i-1=\ell_{i+1}<p$. In other words,
$\lambda_{i+1}-i<p+1$. Since $\lambda_{i+1}-i$ and $p+1$ are integers, this
entails $\lambda_{i+1}-i\leq\left(  p+1\right)  -1=p$. Therefore,
$\lambda_{i+1}\leq i+p<i+p+1$. Thus, $i+p+1>\lambda_{i+1}\geq0$. This shows
that $i+p+1$ is a positive integer.

Thus, Lemma \ref{lem.conj.uniprop} (applied to $i+p+1$ and $i+1$ instead of
$i$ and $j$) yields the logical equivalence
\[
\left(  \lambda_{i+p+1}^{t}\geq i+1\right)  \ \Longleftrightarrow\ \left(
\lambda_{i+1}\geq i+p+1\right)  .
\]
Since the statement $\left(  \lambda_{i+1}\geq i+p+1\right)  $ is false
(because $\lambda_{i+1}<i+p+1$), we thus conclude that the statement $\left(
\lambda_{i+p+1}^{t}\geq i+1\right)  $ is also false. In other words,
$\lambda_{i+p+1}^{t}<i+1$. Hence, $\lambda_{i+p+1}^{t}\leq i$ (since both
sides of the inequality are integers).

Moreover, Lemma \ref{lem.conj.uniprop} (applied to $i+p+1$ and $i$ instead of
$i$ and $j$) yields the logical equivalence\footnote{To be precise, this
argument only works when $i>0$. However, in the remaining case, the conclusion
($\lambda_{i+p+1}^{t}\geq i$) is obvious anyway (since $\lambda_{i+p+1}%
^{t}\geq0$), so we don't need this argument.}%
\[
\left(  \lambda_{i+p+1}^{t}\geq i\right)  \ \Longleftrightarrow\ \left(
\lambda_{i}\geq i+p+1\right)  .
\]
Since the statement $\lambda_{i}\geq i+p+1$ holds, we thus conclude that
$\lambda_{i+p+1}^{t}\geq i$ holds as well.

Combining $\lambda_{i+p+1}^{t}\geq i$ with $\lambda_{i+p+1}^{t}\leq i$, we
find $\lambda_{i+p+1}^{t}=i$. Now, the definition of $\ell_{i+p+1}^{t}$ yields%
\[
\ell_{i+p+1}^{t}=\underbrace{\lambda_{i+p+1}^{t}}_{=i}-\left(  i+p+1\right)
=i-\left(  i+p+1\right)  =-1-p.
\]
Hence,%
\begin{align*}
-1-p  & =\ell_{i+p+1}^{t}\in\left\{  \ell_{1}^{t},\ell_{2}^{t},\ell_{3}%
^{t},\ldots\right\}  \ \ \ \ \ \ \ \ \ \ \left(  \text{since }i+p+1\text{ is a
positive integer}\right)  \\
& =\Delta\left(  \lambda^{t}\right)  ,
\end{align*}
which contradicts $-1-p\notin\Delta\left(  \lambda^{t}\right)  $. This
contradiction shows that our assumption was false. Hence, $\ell_{i}=p$ is proved.

Now, $i\neq0$ (since $\ell_{i}=p\neq\infty=\ell_{0}$) and thus $i\geq1$.
Hence, $\ell_{i}\in\left\{  \ell_{1},\ell_{2},\ell_{3},\ldots\right\}  $.
Thus, from $\ell_{i}=p$, we obtain%
\[
p=\ell_{i}\in\left\{  \ell_{1},\ell_{2},\ell_{3},\ldots\right\}
=\Delta\left(  \lambda\right)  ,
\]
which is precisely what we desired to prove. Thus, the proof of Proposition
\ref{prop.Delta.disjoint-conv} is complete.
\end{proof}

\section{\label{sec.gv-apx}Appendix: Deriving Proposition \ref{prop.flagJT.f} from the literature}

In this short appendix, we shall briefly outline how Proposition \ref{prop.flagJT.f} can be derived
from some known results.

\subsection{\label{subsec.gv-apx.1}Deriving Proposition \ref{prop.flagJT.f} as a consequence of Gessel--Viennot}

First, we shall show how Proposition \ref{prop.flagJT.f} can be obtained
from \cite[Theorem 3]{GesVie89}.

Let us use the notations of \cite[Theorem 3]{GesVie89}, but set $b_{i}:=1$ and
$\mu:=\varnothing$ and $k:=n$.
Furthermore, we define a labelling set
$L:=\mathbb{Z}^{2}$ (consisting of all possible boxes), a relabeling function
$\mathbf{f}$ given by $f_{r}\left(  s\right)  :=\left(  r,s\right)  \in L$,
and a weight function $w$ given by $w\left(  r,s\right)  =x_{s}+y_{s+r}$ for
each $\left(  r,s\right)  \in L$ (where we set $x_{i}=y_{i}=0$ for all
$i\leq0$). Then, $w\left(  f_{r}\left(  s\right)  \right)  =x_{s}+y_{s+r}$ for
any $r$ and $s$. Now, \cite[Theorem 3]{GesVie89} says that
\begin{align}
& \left(  \text{the sum of the weights of }\mathbf{f}\left(  T\right)  \text{
over all tableaux }T\text{ of shape }\lambda\right.  \nonumber\\
& \ \ \ \ \ \ \ \ \ \ \left.  \text{satisfying }1\leq T_{i,j}\leq d_{i}\text{
for all }\left(  i,j\right)  \in Y\left(  \lambda\right)  \right)  \nonumber\\
& =\det\left(  H_{\mathbf{f}}\left(  -i+1,\ 1,\ \lambda_{j}-j+1,\ d_{j}%
\right)  \right)  _{i,j\in\left[  n\right]  }.
\label{eq.from-gessel.1}
\end{align}
However, the left hand side of this equality is easily seen to be exactly our%
\[
\sum_{T\in\operatorname{FSSYT}\tup{\mu,\mathbf{b}}}\ \ \prod_{\left(  i,j\right)
\in Y\left(  \mu\right)  }\left(  x_{T\left(  i,j\right)  }+y_{T\left(
i,j\right)  +j-i}\right)  ,
\]
if we rename $\lambda$ and $d_{i}$ as $\mu$ and $b_{i}$. On the other hand,
the right hand side of \eqref{eq.from-gessel.1} (after the same renaming) becomes our%
\[
\det\left(  h\left(  \mu_{j}-j+i,\ \ b_{j},\ \ 1-i\right)  \right)
_{i,j\in\lbrack n]},
\]
because every $i,j\in\left[  n\right]  $ satisfy
\begin{align*}
& H_{\mathbf{f}}\left(  -i+1,\ 1,\ \lambda_{j}-j+1,\ d_{j}\right)  \\
& =\sum_{1\leq n_{-i+1}\leq n_{-i+2}\leq\cdots\leq n_{\lambda_{j}-j}\leq
d_{j}}w\left(  f_{-i+1}\left(  n_{-i+1}\right)  \right)  w\left(
f_{-i+2}\left(  n_{-i+2}\right)  \right)  \cdots w\left(  f_{\lambda_{j}%
-j}\left(  n_{\lambda_{j}-j}\right)  \right)  \\
& =\sum_{1\leq n_{-i+1}\leq n_{-i+2}\leq\cdots\leq n_{\lambda_{j}-j}\leq
d_{j}}\ \ \prod_{g=-i+1}^{\lambda_{j}-j}\underbrace{w\left(  f_{g}\left(
n_{g}\right)  \right)  }_{=x_{n_{g}}+y_{n_{g}+g}}\\
& =\sum_{1\leq n_{-i+1}\leq n_{-i+2}\leq\cdots\leq n_{\lambda_{j}-j}\leq
d_{j}}\ \ \prod_{g=-i+1}^{\lambda_{j}-j}\left(  x_{n_{g}}+y_{n_{g}+g}\right)
\\
& =\sum_{1\leq m_{1}\leq m_{2}\leq\cdots\leq m_{\lambda_{j}-j+i}\leq d_{j}%
}\ \ \prod_{g=1}^{\lambda_{j}-j+i}\left(  x_{m_{g}}+y_{m_{g}+g-i}\right)  \\
& \ \ \ \ \ \ \ \ \ \ \ \ \ \ \ \ \ \ \ \ \left(  \text{here, we have shifted
the indices by setting }m_{g}=n_{g-i}\right)  \\
& =\sum_{\substack{\left(  m_{1},m_{2},\ldots,m_{\lambda_{j}-j+i}\right)
\in\left[  d_{j}\right]  ^{\lambda_{j}-j+i};\\m_{1}\leq m_{2}\leq\cdots\leq
m_{\lambda_{j}-j+i}}}\ \ \prod_{g=1}^{\lambda_{j}-j+i}\left(  x_{m_{g}%
}+y_{m_{g}+g-i}\right)  \\
& =h\left(  \lambda_{j}-j+i,\ \ d_{j},\ \ 1-i\right)
\ \ \ \ \ \ \ \ \ \ \left(  \text{using our Definition \ref{defh}}\right)  .
\end{align*}

Hence, the equality \eqref{eq.from-gessel.1} becomes
\begin{align*}
\sum_{T\in\operatorname{FSSYT}\tup{\mu,\mathbf{b}}}\ \ \prod_{\left(
i,j\right)  \in Y\left(  \mu\right)  }\left(  x_{T\left(  i,j\right)
}+y_{T\left(  i,j\right)  +j-i}\right) 
& =\det\left(  h\left(  \mu_{j}-j+i,\ \ b_{j},\ \ 1-i\right)  \right)
_{i,j\in\lbrack n]}\\
& =\det\left(  h\left(  \mu_{i}-i+j,\ \ b_{i},\ \ 1-j\right)  \right)
_{i,j\in\lbrack n]}%
\end{align*}
(since $\det\left(  A^{T}\right)  =\det A$ for any square matrix $A$).
Proposition \ref{prop.flagJT.f} thus follows.

\subsection{\label{subsec.gv-apx.2}Deriving Proposition \ref{prop.flagJT.f} from Chen--Li--Louck}

In their study of flagged double Schur functions, Chen, Li and Louck have obtained a determinantal formula \cite[second bullet point after Theorem 4.2]{ChLiLo02} that can, too, be used to prove Proposition \ref{prop.flagJT.f}.
We shall recall this formula, and then very tersely outline the derivation.

For any $a, b, c \in \ZZ$, we define an element $h_a\tup{X_b / Y_c}$ of $R$ by
\[
h_a\tup{X_b / Y_c} :=
\tup{\text{the coefficient of } t^a \text{ in the power series }
\dfrac{\prod_{i=1}^c \tup{1 + y_it}}{\prod_{j=1}^b \tup{1 - x_jt}}
\in R\ive{\ive{t}}}
\]
(understanding empty products as $1$ as usual).
Then, \cite[second bullet point after Theorem 4.2]{ChLiLo02} (translated into our language, and upon the substitution $y_j \mapsto -y_j$) says that
\begin{align*}
	\sum_{T \in \FSSYT(\mu, \bb)}\ \ \prod_{\tup{i,j}\in Y\tup{\mu}} \tup{x_{T\tup{i,j}} + y_{T\tup{i,j}+j-i}} =
    \det\tup{h_{\mu_i - i + j}\tup{X_{b_i} / Y_{\lambda_i + b_i - i}}}_{i, j \in \ive{n}}
\end{align*}
(with the notations of Theorem~\ref{prop.flagJT.f}).
In order to derive Proposition \ref{prop.flagJT.f} from this, it suffices to show that
\begin{align}
h\tup{a,b,c} = h_a \tup{X_b / Y_{a+b+c-1}}
\label{eq.subsec.gv-apx.2.habc=}
\end{align}
for any $a, b, c \in \ZZ$ satisfying $c \leq 0$.
This equality \eqref{eq.subsec.gv-apx.2.habc=} can, in turn, be derived again from \cite[second bullet point after Theorem 4.2]{ChLiLo02} (applied to $n = 1$, $\lambda_1 = a$ and $b_1 = b$). (To be more precise, this proves \eqref{eq.subsec.gv-apx.2.habc=} for $c = 0$, but then the general case follows by shifting the $y$-variables.)

\end{document}